\documentclass[fleqn,a4paper]{cas-sc}
\usepackage{graphicx,amsmath,amssymb,cite,epsfig,float}
\usepackage{algorithm}
\usepackage{enumitem,comment}

\DeclareSymbolFontAlphabet{\amsmathbb}{AMSb}%
\usepackage{tikz}
\usetikzlibrary{decorations.pathmorphing,patterns}

\usepackage{multirow}
\usepackage{array}

\usepackage{subfigure,booktabs,acronym}

\usepackage{siunitx}

\DeclareSymbolFontAlphabet{\amsmathbb}{AMSb}

\definecolor{lightblue}{rgb}{0.22,0.45,0.70}

\hypersetup{
    colorlinks=true,
    linkcolor=lightblue,
    citecolor=lightblue,
    urlcolor=lightblue
}

\newcommand\qan{\quad\hbox{and}\quad}

\newcommand\cero{\boldsymbol{0}}
\newcommand{\norm}[1]{\left\|#1\right\|}

\newcommand\vdiv{\mathop{\mathrm{div}}\nolimits}

\newcommand\bcurl{\mathop{\mathbf{curl}}\nolimits}
\newcommand\bdiv{\mathop{\mathbf{div}}\nolimits}

\newcommand\tr{\mathop{\mathrm{tr}}\nolimits}
\renewcommand{\div}{\operatorname*{div}}

\newcommand{\beps}{\boldsymbol{\varepsilon}}
\newcommand\vrot{\mathop{\mathrm{rot}}\nolimits}

\newcommand{\bF}{\mathbf{F}}
\newcommand{\bS}{\mathbf{S}}
\newcommand{\bB}{\mathbf{B}}

\newcommand{\bC}{\mathbf{C}}
\newcommand{\bbS}{\amsmathbb{S}}
\newcommand{\bbR}{\amsmathbb{R}}
\newcommand\RR{\bbR}

\newcommand{\Om}{\Omega}

\newcommand{\bu}{\boldsymbol{u}}

\newcommand{\bn}{\boldsymbol{n}}

\newcommand{\bv}{\boldsymbol{v}}
\newcommand{\bw}{\boldsymbol{w}}
\newcommand{\bx}{\boldsymbol{x}}

\newcommand{\bz}{\boldsymbol{z}}

\newcommand{\bp}{\mathbf{p}}
\newcommand{\bm}{\mathbf{m}}
\newcommand{\ff}{\boldsymbol{f}}

\newcommand{\bnabla}{\boldsymbol{\nabla}}
\newcommand{\bkappa}{\boldsymbol{\kappa}}

\newcommand{\bxi}{\boldsymbol{\xi}}
\newcommand{\bzeta}{\boldsymbol{\zeta}}

\newcommand{\bPi}{\boldsymbol{\Pi}}

\newcommand\cA{\mathcal{A}}
\newcommand\cB{\mathcal{B}}

\newcommand\cC{\mathcal{C}}

\newcommand\cG{\mathcal{G}}
\newcommand\cT{\mathcal{T}}

\newcommand\cF{\mathcal{F}}

\newcommand\cN{\mathcal{N}}

\newcommand\cJ{\mathcal{J}}

\newcommand\bP{\mathbf{P}}
\newcommand\bG{\mathbf{G}}
\newcommand\bM{\mathbf{M}}
\newcommand\bU{\mathbf{U}}
\newcommand\bQ{\mathbf{Q}}

\newcommand\bV{\mathbf{V}}
\newcommand{\bW}{\mathbf{W}}

\newcommand\rP{\mathrm{P}}

\newcommand\rM{\mathrm{M}}
\newcommand\rQ{\mathrm{Q}}
\newcommand\rH{\mathrm{H}}
\newcommand\bbV{\amsmathbb{V}}
\newcommand\bbH{\amsmathbb{H}}
\newcommand\bbP{\amsmathbb{P}}

\newcommand\bbM{\amsmathbb{M}}
\newcommand\bH{\mathbf{H}}
\newcommand\bL{\mathbf{L}}
\newcommand\bJ{\mathbf{J}}
\newcommand\rL{\mathrm{L}}
\newcommand\rW{\mathrm{W}}
\newcommand\rF{\mathrm{F}}
\newcommand\bbL{\amsmathbb{L}}

\newcommand\bbPi{\mathbb{\Pi}}
\newcommand\bbI{\amsmathbb{I}}
\newcommand\bbF{\amsmathbb{F}}

\newcommand\bOne{\mathbf{1}}
\newcommand\bbOne{\mathbb{1}}

\newcommand\bsigma{\boldsymbol{\sigma}}
\newcommand\btau{\boldsymbol{\tau}}

\numberwithin{equation}{section}
\numberwithin{figure}{section}
\newtheorem{remark}{Remark}[section]
\newtheorem{theorem}{Theorem}[section]
\newtheorem{lemma}[theorem]{Lemma}
\newtheorem{proposition}[theorem]{Proposition}

\newenvironment{proof}{\noindent{\it Proof.}}{\hfill$\square$}
\allowdisplaybreaks
\newcommand{\isaac}[2]{\def\tmp{#1}{\textcolor{red}{\ifx\tmp\@empty\else[\tmp]\fi}}{\textcolor{teal}{{#2}}}}

\usepackage[acronym]{glossaries}
\newacronym{sad}{SAD}{stress-assisted diffusion}
\newacronym{fem}{FEM}{finite element method}
\newacronym{vem}{VEM}{virtual element method}

\newacronym{hr}{HR}{Hellinger--Reissner}

\begin{document}
\let\WriteBookmarks\relax
\def\floatpagepagefraction{1}
\def\textpagefraction{.001}

\shortauthors{Bermudez \textit{et al.}}
\shorttitle{Fully mixed VEM for poroelastic stress-assisted diffusion in the brain}
\title[mode=title]{Fully mixed virtual element schemes for a new model of steady-state  poroelastic stress-assisted diffusion in the brain}

\author[1]{Isaac Bermudez}[orcid= 0009-0005-2772-2294]
\ead{isaac.bermudez@ucr.ac.cr}

\author[1]{Bryan G\'omez-Vargas}[orcid=0000-0001-6068-4510]
\ead{bryan.gomezvargas@ucr.ac.cr}

\author[2]{Kent-Andr\'e Mardal}[orcid=0000-0002-4946-1110]
\ead{kent-and@simula.no}

\author[3]{Andr\'es E. Rubiano}[orcid=0000-0002-5557-4963]
\ead{Andres.RubianoMartinez@monash.edu}

\author[3]{Ricardo Ruiz-Baier}[orcid=0000-0003-3144-5822]
\ead{Ricardo.RuizBaier@monash.edu}
\cormark[1]

\affiliation[1]{organization={Secci\'on de Matem\'atica, Sede de Occidente, Universidad de Costa Rica},    
city={San Ram\'on, Alajuela}, country={Costa Rica}}

\affiliation[2]{organization={Simula Research Laboratory and University of Oslo},    
city={Oslo}, country={Norway}}

\affiliation[3]{organization={School of Mathematics, Monash University},     addressline={9 Rainforest Walk}, postcode={3800},  city={Melbourne}, state={Victoria}, country={Australia}}
\cortext[cor1]{Corresponding author}

\begin{abstract}
We propose a fully mixed \gls{vem} for the numerical approximation of the coupling between  linear poroelasticity equations with strong symmetry of total poroelastic stress (using the \gls{hr} principle) and stress-altered solute diffusion (where diffusive flux depends on the poroelastic stress and nonlinearly on the concentration gradient). 
Because of the nonlinear coupling, the function spaces associated with the nonlinear diffusion sub-problem are of Banach type. To handle this structure, the solvability of both the continuous and discrete problems is established through a decoupled fixed-point strategy. The linear poroelasticity component is analysed using the theory for perturbed saddle-point problems, whereas {the analysis of the} nonlinear diffusion problem relies on the classical Minty--Browder theorem for monotone global operators. The existence of solutions for the fully coupled system is rigorously proven via Schauder's fixed-point theorem. Additionally, we establish rigorous a priori error estimates for the discrete scheme, successfully handling the strongly cross-coupled nonlinearities. These findings are supported by computational evidence, demonstrating that the formulation asymptotically recovers optimal convergence rates in practice. As a key contribution, both the numerical scheme and its underlying analysis prove to be robust with respect to the poromechanical parameters. Finally, several numerical examples are presented to illustrate the properties and applicability of the proposed scheme in the study of solute transport in the context of brain multiphysics. 
\end{abstract}

\begin{keywords} Coupled elasticity-diffusion \sep Mixed virtual element methods \sep  Stress-based formulations \sep Analysis in Banach spaces \sep  Error estimates
\MSC[2020]  65N30 \sep 65N15 \sep 74D05
\end{keywords}

\maketitle

\section{Introduction}\label{sec:intro}
\paragraph{Scope.} {We address the mathematical and numerical treatment of a new model for cerebral fluid transport based on models of poromechanics (see, e.g. \cite{botti2021hybrid,tully2011cerebral,guo2019validation}). From a biomechanical perspective, the brain parenchyma  behaves as a fluid-saturated porous matrix where the flow of cerebrospinal and interstitial fluids is inextricably linked to the deformation of the solid tissue. A major driver of this fluid movement is the glymphatic system \cite{Xie2013-me}, a brain-wide clearance network that flushes neurotoxic metabolic waste from the tissue, a process that is particularly active during sleep. Models of brain biomechanics frequently rely on Biot's theory of poroelasticity. To tackle the resulting multi-physics problems, a wide variety of alternative numerical strategies, ranging from standard and mixed finite element formulations to discontinuous Galerkin methods, hybrid, and finite volume schemes have been developed (see, for instance, \cite{cadq2023,teichtmeister2019aspects,elyes18,li20,lee16,riviere2017error,asadi2014finite,eliseussen2023posteriori,fumagalli25} and the references therein).  
In this and many other} physical and engineering contexts—going from metal processing and subsurface geoscience to biological tissue mechanics—structural distortion of a material and the migration of dissolved species evolve together. In {some of these}  scenarios, the coupling is such that the applied stress modifies the internal pore structure or transport pathways of the medium. This class of behaviour is known as \gls{sad}, and its analysis and numerical approximation have been explored under diverse modelling frameworks in the existing literature \cite{gatica18,gatica19,gatica22}. When the solid matrix is poroelastic, the interplay becomes more intricate: the total stress governing deformation incorporates contributions not only from the pore fluid pressure but also from the concentration of the diffusing substance itself. The governing system for poroelastic media under \gls{sad} has been examined through mixed \gls{fem} in \cite{gomez23} (with a related methodology applied to poroelasticity coupled with thermal effects in the more recent study \cite{careaga25}, and extended to twofold saddle-point structures for cases involving nonlinear hydraulic conductivity in \cite{khan25,lamichhane24}).

The present work proposes a momentum and mass conservative, robust, and Biot-locking-free virtual element formulation for poroelastic \gls{sad} systems. It constitutes an extension of the \gls{sad} \gls{vem}  from \cite{khot25} to the fully mixed poroelastic setting.  The underlying model problem is based on \cite{gomez23,careaga25}, but the formulation here is based on the \gls{hr} variational principle and imposing strong symmetry of the total Cauchy stress. For this we follow the similar works \cite{artioli17b,artioli20b,dassi20}. 
In this context we also mention other recent polytopal discretisations for poroelasticity in mixed form, as proposed in \cite{tang2021locking,cadq2023,khot25,botti2021hybrid,botti2025fullymixedvirtualelementmethod,liang2024parameter,kumar2024numerical,guo2022robust,liu2023virtual,droniou2024bubble,wmdh2025}. 
{In biological tissues, the transport of chemical species such as growth factors, nutrients, or signalling molecules is often tightly coupled with the mechanical environment. As mentioned above, tissue deformation generates poroelastic stresses that can alter the diffusivity of solutes through the extracellular matrix. To capture this interplay, the model we use here for the  concentration field satisfies a reaction–diffusion equation with diffusive flux depending implicitly on the  poroelastic stress  tensor (reflecting anisotropic changes in diffusivity due to matrix compression or tension, similarly as used in \cite{gatica18,gatica19,gatica22}), and in addition it depends nonlinearly on concentration gradients. This nonlinear term accounts for concentration‑dependent mobility or steric hindrance at high gradients, while the stress‑dependent term captures mechanotransductive feedback \cite{bollenbach2007morphogen}.}

The mathematical analysis of the steady-state poroelastic \gls{sad} problem presents   challenges due to the strong cross-coupling between the solid deformation, fluid flow, and nonlinear diffusion. This inherent complexity is further exacerbated by the inclusion of a power law  nonlinear term  (resembling the inverse relation in a Forchheimer-type problem) in the diffusion equation. This requires us to deviate from the classical Hilbert space framework used earlier  \cite{gatica18,gatica19,gomez23}, and now   the function space for diffusive flux is $\mathbf{L}^4(\Omega)$. The Banach space setting provides a distinct advantage: rather than imposing artificial constraints on the continuous problem, it provides a general functional framework that captures the exact physics of the strong non-linear coupling across arbitrary domain geometries and spatial dimensions. 

For the   solvability and stability analysis we use  a fixed-point argument separating a linear poroelasticity structure and a nonlinear diffusion problem. The   Biot system is analysed using  the Babu\v{s}ka--Brezzi theory. Conversely, the difficulties encountered in the altered \gls{sad} problem are treated by invoking abstract theory for global nonlinear operators. Specifically, we recast the system into a single, comprehensive mapping that encapsulates all constituent terms. By proving that this operator is coercive, bounded, monotone, and hemicontinuous, its surjectivity is deduced via the classical Minty--Browder theorem, which enables us to recover the original formulation and rigorously establish its existence and uniqueness. The solvability of the fully coupled  system is then established through Schauder's fixed-point theorem. By rigorously proving the Hölder continuity of the  fixed-point map, we conclude the existence of {a} solution.  

On the discrete level, we propose a \gls{vem}   for the coupled model. The solvability and stability of the \gls{vem} is established by closely mirroring the continuous analysis, employing analogous saddle-point and fixed-point arguments alongside specialised tools, to handle the cross-coupled stress terms. Following this, we conduct a rigorous a priori error analysis of the discrete scheme. Due to the strongly coupled nonlinearity, we need to repeatedly apply H\"older and Young inequalities that turn the resulting theoretical error bounds more restrictive than usual. Although these a priori bounds account for the most severe analytical scenarios, our computational examples demonstrate that the proposed scheme is highly robust, asymptotically recovering optimal convergence rates in practice.

Finally, we demonstrate the applicability of the proposed model by simulating two key processes linked to neurodegenerative pathology {(such} as Alzheimer's or Parkinson's disease). First, we examine sleep-driven molecular clearance within brain tissue, a critical glymphatic mechanism responsible for flushing metabolic waste. Second, we model the nonlinear transport and accumulation of Amyloid-$\beta$, a peptide aggregate associated with Alzheimer's disease progression. Both cases are of independent biomedical interest and offer valuable insights into how nonstandard models can capture complex multiscale physiological phenomena.

\paragraph{Outline.} The remainder of the paper has been organised in the following manner. In the rest of this section we recall usual notational convention for the domain and the used functional spaces. We also state the governing equations of Biot--\gls{sad}, giving also assumptions on the nonlinear diffusion coefficient and the rest of the model parameters. Section~\ref{sec:continuous_wp} contains the derivation of the weak formulation in double saddle-point structure, it specifies the splitting of kernels of suitable operators, and it examines the properties of all bilinear forms (including stability and boundedness). This section also addresses the unique solvability of the separate Biot and altered diffusion problems. In Section~\ref{sec:discrete} we formulate the \gls{vem} for the coupled model problem, introducing the needed discrete spaces, polynomial interpolation and projection operators, appropriate stabilisation, and discrete operators. In Section~\ref{sec:discrete_wp} we show the solvability and stability of that scheme by using a similar fixed-point argument as in the continuous case. A priori error estimates are presented in Section~\ref{sec:error}, and simple numerical tests are provided in Section~\ref{sec:results}, including the verification of optimal convergence, simulation of classical benchmark tests for poromechanics, and {two specific applications} for coupled \gls{sad} and nonlinear transport within brain multiphysics. {We close in Section \ref{sec:concl} with a brief summary of the obtained results and a discussion on possible extensions.}

\paragraph{Recurrent notation.}
Let us consider a simply connected bounded and Lipschitz domain $\Omega\subset\RR^d$, $d \in \{2,3\}$ occupied by a poroelastic body. 
The domain boundary $\partial\Omega$ is partitioned into disjoint sub-boundaries where homogeneous displacement and traction-type boundary conditions are imposed  $\partial\Omega:= \overline{\Gamma_{\mathrm{D}}} \cup \overline{\Gamma_{\mathrm{N}}}$, and it is assumed for sake of simplicity that both sub-boundaries are non-empty $|\Gamma_{\mathrm{D}}|\cdot|\Gamma_{\mathrm{N}}|>0$. Throughout the text, given a normed space $S$, by $\bS$ and $\bbS$ we will denote the vector and tensor extensions $S^d$ and $S^{d\times d}$, respectively. We define the Hilbert spaces $\bH(\vdiv,\Omega)=\left\{\bw \in \bL^2(\Omega): \vdiv\bw \in \rL^2(\Omega) \right\}$ and $\bH_{\mathrm{N}}(\vdiv,\Omega) := \{\bw\in \bH(\vdiv,\Omega): \bw \cdot \bn = 0 \text{ on } \Gamma_{\mathrm{N}}\}$ with its norm $\norm{\bw}_{\vdiv,\Omega}^2:=\norm{\bw}^2_{0,\Omega}+\norm{\vdiv\bw}^2_{0,\Omega}$. We also define the Banach spaces $\bH^{4}(\vdiv,\Omega)=\{\bw \in \bL^4(\Omega): \vdiv\bw \in \rL^2(\Omega) \}$ and $\bH^{4}_{\mathrm{N}}(\vdiv,\Omega) := \{\bw\in \bH^{4}(\vdiv,\Omega): \bw \cdot \bn = 0 \text{ on } \Gamma_{\mathrm{N}}\}$, both endowed with the norm $\|\bxi\|_{4,\vdiv;\Omega} := \|\bxi\|_{0,4;\Omega} + \|\vdiv\bxi\|_{0,\Omega}$. {This notation follows that of the more general space $\bH^{s}(\vdiv_r,\Omega)=\{\bw \in \bL^s(\Omega): \vdiv\bw \in \rL^r(\Omega) \}$, for Lebesgue exponents $0<r,s<\infty$ used in, e.g., \cite{gatica2021p,cgm-M2AN-2020,caucao2024banach}, where for either $s=2$ or $r=2$ the corresponding super/subscript is dropped.}
Next, we recall the definition of the tensorial Hilbert spaces 
$\bbH(\bdiv,\Omega)=\{\btau \in \bbL^2(\Omega): \bdiv\btau \in \bL^2(\Omega) \}$, $\bbH(\bcurl,\Omega)=\left\{\btau \in \bbL^2(\Omega): \bcurl\btau \in \bbL^2(\Omega) \right\}$, with their usual norms 
$\norm{\btau}_{\bdiv,\Omega}^2:=\norm{\btau}^2_{0,\Omega}+\norm{\bdiv\btau}^2_{0,\Omega}$,   $\norm{\btau}_{\bcurl,\Omega}^2:=\norm{\btau}^2_{0,\Omega}+\norm{\bcurl\btau}^2_{0,\Omega}$,  
where the divergence acts on the rows of $\btau$, and the curl of a tensor is here understood as the tensor formed by the curl of the rows of $\btau$.  
We also define the  following tensor space  
\begin{equation*}
\bbH^{\mathrm{sym}}_{\mathrm{N}}({\mathbf{div}},\Omega):=\left\{\btau \in
\bbH(\bdiv,\Omega): \ \btau = \btau^{\tt t}, \ \btau\bn=\cero\ \text{on}\ \Gamma_{\mathrm{N}}  \right\},\end{equation*}
which is Hilbert with the $\bbH(\bdiv)$ norm. Also, given a domain $\mathcal{O}$ (in $\bbR^d$ or $\bbR^{d-1}$) we denote the inner product in $\rL^2(\mathcal{O})$ (similarly for $\bL^2(\mathcal{O})$ and $\bbL^2(\mathcal{O})$)  by $(\bullet,\bullet)_{\mathcal{O}}$. When $\mathcal{O}= \Omega$ we simply write $(\bullet,\bullet)$. On the other hand, we will denote by $\norm{\cdot}_{\infty,\Omega}$ the norm of the  Banach space $\mathrm L^\infty(\Omega)$
as well as of its vectorial version $\mathbf L^\infty(\Omega)$.

Finally, throughout the paper, when comparing two quantities $a$ and $b$, we use the notation $a \lesssim b$ to indicate that there exists a constant $M$, independent of the mesh size $h$, such that $a \leq M b$.

\paragraph{Strong mixed form.} Let us recall that the steady-state Biot system states momentum and mass balances 
\begin{align*}
    -\bdiv \bsigma  &= \ff   \quad \, \text{in $\Omega $,}  \\
       s_0 {p}+\alpha\vdiv \bu - \vdiv \left(\bkappa \nabla p \right) &= g \quad \text{in $\Omega $,}  
\end{align*}
respectively, where $\bsigma= 2\mu \beps(\bu) + (\lambda \mathrm{div}(\bu)- \alpha p - \beta \varphi) \bbI$ is the poroelastic Cauchy stress tensor including a modulation due to a diffusing quantity $\varphi$, $\alpha$ is the Biot--Willis parameter, $\beta$ is the modulation intensity, $s_0$ is the storativity coefficient,   $\bkappa$ is a symmetric and positive-definite tensor of permeability of the porous media (scaled by the fluid viscosity), i.e., there exist two strictly positive real numbers $\kappa_1$ and $\kappa_2$ satisfying for a.e. $\bx\in \Om$ and all $\bxi\in \bbR^d$ such that $|\bxi|=1$
\begin{equation*}
	0<\kappa_1\leq \bkappa(\bx)\bxi\cdot \bxi\leq \kappa_2.
\end{equation*}
The coefficients $\lambda,\mu$ are the Lam\'e parameters of Hooke's law, $\ff:\Omega \to \RR^d$ is the vector field of body loads and $g:\Omega\to \RR$ is a scalar source/sink of fluid, and $\beps(\bu)= \frac12(\bnabla \bu+(\bnabla\bu)^{\tt t})$ is the infinitesimal strain tensor. 
 
In addition to the mixed Biot equations, we also consider the presence of a solvent within the poroelastic domain. We denote its concentration by $\varphi:\Omega\to\bbR$ and its movement in the body for given volumetric source $\ell$ is governed by  {a} nonlinear diffusion process, to account for high-velocity flows.  More precisely, given $\bsigma$, let $\cA_{\bsigma}: \bbR^{d} \to \bbR^{d}$ be the non-linear operator defined by
\begin{align}\label{eq:nonlinear-A-difussion}
\cA_{\bsigma}(\bw):= \varrho(\bsigma)^{-1}\bw+\eta |\bw|^2\bw,
\end{align}
 where the scalar function  $\varrho:\RR^{d\times d} \rightarrow \RR$ is a stress-dependent diffusivity accounting for altered diffusion acting in the poroelastic domain and indicating a change in microstructure due to poroelastic stress generation and $\eta>0$ is a model coefficient. We claim that this operator admits a well-defined continuous inverse $\cA_{\bsigma}^{-1}$. The strong formulation for the concentration $\varphi$ is then governed by
\begin{equation}\label{eq:phi}
\varphi   
- \vdiv( \cA_{\bsigma}^{-1}(\nabla \varphi)) =  \ell \qquad \text{in $\Omega$}.
\end{equation}

We assume that the \gls{sad} function $\varrho$ takes the form
 \begin{equation}\label{eq:D}
\varrho(\bsigma) = \eta_0 \varrho_0 + \exp(-\eta_1[ \tr\bsigma]^2),
 \end{equation}
where $\varrho_0>0$ is the base-line effective diffusion (in the absence of stress assistance) and $\eta_0,\eta_1$ are positive modulation parameters (the treatment can also be modified to accommodate anisotropy with tensor-valued diffusivity). For 
the sake of analysis, we   require $\varrho^{-1}(\bullet)$ to be uniformly bounded away from zero and Lipschitz continuous with respect to $\bsigma \in \bbL^{2}(\Omega)$. More precisely, there exist positive constants $\varrho_{1}$, $\varrho_{2}$ and $L_{\varrho}$, such that 
\begin{align*}
    0 < \varrho_{1} \leq \varrho^{-1}(\bullet) \leq \varrho_{2} < \infty \qan \|\varrho^{-1}(\bsigma)-\varrho^{-1}(\btau)\|_{0,\Omega} \leq L_{\varrho}\|\bsigma-\btau\|_{0,\Omega},
\end{align*}
for all $\bsigma,\btau \in \bbL^{2}(\Omega)$.

Note that \eqref{eq:phi} is the mass balance of the diffusive substance (drug, tracer, nutrient, signalling molecule, etc.) with linear decay or uptake and with diffusive flux $\bJ=\cA_{\bsigma}^{-1}(\nabla \varphi)$. 
For isothermal transport, the flux is related to the gradient of the chemical potential, and the nonlinear relation can be derived from a dissipation potential $\Psi(\bzeta)$ via $\nabla \varphi = \frac{\partial\Psi}{\partial \bzeta}$. Here we consider a potential of the form 
\[ 
\Psi(\bJ) = \frac12\varrho(\bsigma)^{-1}|\bJ|^2+\frac14\eta|\bJ|^4,
\]
so then $\nabla \varphi = \varrho(\bsigma)^{-1} \bJ + \eta |\bJ|^2\bJ$, which coincides exactly with \eqref{eq:nonlinear-A-difussion}. The quadratic term represents Darcy-type dissipation, the term involving $\varrho$ models \gls{sad}, and the quartic term accounts for nonlinear drag at high fluxes (which is observed in densely cross-linked extracellular matrices). While the form in \eqref{eq:phi} is sufficient to perform the analysis in mixed form that will follow, we observe that the implicit relation can also be made explicit. Following, e.g., \cite{mielke2011gradient}, taking the Legendre transformation of the dissipation provides a dual potential $\widetilde{\Psi}(\nabla \varphi)$ such that $\bJ = \frac{\partial \widetilde{\Psi}}{\partial (\nabla \varphi)}$. For the specific $\Psi$, since $\varrho$ is a scalar function, $\bJ$ is collinear with $\nabla \varphi$. Thus, the diffusive flux can be written in the more familiar form $\bzeta = D(\bsigma,|\nabla \varphi|)\nabla \varphi$, where $D(\bullet,\bullet)>0$ solves 
\begin{equation}\label{eq:Cardano} 
\eta |\nabla \varphi|^2 D^3 + \varrho(\bsigma)^{-1}D - 1 =0.
\end{equation}
It is possible to obtain an explicit expression for $D$ (using Cardano's formula), and from \eqref{eq:Cardano} it readily follows that $D$ is decreasing with respect to $|\nabla \varphi|$. This corresponds to a shear-thinning behaviour typical in biological gels.

The material properties are described at each point by the compliance tensor (the inverse of the fourth-order  linear isotropic stiffness tensor $\bC$) $\bC^{-1}$, which is identified as a symmetric, bounded, and uniformly positive definite linear operator characterised by its action 
\begin{equation*}
\bC\beps(\bu) = 2\mu\beps(\bu) + \lambda(\vdiv\bu)\bbI,\qquad \bC^{-1}\,\bsigma = \frac{1}{2\mu}\biggl(\bsigma - \frac{\lambda}{2\mu+d\lambda} \tr(\bsigma)\bbI\biggr),\end{equation*}
and $\bsigma= \bC\beps(\bu)-\{ \alpha p + \beta \varphi\}\bbI$ to obtain $\bC^{-1} (\bsigma + \{\alpha p + \beta \varphi\} \bbI) =  \beps(\bu).$

The problem is rewritten, considering the elasticity equations with strong symmetric stress imposition, which are coupled with the fluid phase obeying Darcy's law for filtration in porous media, and a mixed form associated with \eqref{eq:phi}. The unknowns are the effective poroelastic Cauchy stress tensor $\bsigma$, the displacement vector $\bu$, the filtration flux vector $\bz$,   the fluid pressure $p$, the diffusive flux $\bzeta$, and the concentration $\varphi$ such that 
\begin{subequations}\label{eq:poroelast}
\begin{alignat}{2}
 \bC^{-1}\, \bsigma & =  \beps(\bu) - \frac{\alpha p + \beta \varphi}{2\mu + d\lambda} \bbI && \quad \text{in $\Omega$}, \label{eq:poroelast-1}\\
 -\bdiv\bsigma & = \ff  && \quad \text{in $\Omega$}, \label{eq:poroelast-2}\\
 \bsigma & = \bsigma^{\tt t} && \quad \text{in $\Omega$}, \label{eq:poroelast-3}\\
\bkappa^{-1} \bz + \nabla p & = \cero  && \quad \text{in $\Omega$}, \label{eq:poroelast-4}\\
s_0 p + \alpha \tr\bC^{-1} [\bsigma + (\alpha p + \beta \varphi)\bbI] + \vdiv\bz & = g && \quad \text{in $\Omega$}, \label{eq:poroelast-5}\\
\cA_{\bsigma}(\bzeta)+\nabla \varphi & = \cero  && \quad \text{in $\Omega$}, \label{eq:poroelast-6}\\
 \varphi + \vdiv\bzeta & = \ell && \quad \text{in $\Omega$}, \label{eq:poroelast-7}\\
 \bu  = \bu_{\mathrm{D}}, \quad p = p_{\mathrm{D}}, \quad \varphi&= \varphi_{\mathrm{D}} && \quad \text{on $\Gamma_{\mathrm{D}}$}, \label{eq:poroelast-bd1}\\
\bsigma\bn  = \cero, \quad  \bz\cdot \bn=0, \quad \bzeta \cdot \bn& = 0 && \quad \text{on $\Gamma_{\mathrm{N}}$}, \label{eq:poroelast-bd2}
\end{alignat}\end{subequations}
stating a rescaling of the stress constitutive relation, the balance of linear momentum, the balance of angular momentum, Darcy's law, the balance of the total amount of fluid, the constitutive equation for the diffusive flux, the concentration balance, and the mixed-loading boundary conditions of homogeneous type, respectively. 

{We remark that other typical application-oriented works targeting fluid transport in the brain are based on intrinsically transient models \cite{tully2011cerebral,guo2019validation}. We can easily include time-dependence in the model, discretisation, and simulation, but its analysis requires a substantially different analytical framework following, e.g., \cite{showalter1997}. Nevertheless, the steady-state framework from \eqref{eq:poroelast} still provides many of the underlying coupling mechanisms that one would require to analyse when targeting the transient counterpart of the system. On the other hand, 
the steady-state can be motivated by the scale separation between rapid mechanical pulsations (cardiac/respiratory cycles on the order of $1\,\text{s}$) and the characteristic time scales of interstitial fluid perfusion and solute clearance in the brain parenchyma (spanning several hours to days). Over long transport horizons, the system approaches a stationary equilibrium governed by time-averaged total stress and fluid pressure. }

\section{Weak formulation and continuous analysis}\label{sec:continuous_wp}
The functional framework and weak formulation of the coupled problem \eqref{eq:poroelast} are developed in this section. We establish the existence and uniqueness of solutions for the decoupled equations—relying on a nonlinear abstract result for the diffusion sub-problem—and subsequently apply Schauder's fixed-point theorem to guarantee the existence of a weak solution for the fully coupled system.
\subsection{Derivation and main properties}\label{sec:main_properties} We apply algebraic manipulations and multiply the strong form of the balance equations and constitutive relations by suitable test functions, integrate by parts in the constitutive relations and in the diffusion term, and employ the boundary conditions to obtain the weak formulation: for $\ff\in \bL^2(\Omega)$, $g,\ell\in \rL^2(\Omega)$,  $\bu_{\mathrm{D}} \in \bH^{1/2}_{00}(\Gamma_{\mathrm{D}})$, and 
$p_{\mathrm{D}},\varphi_{\mathrm{D}} \in \rH^{1/2}_{00}(\Gamma_{\mathrm{D}})$; 
find $(\bsigma,p,\bu,\bz,\bzeta,\varphi) \in \bbH^{\mathrm{sym}}_{\mathrm{N}}(\bdiv,\Omega)\times \rL^2(\Omega)\times \bL^2(\Omega)\times \bH_{\mathrm{N}}(\vdiv,\Omega)\times \bH^{4}_{\mathrm{N}}(\vdiv,\Omega)\times \rL^2(\Omega)$
 such that 
\begin{subequations}\label{eq:weak1}
\begin{alignat}{2}
(\bC^{-1}\bsigma,\btau) + 
 \bigl(\frac{\alpha p}{2\mu + d\lambda}, \tr\btau \bigr) 
 + (\bdiv\btau,\bu)  +  \bigl( \frac{\beta \varphi}{2\mu + d\lambda}, \tr\btau \bigr)  & = \langle \bu_{\mathrm{D}},\btau\bn\rangle_{\Gamma_{\mathrm{D}}} &   \forall \btau \!\in  \bbH^{\mathrm{sym}}_{\mathrm{N}}(\bdiv,\!\Omega),\\
  (\tr \bsigma,\frac{\alpha q}{2\mu + d\lambda})   + \bigl[s_0+\frac{d\alpha^2}{2\mu + d\lambda}\bigr]
 ( p,q)  + (q,\vdiv\bz) + \alpha  \bigl(\frac{d\beta \varphi}{2\mu + d\lambda} 
,q\bigr)
& = (g,q) &\  \forall q \in \rL^2(\Omega),\\
 (\bdiv\bsigma,\bv) & = -(\ff,\bv) &\  \forall \bv\in \bL^2(\Omega),\\
 (\bkappa^{-1}\bz,\bw) - (p,\vdiv\bw) & = - \langle p_{\mathrm{D}},\bw\cdot\bn\rangle_{\Gamma_{\mathrm{D}}}  &\  \forall \bw\in\bH_{\mathrm{N}}(\vdiv,\Omega),\\
(\varrho(\bsigma)^{-1}\bzeta+\eta |\bzeta|^2\bzeta, \bxi) - (\varphi,\vdiv\bxi) & = - \langle \varphi_{\mathrm{D}},\bxi\cdot\bn\rangle_{\Gamma_{\mathrm{D}}} &\ \forall \bxi\in \bH_{\mathrm{N}}^4(\vdiv,\Omega),\label{eq:diffWeak}\\
-(\psi,\vdiv\bzeta) - ( \varphi,\psi)  & = -(\ell,\psi)  & \ \forall \psi \in \rL^2(\Omega),
\end{alignat}
\end{subequations}
where the ordering of the unknowns obeys to the subsequent structure of the analysis. Indeed, we group the Biot function spaces as well as trial and test functions for stress-pressure and displacement-discharge flux as follows 
\[\bbV:= \bbH^{\mathrm{sym}}_{\mathrm{N}}(\bdiv,\Omega)\times \rL^2(\Omega), \quad \bQ:= \bL^2(\Omega)\times \bH_{\mathrm{N}}(\vdiv,\Omega), \]
(endowed with the canonical graph norms of the product spaces) and 
 \[    \vec{\bsigma} := (\bsigma, p), \  \vec{\btau} := (\btau, q) \in \bbV, \qan \vec{\bu} := (\bu, \bz), \  \vec{\bv} := (\bv, \bw) \in \bQ,\]
respectively. Then, \eqref{eq:weak1} consists in finding  $(\vec{\bsigma}, \vec{\bu}) \in \bbV \times \bQ$ and $( \bzeta, \varphi) \in  \bH^4_{\mathrm{N}}(\vdiv,\Omega)\times \rL^2(\Omega)$, such that 
\begin{subequations}\label{eq:weak2}
    \begin{align}
         A(\vec{\bsigma}, \vec{\btau})+\,B(\vec{\btau}, \vec{\bu}) + D(\varphi, \vec{\btau}) &   \, =\, F(\vec{\btau})  &&\forall \vec{\btau} \in \bbV,\\
    B(\vec{\bsigma},  \vec{\bv}) -\,C(\vec{\bu}, \vec{\bv}) &  \,=\, G(\vec{\bv})&&\forall \vec{\bv} \in \bQ,  \\ 
    [\cA_{\bsigma}(\bzeta), \bxi]+\,[\cB(\bxi), \varphi] &   \, =\, [\cF,\bxi] &&\forall \bxi \in \bH_{\mathrm{N}}^{4}(\vdiv,\Omega), \label{eq:weak-diff-1}  \\
   [\cB(\bzeta),  \psi]-\,[\cC(\varphi), \psi] &  \,=\, [\cG,\psi] &&\forall \psi \in \rL^2(\Omega), \label{eq:weak-diff-2}  
    \end{align}
\end{subequations}
where the bilinear forms $A:\bbV\times\bbV\to\RR$, $B:\bbV\times\bQ\to\RR$, $C:\bQ\times\bQ\to\RR$, {$D:\rL^2(\Omega) \times \bbV \to \RR$}, the linear operators  $\cB:  \bH^{4}_{\mathrm{N}}(\vdiv,\Omega) \to \rL^2(\Omega)$, $\cC:\rL^2(\Omega)\to \rL^2(\Omega)$, and (for a given $\widehat{\bsigma} \in \bbL^2(\Omega)$) the nonlinear operator $\cA_{\widehat{\bsigma}}:\bH^{4}_{\mathrm{N}}(\vdiv,\Omega)\to (\bH^{4}_{\mathrm{N}}(\vdiv,\Omega))^{\prime}$, are defined as 
\begin{gather*}
   A(\vec{\bsigma}, \vec{\btau}) := (\bC^{-1}\bsigma, \btau) +  \bigl(\frac{\alpha p}{2\mu + d\lambda}, \tr\btau \bigr) +  
   \bigl(\frac{\alpha q}{2\mu + d\lambda}, \tr\bsigma \bigr) 
   + \biggl[s_0+\frac{d\alpha^2}{2\mu + d\lambda}\biggr]
 ( p,q),\\
  B(\vec{\btau}, \vec{\bv}) := (\bv, \bdiv \btau)+ (q , \vdiv \bw), \quad 
  C(\vec{\bu}, \vec{\bv}):=(\bkappa^{-1}\bz,\bw), \quad 
  D(\psi,\vec{\btau}):= ( \frac{\beta \psi}{2\mu + d\lambda}, \tr\btau + \alpha d q ), \\
 [\cA_{\widehat{\bsigma}}(\bzeta),\bxi]  := (\varrho(\widehat{\bsigma})^{-1}\bzeta,\bxi)+\eta(|\bzeta|^{2}\bzeta, \bxi), \quad  [\cB(\bxi), \psi]:= -(\psi, \vdiv \bxi), \quad  [\cC(\varphi), \psi]:=(\varphi, \psi).
\end{gather*}
Similarly, the linear functionals $F: \bbV\to \RR$, $G: \bQ\to\RR$, $\cF: \bH_{\mathrm{N}}^{4}(\vdiv,\Omega)\to \RR$, and $\cG: \rL^2(\Omega)\to \RR$ are 
\begin{gather*} F(\vec{\btau}) : = \langle \bu_{\mathrm{D}},\btau\bn\rangle_{\Gamma_{\mathrm{D}}} + (g,q), \quad G(\vec{\bv}): = - (\ff,\bv) + \langle p_{\mathrm{D}},\bw\cdot\bn\rangle_{\Gamma_{\mathrm{D}}}, \\
[\cF,\bxi]:= - \langle \varphi_{\mathrm{D}},\bxi\cdot\bn\rangle_{\Gamma_{\mathrm{D}}}, \quad 
[\cG,\psi]:= -(\ell,\psi).
\end{gather*}

We proceed to examine the properties of the bilinear forms and linear functionals for the Biot block. As an intermediate step we denote by $\bB$ and $\bB^*$ the operators induced by the bilinear form $B(\bullet,\bullet)$. Their kernels admit the following characterisations:
\begin{subequations}\label{eq:kernel-V0}
\begin{align}
\bbV_0 &:= \mathrm{ker}(\bB) = \{ \vec{\btau} \in \bbV: \ B(\vec{\btau},\vec{\bv}) = 0 ,\, \forall \vec{\bv}\in \bQ\} \notag \\
&=: \bbV_{01} \times \bbV_{02} \equiv \{ \btau \in \bbH^{\mathrm{sym}}_{\mathrm{N}}(\bdiv,\Omega): \ \bdiv\btau = \cero \ \text{in}\ \Omega\} \times \{0\},\\
\bQ_0 &:= \mathrm{ker}(\bB^*) = \{ \vec{\bv} \in \bQ: \ B(\vec{\btau},\vec{\bv}) = 0 ,\, \forall \vec{\btau}\in \bbV\} \notag \\
&=: \bQ_{01} \times \bQ_{02} \equiv \{\cero\}\times \{ \bw \in \bH_{\mathrm{N}}(\vdiv,\Omega): \ \vdiv\bw = 0 \ \text{in}\ \Omega\}. \label{eq:kernel-Q0}
\end{align}
\end{subequations}
 The characterisation of $\bbV_{02}$ (and similarly for $\bQ_{01}$) follows as in \cite[Section 3.3]{cgm-M2AN-2020}. It is possible to realise that $\nabla q =\cero$ in the distributional sense, which gives $q \in \rH^{1}_{}(\Omega)$. Moreover, integrating by parts $(q, \vdiv \bw)$ in \eqref{eq:kernel-Q0}, we arrive at $\langle \bw \cdot \bn, q\rangle_{\Gamma_{\mathrm{D}}} = 0$  for all $\bw \in \bH_{\mathrm{N}}(\vdiv,\Omega)$. Next, using the surjectivity of the normal trace from $\bH_{\mathrm{N}}(\vdiv,\Omega)$ onto $\rH^{-1/2}_{00}(\Gamma_{\mathrm{D}})$ (cf. \cite[Lemma 51.5]{guermond2021feII}), yields $q=0$ on $\Gamma_{\mathrm{D}}$, and hence $q \in \rH^{1}_{\mathrm{D}}(\Omega)$.

In turn, the spaces $\bbV_{01}^{\perp}$, $\bbV_{02}^{\perp}$, $\bQ_{01}^{\perp}$ and $\bQ_{02}^{\perp}$ are characterised as follows:
\begin{align*}
    \bbV_{01}^{\perp} &\equiv \{ \bsigma \in \bbH^{\mathrm{sym}}_{\mathrm{N}}(\bdiv,\Omega):\ (\bsigma,\btau) =0 ,\ \forall \btau \in  \bbV_{01}\}, \quad \bbV_{02}^{\perp}\equiv \rL^{2}(\Omega),\\
    \bQ_{01}^{\perp} &\equiv \bL^{2}(\Omega), \quad \bQ_{02}^{\perp} \equiv \{ \bz \in \bH_{\mathrm{N}}(\vdiv,\Omega):\ (\bz,\bw) =0, \ \forall \bw \in  \bQ_{02}\},
\end{align*}
and hence $\bbV_0^{\perp}= \bbV_{01}^{\perp} \times \bbV_{02}^{\perp}$ and $\bQ_0^{\perp}= \bQ_{01}^{\perp} \times \bQ_{02}^{\perp}$  are closed subspaces of $\bbV$ and $\bQ$, respectively.
\begin{lemma}[boundedness of the bilinear forms]\label{lem:boundedness_bilinear}
The bilinear forms $A(\bullet,\bullet)$, $B(\bullet,\bullet)$, $C(\bullet,\bullet)$, and $D(\bullet,\bullet)$ are bounded. That is:
\begin{align*}
|A(\vec{\bsigma}, \vec{\btau})| &\leq \|A\| \|\vec{\bsigma}\|_{\bbV} \|\vec{\btau}\|_{\bbV} \quad &&\forall \vec{\bsigma}, \vec{\btau} \in \bbV, \\
|B(\vec{\btau}, \vec{\bv})| &\leq \|B\| \|\vec{\btau}\|_{\bbV} \|\vec{\bv}\|_{\bQ} \quad &&\forall \vec{\btau} \in \bbV, \forall \vec{\bv} \in \bQ, \\
|C(\vec{\bu}, \vec{\bv})| &\leq \|C\| \|\vec{\bu}\|_{\bQ} \|\vec{\bv}\|_{\bQ} \quad &&\forall \vec{\bu}, \vec{\bv} \in \bQ, \\
|D(\psi, \vec{\btau})| &\leq \|D\| \|\psi\|_{0,\Omega} \|\vec{\btau}\|_{\bbV} \quad &&\forall \psi \in \rL^2(\Omega), \forall \vec{\btau} \in \bbV, 
\end{align*}
where the boundedness constants are given by
\begin{gather*}
\|A\| := \max\left\{ \frac{1}{2\mu} + \frac{\lambda}{2\mu(2\mu+d\lambda)}, \frac{\alpha\sqrt{d}}{2\mu+d\lambda}, s_0+\frac{d\alpha^2}{2\mu+d\lambda} \right\}, \quad \|B\| := 1, \\ \|C\| := \frac{1}{\kappa_1}, \quad 
\|D\| := \frac{\beta\sqrt{d}(1+\alpha)}{2\mu+d\lambda}.
\end{gather*}
\end{lemma}
\begin{proof}
The boundedness of the bilinear forms is a direct consequence of the Cauchy--Schwarz inequality, and thus further details are omitted.
\end{proof}

\begin{lemma}[symmetry and positive semi-definiteness of diagonal forms]\label{lemma:positive-definite-diagonal}
The bilinear forms $A(\bullet,\bullet)$ and  $C(\bullet,\bullet)$ are symmetric and positive semi-definite.
\end{lemma}
\begin{proof}
  It is clear that $A(\bullet,\bullet)$ and $C(\bullet,\bullet)$  are symmetric, {whereas}   $C(\bullet,\bullet)$ is positive semi-definite. To prove that $A(\bullet,\bullet)$ is positive semi-definite,  note that given $\vec{\btau} \in \bbV$, we have
\begin{align}\label{positive-A}
A(\vec{\btau}, \vec{\btau}) &= (\bC^{-1}(\btau+\alpha q\bbI), \btau+\alpha q\bbI)+s_0\|q\|^{2}_{0,\Om} \geq 0,
\end{align}
since $\bC^{-1}$ is {positive definite} and $s_{0}\geq0$.
\end{proof}

We proceed similarly to \cite[Section 2.3]{gatica2006priori} to show that $A(\bullet,\bullet)$ is $\bbV_{0}$-elliptic. To do that, we recall the decomposition
\begin{equation*}
 \bbH(\bdiv_{},\Omega)  = \bbH_0(\bdiv_{},\Omega) \, \oplus \, {m \,\bbI}, \hbox{ with }{m\in \mathbb{R}}, \quad \text{and} \quad 
  \bbH_0(\bdiv_{},\Omega)  := \left\{\btau \in \bbH(\bdiv_{},\Omega):\ \int_{\Omega}\tr\btau = 0\right\}.
\end{equation*}
We also recall two useful estimates, whose proofs can be found in \cite[Lemma 2.3]{gatica2006priori} and \cite[Lemma 2.4]{gatica2006priori}. Specifically, there exists $C_{1}>0$, depending only on $\Omega$, such that 
\begin{equation}\label{DES}
   C_{1} \,\|\btau_{0}\|_{0,\Omega} \, \leq \, \| \btau^{\tt d}\|_{0,\Omega} + \| \bdiv \btau\|_{0,\Omega} \quad \forall \, \btau \in \bbH(\bdiv,\Omega), \qan
\end{equation}
 there exists $C_{2}>0$, depending only on $\Gamma_{\mathrm{N}}$ and $\Omega$, such that 
\begin{align}\label{DES-GN}
    C_{2} \, \|\btau \|_{\bdiv,\Omega}  \,\leq \,   \|\btau_{0}\|_{\bdiv,\Omega}  \quad \forall \, \btau:= \btau_{0}+ {m} \bbI \in   \bbH_{\mathrm{N}}(\bdiv,\Omega),
\end{align}
with $\btau_{0} \in \bbH_{0}(\bdiv,\Omega)$ and $m  \in \bbR$. Then, we have the following result.

\begin{lemma}[coercivity for the main diagonal forms]\label{lemma:coercitivity-as}
There exist constants $\alpha_A,\alpha_c,\alpha_{C}>0$ such that 
\begin{subequations}
\begin{alignat}{2}
  A(\vec{\btau}, \vec{\btau}) & \geq  \alpha_A \|\vec{\btau}\|^{2}_{\bbV} \quad &&\forall \, \vec{\btau} \in \bbV_0, \label{coercitivity-A}\\
    C(\vec{\bv}, \vec{\bv}) &\geq   \alpha_C \|\vec{\bv}\|^{2}_{\bQ} \quad &&\forall \, \vec{\bv} \in \bQ_{0}. \label{coercitivity-C}
\end{alignat}\end{subequations}
\end{lemma}
\begin{proof}
For \eqref{coercitivity-A}, we let $\vec{\btau}=(\btau, q) \in  \bbV_{01} \times \bbV_{02}$. This means, according to \eqref{positive-A}, \eqref{DES} and \eqref{DES-GN}, that
\begin{align*}
    A(\vec{\btau}, \vec{\btau}) \geq \alpha_A  \|\vec{\btau}\|^{2}_{\bbV},
\end{align*}
with $\alpha_A = \displaystyle  {} \frac{C_{1}\,C_{2}}{4\mu}$.
On the other hand, given $\vec{\bv} \in \bQ_{0}$ (cf. \eqref{eq:kernel-Q0}),  we have
\begin{align*}
    C(\vec{\bv}, \vec{\bv}) \geq \frac{1}{ \kappa_2}\|\bw\|^{2}_{0,\Om}= \frac{1}{\kappa_2} \left\{ \|\bw\|^{2}_{0,\Om}+\|\vdiv \bw\|^{2}_{0,\Om}+\|\bv\|^{2}_{0,\Om}\right\},
\end{align*}
which shows \eqref{coercitivity-C} with $\displaystyle \alpha_C= \frac{1}{\kappa_2}$.
\end{proof}

\begin{lemma}[continuous inf-sup condition]\label{lemma:inf-sup-B}
There exist positive constants $\beta_B,\beta_b$ such that
\begin{equation}\label{eq:inf-sup-B}
    \sup_{ \vec{\btau} \in \bbV\setminus\{\boldsymbol{0}\}}\frac{ B (\vec{\btau}, \vec{\bv})}{\|\vec{\btau}\|_{\bbV}}  \geq  \beta_{B} \|\vec{\bv}\|_{\bQ} \qquad \forall \, \vec{\bv} \in \bQ_0^{\perp}.
\end{equation}
\end{lemma}
\begin{proof}
To prove \eqref{eq:inf-sup-B}, it suffices to establish the following two independent inf–sup conditions, which follow from the diagonal structure of $B(\bullet,\bullet)$:
 \begin{subequations}
\begin{equation}\label{eq:inf-sup-b1}
   \displaystyle \sup_{\btau \in \bbH^{\mathrm{sym}}_\mathrm{N}(\bdiv,\Omega) \setminus \{\mathbf{0}\}}\frac{  (\bv, \bdiv \btau)}{\|\btau\|_{\bdiv,\Omega}}  \geq  \beta_{1} \|\bv\|_{0,\Om} \quad \forall \, \bv \in \bQ^{\perp}_{01},
\end{equation}
\begin{equation}\label{eq:inf-sup-b2}
   \displaystyle \sup_{q \in \rL^{2}(\Om) \setminus \{0\}}\frac{(q, \vdiv \bw)}{\|q\|_{0,\Om}}  \geq  \beta_{2} \|\bw\|_{\vdiv,\Om} \quad \forall \, \bw \in \bQ^{\perp}_{02}.
\end{equation}\end{subequations}
For \eqref{eq:inf-sup-b1}, we refer to \cite[Lemma 2.2, eq. (14)]{gatica2006priori}, whereas \eqref{eq:inf-sup-b2} holds by virtue of the existence of a constant $\widehat{\beta}_{2}>0$, such that  (see, e.g., \cite[Lemma 3.2]{cgorv-cmame-2015})
\[
    \sup_{\bw \in \bH^{}_\mathrm{N}(\vdiv,\Om) \setminus \{\mathbf{0}\}}\frac{  (q, \vdiv \bw)}{\|\bw\|_{\vdiv,\Om}}  \geq  \widehat{\beta}_{2} \|q\|_{0,\Om} \quad \forall \, q \in \bbV^{\perp}_{02}= \rL^{2}(\Om),
\]
 and the identity given by \cite[eq. (4.3.18)]{boffi13}, which also implies that $\beta_{2}= \widehat{\beta}_{2}$. Thus, the required inequality \eqref{eq:inf-sup-B} is obtained with $\beta_{B}= \frac{\beta_{1}+\beta_{2}}{4}$.
\end{proof}

For the altered diffusion block, the well-posedness analysis will rely on the classical Browder--Minty theorem for a suitable nonlinear operator to be defined next. The strategy consists in reducing the perturbed mixed saddle-point system to a single non-linear operator equation in the space $\bH^4_{\mathrm{N}}(\vdiv,\Omega)$. 
More precisely, since $\cC$ represents the identity mapping in the inner product space $L^2$, it is clear that it is bounded, symmetric, strongly coercive, and has a well-defined bounded inverse $\cC^{-1}: \rL^{2}(\Omega) \to \rL^{2}(\Omega)$. Thus, from  \eqref{eq:weak-diff-2} the concentration variable can be rewritten as $\varphi = \cC^{-1}(-\cG + \mathcal{B}(\bzeta))$, where $\cG \in \rL^{2}(\Omega)$.

Substituting this relation back into \eqref{eq:weak-diff-1} motivates the definition of the global nonlinear operator $\cN: \bH^4_{\mathrm{N}}(\vdiv,\Omega) \to \bH^4_{\mathrm{N}}(\vdiv,\Omega)'$:
\begin{equation*} 
    \cN(\bzeta) := \mathcal{A}_{{\bsigma}}(\bzeta) + \mathcal{B}^* \cC^{-1} \mathcal{B}(\bzeta),
\end{equation*}
where $\mathcal{B}^*: \rL^{2}(\Omega) \to \bH^4_{\mathrm{N}}(\vdiv,\Omega)'$ is the adjoint of $\mathcal{B}$.

We observe that the second term yields $[\mathcal{B}^* \cC^{-1} \mathcal{B}(\bzeta), \bxi] = [\cC^{-1} \mathcal{B}(\bzeta), \mathcal{B}(\bxi)] = (\vdiv \bzeta, \vdiv \bxi)$. Thus
\begin{equation} \label{eq:def_N_variational}
    [\cN(\bzeta), \bxi] = (\varrho(\widehat{\bsigma})^{-1}\bzeta, \bxi) + \eta (|\bzeta|^2\bzeta, \bxi) + (\vdiv \bzeta, \vdiv \bxi) \quad \forall \bzeta, \bxi \in \bH^4_{\mathrm{N}}(\vdiv,\Omega).
\end{equation}
In what follows, we proceed to examine some properties of the nonlinear operator $\cN$. 
\begin{lemma}[boundedness]\label{lem:boundedness}
    The operator $\cN$ is bounded. That is, for any $M > 0$, there exists a constant $C_\star > 0$ such that $\|\cN(\bzeta)\|_{(\bH^4_{\mathrm{N}}(\vdiv,\Omega))'} \le C_\star$ whenever $\|\bzeta\|_{{{4,\vdiv;\Omega}}} \le M$.
\end{lemma}
\begin{proof}
    Let $M > 0$ and assume $\bzeta \in\bH^4_{\mathrm{N}}(\vdiv,\Omega)$ such that $\|\bzeta\|_{{4,\vdiv;\Omega}} \le M$. For any $\bxi \in \bH^4_{\mathrm{N}}(\vdiv,\Omega)$, we bound the duality pairing using the Cauchy-Schwarz and Hölder inequalities
    \begin{align*}
        |[\cN(\bzeta), \bxi]|
        \le \rho_2 \|\bzeta\|_{0, \Omega} \|\bxi\|_{0, \Omega} + \eta \|\bzeta\|_{0,4;\Omega}^3 \|\bxi\|_{0,4;\Omega} + \|\vdiv \bzeta\|_{0,\Omega} \|\vdiv \bxi\|_{0,\Omega}.
    \end{align*}
    Next, thanks to the continuous Sobolev embedding $\mathbf{L}^4(\Omega) \hookrightarrow \mathbf{L}^2(\Omega)$, there exists a constant $C_{\text{emb}} > 0$ such that
    \begin{align*}
        |[\cN(\bzeta), \bxi]| &\le \rho_2 C_{\text{emb}}^2 \|\bzeta\|_{0,4;\Omega} \|\bxi\|_{0,4;\Omega} + \eta \|\bzeta\|_{0,4;\Omega}^3 \|\bxi\|_{0,4;\Omega} + \|\vdiv \bzeta\|_{0,\Omega} \|\vdiv \bxi\|_{0,\Omega} \\
        &\le \left( \rho_2 C_{\text{emb}}^2 \|\bzeta\|_{{4,\vdiv;\Omega}} + \eta \|\bzeta\|_{{4,\vdiv;\Omega}}^3 + \|\bzeta\|_{{4,\vdiv;\Omega}} \right) \|\bxi\|_{{4,\vdiv;\Omega}},
    \end{align*}
    which says that
    \begin{equation*}
        \|\cN(\bzeta)\|_{\bH^4_{\mathrm{N}}(\vdiv,\Omega)'} \le (\rho_2 C_{\text{emb}}^2 + 1) M + \eta M^3 := C_\star.
    \end{equation*}
    Since $C_\star$ is finite and depends only on $M$ and physical parameters, $\cN$ maps bounded sets to bounded sets.
\end{proof}
\begin{lemma}[hemicontinuity]\label{lem:hemicontinuity} $\cN$ is hemicontinuous, that is, for each $ \bzeta, \bxi \in \bH^4_{\mathrm{N}}(\vdiv,\Omega)$, the real mapping $J:\bbR \to \bbR $, $t \mapsto [\cN(\bzeta + t\bxi), \bxi]$ is continuous.
\end{lemma}
\begin{proof}
    From \eqref{eq:def_N_variational} we note that the operator $\cN$ consists of two continuous linear parts and one nonlinear part. The linear parts are clearly continuous with respect to $t \in \bbR$. Regarding the nonlinear part,  we can adapt the result given in \cite[Proposition 3]{Girault2008} (see also \cite[Lemma 3.2]{Caucao2020}) to show that the map $t \mapsto \displaystyle \eta \int_\Omega |\bzeta + t\bxi|^2 (\bzeta + t\bxi) \cdot \bxi$ is continuous for any $\bzeta, \bxi \in \bL^4(\Omega)$. Since $ \bH^4_{\mathrm{N}}(\vdiv,\Omega) \subset \bL^4(\Omega)$, the mapping $t \mapsto [\cN(\bzeta + t\bxi), \bxi]$ is continuous, and thus $\cN$ is hemicontinuous.
\end{proof}
\begin{lemma}[monotonicity]\label{lem:monotonicity}
   The operator $\mathcal{N}$ is monotone, that is
\[[\cN(\bzeta) - \cN(\bxi), \bzeta - \bxi] \ge 0 \quad \forall \bzeta, \bxi \in \bH^4_{\mathrm{N}}(\vdiv,\Omega).\]
\end{lemma}
\begin{proof}
    Let $\bzeta, \bxi \in \bH^4_{\mathrm{N}}(\vdiv,\Omega)$. From the definition of $\mathcal{N}$ (cf. \eqref{eq:def_N_variational}), we have
    \begin{align*}
        [\cN(\bzeta) - \cN(\bxi), \bzeta - \bxi] = (\varrho(\widehat{\bsigma})^{-1}(\bzeta - \bxi), \bzeta - \bxi) + (\vdiv(\bzeta - \bxi), \vdiv(\bzeta - \bxi)) + \eta (|\bzeta|^2\bzeta - |\bxi|^2\bxi, \bzeta - \bxi).
    \end{align*}
    For the first and second terms on the right-hand side, it is easy to see that
    \begin{equation}\label{eq:TDif}
        (\varrho(\widehat{\bsigma})^{-1}(\bzeta - \bxi), \bzeta - \bxi) \geq \rho_1 \|\bzeta - \bxi\|_{0,\Omega}^2 \ge 0 \qan   \|\vdiv(\bzeta - \bxi)\|_{0,\Omega}^2 \ge 0.
    \end{equation} 
    For the nonlinear term, we proceed as in \cite[Section 2.3]{Caucao2020}. More precisely, thanks to \cite[Lemma 5.1]{Glowinski1975}, there exists a constant $C_{\text{F}} > 0$, depending only on $\Omega$, such that
    \begin{equation*}
        (|\mathbf{u}|^2\mathbf{u} - |\mathbf{v}|^2\mathbf{v}) \cdot (\mathbf{u} - \mathbf{v}) \ge C_{\text{F}} |\mathbf{u} - \mathbf{v}|^4 \quad \forall \mathbf{u}, \mathbf{v} \in \bbR^d.
    \end{equation*}
    Setting $\mathbf{u} = \bzeta$ and $\mathbf{v} = \bxi$, we obtain
    \begin{equation}\label{eq:TF}
        (|\bzeta|^2\bzeta - |\bxi|^2\bxi, \bzeta - \bxi) \ge C_{F} \|\bzeta - \bxi\|_{0,4;\Omega}^4 \ge 0.
    \end{equation}
    Finally, combining \eqref{eq:TDif}, and \eqref{eq:TF}, we conclude that
    \begin{equation*}
        [\cN(\bzeta) - \cN(\bxi), \bzeta - \bxi] = \rho_1 \|\bzeta - \bxi\|_{0,\Omega}^2+\eta C_F \|\bzeta - \bxi\|_{0,4;\Omega}^4 + \|\vdiv(\bzeta - \bxi)\|_{0,\Omega}^2 \ge 0,
    \end{equation*}
    which proves the desired result.
\end{proof}
\begin{lemma}[coercivity]\label{lem:coercivity}
    The operator $\cN$ is coercive on $\bH^4_{\mathrm{N}}(\vdiv,\Omega)$, implying that
\[\lim_{\|\bzeta\|_{{4,\vdiv;\Omega}} \to \infty} \frac{[\cN(\bzeta), \bzeta]}{\|\bzeta\|_{{4,\vdiv;\Omega}}} = \infty.\]
\end{lemma}
\begin{proof}
    Let $\bzeta \in \bH^4_{\mathrm{N}}(\vdiv,\Omega)$. It is clear that
    \begin{equation*}
        [\cN(\bzeta), \bzeta] \ge \eta C_F \|\bzeta\|_{0,4;\Omega}^4 + \|\vdiv \bzeta\|_{0,\Omega}^2,
    \end{equation*}
    which means that
    \begin{equation*}
        \frac{[\cN(\bzeta), \bzeta]}{\|\bzeta\|_{{4,\vdiv;\Omega}}} \ge \frac{\eta C_F \|\bzeta\|_{0,4;\Omega}^4 + \|\vdiv \bzeta\|_{0,\Omega}^2}{\|\bzeta\|_{0,4;\Omega} + \|\vdiv \bzeta\|_{0,\Omega}}.
    \end{equation*}
    We now take $\|\bzeta\|_{{4,\vdiv;\Omega}} \to \infty$. This means that at least one of its norm terms, $\|\bzeta\|_{0,4;\Omega}$ or $\|\vdiv \bzeta\|_{0,\Omega}$, must tend to infinity. Since the exponents in the numerator strictly dominate those in the denominator for each respective term, the quotient diverges to infinity, proving that $\cN$ is coercive.
\end{proof}

To analyse the solvability of the coupled system \eqref{eq:weak2}, we simply decouple the problem into the poroelasticity equations (the first two equations in that system) and the remaining diffusion equation in mixed form. We separate the analysis for each problem in the following sub-section. 
\subsection{Unique solvability of decoupled Biot equations and nonlinear diffusion equations}\label{sec:wp-decoupled}

Having verified the properties discussed in the previous subsection, we are now in a position to establish the well-posedness of the decoupled sub-problems. To this end, the analysis of the Biot equations is based on the classical theory for perturbed saddle-point problems from \cite[Theorem 4.3.1]{boffi13}. Meanwhile, the unique solvability of the nonlinear diffusion equations is based on the abstract framework given by \cite[Corollary 2.2]{showalter1997}. We emphasize that the application of this abstract result to our specific setting is not immediate; rather, it is a consequence of the rigorous analysis of boundedness, hemicontinuity, monotonicity, and coercivity carried out in the preceding lemmas. By satisfying these sufficient conditions, the corollary guarantees the surjectivity of $\cN$. Consequently, due to our perturbed saddle-point structure, the well-posedness of the diffusion sub-problem can be readily deduced.

Firstly, let us assume that $\widehat{\varphi} \in \rL^2(\Omega)$ is prescribed. Then, we have the following result.

\begin{theorem}[well-posedness of the Biot equations]\label{th:well-posed-Biot}
 There exists a unique  $( \vec{\bsigma}, \vec{\bu}) \in \bbV\times \bQ$ such that 
 \begin{subequations}\label{eq:well-posed-Biot}
     \begin{align}
            A(\vec{\bsigma}, \vec{\btau})+\,B(\vec{\btau}, \vec{\bu}) &   \, =\, -D(\widehat{\varphi},\vec{\btau}) + F(\vec{\btau})  &&\forall \vec{\btau} \in \bbV,\\ 
    B(\vec{\bsigma},  \vec{\bv}) -\,C(\vec{\bu}, \vec{\bv}) &  \,=\, G(\vec{\bv})&&\forall \vec{\bv} \in \bQ,  
     \end{align}
 \end{subequations}
 and moreover 
 \[ \| ( \vec{\bsigma}, \vec{\bu}) \|_{\bbV{\times \bQ}} \lesssim  \frac{(1+\alpha d)\beta}{2\mu + d\lambda}\|\widehat{\varphi}\|_{0,\Omega} + \|\ff\|_{0,\Omega} + \|\bu_{\mathrm{D}}\|_{1/2,00;\Gamma_{\mathrm{D}}} + \|g\|_{0,\Omega} + \|p_{\mathrm{D}}\|_{1/2,00;\Gamma_{\mathrm{D}}}. \]
\end{theorem}
 \begin{proof}
   It follows from Lemma \ref{lemma:positive-definite-diagonal}, and equations \eqref{coercitivity-A}, \eqref{coercitivity-C}, and \eqref{eq:inf-sup-B} of Lemmas \ref{lemma:coercitivity-as} and \ref{lemma:inf-sup-B}, respectively, and a straightforward application of \cite[Theorem 4.3.1]{boffi13}.
\end{proof}  

Similarly, for a prescribed $\widehat{\bsigma} \in \bbH^{\mathrm{sym}}_{\mathrm{N}}(\bdiv,\Omega)$, we have the following result.
\begin{theorem}[well-posedness of the mixed perturbed nonlinear diffusion equations]\label{th:well-posed-diff}
Given $\cF \in \bH^{4}_{\mathrm{N}}(\vdiv,\Omega)'$ and $\cG \in \rL^2(\Omega)$, there exists a unique $( \bzeta, \varphi) \in  \bH^{4}_{\mathrm{N}}(\vdiv,\Omega) \times \rL^2(\Omega)$ such that
    \begin{align*}
      [\mathcal{A}_{\widehat{\bsigma}}(\bzeta), \bxi] + [\mathcal{B}(\bxi), \varphi] &= [\cF,\bxi] &&\forall \bxi \in \bH^{4}_{\mathrm{N}}(\vdiv,\Omega),\\
     [\mathcal{B}(\bzeta), \psi] - [\cC(\varphi), \psi] &= [\cG,\psi] &&\forall \psi \in \rL^2(\Omega),  
\end{align*}   
 furthermore,
    \begin{equation}\label{eq:apriori_bound}
        \|\bzeta\|_{{4,\vdiv;\Omega}} + \|\varphi\|_{0, \Omega} \le \mathcal{M}(\ell, \varphi_{\mathrm{D}}),
    \end{equation}
where 
    \begin{subequations}\begin{align}\label{eq:def_M_data}
        \mathcal{M}(\ell, \varphi_{\mathrm{D}})& := \left( \frac{2}{\eta} \right)^{1/4} \mathcal{H}(\ell, \varphi_{\mathrm{D}})^{1/4} + 4 \mathcal{H}(\ell, \varphi_{\mathrm{D}})^{1/2} + \|\ell\|_{0,\Omega},\quad \text{with}
\\ \label{eq:def_F_data}
        \mathcal{H}(\ell, \varphi_{\mathrm{D}}) &:= \frac{3}{4(2\eta)^{1/3}} \left( \max\{C_{\mathrm{emb}}, 1\} \|\varphi_{\mathrm{D}}\|_{1/2,00;\Gamma_{\mathrm{D}}} \right)^{4/3} + \frac{3}{2} \left( \max\{C_{\mathrm{emb}}, 1\} \|\varphi_{\mathrm{D}}\|_{1/2,00;\Gamma_{\mathrm{D}}} \right)^2 + \|\ell\|_{0,\Omega}^2.
    \end{align}\end{subequations}
\end{theorem}
\begin{proof}
     Let $\overline{\cF} \in \bH^{4}_{\mathrm{N}}(\vdiv,\Omega)'$ such that $[\overline{\cF}, \bxi] := [\cF, \bxi] + [\cC^{-1}\cG, \mathcal{B}(\bxi)]$. Thanks to Lemmas \ref{lem:boundedness}-\ref{lem:coercivity}, and a direct application of \cite[Corollary 2.2]{showalter1997}, the operator $\mathcal{N}$ is surjective. Thus, there exists $\bzeta \in \bH^{4}_{\mathrm{N}}(\vdiv,\Omega)$ such that $[\mathcal{N}(\bzeta), \bxi] = [\overline{\cF}, \bxi]$ for all $\bxi \in \bH^{4}_{\mathrm{N}}(\vdiv,\Omega)$.\\ 
    We define the concentration variable as $\varphi := \cC^{-1}(\mathcal{B}(\bzeta) - \cG) \in \rL^2(\Omega)$. 
    Let us prove that the pair $(\bzeta, \varphi)$ satisfies the original system. Indeed, from the definition of $\varphi$, we have $\cC(\varphi) = \mathcal{B}(\bzeta) - \cG$, which directly implies that
    \begin{equation*}
        [\mathcal{B}(\bzeta), \psi] - [\cC(\varphi), \psi] = [\cG, \psi] \quad \forall \psi \in \rL^2(\Omega),
    \end{equation*}
    satisfying the second equation of the system.
   For the first equation, we use the fact that $[\mathcal{N}(\bzeta), \bxi] = [\overline{\cF}, \bxi]$ and expand both sides using the definitions of $\mathcal{N}$ and $\overline{\cF}$, and rearrange terms, giving 
    \begin{equation*}
        [\mathcal{A}_{\widehat{\bsigma}}(\bzeta), \bxi] + [\mathcal{B}(\bxi), \cC^{-1}(\mathcal{B}(\bzeta))] = [\cF, \bxi] + [\mathcal{B}(\bxi), \cC^{-1}(\cG)].
    \end{equation*}
    Thus, using the definition of the concentration variable, we find that
    \begin{equation*}
        [\mathcal{A}_{\widehat{\bsigma}}(\bzeta), \bxi] + [\mathcal{B}(\bxi), \varphi] = [\cF, \bxi] \quad \forall \bxi \in \bH^{4}_{\mathrm{N}}(\vdiv,\Omega).
    \end{equation*}
    Hence, existence is proved.
    
    For the uniqueness we suppose that there exist two solutions $(\bzeta_1, \varphi_1)$ and $(\bzeta_2, \varphi_2)$ in $\bH^{4}_{\mathrm{N}}(\vdiv,\Omega) \times \rL^2(\Omega)$ satisfying the mixed system. Subtracting their respective equations yields
    \begin{align*}
        [\mathcal{A}_{\widehat{\bsigma}}(\bzeta_1) - \mathcal{A}_{\widehat{\bsigma}}(\bzeta_2), \bxi] + [\mathcal{B}(\bxi), \varphi_1 - \varphi_2] &= 0 \quad \forall \bxi \in \bH^{4}_{\mathrm{N}}(\vdiv,\Omega),\\ 
        [\mathcal{B}(\bzeta_1 - \bzeta_2), \psi] - [\cC(\varphi_1 - \varphi_2), \psi] &= 0 \quad \forall \psi \in \rL^2(\Omega). 
    \end{align*}
    Choosing the test functions $\bxi = \bzeta_1 - \bzeta_2$ and $\psi = \varphi_1 - \varphi_2$, and  subtracting the second equation from the first one,  we obtain
    \begin{equation*}
        [\mathcal{A}_{\widehat{\bsigma}}(\bzeta_1) - \mathcal{A}_{\widehat{\bsigma}}(\bzeta_2), \bzeta_1 - \bzeta_2] + [\cC(\varphi_1 - \varphi_2), \varphi_1 - \varphi_2] = 0.
    \end{equation*}
    Next, since the operators satisfy $[\mathcal{A}_{\widehat{\bsigma}}(\bzeta_1) - \mathcal{A}_{\widehat{\bsigma}}(\bzeta_2), \bzeta_1 - \bzeta_2] \ge \eta C_F \|\bzeta_1 - \bzeta_2\|_{0,4;\Omega}^4$, and $[\cC(\varphi_1 - \varphi_2), \varphi_1 - \varphi_2] = \|\varphi_1 - \varphi_2\|_{0,\Omega}^2$, it follows that
    \begin{equation*}
        \eta C_F \|\bzeta_1 - \bzeta_2\|_{0,4;\Omega}^4 + \|\varphi_1 - \varphi_2\|_{0,\Omega}^2 \le 0.
    \end{equation*}
 Consequently, $\bzeta_1 = \bzeta_2$ and $\varphi_1 = \varphi_2$, proving the uniqueness of the full system.
 
 Finally, we establish the \emph{a priori} bound. We test the first equation of the system with $\bxi = \bzeta$ and the second with $\psi = -\varphi$, to obtain that
    \begin{equation*}
        [\mathcal{A}_{\widehat{\bsigma}}(\bzeta), \bzeta] + [\cC(\varphi), \varphi] = [\cF, \bzeta] - [\cG, \varphi].
    \end{equation*}
    Using that $[\mathcal{A}_{\widehat{\bsigma}}(\bzeta), \bzeta] \ge \eta \|\bzeta\|_{0,4;\Omega}^4$ and the definition of $\cC$, we obtain
    \begin{equation}\label{eq:energy_bound1}
        \eta \|\bzeta\|_{0,4;\Omega}^4 + \|\varphi\|_{0, \Omega}^2 \le \|\cF\|_{\bH^{4}_{\mathrm{N}}(\vdiv,\Omega)'} \|\bzeta\|_{{{4,\vdiv;\Omega}}} + \|\cG\|_{0, \Omega} \|\varphi\|_{0, \Omega}.
    \end{equation}
    Applying Young's inequality to the last term, we have $\|\cG\|_{0, \Omega} \|\varphi\|_{0, \Omega} \le \frac{1}{2}\|\cG\|_{0, \Omega}^2 + \frac{1}{2}\|\varphi\|_{0, \Omega}^2$, which yields
    \begin{equation}\label{eq:energy_bound2}
        \eta \|\bzeta\|_{0,4;\Omega}^4 + \frac{1}{2}\|\varphi\|_{0, \Omega}^2 \le \|\cF\|_{\bH^{4}_{\mathrm{N}}(\vdiv,\Omega)'} \|\bzeta\|_{{4,\vdiv;\Omega}} + \frac{1}{2}\|\cG\|_{0, \Omega}^2.
    \end{equation}
    On the other hand, from the second equation of the system, $[\mathcal{B}(\bzeta), \psi] = [\cC(\varphi), \psi] + [\cG, \psi]$. Choosing  $\psi = -\vdiv \bzeta$, we get
    \begin{equation*}
        \|\vdiv \bzeta\|_{0,\Omega}^2  \le (\|\varphi\|_{0, \Omega} + \|\cG\|_{0, \Omega}) \|\vdiv \bzeta\|_{0,\Omega},
    \end{equation*}
    which yields 
    \[\|\vdiv \bzeta\|_{0,\Omega} \le \|\varphi\|_{0, \Omega} + \|\cG\|_{0, \Omega}.\]
    Recalling that $\|\bzeta\|_{{4,\vdiv;\Omega}} = \|\bzeta\|_{0,4;\Omega} + \|\vdiv \bzeta\|_{0,\Omega}$, we have
    \begin{equation}\label{eq:norm_V_bound}
        \|\bzeta\|_{{4,\vdiv;\Omega}} \le \|\bzeta\|_{0,4;\Omega} + \|\varphi\|_{0, \Omega} + \|\cG\|_{0, \Omega}.
    \end{equation}
    Substituting \eqref{eq:norm_V_bound} back into \eqref{eq:energy_bound2} yields
    \begin{equation*}
        \eta \|\bzeta\|_{0,4;\Omega}^4 + \frac{1}{2}\|\varphi\|_{0, \Omega}^2 \le \|\cF\|_{\bH^{4}_{\mathrm{N}}(\vdiv,\Omega)'} \|\bzeta\|_{0,4;\Omega} + \|\cF\|_{\bH^{4}_{\mathrm{N}}(\vdiv,\Omega)'} \|\varphi\|_{0, \Omega} + \|\cF\|_{\bH^{4}_{\mathrm{N}}(\vdiv,\Omega)'} \|\cG\|_{0, \Omega} + \frac{1}{2}\|\cG\|_{0, \Omega}^2.
    \end{equation*}
    We now apply Young's inequalities to the terms involving $\|\cF\|_{\bH^{4}_{\mathrm{N}}(\vdiv,\Omega)'}$ as follows. For the first term, applying $ab \le \frac{\epsilon^4}{4} a^4 + \frac{3}{4\epsilon^{4/3}} b^{4/3}$ with $\epsilon^4 = 2\eta$ yields
    \begin{equation*}
        \|\cF\|_{\bH^{4}_{\mathrm{N}}(\vdiv,\Omega)'} \|\bzeta\|_{0,4;\Omega} \le \frac{\eta}{2} \|\bzeta\|_{0,4;\Omega}^4 + \frac{3}{4(2\eta)^{1/3}} \|\cF\|_{\bH^{4}_{\mathrm{N}}(\vdiv,\Omega)'}^{4/3}.
    \end{equation*}
    For the remaining terms, we get
    \begin{align*}
        \|\cF\|_{\bH^{4}_{\mathrm{N}}(\vdiv,\Omega)'} \|\varphi\|_Q \le \frac{1}{4} \|\varphi\|_{0, \Omega}^2 + \|\cF\|_{\bH^{4}_{\mathrm{N}}(\vdiv,\Omega)'}^2, \qan 
        \|\cF\|_{\bH^{4}_{\mathrm{N}}(\vdiv,\Omega)'} \|\cG\|_{0, \Omega} \le \frac{1}{2}\|\cF\|_{\bH^{4}_{\mathrm{N}}(\vdiv,\Omega)'}^2 + \frac{1}{2}\|\cG\|_{0, \Omega}^2.
    \end{align*}
     Inserting these bounds into the energy inequality and absorbing the terms $\frac{\eta}{2} \|\bzeta\|_{0,4;\Omega}^4$ and $\frac{1}{4} \|\varphi\|_{0, \Omega}^2$ into the left-hand side, we arrive at
    \begin{equation}\label{eq:pre_M_bound}
        \frac{\eta}{2} \|\bzeta\|_{0,4;\Omega}^4 + \frac{1}{4}\|\varphi\|_{0, \Omega}^2 \le  \frac{3}{4(2\eta)^{1/3}} \|\cF\|_{\bH^{4}_{\mathrm{N}}(\vdiv,\Omega)'}^{4/3} + \frac{3}{2} \|\cF\|_{\bH^{4}_{\mathrm{N}}(\vdiv,\Omega)'}^2 + \|\cG\|_{0, \Omega}^2.
    \end{equation}
    To complete the bound in terms of the physical data, we recall that
    \begin{equation*}
        \|\cG\|_{0, \Omega} \le \|\ell\|_{0,\Omega} \qan 
        \|\cF\|_{\bH^{4}_{\mathrm{N}}(\vdiv,\Omega)'} \le \max\{C_{\mathrm{emb}}, 1\} \|\varphi_{\mathrm{D}}\|_{1/2,00;\Gamma_{\mathrm{D}}}.
    \end{equation*}
    Substituting these data bounds into the right-hand side of \eqref{eq:pre_M_bound} yields $\mathcal{H}(\ell, \varphi_{\mathrm{D}})$ as defined in \eqref{eq:def_F_data}.
    From this, it is clear that
    \begin{equation*}
        \|\bzeta\|_{0,4;\Omega} \le \left( \frac{2}{\eta}  \mathcal{H}(\ell, \varphi_{\mathrm{D}}) \right)^{1/4} \quad \text{and} \quad \|\varphi\|_{0, \Omega} \le 2  \mathcal{H}(\ell, \varphi_{\mathrm{D}})^{1/2}.
    \end{equation*}
    Finally, recalling from \eqref{eq:norm_V_bound} that $\|\vdiv \bzeta\|_{0,\Omega} \le \|\varphi\|_{0, \Omega} + \|\cG\|_{0, \Omega} \le \|\varphi\|_{0, \Omega} + \|\ell\|_{0,\Omega}$, we have
    \begin{align*}
        \|\bzeta\|_{{4,\vdiv;\Omega}} + \|\varphi\|_{0, \Omega} &= \|\bzeta\|_{0,4;\Omega} + \|\vdiv \bzeta\|_{0,\Omega} + \|\varphi\|_{0, \Omega} \\[1ex]
        &\le \|\bzeta\|_{0,4;\Omega} + 2\|\varphi\|_{0, \Omega} + \|\ell\|_{0,\Omega} \\[1ex]
        &\le \left( \frac{2}{\eta} \right)^{1/4}  \mathcal{H}(\ell, \varphi_{\mathrm{D}})^{1/4} + 4  \mathcal{H}(\ell, \varphi_{\mathrm{D}})^{1/2} + \|\ell\|_{0,\Omega}.
    \end{align*}
    This establishes the \emph{a priori} stability estimate \eqref{eq:apriori_bound} defined by $\mathcal{M}(\ell, \varphi_{\mathrm{D}})$, thus completing the proof.
\end{proof}

\subsection{Solvability of the coupled problem via fixed-point theory}\label{sec:wp}
We define the following map
\[ 
\cJ^{\mathrm{Biot}}:\rL^2(\Omega)\to \bbV \times \bQ, \quad \widehat{\varphi} \mapsto \cJ^{\mathrm{Biot}}(\widehat{\varphi})= \left(\left(\cJ^{\mathrm{Biot}}_1(\widehat{\varphi}), \cJ^{\mathrm{Biot}}_2(\widehat{\varphi})\right),\cJ^{\mathrm{Biot}}_3(\widehat{\varphi})\right) := ( (\bsigma, p),  \vec{\bu}) = ( \vec{\bsigma}, \vec{\bu}), 
\]
where $( \vec{\bsigma}, \vec{\bu}) \in \bbV \times \bQ$ is the unique solution of the poroelasticity equations as stated in Theorem~\ref{th:well-posed-Biot}.

In turn, taking into account the reaction perturbation, we define the solution operator associated with the altered diffusion equations as 
\[ 
\cJ^{\mathrm{diff}}: \bbH^{\mathrm{sym}}_{\mathrm{N}}(\bdiv,\Omega) \to  \bH^4_\mathrm{N}(\vdiv,\Omega)\times \rL^2(\Omega), \quad \widehat{\bsigma} \mapsto \cJ^{\mathrm{diff}}(\widehat{\bsigma}) = \left(\cJ^{\mathrm{diff}}_1(\widehat{\bsigma}),\cJ^{\mathrm{diff}}_2(\widehat{\bsigma})\right) := ( \bzeta, \varphi),
\]
where $(\bzeta, \varphi)$ is the unique solution of the nonlinear diffusion equations as stated in Theorem~\ref{th:well-posed-diff}. 
These maps are well-defined and so it is the following one
\begin{equation}\label{Operator-J}
\mathcal{J}: \rL^{2}(\Omega)  \to  \rL^{2}(\Omega), \quad \widehat{\varphi} \mapsto \cJ(\widehat{\varphi}):= \cJ^{\mathrm{diff}}_2\left(\cJ^{\mathrm{Biot}}_1(\widehat{\varphi})\right).
\end{equation}

Finding a fixed point $\varphi$ of $\cJ$ is therefore equivalent to solving the fully coupled continuous problem \eqref{eq:weak2}. For this we use the  Schauder fixed-point theorem and start by considering, for a generic $r>0$, the following closed ball
\[ 
\rW : = \{ \widehat{\varphi} \in \rL^{2}(\Omega): \|\widehat{\varphi}\|_{0,\Omega}  \leq r \},
\]
and proceed next to show that {$\cJ$ maps $\rW$ into itself and that $\mathcal{J}$ is continuous}.
\begin{lemma}[ball mapping property]\label{lemma:ball_mapping}
    Under the assumption
    \begin{equation}\label{eq:data_assumption}
        \mathcal{M}(\ell, \varphi_{\mathrm{D}}) \le r,
    \end{equation}
    where $\mathcal{M}(\ell, \varphi_{\mathrm{D}})$ is the continuous data bound defined in \eqref{eq:def_M_data}, there holds $\mathcal{J}(\mathrm{W}) \subseteq \mathrm{W}$.
\end{lemma}
\begin{proof}
    Given $\widehat{\varphi} \in \mathrm{W}$, by the \emph{a priori} stability estimate \eqref{eq:apriori_bound} and assumption \eqref{eq:data_assumption}, we have
    \begin{equation*}
        \|\mathcal{J}(\widehat{\varphi})\|_{0,\Omega} = \| \cJ^{\mathrm{diff}}_2\left(\cJ^{\mathrm{Biot}}_1(\widehat{\varphi})\right)\|_{0,\Omega} \le \mathcal{M}(\ell, \varphi_{\mathrm{D}}) \le r,
    \end{equation*}
    which means that $\mathcal{J}(\mathrm{W}) \subseteq \mathrm{W}$.
\end{proof}

We continue the analysis with the continuity property of ${\mathcal{J}}$. To this end, we recall that Theorem \ref{th:well-posed-Biot} establishes the existence of a positive constant $C_{\mathrm{B}}$ such that 
\begin{align}\label{global-inf-sup-Jhat}
&\sup_{\stackrel{(\vec{\btau},\vec{\bv}) \in \bbV\times \bQ}{ (\vec{\btau},\vec{\bv}) \neq \mathbf{0}}}\frac{A(\vec{\bzeta},\vec{\btau})+B(\vec{\btau},\vec{\bw})+ B(\vec{\bzeta},\vec{\bv})-C(\vec{\bw},\vec{\bv})}{\|(\vec{\btau},\vec{\bv})\|_{\bbV\times \bQ}}\geq \,C_{\mathrm{B}}\|(\vec{\bzeta},\vec{\bw})\|_{\bbV\times \bQ} \ \forall\, (\vec{\bzeta},\vec{\bw})\in \bbV\times \bQ.
\end{align}
\begin{lemma}[Hölder continuity]\label{lemma:continuity-J} 
 There exists a positive constant $\mathcal{L}_{\mathcal{J}}$ such that 
 \begin{equation}\label{continuity-J}
  \|\mathcal{J}(\varphi_{1})-\mathcal{J}(\varphi_{2})\|_{0,\Omega} \leq \mathcal{L}_{\mathcal{J}}\|\varphi_{1}-\varphi_{2}\|_{0, \Omega}^{2/3} \qquad \forall \varphi_1,\varphi_2 \in \rL^2(\Omega).
 \end{equation} 
 \end{lemma}
\begin{proof}
    Given $\varphi_{1}, \varphi_{2} \in \rL^2(\Omega)$, we let $\cJ^{\mathrm{Biot}}(\varphi_{1})=(\vec{\bsigma}_{1}, \vec{\bu}_{1}) \in \bbV\times \bQ$ and $\cJ^{\mathrm{Biot}}(\varphi_{2})=(\vec{\bsigma}_{2}, \vec{\bu}_{2}) \in \bbV\times \bQ$ be the unique solutions of \eqref{eq:well-posed-Biot}. Then, applying the inf-sup condition  \eqref{global-inf-sup-Jhat} with $(\vec{\bzeta}, \vec{\bw})=(\vec{\bsigma}_{1}-\vec{\bsigma}_{2},  \vec{\bu}_{1}- \vec{\bu}_{2})$, it follows that
    \begin{align*}
     \displaystyle C_{\mathrm{B}}\|(\vec{\bsigma}_{1}-\vec{\bsigma}_{2},\vec{\bu}_{1}- \vec{\bu}_{2})\|_{\bbV\times \bQ} &\leq \sup_{\stackrel{(\vec{\btau},\vec{\bv}) \in \bbV\times \bQ}{ (\vec{\btau},\vec{\bv}) \neq \mathbf{0}}}\frac{A(\vec{\bsigma}_{1}-\vec{\bsigma}_{2},\vec{\btau})+B(\vec{\btau},\vec{\bu}_{1}- \vec{\bu}_{2})+ B(\vec{\bsigma}_{1}-\vec{\bsigma}_{2},\vec{\bv})-C(\vec{\bu}_{1}- \vec{\bu}_{2},\vec{\bv})}{\|(\vec{\btau},\vec{\bv})\|_{\bbV\times \bQ}}\\
     & = \sup_{\stackrel{(\vec{\btau},\vec{\bv}) \in \bbV\times \bQ}{ (\vec{\btau},\vec{\bv}) \neq \mathbf{0}}}\frac{D({\varphi}_{1},\vec{\btau})-D({\varphi}_{2},\vec{\btau})}{\|(\vec{\btau},\vec{\bv})\|_{\bbV\times \bQ}} \\
     & \leq \frac{(1+\alpha d)\beta}{2\mu + d\lambda}\|\varphi_{1}-\varphi_{2}\|_{0,\Omega}.
\end{align*}
The bound above implies that
\begin{align}\label{continuity-hat-biot}
\|\cJ_{1}^{\mathrm{Biot}}(\varphi_{1})-\cJ_{1}^{\mathrm{Biot}}(\varphi_{2})\|_{\bdiv, \Omega}
\leq  \mathcal{L}_{\mathrm{Biot}} \|\varphi_{1}-\varphi_{2}\|_{0,\Omega}, \quad \text{with} \quad \mathcal{L}_{\mathrm{Biot}} := \frac{(1+\alpha d)\beta}{C_{\mathrm{B}}(2\mu + d\lambda)}.
\end{align}
On the other hand, given ${\bsigma}_{1}, {\bsigma}_{2} \in \bbH^{\mathrm{sym}}_{\mathrm{N}}(\bdiv,\Omega)$, we let $\cJ^{\mathrm{diff}}({\bsigma}_{1})=(\bzeta_{1}, \varphi_{1})$ and $\cJ^{\mathrm{diff}}({\bsigma}_{2})=(\bzeta_{2}, \varphi_{2})$ in $\bH^{4}_{\mathrm{N}}(\vdiv,\Omega)\times \rL^2(\Omega)$ be the unique solutions of the nonlinear diffusion sub-problem. This means that their respective differences satisfy the error system
    \begin{subequations}\label{eq:diff_error_system}
    \begin{align}
        [\cA_{{\bsigma}_{1}}(\bzeta_{1}) - \cA_{{\bsigma}_{2}}(\bzeta_{2}), \bxi] + [\cB(\bxi), \varphi_{1}-\varphi_{2}] &= 0 \qquad \forall \bxi \in \bH_{\mathrm{N}}^4(\vdiv,\Omega), \label{eq:diff_error_1} \\
        [\cB(\bzeta_{1}-\bzeta_{2}),  \psi] - [\cC(\varphi_{1}-\varphi_{2}), \psi] &= 0 \qquad \forall \psi \in \rL^2(\Omega). \label{eq:diff_error_2}
    \end{align}
    \end{subequations}
To exploit the saddle-point structure, we test the first equation \eqref{eq:diff_error_1} with $\bxi = \bzeta_1 - \bzeta_2$ and the second equation \eqref{eq:diff_error_2} with $\psi = -(\varphi_1 - \varphi_2)$. Adding both equations we obtain
\begin{align*}
    [\cA_{{\bsigma}_1}(\bzeta_1) - \cA_{{\bsigma}_2}(\bzeta_2), \bzeta_1 - \bzeta_2] + \|\varphi_1 - \varphi_2\|_{0,\Omega}^2 = 0.
\end{align*}
Thus, adding and subtracting $a_{{\bsigma}_1}(\bzeta_2, \bzeta_1 - \bzeta_2)$, utilising the strong monotonicity of the nonlinear operator, we obtain
\begin{align*}
   \eta C_F \|\bzeta_1 - \bzeta_2\|_{0,4;\Omega}^4 + \|\varphi_1 - \varphi_2\|_{0,\Omega}^2 &\leq [\cA_{{\bsigma}_2}(\bzeta_2) - \cA_{{\bsigma}_1}(\bzeta_2), \bzeta_1 - \bzeta_2] \nonumber \\
    &= \int_{\Omega} \left( \varrho({\bsigma}_2)^{-1} - \varrho({\bsigma}_1)^{-1} \right) \bzeta_2 \cdot (\bzeta_1 - \bzeta_2) .
\end{align*}
Next, applying Hölder's inequality, the Lipschitz continuity of $\varrho^{-1}$, the a priori bound $\|\bzeta_2\|_{0,4;\Omega} \leq \mathcal{M}(\ell, \varphi_{\mathrm{D}})$ given by \eqref{eq:apriori_bound}, and Young's inequality $\displaystyle ab \leq \frac{\epsilon^4}{4}a^4 + \frac{3}{4\epsilon^{4/3}}b^{4/3}$ choosing $\epsilon^4 = 2\eta C_F$, we find that
\begin{align*}
   \eta C_F \|\bzeta_1 - \bzeta_2\|_{0,4;\Omega}^4 + \|\varphi_1 - \varphi_2\|_{0,\Omega}^2 &
    \leq L_\varrho \mathcal{M}(\ell, \varphi_{\mathrm{D}}) \|{\bsigma}_1 - {\bsigma}_2\|_{0,\Omega} \|\bzeta_1 - \bzeta_2\|_{0,4;\Omega}\\
    & \leq \frac{\eta C_F}{2} \|\bzeta_1 - \bzeta_2\|_{0,4;\Omega}^4 + \frac{3}{4(2\eta C_F)^{1/3}} \left( L_\varrho \mathcal{M}(\ell, \varphi_{\mathrm{D}}) \|{\bsigma}_1 - {\bsigma}_2\|_{0,\Omega} \right)^{4/3}.
\end{align*}
It follows immediately that
\begin{align}\label{continuity-tilde-assited}
\|\cJ_{2}^{\mathrm{diff}}({\bsigma}_{1})-\cJ_{2}^{\mathrm{diff}}({\bsigma}_{2})\|_{0,\Omega} \leq \mathcal{L}_{\mathrm{Diff}} \|{\bsigma}_{1}-{\bsigma}_{2}\|_{0, \Omega}^{2/3}, \quad \text{with} \quad \mathcal{L}_{\mathrm{Diff}} := \frac{\sqrt{3} (L_\varrho \mathcal{M}(\ell, \varphi_{\mathrm{D}}))^{2/3}}{2(2\eta C_F)^{1/6}}.
\end{align}
Then, the global estimate in \eqref{continuity-J} follows directly from the fact that $\mathcal{J} = \cJ_{2}^{\mathrm{diff}} \circ \cJ_{1}^{\mathrm{Biot}}$, the Hölder continuity of $\cJ_{2}^{\mathrm{diff}}$ (cf. \eqref{continuity-tilde-assited}), and the Lipschitz continuity of $\cJ_{1}^{\mathrm{Biot}}$ (cf. \eqref{continuity-hat-biot}):
\begin{align*}
    \|\mathcal{J}(\varphi_{1})-\mathcal{J}(\varphi_{2})\|_{0,\Omega} 
    &= \|\cJ_{2}^{\mathrm{diff}}(\cJ_{1}^{\mathrm{Biot}}(\varphi_{1})) - \cJ_{2}^{\mathrm{diff}}(\cJ_{1}^{\mathrm{Biot}}(\varphi_{2}))\|_{0,\Omega} \\
    &= \|\cJ_{2}^{\mathrm{diff}}({\bsigma}_{1})-\cJ_{2}^{\mathrm{diff}}({\bsigma}_{2})\|_{0,\Omega} \\
    &\leq \mathcal{L}_{\mathrm{Diff}} \|{\bsigma}_{1}-{\bsigma}_{2}\|_{0, \Omega}^{2/3} 
    \leq \mathcal{L}_{\mathrm{Diff}} \|\cJ_{1}^{\mathrm{Biot}}(\varphi_{1})-\cJ_{1}^{\mathrm{Biot}}(\varphi_{2})\|_{\bdiv, \Omega}^{2/3} \\
    & \leq \mathcal{L}_{\mathrm{Diff}} \left( \mathcal{L}_{\mathrm{Biot}}\|\varphi_1 - \varphi_2\|_{0,\Omega} \right)^{2/3},
\end{align*}
which yields exactly $\|\mathcal{J}(\varphi_{1})-\mathcal{J}(\varphi_{2})\|_{0,\Omega} \leq \mathcal{L}_{\mathcal{J}} \|\varphi_1 - \varphi_2\|_{0,\Omega}^{2/3}$, where
 $\displaystyle \mathcal{L}_{\mathcal{J}}:= \mathcal{L}_{\mathrm{Diff}} \mathcal{L}_{\mathrm{Biot}}^{2/3}$.    
\end{proof}

 Owing to the above analysis, we now establish the main result of this section.
\begin{theorem}[existence of solution for the fully-coupled continuous problem]\label{th:J-unique-solvability-fixed}
 Suppose that $\mathcal{M}(\ell, \varphi_{\mathrm{D}}) \leq r$.  Then, the coupled problem \eqref{eq:weak2} has at least one solution $(\vec{\bsigma}, \vec{\bu}) \in \bbV \times \bQ$ and $( \bzeta, \varphi) \in   \bH^{4}_{\mathrm{N}}(\vdiv,\Omega) \times \rL^2(\Omega)$. Moreover, we have 
\begin{align}
    \|(\vec{\bsigma}, \vec{\bu},\bzeta, \varphi)\|_{\bbV \times \bQ \times \bH^{4}_{\mathrm{N}}(\vdiv,\Omega) \times \rL^2(\Omega)}  &  \lesssim \biggl(\frac{(1+\alpha d)\beta}{2\mu + d\lambda}+1\biggr)\mathcal{M}(\ell, \varphi_{\mathrm{D}}) + \|\ff\|_{0,\Omega}  + \|g\|_{0,\Omega} \notag \\ &\quad + \|\bu_{\mathrm{D}}\|_{1/2,00;\Gamma_{\mathrm{D}}} +\, \|p_{\mathrm{D}}\|_{1/2,00;\Gamma_{\mathrm{D}}}. \label{estimate su}
\end{align}
\end{theorem}
\begin{proof}
We first recall that Lemma \ref{lemma:ball_mapping} guarantees that $\mathcal{J}$ maps $\rW$ into itself. Then, bearing in mind the Hölder-continuity of $\mathcal{J} : \rW \to \rW$ given by Lemma \ref{lemma:continuity-J}, a direct application of the Schauder fixed-point Theorem yields the existence of at least one fixed point $\varphi \in \rW$ of this operator, and hence at least one solution to \eqref{eq:weak2}. In addition, the {\it a priori} estimates provided by Theorem \ref{th:well-posed-Biot} and \ref{th:well-posed-diff} yield \eqref{estimate su}, which completes the proof.
\end{proof}

\begin{remark}[on the uniqueness of the continuous solution]Theorem \ref{th:J-unique-solvability-fixed} establishes the existence of at least one weak solution to problem \eqref{eq:weak2}, but does not assert uniqueness. This is due to the highly nonlinear nature of the \gls{sad} operator, which renders the fixed-point operator $\mathcal{J}$ H\"{o}lder continuous rather than Lipschitz. Far from being an isolated limitation of our model, this analytical hurdle is a well-documented characteristic of such coupled problems; indeed, similar mathematical challenges are encountered in recent literature, such as in \cite{Allendes2023}, where only global existence is established for coupled Darcy--Forchheimer equations, reflecting the fact that unconditional uniqueness for this class of systems remains a prominent open question.
\end{remark}

\section{Virtual element discretisation}\label{sec:discrete}
This section introduces the \gls{vem}-based discrete formulation for the fully-coupled problem \eqref{eq:weak1}-\eqref{eq:weak2}. We employ a \gls{vem} for both 2D/3D linear elasticity problems based on the \gls{hr} variational principle (cf. \eqref{eq:poroelast-1}-\eqref{eq:poroelast-3}). The main advantage of this type of VE space is that it allows the symmetry of the discrete tensor to be enforced strongly. Moreover, its definition is unified in both 2D and 3D, taking into account that in 3D, facets correspond to the faces of the polyhedral element, while in 2D, they correspond to the edges of the polygonal element. On the other hand, the \gls{vem} employed here for mixed second order elliptic problems (corresponding to the equations \eqref{eq:poroelast-4}-\eqref{eq:poroelast-5} and \eqref{eq:poroelast-6}-\eqref{eq:poroelast-7})   requires separate definitions in 2D and 3D. In addition, we introduce appropriate polynomial projection, interpolation and stabilisation operators to guarantee computability of the discrete formulation. 

We recall that the detailed construction, unisolvence in terms of the corresponding {Degrees of Freedom} (DoFs), additional properties of the VE spaces; as well with the properties of the polynomial spaces and the computability of the polynomial projection operators in terms of the (respective) DoFs presented in this section are provided in \cite{artioli18,visinoni24,beirao16,beirao16b}.

\paragraph{Assumptions on the mesh.} Let $\cT_h$ be a collection of polygonal/polyhedral meshes on $\Omega$ and $\cF_h$ be the set of all facets in 3D (edges in 2D). The diameter of a polygon/polyhedron $K$ is represented by $h_K$ and the length/area of a facet $f$ is represented by $h_f$. The maximum diameter of elements in $\cT_h$ is represented by $h$. It is assumed that there exists a uniform positive constant {$\gamma$} such that
\begin{enumerate}[label={\textbf{(A\arabic*)}}, align=left, leftmargin=*, labelwidth=!, labelsep=1em]
    \item \label{A1} Every element $K$ has a star-shaped interior with respect to a ball with a radius greater than {$\gamma h_K$}. 
    \item \label{A2} Every facet $f\in \partial K$ has a star-shaped interior with respect to a ball with a radius greater than {$\gamma h_K$}.
    \item \label{A3} Every facet $f\in \partial K$ satisfies the inequality {$h_f \geq \gamma h_K$}.
\end{enumerate}

\begin{remark}
    {Regarding practical applications to complex brain geometries, \gls{vem} provides an  advantage over standard conformal \gls{fem}—where capturing intricate cortical folds typically demands excessive local mesh refinement. \gls{vem}   supports element agglomeration and polygonal boundary fitting, allowing fine initial meshes that resolve detailed anatomical boundaries to be merged into larger polytopal elements. This procedure preserves the face-to-element size ratio $\gamma$ in \ref{A3} and maintains optimal convergence rates without inflating the global element count. As an example of an agglomeration framework, we refer the reader to the recent work \cite{antonietti26}.}
\end{remark}

\paragraph{Polynomial spaces.} In this paper, we consider an arbitrary polynomial degree $k\geq 1$. The space of polynomials of total degree at most $k$ defined locally on $K\in \cT_h$ (or facet $f\in \cF_h$) is represented by $\rP_k (K)$, and its vector and tensor counterparts are represented by $\bP_k(K)$ and $\bbP_k(K)$, respectively. We also consider the standard notation $\rP_{-1}(K)=\{0\}$. 

The spaces $\bG_k(K):= \nabla (\rP_{k+1}(K))$ and $\bG_k^\oplus(K)$ denote the gradients of polynomials of degree $\leq k+1$ on $K$ and the complement of the space $\bG_k(K)$ in the vector polynomial space $\bP_k(K)$ such that the direct sum $\bP_k(K) = \bG_k(K) \oplus  \bG_k^\oplus(K)$ holds, respectively. In particular, we select $\bG_k^\oplus(K)= \bx^\perp \rP_{k-1}(K)$ (resp. $\bG_k^\oplus(K):= \bx\wedge(\bP_{k-1}(K)$) where $\bx^\perp = (x_2,-x_1)^{\tt t}$ in 2D (resp. $\bx := (x_1,x_2,x_3)^{\tt t}$ and $\wedge$ the usual external product in 3D). 

Let $\bx_K = (x_{1,K},x_{2,K})^{\tt t}$ (resp. $\bx_K = (x_{1,K},x_{2,K},x_{3,K})^{\tt t}$) denote the barycentre of $K$ and let $\bM_k(K)$ be the set of vector scaled monomials as
\begin{align*}
    \bM_k(K):=\left\{ \left( \frac{\bx-\bx_K}{h_K} \right)^{\boldsymbol{\alpha}}\in \bP_k(K):\ 0\leq |\boldsymbol{\alpha}|\leq k \right\},
\end{align*}
where $\boldsymbol{\alpha}=(\alpha_1,\alpha_2)^{\tt t}$ (resp. $\boldsymbol{\alpha}=(\alpha_1,\alpha_2,\alpha_3)^{\tt t}$) is a non-negative multi-index with $|\boldsymbol{\alpha}|=\alpha_1+\alpha_2$ and $\bx^{\boldsymbol{\alpha}}=x_1^{\alpha_1}x_2^{\alpha_2}$ in 2D (resp. $|\boldsymbol{\alpha}|=\alpha_1+\alpha_2+\alpha_3$ and $\bx^{\boldsymbol{\alpha}}=x_1^{\alpha_1}x_2^{\alpha_2}x_3^{\alpha_3}$ in 3D), with analogous definition for the scalar and tensor version $\rM_k$ and $\bbM_k$. Notice that the polynomial decompositions presented before hold also in terms of the scaled monomial. For example, in the 2D case, we can take the sets $\bG_{k}(K)$ and $\bG_{k}^\oplus(K)$ as $\bM_{k}^{\nabla}(K):=\nabla\rM_{k+1}(K)\setminus\{\mathbf{0}\}$ and $\bM_{k}^{\oplus}(K):=\bm^\perp \rM_{k-1}(K)$, with $\bm^\perp := (\frac{x_2-x_{2,E}}{h_K},\frac{x_{1,K}-x_1}{h_K})^{\tt t}$ and $\bm :=  \frac{\bx-\bx_K}{h_K}$, respectively; and providing the decomposition $\bM_{k}(K)=\bM_{k}^{\nabla}(K)\oplus \bM_{k}^{\oplus}(K)$.

The set of polynomials that solves locally the constitutive law in linear elasticity is defined as $\widetilde{\bbM}_k(K):= \{\widetilde{\bm}_k \in \bbM_k(K):\ \widetilde{\bm}_k=\bC\beps(\bm_{k+1})\ \text{for some}\ \bm_{k+1}\in \bM_{k+1}(K)\}$. On the other hand, the set of scaled rigid body motions of an element $K$ is given by
\begin{align*}
    \textbf{RBM}(K) &:= 
    \begin{cases}
        \begin{aligned}
            &\left\{ \begin{pmatrix} \frac{1}{h_E} \\ 0 \end{pmatrix}, \begin{pmatrix} 0 \\ \frac{1}{h_E} \end{pmatrix}, \begin{pmatrix} \frac{x_{2,E}-x_2}{h_E} \\ \frac{x_1-x_{1,E}}{h_E} \end{pmatrix} \right\}&\quad \text{in 2D},\\
            &\left\{ \begin{pmatrix} \frac{1}{h_P} \\ 0 \\ 0 \end{pmatrix}, \begin{pmatrix} 0 \\ \frac{1}{h_P} \\ 0 \end{pmatrix}, \begin{pmatrix} 0 \\ 0 \\ \frac{1}{h_P} \end{pmatrix}, \begin{pmatrix} \frac{x_{2,P}-x_2}{h_P} \\ \frac{x_1-x_{1,P}}{h_P} \\ 0 \end{pmatrix}, \begin{pmatrix} 0 \\ \frac{x_{3,P}-x_3}{h_P} \\ \frac{x_2-x_{2,P}}{h_P} \end{pmatrix}, \begin{pmatrix} \frac{x_3-x_{3,P}}{h_P} \\ 0 \\ \frac{x_{1,P}-x_1}{h_P} \end{pmatrix}\right\} &\quad \text{in 3D}.
        \end{aligned}
    \end{cases}
\end{align*}
In this case, the polynomial decomposition $\bP_k(K) = \textbf{RBM}(K) \oplus \textbf{RBM}^\perp_k(K)$, holds with
\begin{align*}
    \textbf{RBM}^\perp(K) := \left\{ \bm_k \in \bM_k: \ \int_{K} \bm_k \cdot \bm_{{\textbf{RBM}}} = 0 ,\, \forall \bm_{{\textbf{RBM}}}\in  \textbf{RBM}(K) \right\}.
\end{align*}

\subsection{\gls{vem} for \gls{hr} linear elasticity}\label{sec:HRVEM} The associated (conforming) VE space for $\bbH({\mathbf{div}},\Omega)$ in both 2D and 3D locally solves the constitutive law in linear elasticity \cite{artioli18,visinoni24} and its defined by 
\begin{align*} 
    \bbS^{h,k}(K):= \{\btau_h\in \bbH({\mathbf{div}},K) :\ &\btau_h\bn|_{{f}} \in \bP_k(f) ,\, \forall f\in \partial K, \\
    & \bdiv\btau_h \in \bP_k(K),\ \btau_h=\bC\beps(\bv^*) \ \text{for some}\ \bv^*\in \bH^1(K) \}.
\end{align*}
Notice that the polynomial space $\widetilde{\bbP}_k(K):=\{\widetilde{\bp}_k \in \bbP_k(K):\ \widetilde{\bp}_k=\bC\beps(\bp_{k+1})\ \text{for some}\ \bp_{k+1}\in \bP_{k+1}(K)\}\subseteq \bbS^{h,k}(K)$. To define the global discrete spaces we patch together the local spaces in the following way
\begin{align*}
    \bbS^{h,k} &:= \{ \btau_h \in \bbH^{\mathrm{sym}}_{\mathrm{N}}(\bdiv,\Omega) \colon \btau_h|_K \in \bbS^{h,k}(K) ,\, \forall K\in {\cT_h} \}, \\
    \bU^{h,k} &:= \{ \bv_h\in \bL^2(\Omega) \colon \bv_h|_K\in \bP_{k}(K) ,\, \forall K\in \cT_h \}.
\end{align*}
The associated DoFs for $\btau_h \in \bbS^{h,k}(K)$ and $\bv_h \in \bU^{h,k}(K):=\bP_{k}(K)$ are given as follows
\begin{align*}
    &\bullet \frac{1}{h_f}\int_f \btau_h\bn \cdot \bm_k ,\quad \forall \bm_k \in \bM_k(f), \\
    &\bullet \frac{1}{h_K}\int_K \bdiv\btau_h\cdot\bm_{{\textbf{RBM}}^\perp} ,\quad  \forall \bm_{{\textbf{RBM}}^\perp}\in \textbf{RBM}^\perp(K), \\
    &\bullet \frac{1}{h_K}\int_K \bv_h \cdot \bm_k ,\quad \forall \bm_k \in \bM_k(K).
\end{align*}

\subsection{\gls{vem} for perturbed mixed second-order elliptic problems}\label{sec:MixedVEM} The conforming VE approximation for the space  $\bH(\vdiv,\Omega)$ in 2D is defined locally by solving a $\vdiv$-$\vrot$ problem \cite{beirao16}, as follows  
\begin{align*}
    \bV_{\text{2D}}^{h,k}(K) := \{ \bxi_h\in \bH(\vdiv,K)\cap \bH(\vrot,K) \colon\ &\bxi_h\cdot \bn|_{{f}}\in \rP_{k}(f),\, \forall f\subset \partial K, \\
    &\vdiv\bxi_h \in \rP_{k}(K), \ \vrot\bxi_h \in \rP_{k-1}(K)\}.
\end{align*}
Observe that $\bP_{k}(K)\subseteq \bV_{\text{2D}}^{h,k}(K)$. In turn, the global discrete space is defined as 
\begin{align*}
    \bV_{\text{2D}}^{h,k} &:= \{ \bxi_h \in \bH_{\mathrm{N}}(\vdiv,\Omega) \colon \bxi_h|_K \in \bV_{\text{2D}}^{h,k}(K) ,\, \forall K\in \cT_h \}.
\end{align*}
We consider the following DoFs for $\bxi_h\in \bV_{\text{2D}}^{h,k}(K)$:
\begin{align*}
    &\bullet \text{The values of } \bxi_h\cdot \bn \text{ at the $k+1$ Gauss--Lobatto quadrature points of each edge of $K$}, \\
    &\bullet \frac{1}{h_K}\int_K \bxi_h\cdot \bm_{k-1}^\nabla ,\quad \forall \bm_{k-1}^{\nabla} \in \bM_{k-1}^{\nabla}(K),\\
    &\bullet \frac{1}{h_K}\int_K \bxi_h \cdot \bm_{k}^{\oplus} ,\quad \forall \bm_{k}^{\oplus} \in \bM_{k}^{\oplus}(K).
\end{align*}

In contrast, the 3D version of the conforming VE approximation for the space $\bH(\vdiv,\Omega)$ locally solves a $\nabla(\vdiv)$--$\bcurl\bcurl$ problem \cite{beirao16b} and it is defined as
\begin{align*}
    \bV_{\text{3D}}^{h,k+1}(K) := \{ \bxi_h \in \bH(\vdiv,K)\cap \bH(\bcurl,K) \colon\ &\bxi_h \cdot \bn|_{{f}} \in \rP_{k+1}(f),\, \forall f\in \partial K, \\
    & \nabla (\vdiv \bxi_h) \in \bG_{k-1}(K),\ \bcurl\bcurl \bxi_h \in \bP_{k}(K)\}.
\end{align*}
Note that $\bP_{k+1}(K)\subseteq\bV_{\text{3D}}^{h,k+1}(K)$. Then, the discrete global spaces are defined by
\begin{align*}
    \bV_{\text{3D}}^{h,k+1} &:= \{ \bxi_h \in \bH_{\mathrm{N}}(\vdiv,\Omega) \colon \bxi_h|_K \in \bV_{\text{3D}}^{h,k+1}(K) ,\, \forall K\in \cT_h \}.
\end{align*}
The set of DoFs for $\bxi_h\in \bV_{\text{3D}}^{h,k+1}(K)$ is provided next
\begin{align*}
    &\bullet \text{The values of } \bxi_h\cdot \bn \text{ at the $k+2$ quadrature points on each face of $K$}, \\
    &\bullet \frac{1}{h_K}\int_K \bxi_h\cdot \bm_{k-1}^\nabla ,\quad \forall \bm_{k-1}^{\nabla} \in \bM_{k-1}^{\nabla}(K),\\
    &\bullet \frac{1}{h_K}\int_K \bxi_h \cdot \bm_{k+1}^{\oplus} ,\quad \forall \bm_{k+1}^{\oplus} \in \bM_{k+1}^{\oplus}(K).
\end{align*}

Finally, the global discrete space for the space $\rL^2(\Omega)$ is defined in general for 2D and 3D as
\begin{align*}
    \mathrm{Q}^{h,k} &:= \{ \psi_h\in \rL^2(\Omega) \colon \psi_h|_K\in \rP_{k}(K) ,\, \forall K\in \cT_h \},
\end{align*}
and, for a given $\psi_h\in \mathrm{Q}^{h,k}(K):=\rP_{k}(K)$, the DoFs for the space above are defined by
\begin{align*}
    &\bullet \frac{1}{h_K}\int_K \psi_h m_{k} ,\quad  \forall m_{k}\in \rM_{k}(K).
\end{align*}

\subsection{Polynomial projection and interpolation operators}\label{sec:proj_inter_op} For each element $K$, we introduce the following local polynomial projection operators:
\begin{itemize}
    \item   The $\bC$-energy projection is defined as $\bbPi_{k}^{\bC,K}: \widetilde{\bbS}(K)\rightarrow \widetilde{\bbM}_k(K)$ by
            \begin{align*}
                \int_K {\bC^{-1}\left(\btau-\bbPi_{k}^{\bC,K}\btau\right)} : \widetilde{\bm}_k = 0, \quad  
                \forall \widetilde{\bm}_k\in \widetilde{\bbM}_k(K),
            \end{align*}
            where $\widetilde{\bbS}(K):=\{\btau\in \bbH(\bdiv,K):\ \btau= \bC\beps(\bv)\ \text{for some}\ \bv\in \bH^1(K)\}$.
    \item   The $\bL^2$ projection is defined by $\bPi_{k}^{0,K}: \bL^2(K)\rightarrow \bM_{k}(K)$ where
            \begin{align}\label{PiVector}
                \int_K \left(\bxi-\bPi_{k}^{0,K}\bxi\right)\cdot \bm_{k} = 0, \quad 
                \forall \bm_{k} \in \bM_{k}(K),
            \end{align}
            with an analogous definition for scalar functions.
\end{itemize}
The detailed proof of computability of these operators in terms of the respective DoFs can be found in \cite{artioli18,visinoni24,beirao16,beirao16b}. In addition, we state a result involving classical polynomial approximation theory \cite{brenner08}. The estimate is presented for scalar functions, but it also holds in general for vector and tensor functions.
\begin{proposition}[polynomial approximation] \label{prop:est_poly}
    Given $K\in \cT_h$, assume that $v\in  \rH^{{\overline{s}}}(K)$, with  $1 \leq {\overline{s}} \leq k+1$. Then, there exist $v_\pi \in \rP_k(K)$ and a positive constant that depends only on {the mesh regularity parameter $\gamma$} ({cf.} \ref{A1}-\ref{A3}) such that for $0\leq \overline{r}\leq {\overline{s}}$ the following estimate holds
    \begin{align*}
        |v - v_\pi|_{\overline{r},K} &\lesssim  h_K^{{\overline{s}}-\overline{r}}|v|_{{\overline{s}},K}.
    \end{align*} 
\end{proposition}

Next, {the} nature of the space defined in Section~\ref{sec:HRVEM} allow us to define locally the Fortin--like interpolation operator $\bbF^{k,K}:\bbH^{1}(K)\rightarrow \bbS^{h,k}(K)$ through the associated DoFs in a unified way for an element $K$ ({cf.} \cite{botti25}). Whereas, following Section~\ref{sec:MixedVEM}, the Fortin--like interpolation operators $\bF_{\text{2D}}^{k,K}:\mathbf{H}^{1}(K)\rightarrow \bV_{\text{2D}}^{h,k}(K)$ and $\bF_{\text{3D}}^{k+1,K}:\mathbf{H}^{1}(K)\rightarrow \bV_{\text{3D}}^{h,k+1}(K)$ are defined by their associated DoFs, taking into account that the element $K$ refers to a polygon in 2D and a polyhedral in 3D. See, e.g.,  \cite[Section 3.2]{beirao16} and \cite[Section 4.1]{beirao22} for their respective constructions. Moreover, the associated commutative property holds for each operator as follows: for each $K\in \cT_h$, we have
\begin{align}\label{commutative_property}
    \bdiv(\bbF^{k,K}\btau) = \bPi^{0,K}_k (\bdiv\btau), \quad \vdiv (\bF_{\text{2D}}^{k,K}\bxi) = \Pi_{k}^{0,K} (\vdiv \bxi ), \quad \vdiv (\bF_{\text{3D}}^{k+1,K}\bxi) = \Pi_{k}^{0,K} (\vdiv \bxi).
\end{align}

\begin{proposition}[\gls{hr} \gls{vem} interpolation estimates] \label{prop:est_int_HR}
    Given $K\in \cT_h$, assume that $\btau\in  \bbH(\bdiv,K)\cap\bbH^{{\overline{s}}}(K)$, with  $1 \leq {\overline{s}} \leq k+1$. Then, there {exists} a positive constant that depends only on {$\gamma$} ({cf.} \ref{A1}-\ref{A3}) such that, for $0\leq \overline{r}\leq {\overline{s}}$, the following estimate holds
    \begin{align*}
        |\btau - \bbF^{k,K}\btau|_{\overline{r},K} &\lesssim h_K^{{\overline{s}}-\overline{r}}|\btau|_{{\overline{s}},K}.
    \end{align*} 
\end{proposition}

\begin{proposition}[mixed (Poisson) \gls{vem} interpolation estimates] \label{prop:est_int_mixed}
    {Given $K\in \cT_h$ 
    and $1 \leq {\overline{s}} \leq k+1$, there exist positive constants that depend} only on {$\gamma$} ({cf.} \ref{A1}-\ref{A3}) such that for $0\leq \overline{r}\leq {\overline{s}}$ the following estimates hold
    \begin{gather*}
       \|\bxi - \bF_{\mathrm{2D}}^{k,K}\bxi\|_{0,\overline{l};K} \lesssim h_K^{{\overline{s}}-\overline{r}}|\bxi|_{{\overline{s}},\overline{l};K},\qquad \|\bxi - \bF_{\mathrm{3D}}^{k+1,K}\bxi\|_{0,\overline{l};K} \lesssim h_K^{{\overline{s}}-\overline{r}}|\bxi|_{\overline{l},\overline{s};K} \qquad \forall \bxi \in \bW^{\overline{s},\overline{l}}(K).
    \end{gather*}
\end{proposition}
\begin{remark}
 Typically, one also requires an interpolation property for the divergence part of the flux (or stress) space to obtain error estimates (see, for example, \cite{gatica2021p,gharibi2025mixed,gharibi2025stabfree} for mixed \gls{vem} in the $L^p$ context). Such a property calls for additional regularity for the divergence part, for example (using the notation from Proposition~\ref{prop:est_int_mixed} in the 2D case)
\[\|\vdiv(\bxi - \bF_{\mathrm{2D}}^{k,K}\bxi)\|_{0,K} \lesssim h_K^{{\overline{s}}-\overline{r}} |\vdiv\bxi|_{\overline{s},K} \qquad \forall \bxi \in \bW^{1,1}(K) \ \text{such that} \ \vdiv\bxi \in \rH^{\overline{s}}(K).\]
Here we proceed differently and derive estimates involving the divergence by simply using the commutativity property \eqref{commutative_property} and applying Proposition~\ref{prop:est_poly}. This avoids the assumption of more regularity for the divergence, but rather asking it for the concentration. \end{remark}

\subsection{Discrete problem} \label{sec:discrete-problem} 
Without losing generality we denote by $\bV^{h,\overline{k}}$ the global discrete spaces defined in Section~\ref{sec:MixedVEM}, i.e., $\bV^{h,\overline{k}}=\bV^{h,k}_{\text{2D}}$ (resp. $\bV^{h,\overline{k}}=\bV^{h,k+1}_{\text{3D}}$) for polygonal elements (resp. polyhedral elements), we also introduce the discrete product spaces $\bbV^{h,k}:= \bbS^{h,k}\times \rQ^{h,k}$ and $\bQ^{h,k}:=\bU^{h,k}\times \bV^{h,\overline{k}}$, 
and note that the space $\widetilde{\bV}^{h,\overline{k}}:=\bV^{h,\overline{k}}\cap\bH_{\mathrm{N}}^4(\vdiv,\Omega)$ consists of the discrete space $\bV^{h,\overline{k}}$ equipped with  the norm  $\norm{\cdot}_{4,\vdiv;\Omega}$. For brevity, (and wherever needed) the polynomial projections of $\bsigma_h\in \bbS^{h,k}$, $\bz_h\in \bV^{h,\overline{k}}$ and $\bzeta_h \in \widetilde{\bV}^{h,\overline{k}}$ are denoted by  $\bsigma_h^{\bbPi}:=\bbPi_{k}^{\bC,K}\bsigma_h$, $\bz_h^{\bPi}:=\bPi_{\overline{k}}^{0,K}\bz_h$ and $\bzeta_h^{\bPi}:=\bPi_{\overline{k}}^{0,K}\bzeta_h$, respectively, where the projection $\bPi_{\overline{k}}^{0,K}$ refers to $\bPi_{k}^{0,K}$ in the two dimensional setting (resp. $\bPi_{k+1}^{0,K}$ in the three dimensional setting). We recall that the computability of the discrete formulation (introduced below) follows directly from the computability of the projection operators discussed in Section~\ref{sec:proj_inter_op}.

Given $\vec{\bsigma}_h := (\bsigma_h, p_h), \  \vec{\btau}_h := (\btau_h, q_h) \in \bbV^{h,k}$,  $\vec{\bu}_h := (\bu_h, \bz_h)$, and $\vec{\bv}_h := (\bv_h, \bw_h) \in \bQ^{h,k}$, the computable discrete bilinear forms $A_h:\bbV^{h,k}\times\bbV^{h,k}\to\RR$, $B:\bbV^{h,k}\times\bQ^{h,k}\to\RR$, $C_h:\bQ^{h,k}\times\bQ^{h,k}\to\RR$, $D_h:\rQ^{h,k}\times\bbV^{h,k}\to\RR$,  
the computable discrete lineal operators $\cB:  \widetilde{\bV}^{h,\overline{k}} \to (\rQ^{h,k})^\prime$, $\cC:\rQ^{h,k}\to (\rQ^{h,k})^\prime$, and (for a given polynomial $\widehat{\bsigma}_h^{\bbPi}\in \widetilde{\bbP}_{k}(K)$) the nonlinear operator $\cA_{h,\widehat{\bsigma}_h^{\bbPi}}:\widetilde{\bV}^{h,\overline{k}}\to (\widetilde{\bV}^{h,\overline{k}})^{\prime}$, are defined as 
\begin{align*}
    A_h(\vec{\bsigma}_h, \vec{\btau}_h) &= \sum_{K\in \cT_h} A_h^K(\vec{\bsigma}_h, \vec{\btau}_h) \\
    &:=  \sum_{K\in \cT_h} \biggl[ (\bC^{-1}(\bbPi_{k}^{\bC,K}\bsigma_h), \bbPi_{k}^{\bC,K}\btau_h)_K + S_1^{\bC,K}((\bbOne-\bbPi_{k}^{\bC,K})\bsigma_h,(\bbOne-\bbPi_{k}^{\bC,K})\tau_h) \\
    &\quad + \bigl(\frac{\alpha p_h}{2\mu + d\lambda}, \tr(\bbPi_{k}^{\bC,K}\btau_h) \bigr)_K + \bigl(\frac{\alpha q_h}{2\mu + d\lambda}, \tr(\bbPi_{k}^{\bC,K}\bsigma_h) \bigr)_K + \bigl[s_0+\frac{d\alpha^2}{2\mu + d\lambda}\bigr] (p_h,q_h)_K\biggr],\\
    B(\vec{\btau}_h, \vec{\bv}_h) &= \sum_{K\in \cT_h} B^K(\vec{\btau}_h, \vec{\bv}_h) := \sum_{K\in \cT_h} \biggl[(\bv_h, \bdiv \btau_h)_K + (q_h , \vdiv \bw_h)_K \biggr], \\ 
    C_h(\vec{\bu}_h, \vec{\bv}_h) &= \sum_{K\in \cT_h}C_h^K(\vec{\bu}_h, \vec{\bv}_h) \\
    &:= \sum_{K\in \cT_h} \biggl[(\bkappa^{-1}(\bPi_{\overline{k}}^{0,K}\bz_h),\bPi_{\overline{k}}^{0,K}\bw_h)_K + S_2^{0,K}((\bOne-\bPi_{\overline{k}}^{0,K})\bz_h,(\bOne-\bPi_{\overline{k}}^{0,K})\bw_h)\biggr], \\ 
    D_h(\psi_h,\vec{\btau}_h) &= \sum_{K\in \cT_h} D_h^{K}(\psi_h,\vec{\btau}_h) := \sum_{K\in \cT_h} (\frac{\beta \psi_h}{2\mu + d\lambda}, \tr(\bbPi_{k}^{\bC,K}\btau_h) + \alpha d q_h )_K , \\
    [\cA_{h,\widehat{\bsigma}_h^{\bbPi}}(\bzeta_h),\bxi_h] &= \sum_{K\in \cT_h} [\cA_{h,\widehat{\bsigma}_h^{\bbPi}}^K(\bzeta_h),\bxi_h]\\
    &:= \sum_{K\in \cT_h} \biggl[ (\varrho(\widehat{\bsigma}_h^{\bbPi})^{-1}(\bPi_{\overline{k}}^{0,K}\bzeta_h),\bPi_{\overline{k}}^{0,K}\bxi_h)_K \\
    &\quad + \eta(|\bPi_{\overline{k}}^{0,K}\bzeta_h|^{2}\bPi_{\overline{k}}^{0,K}\bzeta_h, \bPi_{\overline{k}}^{0,K}\bxi_h) + S_3^{\rF,K}((\bOne - \bPi_{\overline{k}}^{0,K})\bzeta_h,(\bOne - \bPi_{\overline{k}}^{0,K})\bxi_h)\biggr], \\
    {[\cB(\bxi_h), \psi_h]} &{= \sum_{K\in \cT_h} [\cB^K(\bxi_h), \psi_h] := -\sum_{K\in \cT_h} (\vdiv \bxi_h, \psi_h)_K,} \\
    {[\cC(\varphi_h), \psi_h]} &{= \sum_{K\in \cT_h} [\cC_h^K(\varphi_h), \psi_h] := \sum_{K\in \cT_h} (\varphi_h, \psi_h)_K.}
\end{align*}
The stabilisation terms $S_1^{\bC,K}:\bbV^{h,k}\times\bbV^{h,k}\to\RR$, $S_2^{0,K}:\bQ^{h,k}\times\bQ^{h,k}\to\RR$, and {$S_3^{\rF,K}:\widetilde{\bV}^{h,\overline{k}}\times \widetilde{\bV}^{h,\overline{k}}\to \RR$} are assumed to be any positive semi-definite inner products satisfying the following condition: for each $K\in\cT_h$, there exist positive constants $C_{s_1}$, $C_{s_2}$, $C_{s_3}$ (independent of $h$ and $K$) such that
\begin{subequations}\label{stab_bounds}
    \begin{alignat}{2}
        C_{s1}^{-1}(\bC^{-1}\btau_h, \btau_h)_K&\leq S_1^{\bC,K}(\btau_h,\btau_h)\leq C_{s1}(\bC^{-1}\btau_h, \btau_h)_K\quad &&\forall \btau_h\in \mathrm{ker}(\bbPi_{k}^{\bC,K}),\label{s1}\\
        C_{s2}^{-1}(\bkappa^{-1}\bw_h,\bw_h)_K&\leq S_2^{0,K}(\bw_h,\bw_h)\leq C_{s2}(\bkappa^{-1}\bw_h,\bw_h)_K\quad &&\forall \bw_h\in \mathrm{ker}(\bPi_{\overline{k}}^{0,K}),\label{s2}\\
        {C_{s3}^{-1}(|\bxi_h|\bxi_h,\bxi_h)_K}&{\leq S_3^{\rF,K}(\bxi_h,\bxi_h)\leq C_{s3}(|\bxi_h|\bxi_h,\bxi_h)_K}\quad &&\forall \bxi_h\in \mathrm{ker}(\bPi_{\overline{k}}^{0,K}).\label{s3}
    \end{alignat}
\end{subequations}
Finally, the computable linear functionals $F: \bbV^{h,k}\to \RR$, $G: \bQ^{h,k}\to\RR$, $\cF:\widetilde{\bV}^{h,\overline{k}}\to \RR$, and $\cG: \rQ^{k,k}\to \RR$ are given by
\begin{align*}
    F(\vec{\btau}_h) &= \sum_{K\in \cT_h} F^K(\vec{\btau}_h) := \sum_{K\in \cT_h} \biggl[ (g,q_h)_K  + \sum_{f\in \partial K\cap \Gamma_{\mathrm{D}}}\langle \bu_{\mathrm{D}},\btau_h\bn\rangle_{f} \biggr], \\
    G(\vec{\bv}_h) &= \sum_{K\in \cT_h} G^K(\vec{\bv}_h) := - \sum_{K\in \cT_h} \biggl[ -(\ff,\bv_h)_K + \sum_{f\in \partial K\cap \Gamma_{\mathrm{D}}} \langle p_{\mathrm{D}},\bw_h\cdot\bn\rangle_{f}\biggr], \\
    {[\cF,\bxi_h]} &{= \sum_{K\in \cT_h} [\cF,\bxi_h] := - \sum_{K\in \cT_h} \sum_{f\in \partial K\cap \Gamma_{\mathrm{D}}} \langle \varphi_{\mathrm{D}},\bxi_h\cdot\bn\rangle_{f},} \\
    {[\cG,\psi_h]} &{= \sum_{K\in \cT_h} [\cG^K,\psi_h] := - \sum_{K\in \cT_h} (\ell,\psi_h)_K.}
\end{align*}
The discrete version of \eqref{eq:weak2} is defined next: find $(\vec{\bsigma}_h, \vec{\bu}_h) \in \bbV^{h,k} \times \bQ^{h,k}$ and $(\bzeta_h, \varphi_h) \in \widetilde{\bV}^{h,\overline{k}} \times \rQ^{h,k}$, such that
\begin{subequations}\label{eq:weak_disc}
    \begin{align}
        A_h(\vec{\bsigma}_h, \vec{\btau}_h)+\,B(\vec{\btau}_h, \vec{\bu}_h) + D_h(\varphi_h, \vec{\btau}_h) &   \, =\, F(\vec{\btau}_h)  &&\forall \vec{\btau}_h \in \bbV^{h,k},\label{eq:weak_disc1}\\
        B(\vec{\bsigma}_h,  \vec{\bv}_h) -\,C_h(\vec{\bu}_h, \vec{\bv}_h) &  \,=\, G(\vec{\bv}_h)&&\forall \vec{\bv}_h \in \bQ^{h,k}, \label{eq:weak_disc2} \\  
        [\cA_{h,\bsigma_h^\bbPi}(\bzeta_h), \bxi_h]+\,[\cB(\bxi_h), \varphi_h] &   \, =\, [\cF,\bxi_h] && \forall \bxi_h \in \widetilde{\bV}^{h,\overline{k}},\label{eq:weak_disc3}\\
        [\cB(\bzeta_h),  \psi_h]-\,[\cC(\varphi_h), \psi_h] &  \,=\, [\cG,\psi_h] && \forall \psi_h \in \rQ^{h,k}. \label{eq:weak_disc4}
    \end{align}
\end{subequations} 
{\begin{remark}
    Note that the non--standard stabilisation operator $S_3^{\rF,K}:\widetilde{\bV}^{h,\overline{k}}\times \widetilde{\bV}^{h,\overline{k}}\to \RR$ has to be treated carefully. We propose the operator
    $$S_3^{\rF,K}(\bzeta_h,\bxi_h):=h_K^{-d}\left(\sum_{i}^{\mathrm{\#DoFs}} \mathrm{dof}_i(\bzeta_h)\mathrm{dof}_i(\bxi_h)\right)^2.$$
    Note that this is a modification of the standard $\mathrm{DOFI}$-$\mathrm{DOFI}$ stabilisation, which is well known in the \gls{vem} literature. In addition, we have that
    $$\norm{\bzeta_h}_{0,K}^2\lesssim \sum_{i}^{\mathrm{\#DoFs}} \mathrm{dof}_i(\bzeta_h)\mathrm{dof}_i(\bzeta_h) \lesssim\norm{\bzeta_h}_{0,K}^2\quad \forall \bzeta_h\in \mathrm{ker}(\bPi_{\overline{k}}^{0,K}).$$
    Thus, the inequality above together with the equivalence between the norms $\bL^2$ and $\bL^4$ for finite dimensional spaces (see \cite[Lemma 12.1]{ErnGuermond2021FE1}) given by $C_{\mathrm{eqv}}h_K^{-d/4}\|\bzeta_h\|_{0,K} \le \|\bzeta_h\|_{0,4;K}\le \overline{C}_{\mathrm{eqv}}h_K^{-d/4}\|\bzeta_h\|_{0,K}$ with $C_{\mathrm{eqv}},\overline{C}_{\mathrm{eqv}}>0$, imply that
    $$\norm{\bzeta}_{0,4;K}^4\lesssim S_3^{\rF,K}(\bzeta_h,\bzeta_h)\lesssim\norm{\bzeta}_{0,4;K}^4\quad  \forall \bzeta_h\in \mathrm{ker}(\bPi_{\overline{k}}^{0,K}).$$
\end{remark}}
\section{Discrete well-posedness analysis}\label{sec:discrete_wp}
This section extends the results {shown} in Section~\ref{sec:continuous_wp} to the \gls{vem} formulation proposed in \eqref{eq:weak_disc}. Following the analysis for the continuous problem,
we establish the solvability of the fully coupled discrete problem via a discrete fixed-point argument. We recall that, thanks to stabilisation,  the discrete operators inherit the properties presented in Section~\ref{sec:main_properties}.
\subsection{Properties of the discrete operators}
Note that, for each $K\in \cT_h$, given $\vec{\btau}_h\in \bbV^{h,k}$ and $\vec{\bv}_h\in \bQ^{h,k}$, we have that $\bdiv\btau_h \in \bP_k(K)$, $\vdiv\bw_h \in \rP_{k}(K)$ (see also the definition of the 2D (resp. 3D) \gls{vem} space in Section~\ref{sec:MixedVEM} (resp. \cite[Theorem 8.2]{beirao16b})), $q_h\in \rP_{k}(K)$, and $\bv_h\in \bP_k(K)$. Hence, the the following characterisations hold:
    \begin{align*}
        \bbV_0^h := \mathrm{ker}(\bB |_{\bbV^{h,k}}) &= \{ \vec{\btau}_h \in \bbV^{h,k}: \ B(\vec{\btau}_h,\vec{\bv}_h) = 0 ,\ \forall \vec{\bv}_h\in \bQ^{h,k}\} \nonumber
        \\ &= \bbV_{01}^h \times \bbV_{02}^h \equiv \{ \btau_h \in \bbS^{h,k}: \ \bdiv\btau_h|_K = \cero, \ \forall K\in\cT_h\} \times \{0\}, 
        \\
        \bQ_0^h := \mathrm{ker}(\bB^*|_{\bQ^{h,k}}) &= \{ \vec{\bv}_h \in \bQ^{h,k}: \ B(\vec{\btau}_h,\vec{\bv}_h) = 0 ,\ \forall \vec{\btau}_h\in \bbV^{h,k}\}  \nonumber 
        \\ &= \bQ_{01}^h \times \bQ_{02}^h \equiv \{\cero\}\times \{ \bw_h \in \bV^{h,\overline{k}}: \ \vdiv\bw_h|_K = 0, \ \forall K\in\cT_h\}. 
    \end{align*}
On the other hand, the orthogonal spaces $(\bbV_0^h)^{\perp}= (\bbV_{01}^h)^{\perp} \times (\bbV_{02}^h)^{\perp}$ and $(\bQ_0^h)^{\perp}= (\bQ_{01}^h)^{\perp} \times (\bQ_{02}^h)^{\perp}$ are closed subspaces of $\bbV^{h,k}$ and $\bQ^{h,k}$, where 
    \begin{gather*}
        (\bbV_{01}^h)^{\perp}  \equiv \{ \bsigma_h \in \bbS^{h,k}:\ (\bsigma_h,\btau_h)_K =0 ,\ \forall \btau_h \in  \bbV_{01}^h,\ \forall K\in\cT_h\}, \quad (\bbV_{02}^h)^{\perp}\equiv \rQ^{h,k},\\
        (\bQ_{01}^h)^{\perp} \equiv \bQ^{h,k}, \qan (\bQ_{02}^h)^{\perp} \equiv \{ \bz_h \in \bV^{h,\overline{k}}:\ (\bz_h,\bw_h)_K =0, \ \forall \bw_h \in  \bQ_{02}^h, \ \forall K\in\cT_h\}.
    \end{gather*}
    
In what follows, we prove some key properties of the discrete bilinear forms.
{\begin{lemma}[boundedness of the discrete bilinear forms]\label{lem:boundedness_bilinear_h}
The bilinear forms $A_h(\bullet,\bullet)$, $C_h(\bullet,\bullet)$, and $D_h(\bullet,\bullet)$ are bounded. That is:
\begin{align*}
|A_h(\vec{\bsigma}, \vec{\btau})| &\leq \|A_h\| \|\vec{\bsigma}_h\|_{\bbV} \|\vec{\btau}_h\|_{\bbV} \quad &&\forall \vec{\bsigma}_h, \vec{\btau}_h \in \bbV^{h,k}, \\
|C_h(\vec{\bu}_h, \vec{\bv}_h)| &\leq \|C_h\| \|\vec{\bu}_h\|_{\bQ} \|\vec{\bv}_h\|_{\bQ} \quad &&\forall \vec{\bu}_h, \vec{\bv}_h \in \bQ^{h,k}, \\
|D_h(\psi_h, \vec{\btau}_h)| &\leq \|D_h\| \|\psi_h\|_{0,\Omega} \|\vec{\btau}_h\|_{\bbV} \quad &&\forall \psi_h \in \rQ^{h,k}, \forall \vec{\btau}_h \in \bbV^{h,k},
\end{align*}
where the boundedness constants are given by
\begin{gather*}
\|A_h\| := \max\left\{ \left(\frac{1}{2\mu} + \frac{\lambda}{2\mu(2\mu+d\lambda)}\right)+C_{s1}, \frac{\alpha\sqrt{d}}{2\mu+d\lambda}C_\bbPi, s_0+\frac{d\alpha^2}{2\mu+d\lambda} \right\}, \quad \|C_h\| := \frac{1}{\kappa_1}+C_{s2}, \quad \\
\|D_h\| := \frac{\beta\sqrt{d}(1+\alpha)}{2\mu+d\lambda}C_\bbPi,
\end{gather*}
where $C_\bbPi$ denotes the continuity constant of the projection operator $\bbPi_{k}^{\bC,K}$.
\end{lemma}
\begin{proof}
The result follows from a direct application of the Cauchy--Schwarz inequality, the stabilisation bounds in \eqref{stab_bounds}, and the continuity of the projection operator $\bbPi_{k}^{\bC,K}$.
\end{proof}}
{\begin{lemma}[symmetry and positive semi-definiteness of discrete diagonal forms]\label{lemma:positive-definite-diagonal-discrete}
    The bilinear forms $A_h(\bullet,\bullet)$ and  $C_h(\bullet,\bullet)$ are symmetric and positive semi-definite.
\end{lemma}
\begin{proof}
    The proof reduces to employ the arguments in Lemma~\ref{lemma:positive-definite-diagonal} together with the properties of the stabilisation operators in \eqref{stab_bounds}. Indeed, we can extend \eqref{positive-A} for all $\vec{\btau}_h\in \bbV^{h,k}$ as follows
    \begin{align}
        A_h(\vec{\btau}_h, \vec{\btau}_h) &\geq \frac{1}{2\mu}\|(\bbPi_{k}^{\bC,K}\btau_h)^{\tt d}\|^{2}_{0,\Om}+ S_1^{\bC,K}((\bbOne-\bbPi_{k}^{\bC,K})\btau_h,(\bbOne-\bbPi_{k}^{\bC,K})\btau_h) \nonumber\\
        & \quad + \, \frac{s_{0}}{2}\|q_h\|^{2}_{0,\Om}\, +
        \frac{s_0}{d(s_0(2\mu + d\lambda) + 2d\alpha^2)}\|\tr (\bbPi_{k}^{\bC,K}\btau_h)\|^{2}_{0,\Om} \geq 0, \label{positive-definite-A_h}
    \end{align}
    thanks to the positive semi-definitess of $S_1^{\bC,K}(\bullet,\bullet)$.
\end{proof}
}
{\begin{lemma}[coercivity for the main discrete diagonal forms]\label{lemma:coercitivity-as_h}
There exist  constants $\overline{\alpha}_A,\overline{\alpha_C}>0$ such that 
\begin{subequations}
    \begin{alignat}{2}
        A_h(\vec{\btau}_h, \vec{\btau}_h) & \geq  \overline{\alpha}_A \|\vec{\btau}_h\|^{2}_{\bbV} \quad &&\forall \, \vec{\btau}_h \in \bbV_0^h, \label{coercitivity-A_h}\\
        C_h(\vec{\bv}_h, \vec{\bv}_h) &\geq   \overline{\alpha_C} \|\vec{\bv}_h\|^{2}_{\bQ} \quad &&\forall \, \vec{\bv}_h \in \bQ_{0}^h. \label{coercitivity-C_h}
    \end{alignat}
\end{subequations}
\end{lemma}
\begin{proof}
    Note that \eqref{positive-definite-A_h} and \eqref{s1} imply that for all $\vec{\btau}_h \in \bbV_0^h$
    \begin{align*}
        A_h(\vec{\btau}_h, \vec{\btau}_h) &\geq \min\{\frac{1}{2\mu},C_{s1}^{-1}\}\|\btau_h^{\tt d}\|^{2}_{0,\Om} +  \frac{s_{0}}{2}\|q_h\|^{2}_{0,\Om}\, +
        \min\{\frac{s_0}{d(s_0(2\mu + d\lambda) + 2d\alpha^2)},C_{s1}^{-1}\}\|\tr \btau_h\|^{2}_{0,\Om}. 
    \end{align*}
    Thus, applying \eqref{DES} and \eqref{DES-GN} to $\btau_h$, there exist $\overline{C}_1,\overline{C}_2>0$ such that  
    $A_h(\vec{\btau}_h, \vec{\btau}_h) \geq \overline{\alpha}_A\norm{\vec{\btau}_h}_{\bdiv,\Omega}^2$,
    where $\overline{\alpha}_A = \frac{\overline{C}_1,\overline{C}_2}{4\mu}\min\{1,C_{s1}^{-1}\}$. Finally, in a similar manner to Lemma~\ref{lemma:coercitivity-as}, we obtain that \eqref{coercitivity-C_h} holds with $\overline{\alpha}_C=\min\{\alpha_C,C_{s2}^{-1}\}$.
\end{proof}
}
{\begin{lemma}[discrete inf-sup conditions]\label{lemma:inf-sup-B_h}
There exist a positive constant $\overline{\beta}_B$ such that
\begin{alignat}{2}
    \sup_{ \vec{\btau}_h \in \bbV^{h,k}\setminus\{\boldsymbol{0}\}}\frac{ B (\vec{\btau}_h, \vec{\bv}_h)}{\|\vec{\btau}_h\|_{\bbV}} & \geq  \overline{\beta}_{B} \|\vec{\bv}_h\|_{\bQ} \qquad &&\forall \, \vec{\bv}_h \in (\bQ_0^h)^{\perp}.\label{eq:inf-sup-Bh}
\end{alignat}    
\end{lemma}
\begin{proof}
    We start by recalling from \cite[Proposition 5.6]{artioli17b} the following discrete inf-sup condition
    \begin{align}\label{eq:inf-sup-b1h}
        \sup_{\btau \in \bbS^{h,k} \setminus \{\mathbf{0}\}}\frac{  (\bv_h, \bdiv \btau_h)}{\|\btau_h\|_{\bdiv,\Omega}}  \geq  \overline{\beta}_{1} \|\bv_h\|_{0,\Om} \quad \forall \, \bv_h \in (\bQ_{01}^h)^{\perp}.
    \end{align}
    Similarly to \cite[Proposition 5.4.2]{boffi13}, given that $\vdiv \bw_h \in \bP_{k}(K)$ for all $K\in \cT_h$, $\bw_h\in (\bQ_{02}^h)^\perp\subseteq \bH_\mathrm{N}(\vdiv,\Omega)$, and the definition of $\Pi_{k}^{0,K}$ in \eqref{PiVector}, we have that 
    \begin{align}\label{eq:inf-sup-b2h}
        \sup_{q_h\in \rQ^{h,k}\setminus\{0\}} \frac{(q_h,\vdiv \bw_h)}{\|q_h\|_{0,\Omega}} \geq \sup_{q\in \rL^{2}(\Omega)\setminus\{0\}} \frac{(\Pi_{k}^{0}q,\vdiv \bw_h)}{\|\Pi_{k}^{0}q\|_{0,\Omega}} \geq \sup_{q\in \rL^{2}(\Omega)\setminus\{0\}} \frac{(q,\vdiv \bw_h)}{C_\Pi\|q\|_{0,\Omega}} \geq \overline{\beta}_2 \|\bw_h\|_{\vdiv,\Omega},
    \end{align}
  where in the last inequality we have used the continuous inf-sup condition  \eqref{eq:inf-sup-b2}. Here  $\overline{\beta}_2 = \frac{\beta_2}{C_\Pi}$, $\Pi_{k}^{0}q := \Pi_{k}^{0,K}q|_K$ for all $K\in \cT_h$, and $C_\Pi$ being the associated continuity constant of $\Pi_{k}^{0}$ in the $\rL^2$-norm. Therefore, the bounds in \eqref{eq:inf-sup-b1h}-\eqref{eq:inf-sup-b2h} yields \eqref{eq:inf-sup-Bh} with $\overline{\beta}_{B} = \frac{\overline{\beta}_1+\overline{\beta}_2}{4}$.
\end{proof}}

{Regarding the discrete counterpart of the altered diffusion block, we extend the analysis performed in Section~\ref{sec:wp-decoupled} by setting $\varphi_h = \cC^{-1}(-\cG + \mathcal{B}(\bzeta_h))$ and defining the discrete global nonlinear operator $\cN_h:\widetilde{\bV}^{h,\overline{k}} \to (\widetilde{\bV}^{h,\overline{k}})'$, defined by
\begin{equation*} 
    \cN_h(\bzeta_h) := \mathcal{A}_{h,\bsigma_h^\bbPi}(\bzeta_h) + \mathcal{B}^* \cC^{-1} \mathcal{B}(\bzeta_h).
\end{equation*}
Since the operators in the second term remain unchanged with respect to the continuous setting, we readily see that
\begin{align*} 
    [\cN_h(\bzeta_h), \bxi_h] &= \sum_{K\in \cT_h} [\cN_h^K(\bzeta_h), \bxi_h]\notag \\
    &:= \sum_{K\in \cT_h} \biggl[ (\varrho(\widehat{\bsigma}_h^{\bbPi})^{-1}(\bPi_{\overline{k}}^{0,K}\bzeta_h),\bPi_{\overline{k}}^{0,K}\bxi_h)_K \notag \\
    &\quad + \eta(|\bPi_{\overline{k}}^{0,K}\bzeta_h|^{2}\bPi_{\overline{k}}^{0,K}\bzeta_h, \bPi_{\overline{k}}^{0,K}\bxi_h) + S_3^{\rF,K}((\bOne - \bPi_{\overline{k}}^{0,K})\bzeta_h,(\bOne - \bPi_{\overline{k}}^{0,K})\bxi_h) \notag\\
    &\quad + (\vdiv \bzeta_h, \vdiv \bxi_h)_K\bigg] \quad \forall \bzeta_h, \bxi_h \in \widetilde{\bV}^{h,\overline{k}}.
\end{align*}
}
{Now, we encapsulate the properties of the discrete nonlinear operator $\cN_h$ in the following results.
\begin{lemma}[discrete boundedness]\label{lem:boundedness_h}
    The operator $\cN_h$ is bounded. That is, for any $M > 0$, there exists a constant $\overline{C}_\star> 0$ such that $\|\cN_h(\bzeta_h)\|_{(\bH^4_{\mathrm{N}}(\vdiv,\Omega))'} \le \overline{C}_\star$ whenever $\|\bzeta_h\|_{{{4,\vdiv;\Omega}}} \le M$.
\end{lemma}
\begin{proof}
    Let $M > 0$ and assume $\bzeta_h \in\widetilde{\bV}^{h,\overline{k}}$ such that $\|\bzeta_h\|_{{4,\vdiv;\Omega}} \le M$. For any $\bxi_h \in \widetilde{\bV}^{h,\overline{k}}$, we bound the duality pairing using the Cauchy-Schwarz and Hölder inequalities, together with the equivalence between the norms in $\bL^2$ and $\bL^4$ for discrete spaces (see \cite[Lemma 12.1]{ErnGuermond2021FE1}) given by $C_{\text{eqv}}h^{-d/4}\|\bzeta_h\|_{0,\Omega} \le \|\bzeta_h\|_{0,4;\Omega}\le \overline{C}_{\text{eqv}}h^{-d/4}\|\bzeta_h\|_{0,\Omega}$ with $C_{\text{eqv}},\overline{C}_{\text{eqv}}>0$, as follows
    \begin{align*}
        |[\cN_h(\bzeta_h), \bxi_h]|
        \le \rho_2C_{\bPi}^2 \|\bzeta_h\|_{0, \Omega} \|\bxi_h\|_{0, \Omega} + (\eta+C_{s3}) \|\bzeta_h\|_{0,4;\Omega}^3 \|\bxi_h\|_{0,4;\Omega} + \|\vdiv \bzeta_h\|_{0,\Omega} \|\vdiv \bxi_h\|_{0,\Omega},
    \end{align*}
    which turns into
    \begin{align*}
        |[\cN_h(\bzeta_h), \bxi_h]| &\le \rho_2C_{\bPi}^2C_{\text{emb}}^2 \|\bzeta_h\|_{0,4;\Omega} \|\bxi_h\|_{0,4;\Omega} + (\eta+C_{s3}) \|\bzeta_h\|_{0,4;\Omega}^3 \|\bxi_h\|_{0,4;\Omega} + \|\vdiv \bzeta_h\|_{0,\Omega} \|\vdiv \bxi_h\|_{0,\Omega} \\
        &\le \left( \rho_2C_{\bPi}^2C_{\text{emb}}^2|\bzeta_h\|_{{4,\vdiv;\Omega}} + (\eta+C_{s3})  \|\bzeta_h\|_{{4,\vdiv;\Omega}}^3 + \|\bzeta_h\|_{{4,\vdiv;\Omega}} \right) \|\bxi_h\|_{{4,\vdiv;\Omega}},
    \end{align*}
    where $C_{\bPi}$ is the continuity constant of the projection operator $\bPi_{\overline{k}}^{0,K}$ in the $\bL^4(\Omega)$--norm. Finally, the result holds with 
    $\overline{C}_\star:=\left(\rho_2C_{\bPi}^2C_{\text{emb}}^2 + 1\right) M + (\eta+C_{s3})  M^3$.
\end{proof}}
{\begin{lemma}[discrete hemicontinuity]\label{lem:hemicontinuity_h} $\cN_h$ is hemicontinuous, that is, for each $ \bzeta_h, \bxi_h \in \widetilde{\bV}^{h,\overline{k}}$, the real mapping $J:\bbR \to \bbR $, $t \mapsto [\cN_h(\bzeta_h + t\bxi_h), \bxi_h]$ is continuous.
\end{lemma}
\begin{proof}
    The continuity of both the projection operator $\bPi_{\overline{k}}^{0,K}$ and the stabilisation operator $S_3^{0,K}(\bullet,\bullet)$, together with the arguments used in Lemma~\ref{lem:hemicontinuity} lead to the result.
\end{proof}}
{\begin{lemma}[monotonicity]\label{lem:monotonicity_h}
   The operator $\cN_h$ is monotone, that is
\[[\cN_h(\bzeta_h) - \cN_h(\bxi_h), \bzeta_h - \bxi_h] \ge 0 \quad \forall \bzeta_h, \bxi_h \in \widetilde{\bV}^{h,\overline{k}}.\]
\end{lemma}
\begin{proof}
    The result follows by putting together the arguments in Lemma~\ref{lem:monotonicity} and the positive semi-definiteness of the stabilisation operator $S_3^{\rF,K}(\bullet,\bullet)$.
\end{proof}}
{\begin{lemma}[coercivity]\label{lem:coercivity_h}
    The operator $\cN_h$ is coercive on $\widetilde{\bV}^{h,\overline{k}}$, meaning that
\[\lim_{\|\bzeta_h\|_{{4,\vdiv;\Omega}} \to \infty} \frac{[\cN_h(\bzeta_h), \bzeta_h]}{\|\bzeta_h\|_{{4,\vdiv;\Omega}}} = \infty.\]
\end{lemma}
\begin{proof}
    Let $\bzeta_h \in \widetilde{\bV}^{h,\overline{k}}$. From \eqref{stab_bounds}, we readily see that
    \begin{equation*}
        [\cN_h(\bzeta_h), \bzeta_h] \ge \min\{\eta C_F,C_{s3}^{-1}\} \|\bzeta_h\|_{0,4;\Omega}^4 + \|\vdiv \bzeta_h\|_{0,\Omega}^2.
    \end{equation*}
    Thus, The result follows by dividing the inequality above by $\|\bzeta_h\|_{4,\vdiv;\Omega}$ and taking the limit $\|\bzeta_h\|_{4,\vdiv;\Omega} \to \infty$.
\end{proof}}

\subsection{Solvability of the discrete coupled problem}
We follow the analysis in Section~\ref{sec:wp-decoupled}-\ref{sec:wp} to derive the solvability of the discrete problem \eqref{eq:weak_disc}. Given two computable prescribed functions $\widehat{\varphi}_h \in \rQ^{h,k}$ and $\widehat{\bsigma}_h^{\bbPi} \in \bbS^{h,k}$, the following results imply the well-posedness of the decoupled equations corresponding to the discrete Biot equations \eqref{eq:weak_disc1}-\eqref{eq:weak_disc2} and the discrete altered diffusion equation \eqref{eq:weak_disc3}-\eqref{eq:weak_disc4}. The proof follows as in the continuous case by employing Lemmas~\ref{lemma:positive-definite-diagonal-discrete}-\ref{lemma:inf-sup-B_h} for the Biot equations and Lemmas~\ref{lem:boundedness_h}-\ref{lem:coercivity_h} for the nonlinear diffusion equations, together with the discrete versions of \cite[Theorem 4.3.1]{boffi13} and \cite[Corollary 2.2]{showalter1997}, respectively.
\begin{theorem}[well-posedness of the discrete Biot equations]\label{th:well-posed-Biot-disc}
 There exists a unique  $( \vec{\bsigma}_h, \vec{\bu}_h) \in \bbV^{h,k}\times \bQ^{h,k}$ such that 
 \begin{subequations}\label{eq:well-posed-Biot-disc}
     \begin{align}
        A_h(\vec{\bsigma}_h, \vec{\btau}_h)+\,B(\vec{\btau}_h, \vec{\bu}_h) &   \, =\, -D_h(\widehat{\varphi}_h,\vec{\btau}_h) + F(\vec{\btau}_h)  &&\forall \vec{\btau}_h \in \bbV^{h,k},\\ 
        B(\vec{\bsigma}_h,  \vec{\bv}_h) -\,C_h(\vec{\bu}_h, \vec{\bv}_h) &  \,=\, G(\vec{\bv}_h)&&\forall \vec{\bv}_h \in \bQ^{h,k}.  
     \end{align}
 \end{subequations}
 Moreover, 
 \[ \| ( \vec{\bsigma}_h, \vec{\bu}_h) \|_{\bbV\times \bQ} \lesssim  \frac{(1+\alpha d)\beta}{2\mu + d\lambda}\|\widehat{\varphi}_h\|_{0,\Omega} + \|\ff\|_{0,\Omega} + \|\bu_{\mathrm{D}}\|_{1/2,00;\Gamma_{\mathrm{D}}} + \|g\|_{0,\Omega} + \|p_{\mathrm{D}}\|_{1/2,00;\Gamma_{\mathrm{D}}}. \]
\end{theorem}
\begin{proof}
    The arguments are standard and therefore omitted.
\end{proof}
{\begin{theorem}[well-posedness of the discrete nonlinear diffusion equations]\label{th:well-posed-diff-disc}
There exists a unique $( \bzeta_h, \varphi_h) \in  \widetilde{\bV}^{h,\overline{k}} \times \rQ^{h,k}$ such that
\begin{subequations}\label{eq:well-posed-diff-disc}
    \begin{align}
      [\mathcal{A}_{h,\widehat{\bsigma}_h^\bbPi}(\bzeta_h), \bxi_h] + [\mathcal{B}(\bxi_h), \varphi_h] &= [\cF,\bxi_h] &&\forall \bxi_h \in \widetilde{\bV}^{h,\overline{k}},\label{eq:well-posed-diff-disc1}\\
     [\mathcal{B}(\bzeta_h), \psi_h] - [\cC(\varphi_h), \psi_h] &= [\cG,\psi_h] &&\forall \psi_h \in \rQ^{h,k}, \label{eq:well-posed-diff-disc2} 
\end{align}   \end{subequations}
 furthermore,
    \begin{equation}\label{eq:apriori_bound_h}
        \|\bzeta_h\|_{{4,\vdiv;\Omega}} + \|\varphi_h\|_{0, \Omega} \le \mathcal{M}_h(\ell, \varphi_{\mathrm{D}}),
    \end{equation}
where 
    \begin{equation}\label{eq:def_M_data_h}
        \mathcal{M}_*(\ell, \varphi_{\mathrm{D}}) := \left( \frac{2}{\min\{\eta C_F,C_{s3}^{-1}\}} \right)^{1/4} \mathcal{H}_*(\ell, \varphi_{\mathrm{D}})^{1/4} + 4 \mathcal{H}_*(\ell, \varphi_{\mathrm{D}})^{1/2} + \|\ell\|_{0,\Omega},
    \end{equation}
    with 
    \begin{equation*}\label{eq:def_F_data_h}
        \mathcal{H}_*(\ell, \varphi_{\mathrm{D}}) := \frac{3}{4(2\min\{\eta C_F,C_{s3}^{-1}\})^{1/3}} \left( \max\{C_{\mathrm{emb}}, 1\} \|\varphi_{\mathrm{D}}\|_{1/2,00;\Gamma_{\mathrm{D}}} \right)^{4/3} + \frac{3}{2} \left( \max\{C_{\mathrm{emb}}, 1\} \|\varphi_{\mathrm{D}}\|_{1/2,00;\Gamma_{\mathrm{D}}} \right)^2 + \|\ell\|_{0,\Omega}^2.
    \end{equation*}
\end{theorem}}
\begin{proof}
    {We start by proving the existence of a discrete solution to \eqref{eq:well-posed-diff-disc}. Let $\overline{\cF} \in (\widetilde{\bV}^{h,\overline{k}})^\prime$ such that $[\overline{\cF}, \bxi_h] := [\cF, \bxi_h] + [\cC^{-1}\cG, \mathcal{B}(\bxi_h)]$. We readily see that Lemmas~\ref{lem:boundedness}-\ref{lem:coercivity}, together with \cite[Corollary 2.2]{showalter1997} imply the surjectivity of the operator $\mathcal{N}_h$. Therefore, there exists $\bzeta_h \in \widetilde{\bV}^{h,\overline{k}}$ such that $[\mathcal{N}_h(\bzeta_h), \bxi_h] = [\overline{\cF}, \bxi_h]$ for all $\bxi_h \in \widetilde{\bV}^{h,\overline{k}}$.}
    
     {Similarly to the continuous setting, let us define $\varphi_h := \cC^{-1} (\mathcal{B}(\bzeta_h) - \cG) \in \rQ^{h,k}$. Then, $\cC(\varphi_h) = \mathcal{B}(\bzeta_h) - \cG$ and \eqref{eq:well-posed-diff-disc2} holds, i.e.,
    \begin{equation*}
        [\mathcal{B}(\bzeta_h), \psi_h] - [\cC(\varphi_h), \psi_h] = [\cG, \psi_h] \quad \forall \psi_h \in \rQ^{h,k}.
    \end{equation*}
    Next, we recover \eqref{eq:well-posed-diff-disc1} by using that $[\mathcal{N}_h(\bzeta_h), \bxi_h] = [\overline{\cF}, \bxi_h]$ to obtain
    \begin{equation*}
        [\mathcal{A}_{h,\widehat{\bsigma}_h^\bbPi}(\bzeta_h), \bxi_h] + [\mathcal{B}^* \cC^{-1} \mathcal{B}(\bzeta_h), \bxi_h] = [\cF, \bxi_h] + [\mathcal{B}^* \cC^{-1}(\cG), \bxi_h].
    \end{equation*}
    Thus,
    \begin{equation*}
        [\mathcal{A}_{h,\widehat{\bsigma}_h^\bbPi}(\bzeta_h), \bxi_h] + [\mathcal{B}(\bxi_h), \cC^{-1}(\mathcal{B}(\bzeta_h))] = [\cF, \bxi_h] + [\mathcal{B}(\bxi_h), \cC^{-1}(\cG)],
    \end{equation*}
    and
    \begin{equation*}
        [\mathcal{A}_{h,\widehat{\bsigma}_h^\bbPi}(\bzeta_h), \bxi_h] + [\mathcal{B}(\bxi_h), \varphi_h] = [\cF, \bxi_h] \quad \forall \bxi_h \in \widetilde{\bV}^{h,\overline{k}}.
    \end{equation*}
    }

    {For the uniqueness of the decoupled altered diffusion block, we assume that $(\bzeta_{1h} \varphi_{1h}), (\bzeta_{2h}, \varphi_{2h})\in \widetilde{\bV}^{h,\overline{k}}\times \rQ^{h,k}$ satisfy \eqref{eq:well-posed-diff-disc}. Then, some algebraic manipulations lead to
    \begin{align*}
        [\mathcal{A}_{h,\widehat{\bsigma}_h^\bbPi}(\bzeta_{1h}) - \mathcal{A}_{h,\widehat{\bsigma}_h^\bbPi}(\bzeta_{2h}), \bxi_h] + [\mathcal{B}(\bxi_h), \varphi_{1h} - \varphi_{2h}] &= 0 \quad \forall \bxi_h \in \widetilde{\bV}^{h,\overline{k}},\\ 
        [\mathcal{B}(\bzeta_{1h} - \bzeta_{2h}), \psi_h] - [\cC(\varphi_{1h} - \varphi_{2h}), \psi_h] &= 0 \quad \forall \psi_h \in \rQ^{h,k}. 
    \end{align*}
    We now choose $\bxi_h = \bzeta_{1h} - \bzeta_{2h}$ and $\psi_h = \varphi_{1h} - \varphi_{2h}$, to obtain
    \begin{equation*}
        [\mathcal{A}_{h,\widehat{\bsigma}_h^\bbPi}(\bzeta_{1h}) - \mathcal{A}_{h,\widehat{\bsigma}_h^\bbPi}(\bzeta_{2h}), \bzeta_{1h} - \bzeta_{2h}] + [\cC(\varphi_{1h} - \varphi_{2h}), \varphi_{1h} - \varphi_{2h}] = 0.
    \end{equation*}
    Since $[\mathcal{A}_{h,\widehat{\bsigma}_h^\bbPi}(\bzeta_{1h}) - \mathcal{A}_{h,\widehat{\bsigma}}(\bzeta_{2h}), \bzeta_{1h} - \bzeta_{2h}] \ge \min\{\eta C_F,C_{s3}^{-1}\} \|\bzeta_{1h} - \bzeta_{2h}\|_{0,4;\Omega}^4$ and $[\cC(\varphi_1 - \varphi_2), \varphi_1 - \varphi_2] = \|\varphi_1 - \varphi_2\|_{0,\Omega}^2$, we readily see that $\bzeta_{1,h} = \bzeta_{2,h}$ and $\varphi_{1,h} = \varphi_{2,h}$.}
    
    {Now, we focus in proving \eqref{eq:apriori_bound_h}. First, we test \eqref{eq:well-posed-diff-disc} with $\bxi_h = \bzeta_h$ and $\psi_h = -\varphi_h$, to obtain that
    \begin{equation*}
        [\mathcal{A}_{h,\widehat{\bsigma}_h^\bbPi}(\bzeta_h), \bzeta_h] + [\cC(\varphi_h), \varphi_h] = [\cF, \bzeta_h] - [\cG, \varphi_h].
    \end{equation*}
    We readily see that the bound $[\mathcal{A}_{h,\widehat{\bsigma}_h^\bbPi}(\bzeta_h), \bzeta_h] \ge \min\{\eta C_F,C_{s3}^{-1}\} \|\bzeta_h\|_{0,4;\Omega}^4$ implies that
    \begin{equation*}
        \min\{\eta C_F,C_{s3}^{-1}\}  \|\bzeta_h\|_{0,\Omega}^2 + \|\varphi_h\|_{0, \Omega}^2 \le \|\cF\|_{(\bH^{4}_{\mathrm{N}}(\vdiv,\Omega))^\prime} \|\bzeta_h\|_{{{4,\vdiv;\Omega}}} + \|\cG\|_{0, \Omega} \|\varphi_h\|_{0, \Omega}.
    \end{equation*}
 From this, the proof follows exactly as in Theorem~\ref{th:well-posed-diff} and is therefore omitted.}
\end{proof}

Next, we define the following discrete maps
\begin{align*}
    \cJ_h^{\mathrm{Biot}}:\rQ^{h,k}(\Omega)&\to \bbV^{h,k} \times \bQ^{h,k},\\
    \widehat{\varphi}_h &\mapsto \cJ_h^{\mathrm{Biot}}(\widehat{\varphi}_h)= ((\cJ^{\mathrm{Biot}}_{1h}(\widehat{\varphi}_h), \cJ^{\mathrm{Biot}}_{2h}(\widehat{\varphi}_h)),\cJ^{\mathrm{Biot}}_{3h}(\widehat{\varphi})):= ( (\bsigma_h, p_h),  \vec{\bu}_h)=( \vec{\bsigma}_h, \vec{\bu}_h),
\end{align*} 
where $( \vec{\bsigma}_h, \vec{\bu}_h) \in \bbV^{h,k} \times \bQ^{h,k}$ is given by Theorem~\ref{th:well-posed-Biot-disc}; and 
\begin{align*}
    \cJ_h^{\mathrm{diff}}: \bbS^{h,k} &\to \widetilde{\bV}^{h,\overline{k}} \times \rQ^{h,k}, \\
    \widehat{\bsigma}_h &\mapsto \cJ_h^{\mathrm{diff}}(\widehat{\bsigma}_h)=(\cJ^{\mathrm{diff}}_{1h}(\widehat{\bsigma}_h),\cJ^{\mathrm{diff}}_{2h}(\widehat{\bsigma}_h)):= (\bzeta_h,\varphi_h ),
\end{align*}
with $(\bzeta_h,\varphi_h)$ provided by Theorem~\ref{th:well-posed-diff-disc}. These maps are well-defined, along with the following discrete solution operator  
\begin{align}\label{Operator-Jh}
    \mathcal{J}_h: \rQ^{h,k}  &\to  \rQ^{h,k}, \notag \\
    \widehat{\varphi}_h &\mapsto \cJ_h(\widehat{\varphi}_h):= \cJ^{\mathrm{diff}}_{2h}(\cJ^{\mathrm{Biot}}_{1h}(\widehat{\varphi}_h)).
\end{align}
In what follows, we show {the existence of a solution to} the fully-coupled discrete problem \eqref{eq:weak_disc} through the equivalent fixed-point formulation $\cJ_h(\varphi_h)=\varphi_h$. First, we define the discrete closed ball for some $r>0$
\[ \rW_h : = \{ \widehat{\varphi}_h \in \rQ^{h,k}:
  \|\widehat{\varphi}_h\|_{0,\Omega}  \leq r \}.\]
Next, we prove that $\cJ_h$ maps $\rW_h$ into itself and show the continuity of {$\cJ_h$}.
{\begin{lemma}[discrete ball mapping property]\label{lemma:J-W-delta-W-delta-disc}
    Under the assumption 
    \begin{equation}\label{eq:data_assumption_h}
        \mathcal{M}_*(\ell, \varphi_{\mathrm{D}})\leq r,
    \end{equation}
    where $\mathcal{M}_*(\ell, \varphi_{\mathrm{D}})$ is defined as in \eqref{eq:apriori_bound_h}, it follows that  $\mathcal{J}\big(\rW_h\big)\subseteq\rW_h$.
\end{lemma}
\begin{proof}
  Given $\widehat{\varphi}_h \in \rW_h$, the definition \eqref{Operator-Jh}, the small data assumption \eqref{eq:data_assumption_h}, and the continuous dependence on data provided in \eqref{eq:def_M_data_h}, it holds that
     \begin{align*}
        \|\mathcal{J}_h(\widehat{\varphi}_h)\|_{0,\Omega}= \|  \cJ^{\mathrm{diff}}_{2h}(\cJ^{\mathrm{Biot}}_{1h}(\widehat{\varphi}_h))\|_{0,\Omega} \lesssim \|\ell \|_{0,\Omega} +  \|\varphi_{\mathrm{D}}\|_{1/2,00;\Gamma_{\mathrm{D}}} \leq r.
     \end{align*}
\end{proof}}
{\begin{lemma}[discrete H\"older continuity]\label{lemma:continuity-Jh} 
    There exists a positive constant {$\mathcal{L}_{\cJ_h}$} such that 
    \begin{equation*}
        \|\mathcal{J}_{{h}}({\varphi_{1h}})-\mathcal{J}_h({\varphi_{2h}})\|_{0,\Omega}
        \leq \mathcal{L}_{\mathcal{J}_h}\|\varphi_{1h}-\varphi_{2h}\|_{0, \Omega}^{2/3} \qquad \forall \varphi_{1h},\varphi_{2h} \in {\rQ^{h,k}}.
    \end{equation*} 
\end{lemma}
\begin{proof}
    Given $\varphi_{1h}, \varphi_{2h} \in {\rQ^{h,k}}$, we let $\cJ_h^{\mathrm{Biot}}(\varphi_{1h})=(\vec{\bsigma}_{1h}, \vec{\bu}_{1h}) \in \bbV^{h,k}\times \bQ^{h,k}$ and $\cJ_h^{\mathrm{Biot}}(\varphi_{2h})=(\vec{\bsigma}_{2h}, \vec{\bu}_{2h}) \in \bbV^{h,k}\times \bQ^{h,k}$ be the unique solutions of \eqref{eq:well-posed-Biot-disc}. Then, applying the discrete version of the inf-sup condition \eqref{global-inf-sup-Jhat} with $(\vec{\bzeta}_h, \vec{\bw}_h)=(\vec{\bsigma}_{1h}-\vec{\bsigma}_{2h},  \vec{\bu}_{1h}- \vec{\bu}_{2h})$, imply that $\overline{C}_{\mathrm{B}}\|(\vec{\bzeta}_h, \vec{\bw}_h)\|_{\bbV^{h,k}\times \bQ^{h,k}}$ is bounded by
    \begin{align*}
        \displaystyle & \sup_{(\vec{\btau}_h,\vec{\bv}_h) \in (\bbV^{h,k}\times \bQ^{h,k})\setminus \{(\mathbf{0},\mathbf{0})\}}\frac{A_h(\vec{\bsigma}_{1h}-\vec{\bsigma}_{2h},\vec{\btau}_h)+B(\vec{\btau}_h,\vec{\bu}_{1h}- \vec{\bu}_{2h})+ B(\vec{\bsigma}_{1h}-\vec{\bsigma}_{2h},\vec{\bv}_h)-C_h(\vec{\bu}_{1h}- \vec{\bu}_{2h},\vec{\bv}_h)}{\|(\vec{\btau}_h,\vec{\bv}_h)\|_{\bbV\times \bQ}}\\
        & = \sup_{(\vec{\btau}_h,\vec{\bv}_h) \in (\bbV^{h,k}\times \bQ^{h,k})\setminus \{(\mathbf{0},\mathbf{0})\}}\frac{D_h({\varphi}_{1h},\vec{\btau}_h)-D_h({\varphi}_{2h},\vec{\btau}_h)}{\|(\vec{\btau}_h,\vec{\bv}_h)\|_{\bbV\times \bQ}} \leq \frac{(1+\alpha d)\beta}{2\mu + d\lambda}\|\varphi_{1h}-\varphi_{2h}\|_{0,\Omega}.
    \end{align*}
Thus,
\begin{equation*}
\|\cJ_{1h}^{\mathrm{Biot}}(\varphi_{1h})-\cJ_{1h}^{\mathrm{Biot}}(\varphi_{2h})\|_{4,\div;\Om}
\leq  \frac{(1+\alpha d)\beta}{\overline{C}_{\mathrm{B}}(2\mu + d\lambda)}\|\varphi_{1h}-\varphi_{2h}\|_{0,\Omega}.
\end{equation*}
Similarly, let  $\bsigma_{1h}, \bsigma_{2h} \in  \bbS^{h,k}$, such that $\cJ_{h}^{\mathrm{diff}}(\bsigma_{1h})=( \varphi_{1h}, \bzeta_{1h}) \in  \rQ^{h,k} \times \widetilde{\bV}^{h,\overline{k}}$ and $\cJ_{h}^{\mathrm{diff}}(\bsigma_{2h})=( \varphi_{2h}, \bzeta_{2h}) \in  \rQ^{h,k} \times \widetilde{\bV}^{h,\overline{k}}$ be the unique solutions of \eqref{eq:well-posed-diff-disc}.  Equivalently, we have
    \begin{subequations}
    \begin{align*}
        [\cA_{h,{\bsigma}_{1h}^\bbPi}(\bzeta_{1h}) - \cA_{h,{\bsigma}_{2h}^\bbPi}(\bzeta_{2h}), \bxi_h] + [\cB(\bxi_h), \varphi_{1h}-\varphi_{2h}] &= 0 \qquad \forall \bxi_h \in \widetilde{\bV}^{h,\overline{k}}, \\ 
        [\cB(\bzeta_{1h}-\bzeta_{2h}),  \psi_h] - [\cC(\varphi_{1h}-\varphi_{2h}), \psi_h] &= 0 \qquad \forall \psi_h \in \rQ^{h,K}. 
    \end{align*}
    \end{subequations}
Next, we choose the test functions $\bxi_h = \bzeta_{1h} - \bzeta_{2h}$ and $\psi_h = -(\varphi_{1h} - \varphi_{2h})$ in the equation above to arrive at
\begin{align*}
    [\cA_{h,{\bsigma}_{1h}^{\bbPi}}(\bzeta_{1h}) - \cA_{h,{\bsigma}_{2h}^\bbPi}(\bzeta_{2h}), \bzeta_{1h} - \bzeta_{2h}] + \|\varphi_{1h} - \varphi_{2h}\|_{0,\Omega}^2 = 0.
\end{align*}
Thus, the addition and subtraction of $\cA_{{\bsigma}_{1h}^\bbPi}(\bzeta_{2h}, \bzeta_{1h} - \bzeta_{2h})$ in the previous equation, together with the strong monotonicity of the nonlinear operator lead to
\begin{align*}
    \min\{\eta,C_{s3}^{-1}\} \|\bzeta_{1h} - \bzeta_{2h}\|_{0,4;\Omega}^4 + \|\varphi_{1h} - \varphi_{2h}\|_{0,\Omega}^2 &\leq [\cA_{{\bsigma}_{2h}^\bbPi}(\bzeta_{2h}) - \cA_{{\bsigma}_{1h}^\bbPi}(\bzeta_{2h}), \bzeta_{1h} - \bzeta_{2h}] \nonumber \\
    &= \int_\Omega \left( \varrho({\bsigma}_{2h}^\bbPi)^{-1} - \varrho({\bsigma}_{1h}^\bbPi)^{-1} \right) \bzeta_{2h} \cdot (\bzeta_{1h} - \bzeta_{2h}) .
\end{align*}
The result follows by employing the arguments in Theorem~\ref{continuity-J}, in this case the continuity constant is given by
$$\mathcal{L}_{\cJ_h}:=\left(\frac{\sqrt{3} (L_\varrho C_\bbPi \mathcal{M}(\ell, \varphi_{\mathrm{D}}))^{2/3}}{2(2\min\{\eta,C_{s3}^{-1}\})^{1/6}}\right) \left(\frac{(1+\alpha d)\beta}{\overline{C}_{\mathrm{B}}(2\mu + d\lambda)}\right)^{2/3}.$$
\end{proof}}
 
We are ready to state the main result of this section which is a consequence of Lemmas~\ref{lemma:J-W-delta-W-delta-disc}-\ref{lemma:continuity-Jh} together with the Schauder fixed-point theorem.
\begin{theorem}[existence of solution for the fully-coupled discrete problem]\label{th:J-solvability-fixed-disc}
 Suppose that {$\mathcal{M}_*(\ell, \varphi_{\mathrm{D}})\leq r$}.  Then, the coupled problem \eqref{eq:weak_disc} has at least one solution $(\vec{\bsigma}_h, \vec{\bu}_h) \in \bbV^{h,k} \times \bQ^{h,k}$ and $(\bzeta_h,\varphi_h) \in  \widetilde{\bV}^{h,\overline{k}} \times \rQ^{h,k}$. Moreover, and similarly to the continuous case, we have 
\begin{align*}
    \|(\vec{\bsigma}_h, \vec{\bu}_h, \bzeta_h, \varphi_h)\|_{\bbV \times \bQ \times \bH^4_{\mathrm{N}}(\vdiv,\Omega)\times \rL^2(\Omega)}  &  \lesssim \left(\frac{(1+\alpha d)\beta}{2\mu + d\lambda}+ 1\right){\mathcal{M}_*(\ell, \varphi_{\mathrm{D}})} +  \|\ff\|_{0,\Omega}  + \|g\|_{0,\Omega} \notag\\
    &\quad +\, \|\bu_{\mathrm{D}}\|_{1/2,00;\Gamma_{\mathrm{D}}} +  \|p_{\mathrm{D}}\|_{1/2,00;\Gamma_{\mathrm{D}}} . 
\end{align*}
\end{theorem}

\section{A priori error analysis}\label{sec:error}
This section is devoted to deriving the optimal a priori error estimate. The first step is to establish the Strang-type inequalities which are formulated in the theorem below.
\begin{theorem}[quasi-optimality]\label{theorem:a_priori}
In addition to the assumptions of Theorems \ref{th:J-unique-solvability-fixed} and \ref{th:J-solvability-fixed-disc}, let $(\vec{\bsigma}, \vec{\bu}, \bzeta, \varphi) \in \bbV \times \bQ \times \bH^{4}_{\mathrm{N}}(\vdiv,\Omega) \times \rL^2(\Omega)$ and $(\vec{\bsigma}_h, \vec{\bu}_h, \bzeta_h, \varphi_h) \in \bbV^{h,k} \times \bQ^{h,k} \times \widetilde{\bV}^{h,\overline{k}} \times \rQ^{h,k}$ be solutions to \eqref{eq:weak2} and \eqref{eq:weak_disc}, respectively. Furthermore, assume there exists a constant $\mathcal{K}_1 > 0$, independent of $h$, such that $\|\bF_d^{\overline{k},K}\bzeta\|_{\infty,\Omega} \le \mathcal{K}_1$. Under these conditions, and for all $h \leq h_0$, the following error estimates hold 
\begin{subequations}
\begin{align}
    \|\vec{\bsigma}-\vec{\bsigma}_h\|^2_{\bbV}+\|\vec{\bu}-\vec{\bu}_h\|_{ \bQ }^2  &  \lesssim \|\vec{\bsigma}-\vec{\bsigma}_h^{\bbPi}\|_{\bbV}^2+\|\vec{\bsigma}-\vec{\bsigma}_h^{\bbF}\|^2_{\bbV}+\|\vec{\bu}-\vec{\bu}_h^{\bPi}\|_{\bbV}^2+\|\vec{\bu}-\vec{\bu}_h^{\bF}\|_{\bQ}^2 \notag \\
    &\quad + \, \left(\frac{(1+\alpha d)\beta}{2\mu + d\lambda}\right)^2\|\varphi-\varphi_h\|^2_{{0,\Omega}}, \label{a_priori_poro}\\
    \| \bzeta-\bzeta_h\|_{0,\Omega}^2+\| \bzeta-\bzeta_h\|_{4,\vdiv;\Omega}^4 +\|\varphi-\varphi_h\|_{ 0,\Omega}^2  &  \lesssim \|\bzeta-\bF_d^{\overline{k},K}\bzeta\|^4_{{4,\vdiv;\Omega}}+\|\bzeta-\bF_d^{\overline{k},K}\bzeta\|^2_{{4,\vdiv;\Omega}}+ \|\bzeta-\bF_d^{\overline{k},K}\bzeta\|^2_{{0,\Omega}}\notag\\
    & \quad + \|\bzeta-\bF_d^{\overline{k},K}\bzeta\|^{4/3}_{{0,4/3;\Omega}}+\|\varphi-\Pi_k^{0,K}\varphi\|^{4/3}_{0,\Omega}\,+\, \|\varphi-\Pi_k^{0,K}\varphi\|^{2}_{0,\Omega}\notag\\
    &\quad +\|\varphi-\Pi_k^{0,K}\varphi\|^{4}_{0,\Omega}+\frac{1}{2}L_{\varrho}\mathcal{K}_1\|\bsigma-\bsigma_h^{\bbPi}\|^{2}_{0,\Omega},\label{a_priori_diff}
\end{align}\end{subequations}
where $\vec{\bsigma}_h^{\bbF}:= (\bbF^{k,K}\bsigma,\Pi_{k}^{0,K}p ), \vec{\bsigma}_h^{\bbPi}:= (\bsigma_h^{\bbPi},p_h ), \vec{\bu}_h^{\bF}:= (\bPi^{0,K}_k\bu,\bF_d^{\overline{k},K} \bz), \vec{\bu}_h^{\bPi}:= (\bu_h,{\bPi}_{\overline{k}}^{0,K}\bz)$, and by $\bF_d^{\overline{k},K}$ we represent the Fortin operators either $\bF_{\mathrm{2D}}^{k,K}$ or $\bF_{\mathrm{3D}}^{k+1,K}$, depending on the spatial dimension under consideration. Here, $h_0 > 0$ is a threshold constant whose explicit definition is detailed in the subsequent proof.
\end{theorem}
\begin{proof}
    We proceed in a similar way as in \cite{khot25}, noting from \eqref{eq:weak2} and \eqref{eq:weak_disc} that $(\vec{\bsigma}_h-\vec{\bsigma}_h^{\bbF},\vec{\bu}_h-\vec{\bu}_h^{\bF})\in \bbV^{h,k}\times \bQ^{h,k}$ is the unique solution to 
    \begin{align*}
        A_h(\vec{\bsigma}_h-\vec{\bsigma}_h^{\bbF}, \vec{\btau}_h)+\,B(\vec{\btau}_h, \vec{\bu}_h-\vec{\bu}_h^{\bF}) &   \, =\, \widetilde{F}(\vec{\btau}_h)  &&\forall \vec{\btau}_h \in \bbV^{h,k},\\
        B(\vec{\bsigma}_h-\vec{\bsigma}_h^{\bbF},  \vec{\bv}_h) -\,C_h(\vec{\bu}_h-\vec{\bu}_h^{\bF}, \vec{\bv}_h) &  \,=\, \widetilde{G}(\vec{\bv}_h)&&\forall \vec{\bv}_h \in \bQ^{h,k},
    \end{align*}
    where 
    \begin{align*}
        \widetilde{F}(\vec{\btau}_h)&:=A(\vec{\bsigma},\vec{\btau}_h)-A_h(\vec{\bsigma}_h^{\bbF},\vec{\btau}_h)+B(\vec{\btau}_h,\vec{\bu}-\vec{\bu}_h^{\bF})+D_h(\varphi-\varphi_h,\vec{\btau}_h),\\
        \widetilde{G}(\vec{\bv}_h)&:=B(\vec{\bsigma}-\vec{\bsigma}_h^{\bbF},\vec{\bv}_h)-C(\vec{\bu},\vec{\bv}_h)+C_h(\vec{\bu}_h^{\bF},\vec{\bv}_h).
    \end{align*}
    By exploiting the continuous dependence on data established in Theorem \ref{th:well-posed-Biot-disc}, we can deduce that
    \begin{equation}\label{bound_functionals_biot_a_priori}
    \|(\vec{\bsigma}-\vec{\bsigma}_h^{\bbF},\vec{\bu}-\vec{\bu}_h^{\bF})\|_{\bbV\times \bQ}\lesssim \|\widetilde{F}\|_{\bbV'}+\|\widetilde{G}\|_{\bQ'}.
    \end{equation}
    Now, noting that $A_h(\vec{\bsigma}_h^{\bbPi},\vec{\btau}_h) = A(\vec{\bsigma}_h^{\bbPi},\vec{\btau}_h)$, by applying the continuity of the bilinear forms $A(\bullet,\bullet)$, $A_h(\bullet,\bullet)$, $B(\bullet,\bullet)$, $D(\bullet,\bullet)$, $C(\bullet,\bullet)$, and $C_h(\bullet,\bullet)$, as well as using the triangle inequality, it is possible to deduce that
    \begin{subequations}\label{bounds_a_priori_biot}
    \begin{align}
        |A(\vec{\bsigma},\vec{\btau}_h)-A_h(\vec{\bsigma}_h^{\bbF},\vec{\btau}_h)|&\lesssim (\|\vec{\bsigma}-\vec{\bsigma}_h^{\bbPi}\|_{\bbV}+\|\vec{\bsigma}_h^{\bbPi}-\vec{\bsigma}_h^{\bbF}\|_{\bbV})\|\btau_h\|_{\bbV} \lesssim (\|\vec{\bsigma}-\vec{\bsigma}_h^{\bbPi}\|_{\bbV}+\|\vec{\bsigma}-\vec{\bsigma}_h^{\bbF}\|_{\bbV})\|\btau_h\|_{\bbV},\\
        |B(\vec{\btau}_h,\vec{\bu}-\vec{\bu}_h^{\bF})|&\lesssim \|\vec{\bu}-\vec{\bu}_h^{\bF}\|_{\bQ}\|\btau_h\|_{\bbV},\\
        |D_h(\varphi-\varphi_h,\vec{\btau}_h)|&\lesssim \frac{(1+\alpha d)\beta}{2\mu+d\lambda} \|\varphi-\varphi_h\|_{0,\Omega}\|\btau_h\|_{\bbV},\\
        |B(\vec{\bsigma}-\vec{\bsigma}_h^{\bbF},\vec{\bv}_h)|&\lesssim \|\vec{\bsigma}-\vec{\bsigma}_h^{\bbF}\|_{\bbV}\|\bv_h\|_{\bQ},\\
        |C_h(\vec{\bu}_h^{\bF},\vec{\bv}_h)-C(\vec{\bu},\vec{\bv}_h)|&\lesssim (\|\vec{\bu}-\vec{\bu}_h^{\bPi}\|_{\bQ}+\|\vec{\bu}_h^{\bF}-\vec{\bu}_h^{\bPi}\|_{\bQ})\|\bv_h\|_{\bQ}\lesssim \|(\vec{\bu}-\vec{\bu}_h^{\bPi}\|_{\bQ}+\|\vec{\bu}-\vec{\bu}_h^{\bF}\|_{\bQ})\|\bv_h\|_{\bQ}.
    \end{align}   
    \end{subequations}
    Upon substitution of \eqref{bounds_a_priori_biot} into \eqref{bound_functionals_biot_a_priori}, and invoking the triangle inequality, the result \eqref{a_priori_poro} follows. Conversely, for the diffusivity problem, it follows once more from \eqref{eq:weak2} and \eqref{eq:weak_disc} that $(\bzeta_h - \bF_d^{\overline{k},K}\bzeta , \varphi_h-\Pi_k^{0,K}\varphi) \in \widetilde{\bV}^{h,\overline{k}}\times \rQ^{h,k} $ constitutes the unique solution of
    \begin{subequations}
    \begin{align}
        [\mathcal{A}_{h,\bsigma_h^\bbPi}(\bzeta_h - \bF_d^{\overline{k},K}\bzeta), \bxi_h] + [\mathcal{B}(\bxi_h), \varphi_h-\Pi_k^{0,K}\varphi] & \,=\, [\widetilde{\mathcal{F}},\bxi_h] && \forall \bxi_h \in \widetilde{\bV}^{h,\overline{k}},\label{eq:a_priori_diff_1}\\
        [\mathcal{B}(\bzeta_h - \bF_d^{\overline{k},K}\bzeta), \psi_h] -  [\mathcal{C}(\varphi_h-\Pi_k^{0,K}\varphi), \psi_h]& \,=\, [\widetilde{\mathcal{G}},\psi_h] && \forall \psi_h \in \rQ^{h,k},\label{eq:a_priori_diff_2}
    \end{align}
    where 
    \begin{align*}
        [\widetilde{\mathcal{F}},\bxi_h]&:=[\mathcal{A}_{\bsigma}(\bzeta),\bxi_h]-[\mathcal{A}_{h,\bsigma_h^\bbPi}(\bF_d^{\overline{k},K}\bzeta), \bxi_h]+[\mathcal{B}(\bxi_h),\varphi-\Pi_k^{0,K}\varphi],\\
        [\widetilde{\mathcal{G}},\psi_h]&:=[\mathcal{B}(\bzeta-\bF_d^{\overline{k},K}\bzeta),\psi_h]-[\mathcal{C}(\varphi-\Pi_k^{0,K}\varphi),\psi_h].
    \end{align*}
\end{subequations}
Now, taking $\bxi_h:=\bzeta_h - \bF_d^{\overline{k},K}\bzeta$ and $\psi_h:=\varphi_h-\Pi_k^{0,K}\varphi$ and proceeding as in the proofs of the continuous dependence in Theorems \ref{th:well-posed-diff-disc} and \ref{th:well-posed-diff}, we immediately obtain
\begin{align}\label{eq:pre_M_bound_a_priori}
        &\varrho_1\|\bzeta_h - \bF_d^{\overline{k},K}\bzeta\|_{0,\Omega}^2+{\eta} \|\bzeta_h - \bF_d^{\overline{k},K}\bzeta\|_{0,4;\Omega}^4 + \|\varphi_h-\Pi_k^{0,K}\varphi\|_{0, \Omega}^2 \le  [\mathcal{A}_{\bsigma}(\bzeta)-\mathcal{A}_{h,\bsigma_h^{\bbPi}}(\bF_d^{\overline{k},K}\bzeta),\bzeta_h - \bF_d^{\overline{k},K}\bzeta]\nonumber \\
        & \qquad \qquad  +[\mathcal{B}(\bzeta_h - \bF_d^{\overline{k},K}\bzeta), \varphi-\Pi_k^{0,K}\varphi]- [\mathcal{B}(\bzeta - \bF_d^{\overline{k},K}\bzeta), \varphi_h-\Pi_k^{0,K}\varphi]+[\mathcal{C}(\varphi-\Pi_k^{0,K}\varphi),\varphi_h-\Pi_k^{0,K}\varphi].
    \end{align}
After adding and subtracting suitable terms, applying Young’s inequality and utilising the $\mathrm{L}^\infty$-norm bound provided in the theorem statement, for the first term on the right-hand side of \eqref{eq:pre_M_bound_a_priori} one can derive that 
    \begin{align*}
            |(\varrho(\bsigma)^{-1}\bzeta-\varrho(\bsigma^\Pi_h)^{-1}\bF_d^{\overline{k},K}\bzeta,\bzeta_h - \bF_d^{\overline{k},K}\bzeta)|&\leq \frac{1}{2}L_\varrho\mathcal{K}_1\|\bsigma-{\bsigma}_h^\bbPi\|^2_{0,\Omega}+\left(\frac{1}{2}L_\varrho\mathcal{K}_1+\frac{\varrho_2}{2}\right)\|\bzeta_h-\bF_d^{\overline{k},K}\bzeta\|^2_{{0,\Omega}}\nonumber\\
        &\quad  +\frac{\varrho_2}{2}\|\bzeta-\bF_d^{\overline{k},K}\bzeta\|^2_{{0,\Omega}},
       \quad \text{and}  \\
            |\eta(|\bzeta|^2\bzeta-|\bF_d^{\overline{k},K}\bzeta|^2\bF_d^{\overline{k},K}\bzeta,\bzeta_h - \bF_d^{\overline{k},K}\bzeta)|&\leq \frac{9}{8\epsilon_1^{4/3}} 2^{2/3}\eta(1+\mathrm{C}_F)^{2/3}\mathcal{M}(\ell, \varphi_{\mathrm{D}})^{8/3}\|\bzeta-\bF_d^{\overline{k},K}\bzeta\|^{4/3}_{{0,4/3;\Omega}}\nonumber\\
        &\quad  +\frac{3}{8}\epsilon_1^4\eta\|\bzeta_h-\bF_d^{\overline{k},K}\bzeta\|^{4}_{{0,4;\Omega}},
        \end{align*}
where $\epsilon_1$ is a Young's inequality constant and $\mathrm{C}_F$ denotes the Fortin stability constant. Thus, applying Young's inequality once again, and using the continuity of the operators
$\mathcal{B}(\bullet)$, and $\mathcal{C}(\bullet)$,
we can deduce that
\begin{subequations}\label{bounds_a_priori_diff}
    \begin{align}
        &|[\mathcal{A}_{\bsigma}(\bzeta)-\mathcal{A}_{h,\bsigma_h^{\bbPi}}(\bF_d^{\overline{k},K}\bzeta), \bzeta_h - \bF_d^{\overline{k},K}\bzeta]|\leq\frac{1}{2}L_\varrho\mathcal{K}_1\|\bsigma-{\bsigma}_h^\bbPi\|^2_{0,\Omega}+\left(\frac{1}{2}L_\varrho\mathcal{K}+\frac{\varrho_2}{2}\right)\|\bzeta_h-\bF_d^{\overline{k},K}\bzeta\|^2_{{0,\Omega}}\nonumber\\
        & +\frac{\varrho_2}{2}\|\bzeta-\bF_d^{\overline{k},K}\bzeta\|^2_{{0,\Omega}} + \frac{9}{8\epsilon_1^{4/3}} 2^{2/3}\eta(1+\mathrm{C}_F)^{2/3}\mathcal{M}(\ell, \varphi_{\mathrm{D}})^{8/3}\|\bzeta-\bF_d^{\overline{k},K}\bzeta\|^{4/3}_{{0,4/3;\Omega}}+\frac{3}{8}\epsilon_1^4\eta\|\bzeta_h-\bF_d^{\overline{k},K}\bzeta\|^{4}_{{0,4;\Omega}} ,\\
        &|[\mathcal{B}(\bzeta_h - \bF_d^{\overline{k},K}\bzeta), \varphi-\Pi_k^{0,K}\varphi]|\leq \frac{\epsilon_2^4}{4}\|\bzeta_h - \bF_d^{\overline{k},K}\bzeta\|_{4,\vdiv;\Omega}^4+\frac{3}{4\epsilon_2^{4/3}}\|\varphi-\Pi_k^{0,K}\varphi\|_{0,\Omega}^{4/3},\\
        &|[\mathcal{B}(\bzeta-\bF_d^{\overline{k},K}\bzeta),\varphi_h-\Pi_k^{0,K}\varphi]|\leq \frac{1}{2\epsilon_3}\|\varphi_h-\Pi_k^{0,K}\varphi\|_{0,\Omega}^2+ \frac{\epsilon_3}{2}\|\bzeta-\bF_d^{\overline{k},K}\bzeta\|^2_{4,\vdiv;\Omega} ,\\
        &|[\mathcal{C}(\varphi-\Pi_k^{0,K}\varphi),\varphi_h-\Pi_k^{0,K}\varphi]|\leq \frac{1}{2\epsilon_4}\|\varphi_h-\Pi_k^{0,K}\varphi\|_{0,\Omega}^2+\frac{\epsilon_4}{2}\|\varphi-\Pi_k^{0,K}\varphi\|_{0,\Omega}^2.
    \end{align}   
    \end{subequations}
    Moreover, \eqref{eq:a_priori_diff_2} implies that $\|\vdiv(\bzeta_h - \bF_d^{\overline{k},K}\bzeta) \|_{0,\Omega} \leq \|\varphi_h - \Pi_k^{0,K}\varphi \|_{0,\Omega} + \|\widetilde{G} \|_{0,\Omega}$. Consequently,
\begin{equation*}
    \|\vdiv(\bzeta_h - \bF_d^{\overline{k},K}\bzeta) \|^4_{0,\Omega}\leq 8(\|\varphi_h-\Pi_k^{0,K}\varphi\|^4_{0,\Omega}+\|\widetilde{\mathcal{G}}\|^4_{0,\Omega}).
\end{equation*}
Accordingly, provided there exists a constant $\mathcal{K}_2$ and a threshold $h_0 > 0$ such that $8\|\varphi_h-\Pi_k^{0,K}\varphi\|^4_{0,\Omega} \leq \mathcal{K}_2\|\varphi_h-\Pi_k^{0,K}\varphi\|^2_{0,\Omega}$ for all $h \leq h_0$, we combine \eqref{eq:pre_M_bound_a_priori} and \eqref{bounds_a_priori_diff} to obtain the following result
\begin{align}\label{eq:pre_final_bound_a_priori_diffusion}
        &\varrho_1\|\bzeta_h - \bF_d^{\overline{k},K}\bzeta\|_{0,\Omega}^2+{\eta} \|\bzeta_h - \bF_d^{\overline{k},K}\bzeta\|_{0,4;\Omega}^4+\|\vdiv(\bzeta_h - \bF_d^{\overline{k},K}\bzeta)\|_{0,\Omega}^4 + \|\varphi_h-\Pi_k^{0,K}\varphi\|_{0, \Omega}^2 \nonumber \\
        & \quad \leq   \frac{1}{2}L_\varrho\mathcal{K}\|\bsigma-{\bsigma}_h^\bbPi\|^2_{0,\Omega}+\left(\frac{1}{2}L_\varrho\mathcal{K}_1+\frac{\varrho_2}{2}\right)\|\bzeta_h-\bF_d^{\overline{k},K}\bzeta\|^2_{{0,\Omega}}+\frac{\varrho_2}{2}\|\bzeta-\bF_d^{\overline{k},K}\bzeta\|^2_{{0,\Omega}} \nonumber\\
        &\qquad +\frac{9}{8\epsilon_1^{4/3}} 2^{2/3}\eta(1+\mathrm{C}_F)^{2/3}\mathcal{M}(\ell, \varphi_{\mathrm{D}})^{8/3}\|\bzeta-\bF_d^{\overline{k},K}\bzeta\|^{4/3}_{{0,4/3;\Omega}}+\frac{3}{8}\epsilon_1^4\eta\|\bzeta_h-\bF_d^{\overline{k},K}\bzeta\|^{4}_{{0,4;\Omega}}\nonumber\\
        &\qquad +\frac{\epsilon_2^4}{4}\|\bzeta_h - \bF_d^{\overline{k},K}\bzeta\|_{4,\vdiv;\Omega}^4+\frac{3}{4\epsilon_2^{4/3}}\|\varphi-\Pi_k^{0,K}\varphi\|_{0,\Omega}^{4/3}+\left(\frac{1}{2\epsilon_3}+\frac{1}{2\epsilon_4}+\mathcal{K}_2\right)\|\varphi_h-\Pi_k^{0,K}\varphi\|_{0,\Omega}^2\nonumber\\
        &\qquad + \frac{\epsilon_3}{2}\|\bzeta-\bF_d^{\overline{k},K}\bzeta\|^2_{4,\vdiv;\Omega}+\frac{\epsilon_4}{2}\|\varphi-\Pi_k^{0,K}\varphi\|_{0, \Omega}^2+64\|\varphi-\Pi_k^{0,K}\varphi\|_{0, \Omega}^4+64\|\bzeta_h-\bF_d^{\overline{k},K}\bzeta\|^4_{{4,\vdiv;\Omega}}.
    \end{align}
    Finally, under the assumption that $\left(\frac{1}{2}{L}_\varrho\mathcal{K}_1+\frac{\varrho_2}{2}\right)\leq \varrho_1$, $\left(\frac{1}{2\epsilon_3}+\frac{1}{2\epsilon_4}+\mathcal{K}_2\right)\leq 1$, $\epsilon_1^4\leq\frac{8}{3}$, and ${\epsilon_2^4}\leq \min\{2\eta,4\}$, the desired result follows by combining \eqref{bounds_a_priori_biot} and \eqref{eq:pre_final_bound_a_priori_diffusion} and invoking the triangle inequality.
\end{proof}
\begin{theorem}[convergence rates]\label{th:convergence}
    In addition to the hypotheses of Theorem \ref{theorem:a_priori}, assume that $$\left(\frac{(1+\alpha d)\beta}{2\mu+d\lambda}\right)^2+\frac{1}{2}L_{\varrho}\mathcal{K}_1< \frac{1}{M},$$
     where $M > 0$ denotes the maximum of the generic constants implicitly defined by the $\lesssim$ notation in inequalities \eqref{a_priori_poro} and \eqref{a_priori_diff}. Additionally, suppose that there exist $s \in [1, k+1]$ and $\overline{s} \in [1, \overline{k}+1]$ such that $\bsigma\in \bbH^{s}(\Omega),p\in \mathrm{H}^{s}(\Omega), \bu\in \bH^{s}(\Omega), \bz \in \bH^{\overline{s}}(\Omega),\bzeta \in \bH^{\overline{s}}(\Omega)$ and $\varphi\in \mathrm{H}^{s}(\Omega)$. Then, there holds
    \begin{align*}
        \mathrm{e}_h \lesssim h^{\mathrm{min}\{\frac{s}{3},\frac{\overline{s}}{3}\}}\left(\|\bsigma\|_{s,\Omega}+\|p\|_{s,\Omega}+\|\bu\|_{s,\Omega}+\|\bz\|_{\overline{s},\Omega}+\|\bzeta\|_{\overline{s},\Omega} + \|\varphi\|_{s,\Omega}\right), 
    \end{align*}
    where $\mathrm{e}_h:=\|\vec{\bsigma}-\vec{\bsigma}_h\|_{\bbV} + \|\vec{\bu}-\vec{\bu}_h\|_{\bQ }+     \| \bzeta-\bzeta_h\|_{4,\vdiv;\Omega}+\|\varphi-\varphi_h\|_{0,\Omega}$.
\end{theorem}

\begin{proof} The proof relies on estimates \eqref{a_priori_poro} and \eqref{a_priori_diff}, in conjunction with the smallness assumption and the approximation properties of the spaces stated in Propositions \ref{prop:est_poly}, \ref{prop:est_int_HR}, \ref{prop:est_int_mixed}, and in \cite{brenner08}.
\end{proof}
\begin{remark}
        {The suboptimal convergence rate of $\mathcal{O}(h^{\min\{s/3, \bar{s}/3\}})$ establishes a   theoretical lower bound that is inherently dictated by the  Forchheimer-type nonlinearity and the coupling with the Biot system. Handling the cubic or power-law terms within the momentum and fluid balance equations requires a repeated application of Hölder's and Young's inequalities in $L^p$ spaces, which introduces fractional exponents into the a priori error estimates. Similar suboptimal or reduced theoretical convergence rates have been observed in other coupled fluid-porous media formulations; see, for instance, \cite{Caucao2020,caucao2021}. Nevertheless, as demonstrated by the numerical experiments in Section~\ref{sec:results}, the scheme asymptotically recovers optimal convergence rates ($\mathcal{O}(h^s)$ and $\mathcal{O}(h^{\bar{s}})$). This gap between the theoretical bound and the practical performance suggests that the a priori analysis holds for the worst-case scenario under extreme physical regimes, whereas the smooth manufactured solutions and well-behaved source terms employed in standard benchmarks prevent the system from reaching the more pessimistic nonlinear bounds. To theoretically bridge this gap and formally recover the optimal rates, one could impose stronger regularity assumptions. More specifically, assuming $\mathrm{L}^\infty$-Lipschitz continuity for the nonlinear function $\varrho$ would permit shifting the error analysis from a Banach space framework back to a standard Hilbert  setting. While this restricts the analysis to smoother regimes, it circumvents the pessimistic accumulation of constants and fractional exponents. Combined with tailored duality techniques or sharper interpolation inequalities, this approach could yield the optimal bounds observed in practice.
        
         On the other hand, to derive clean a priori bounds for the stress tensor arising from the multiphysics coupling with the Biot formulation, it was necessary to invoke the $\mathrm{L}^\infty(\Omega)$-boundedness of the Fortin operator $\bF_d^{\overline{k},K}$. This assumption is essential to control the cross-coupling and nonlinear stress terms without demanding excessive solution regularity. Although the $\mathrm{L}^\infty$-stability of interpolation operators adds an extra layer of technicality, it is not unusual in the context of complex fluid flows and nonlinear porous media problems, where similar localised scaling and uniform boundedness arguments are used to balance the stability of the discrete spaces \cite{cangiani20,Dehghan2023}.}
\end{remark}

\section{Numerical results}\label{sec:results}
In this section we illustrate the accuracy and performance of the proposed scheme (cf. Section~\ref{sec:discrete}) through several numerical experiments.  We show the optimal behaviour of the method under different polytopal meshes. Finally, we simulate two application-oriented problems.  

We define the total computable error via the local polynomial approximation of the discrete solutions as $\bar{\mathrm{e}}_{h}^2 := \bar{\text{e}}_{\bsigma_h^{\bbPi}}^2+\bar{\text{e}}_{\bu_h}^2+\bar{\text{e}}_{\bz_h^{\bPi}}^2+\bar{\text{e}}_{p_h}^2+\bar{\text{e}}_{\bzeta_h^{\bPi}}^2+\bar{\text{e}}_{\varphi_h}^2$, with
\begin{align*}
    \bar{\text{e}}_{\bsigma_h^{\bbPi}}^2 &:= \|\bsigma-\bbPi_{k}^{\bC,K}\bsigma_h\|_{0,\Omega}^2 + \|\bdiv \bsigma-\bdiv \bsigma_h\|_{0,\Omega}^2, \quad
    \bar{\text{e}}_{\bu_h}^2 := \|\bu-\bu_h\|_{0,\Omega}^2, \\
    \bar{\text{e}}_{\bz_h^{\bPi}}^2 &:= \|\bz-\bPi_{\overline{k}}^{0,K}\bz_h\|_{0,\Omega}^2 + \|\vdiv \bz - \vdiv \bz_h\|_{0,\Omega}^2, \quad
    \bar{\text{e}}_{p_h}^2 := \|p-p_h\|_{0,\Omega}^2, \\
    \bar{\text{e}}_{\bzeta_h^{\bPi}}^2 &:= \|\bzeta-\bPi_{\overline{k}}^{0,K}\bzeta_h\|_{4,0;\Omega}^2 + \|\vdiv \bzeta - \vdiv \bzeta_h\|_{0,\Omega}^2,\quad
    \bar{\text{e}}_{\varphi_h}^2 := \|\varphi-\varphi_h\|_{0,\Omega}^2.
\end{align*}
The experimental rate of convergence $r(\bullet)$ applied to the total error {$\bar{\textnormal{e}}_h$} (or to any of its components) in the refinement $1\leq j$ is computed from the formula {$r(\bar{\textnormal{e}}_h)^{j+1} = \log(\bar{\textnormal{e}}_h^{j+1}/\bar{\textnormal{e}}_h^{j})/\log(h^{j+1}/h^{j})$}, where $h^j$ denotes the mesh size in the refinement $j$. The fixed-point algorithm has stopping criterion driven based on the $\ell^2$-norm of the increments (i.e., the difference between the DoFs at the iteration $i$ and $i-1$ of the fixed-point algorithm) with a tolerance of $5\times 10^{-6}$. We stress that these experiments were  implemented in the library \texttt{VEM++}  \cite{dassi2023vem++}.

Finally, following \cite{artioli18,visinoni24}, we define the stabilisation term $S_1^{\bC,K}(\bsigma_h,\btau_h) := (h_K\tr(\bC)/2)\int_{\partial K} \bsigma_h\bn \cdot \btau_h\bn $, while $S_2^{0,K}(\bullet,\bullet)$  and $S_3^{0,\bsigma_h^{\bbPi},K}(\bullet,\bullet)$ are given by a scaled \texttt{DOFI-DOFI} stabilisation (see \cite{khot25}), with respective scaling factors given by $\norm{\int_K \bkappa^{-1}}_{F}$ and $|\int_K \varrho^{-1}(\bsigma_h^{\bbPi})|$, where $\norm{\cdot}_F$ denotes the Frobenius norm of the matrix. {Regarding the interplay between the choice of stabilisation and the performance of the fixed-point algorithm; the standard \texttt{DOFI-DOFI} stabilisation is known to suffer from ill-conditioning when applied to badly shaped elements or high-order VEM spaces \cite{beirao17, botti25}. While alternative definitions of the stabilisation term (see e.g., \cite{beirao17}) mitigate this issue for high-order approximations, the \texttt{DOFI-DOFI} choice remains robust across the physical coupling parameters and regular mesh geometries tested here.}

\begin{figure}[!h]
    \centering
    \subfigure[Quadrilateral \label{fig:quad}]  {\includegraphics[width=0.24\textwidth,trim={9.5cm 1.25cm 8.75cm  1.5cm},clip]{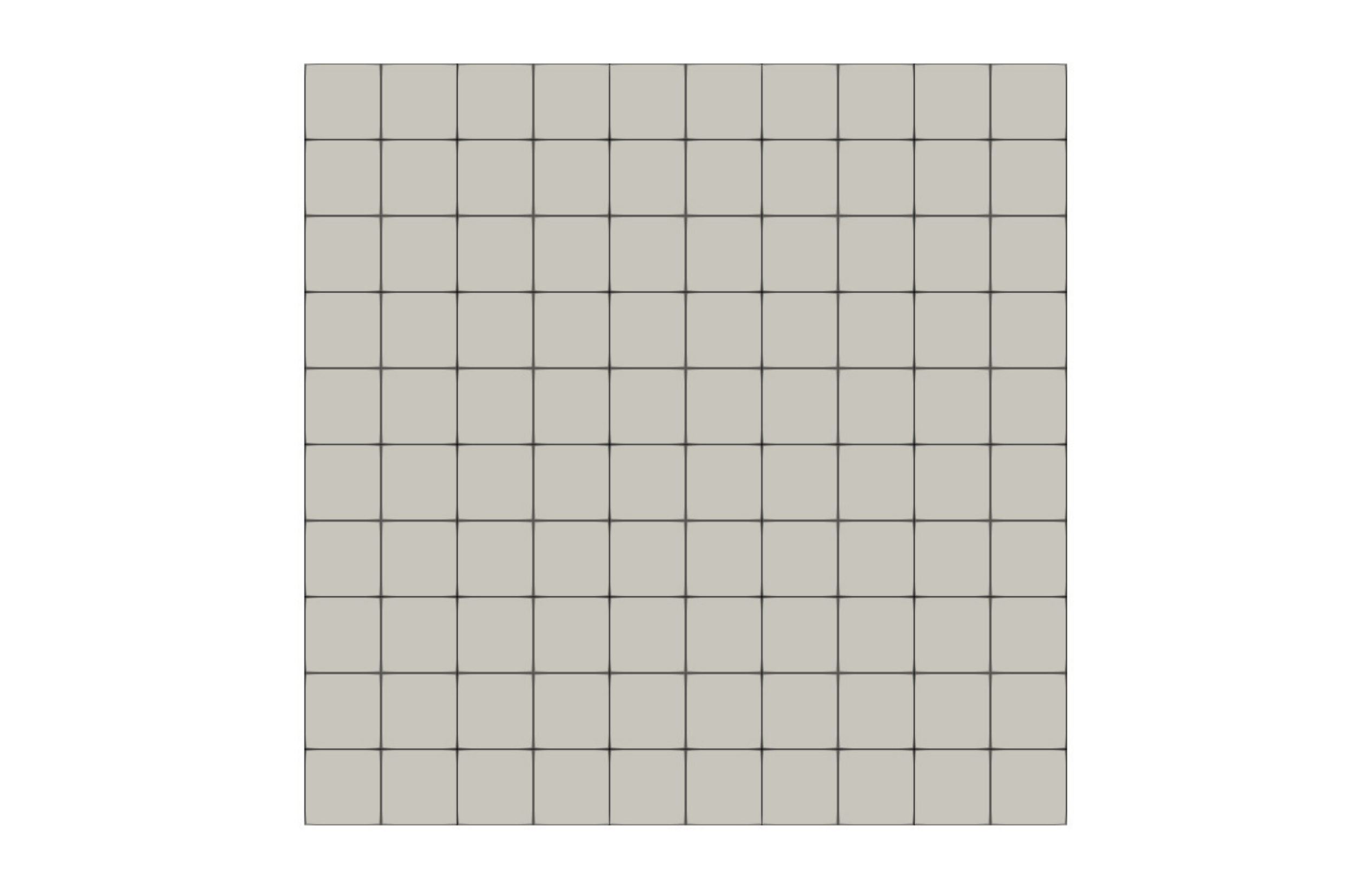}} 
    \subfigure[Distorted \label{fig:hexa}]  {\includegraphics[width=0.24\textwidth,trim={9.5cm 1.25cm 8.75cm  1.5cm},clip]{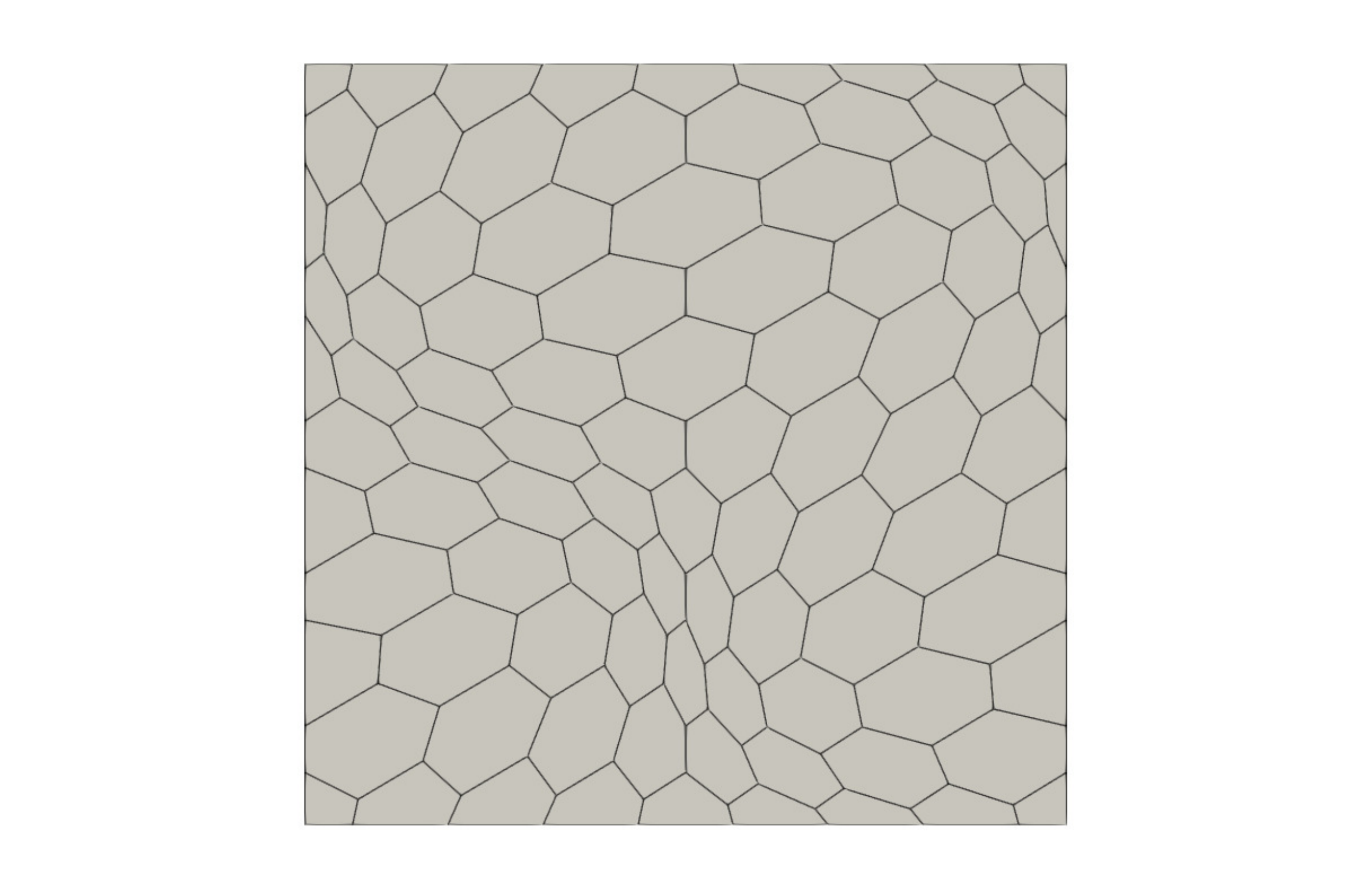}}  
    \subfigure[Hexahedral \label{fig:voro2}] {\includegraphics[width=0.24\textwidth,trim={9.5cm 1.25cm 8.75cm  1.5cm},clip]{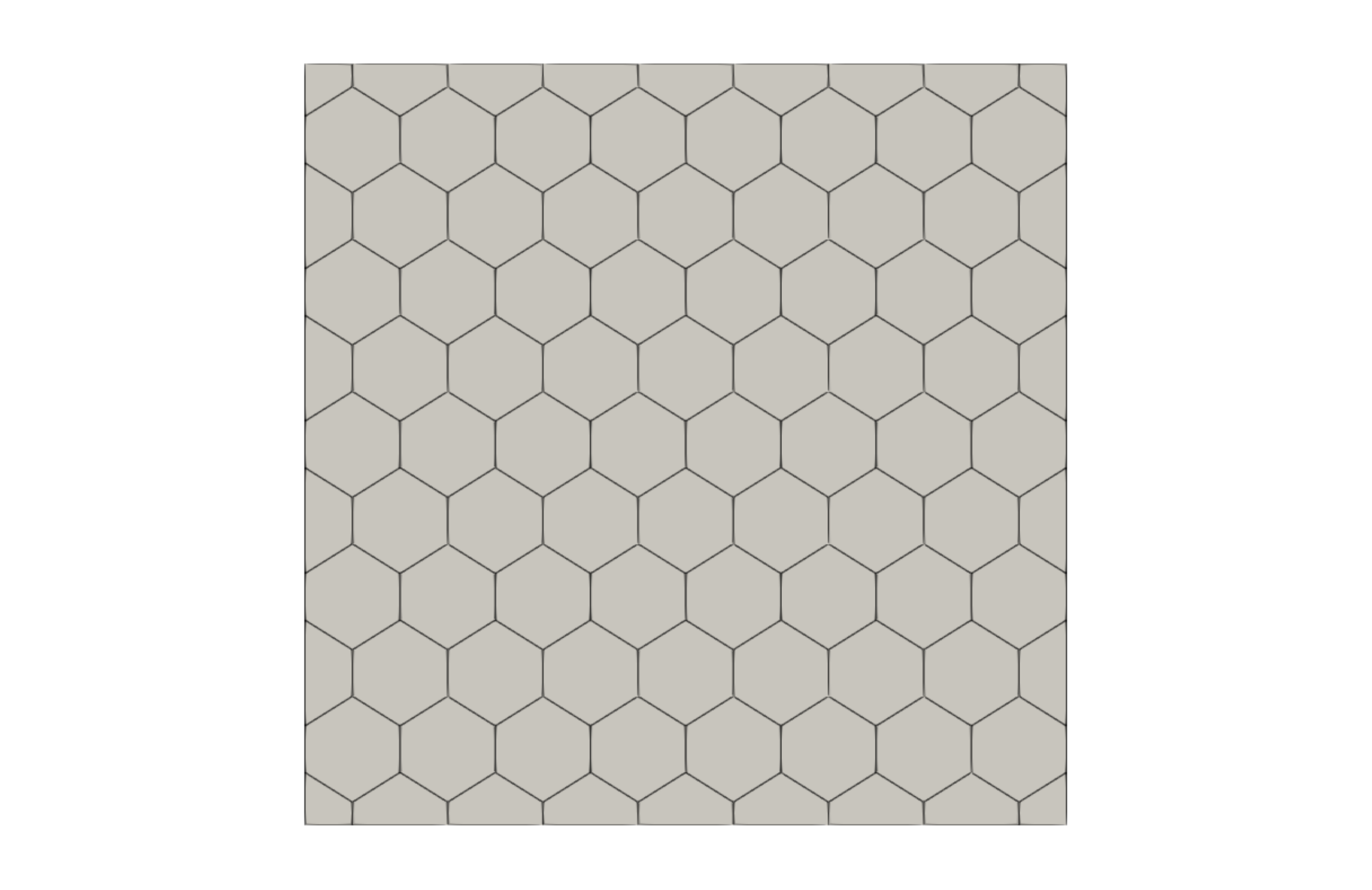}}
    \subfigure[Triangular \label{fig:tria}]  {\includegraphics[width=0.24\textwidth,trim={9.5cm 1.25cm 8.75cm  1.5cm},clip]{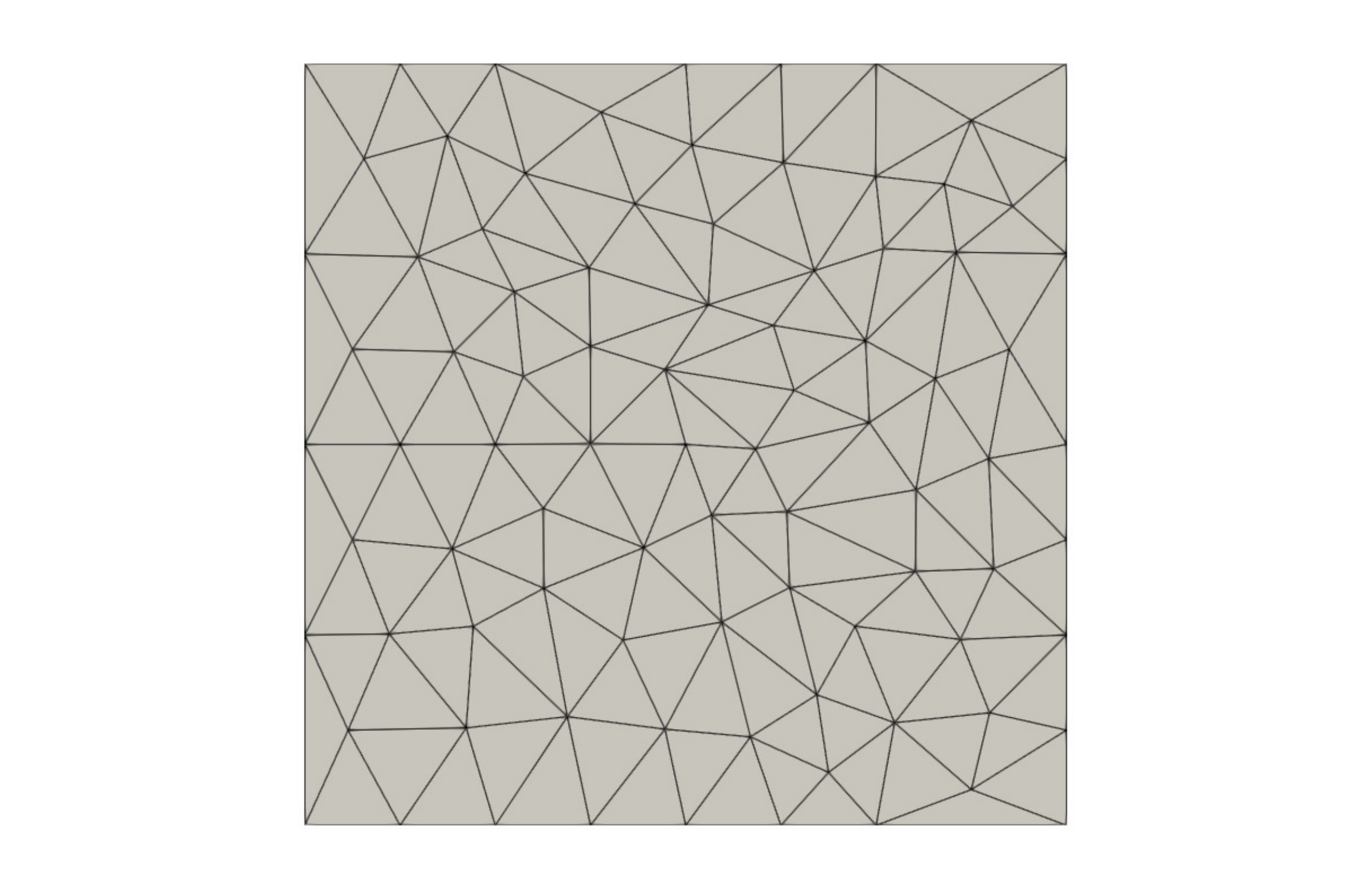}}
    \caption{Example 1. Variety of 2D meshes used in the uniform refinement convergence test.}\label{fig:meshes2D}
\end{figure}

\subsection{Example 1. Convergence rates under uniform refinement: 2D case} \label{sec:ex1}
For this test, the modulation parameter is prescribed as $\eta_1 = 10^{-3}$ (cf. \eqref{eq:D}), the Forchheimer coefficient is set to $\eta=5\times 10^{-4}$, and the permeability tensor is given by $\bkappa=10^{-2}\bbI$.  All remaining model parameters are fixed to unity. The smooth manufactured solutions are defined as follows
\begin{gather*}
  \bu(x_1,x_2) = \left( \cos(4\pi x_1)\cos(4\pi x_2)+e^{-x_1}, \sin(4\pi x_1)\sin(4\pi x_2)+e^{-x_2} \right)^{\tt t}, \\ p(x_1,x_2) = \cos(2\pi x_1)\cos(2\pi x_2) + e^{x_2}, \quad
  \varphi(x_1,x_2) = \sin(2\pi x_1)\sin(2\pi x_2) + e^{x_1}, \\ \bzeta(x_1,x_2) = -\rho(\bsigma(x_1,x_2))\nabla \varphi(x_1,x_2),
\end{gather*}
in the unit square domain $\Omega = (0,1)^2$ with the polygonal discretisations depicted in Figure~\ref{fig:meshes2D}, the boundary conditions are defined in the following sets: $\Gamma_{\mathrm{N}} = \left\{ (x_1,x_2)\in \partial \Omega : x_1=0 \text{ or } x_2=0\right\}$ and $\Gamma_{\mathrm{D}} = \partial \Omega \setminus \Gamma_{\mathrm{N}}$. In particular, the right-hand sides functions ($\mathbf{f}$, $g$ and $\ell$) and the stress-dependent diffusivity (cf. \eqref{eq:D}) are sufficiently smooth, as they are derived from the prescribed manufactured solutions.  {Note that an additional right-hand side term is incorporated into the system to account for the effect of the Forchheimer nonlinearity.} We recall that the polynomial order in two dimensions is given by $k$ for both the \gls{hr} \gls{vem} space and the mixed \gls{vem} space.

The error history is reported in Table~\ref{tab:convergence2D}. Here, we observe optimal rate of convergence $O(h^{k+1})$ ($k=1,2$) as predicted by {Theorem}~\ref{th:convergence} for all the meshes listed in Figure~\ref{fig:meshes2D}. Moreover, we {provide} in detail the computable error for the variables of interest, obtaining their expected optimal convergence rates. {The number of fixed-point iterations required for convergence is listed in the final column. A higher iteration count is anticipated here due to the presence of two nonlinear terms in the altered stress-assisted diffusion model; by comparison, the simplified model examined in Section~\ref{ex3} requires fewer iterations.} Snapshots of the variables of interest are shown in Figure~\ref{fig:manufacturedSols2D} for the Hexahedral mesh (see Figure~\ref{fig:meshes2D}\subref{fig:voro2}) in the last refinement step with polynomial degree $k=2$.

Finally, Table~\ref{tab:convergence2DRobust} illustrates the performance of the scheme under large variations of the physical parameters. The test considers nearly incompressible materials ($\lambda = 10^6$), very small storativity ($s_0 = 10^{-8}$), and weak Biot--Willis coupling ($\alpha = 10^{-6}$). The mesh is fixed to the Hexahedral case (cf. Figure~\ref{fig:meshes2D}\subref{fig:voro2}) and we set the polynomial degree $k=1$. Once again, we observe the expected optimal convergence rates, confirming the robustness of the method in these extreme settings. {Note that robustness with respect to the remaining coefficients cannot be guaranteed in the present setup and would necessitate an alternative formulation (e.g., the weighted-norm approach in \cite{khot25}).}

\begin{figure}[!t]
    \centering
    \subfigure[$|\bsigma_h^{\bbPi}|$.] {\includegraphics[width=0.325\textwidth,trim={7.5cm 4.25cm 2.75cm 4.25cm},clip]{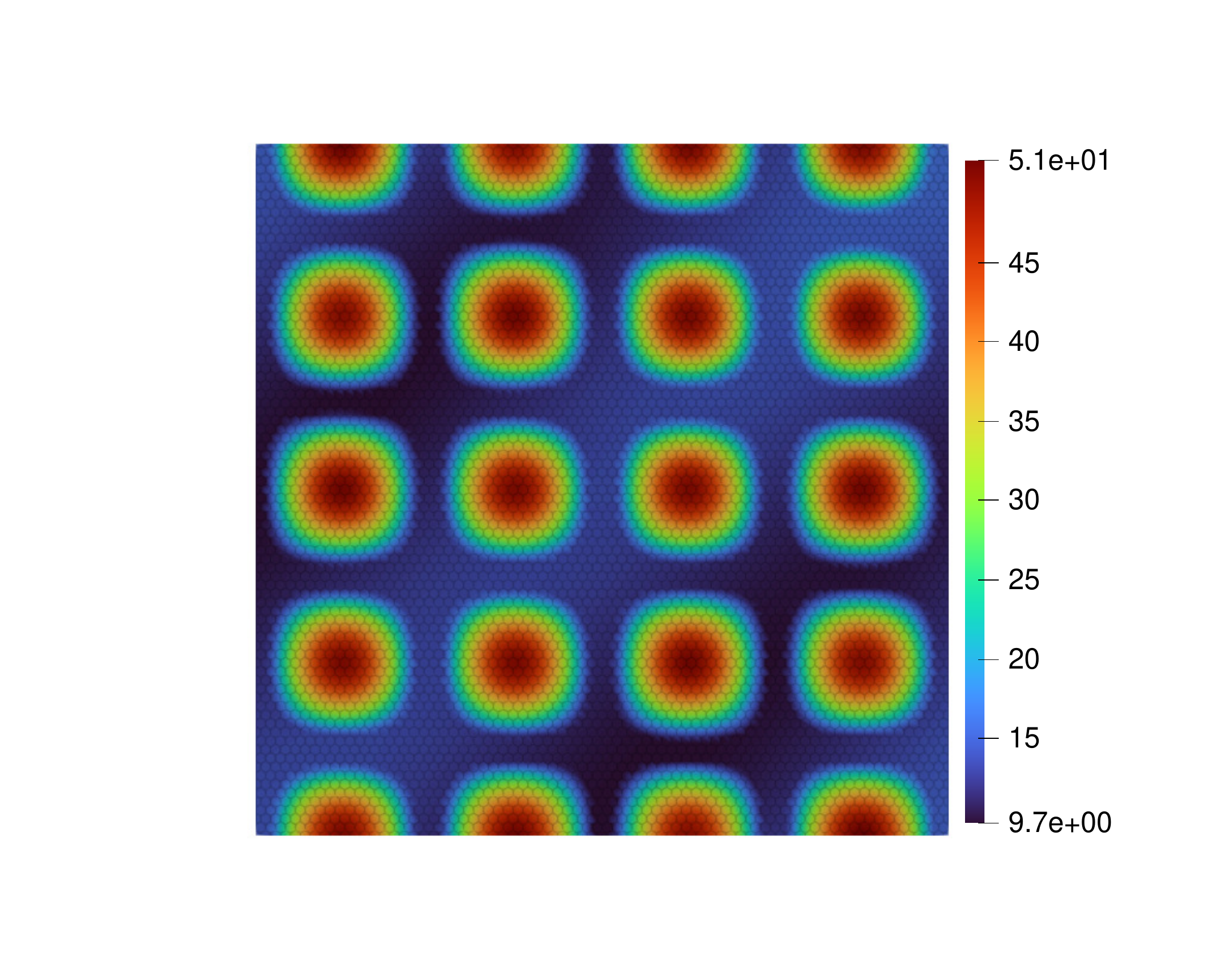}}  
    \subfigure[$|\bz_h^{\bPi}|$.]     {\includegraphics[width=0.325\textwidth,trim={7.5cm 4.25cm 2.75cm 4.25cm},clip]{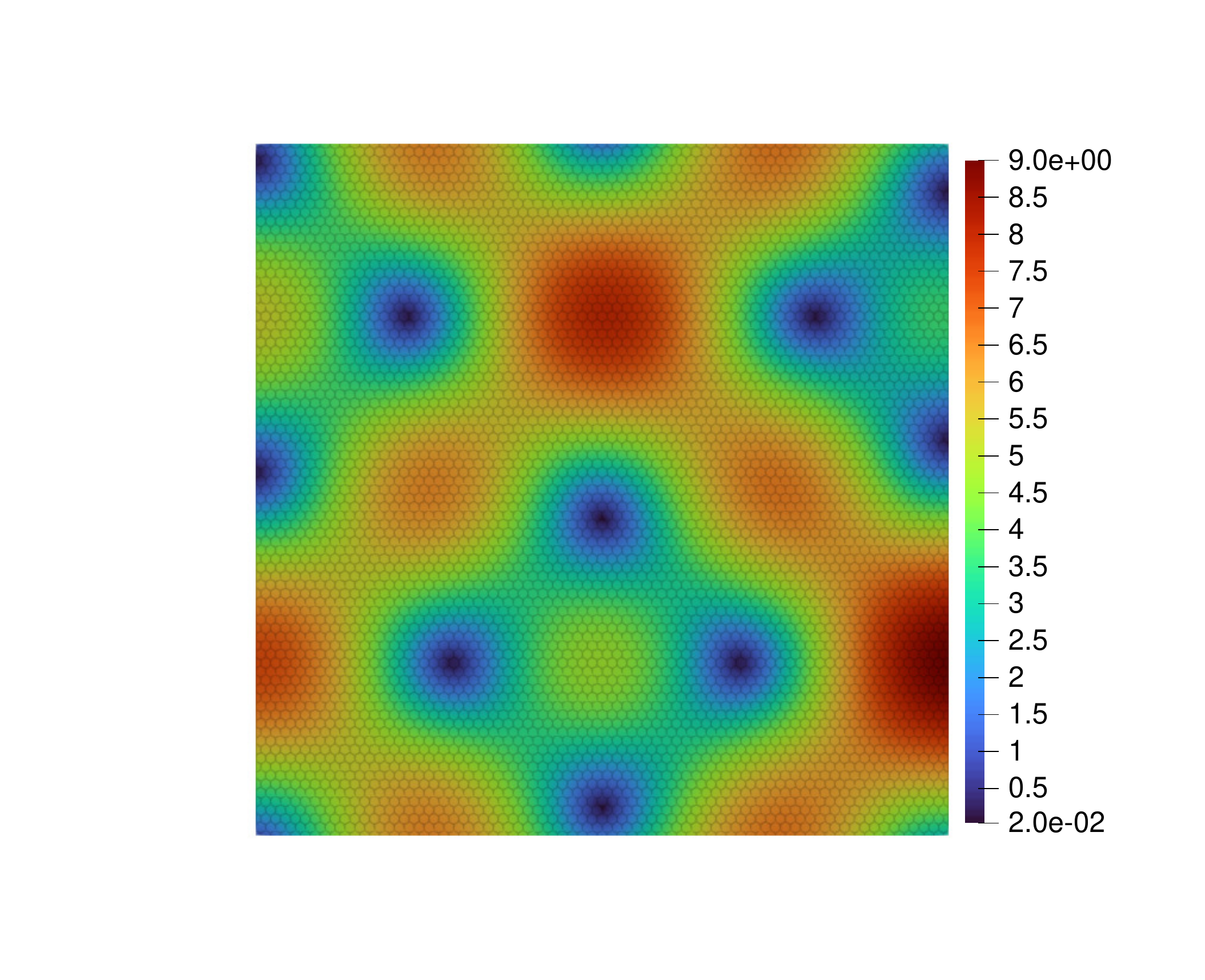}} 
    \subfigure[$|\bzeta_h^{\bPi}|$.]  {\includegraphics[width=0.325\textwidth,trim={7.5cm 4.25cm 2.75cm 4.25cm},clip]{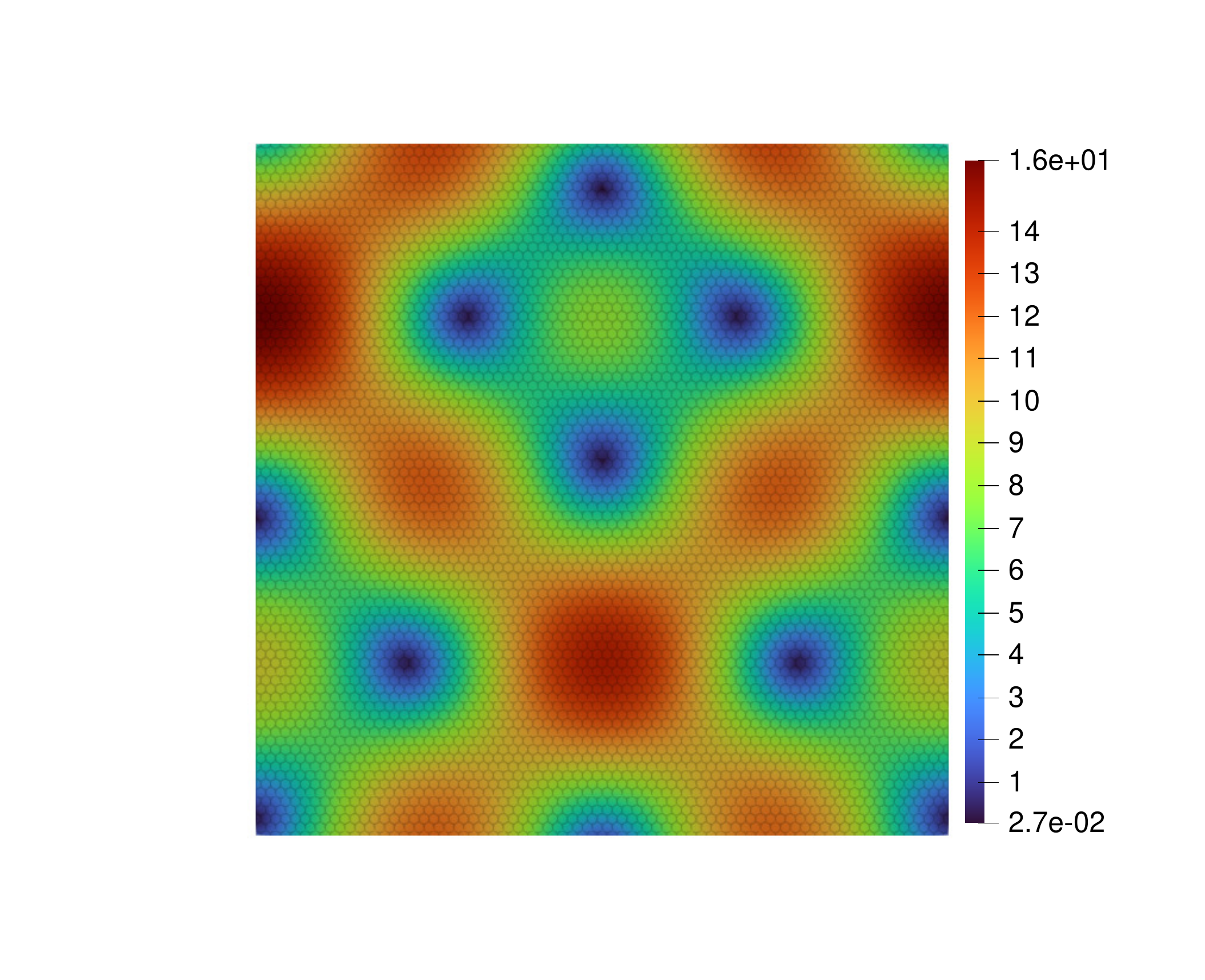}}
    \subfigure[$|\bu_h|$.]         {\includegraphics[width=0.325\textwidth,trim={7.5cm 4.25cm 2.75cm 4.25cm},clip]{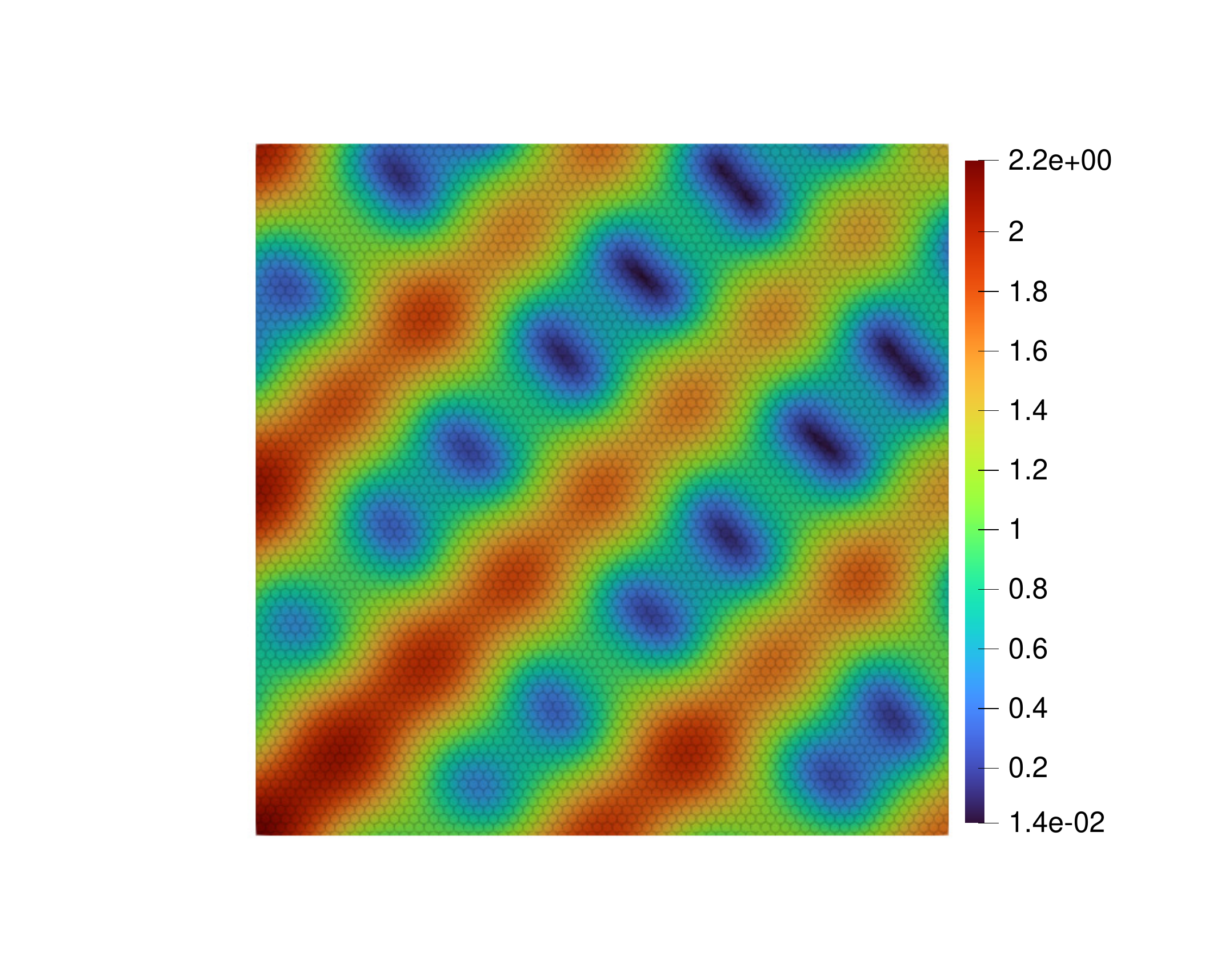}}
    \subfigure[$p_h$.]             {\includegraphics[width=0.325\textwidth,trim={7.5cm 4.25cm 2.75cm 4.25cm},clip]{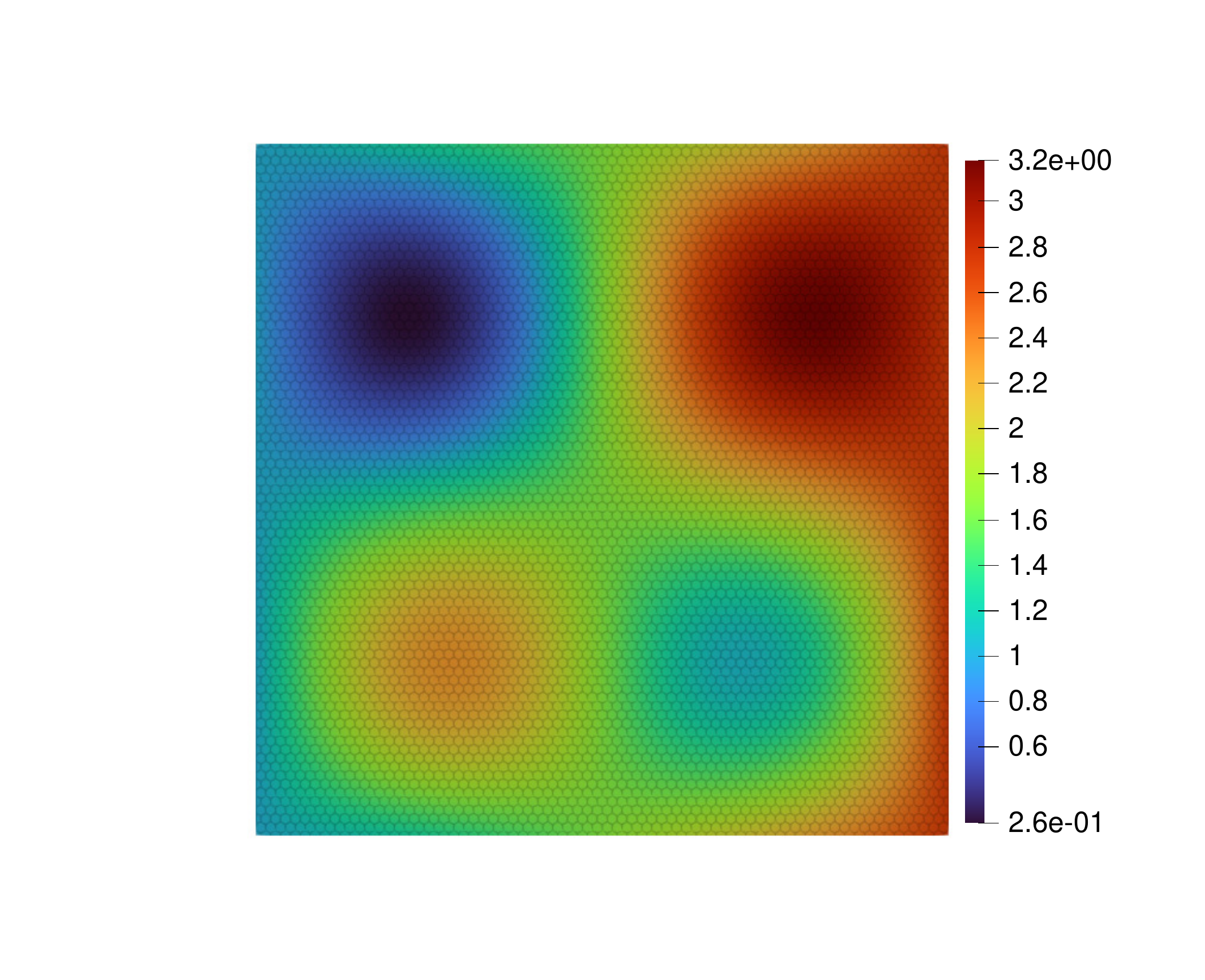}}
    \subfigure[$\varphi_h$.]       {\includegraphics[width=0.325\textwidth,trim={7.5cm 4.25cm 2.75cm 4.25cm},clip]{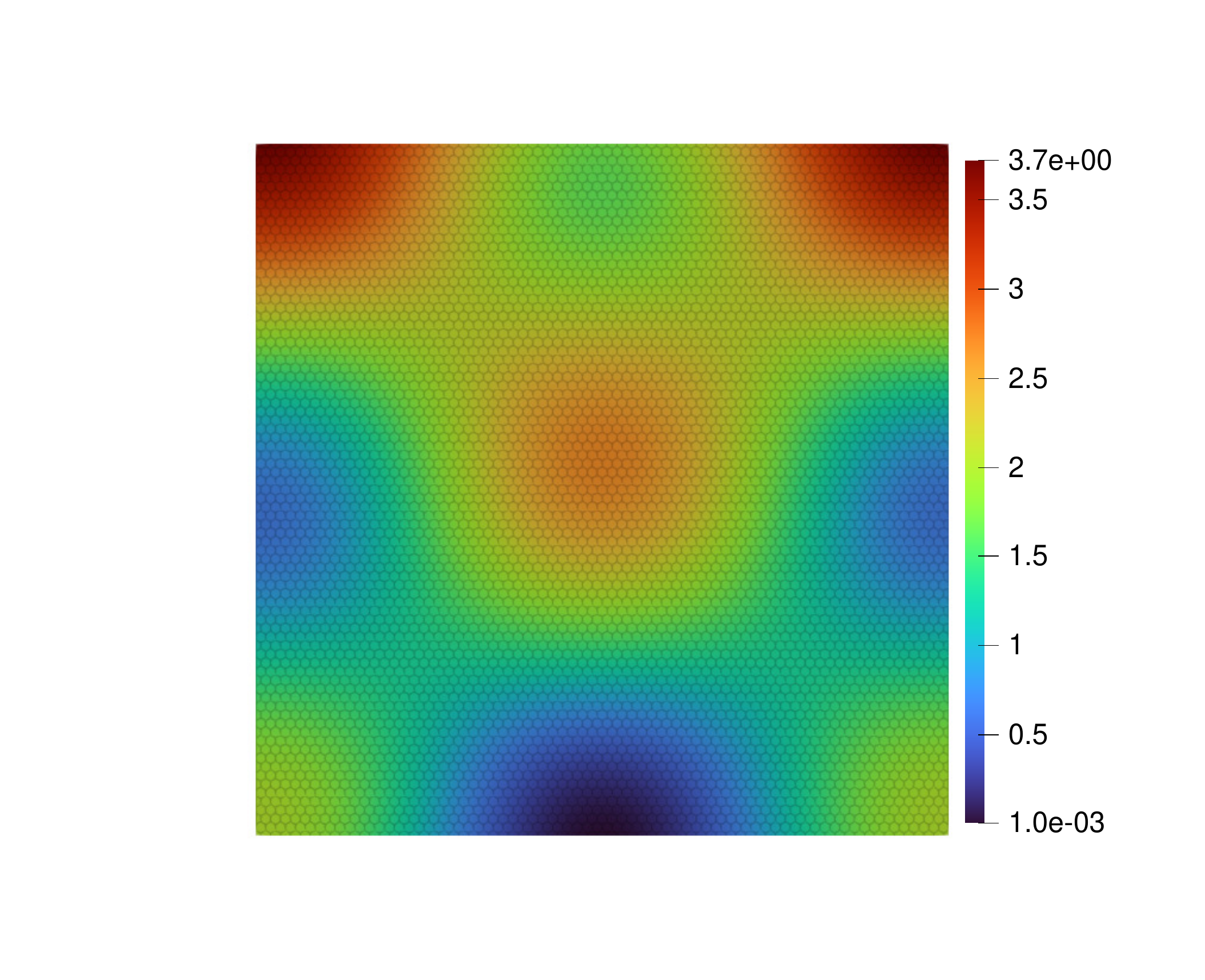}}
    \caption{Example 1. Snapshots of the variables of interest for the Hexahedral mesh in the last refinement step with $k=2$. The parameters are set to unity, except for $\eta_1 = 10^{-3}$, $\eta=5\times10^{-3}$, and $\bkappa=10^{-2}\bbI$.}\label{fig:manufacturedSols2D}
\end{figure}

\begin{table}[!t]
    \setlength{\tabcolsep}{2pt}
    \begin{center}
        \resizebox{\textwidth}{!}{ 
            \begin{tabular}{| c | c | c | c | c | c | c | c | c | c | c | c | c | c | c | c | c | c |}
                \hline
                {$k$} & 
                {$\cT_h$} & 
                {$h$} & 
                {$\bar{\mathrm{e}}_h$} & 
                {$r(\bar{\mathrm{e}}_h)$} & 
                {$\bar{\text{e}}_{\bsigma_h^{\bbPi}}$} & 
                {$r(\bar{\text{e}}_{\bsigma_h^{\bbPi}})$} & 
                {$\bar{\mathrm{e}}_{\bu_h}$} & 
                ${r(\bar{\mathrm{e}}_{\bu_h})}$ & 
                {$\bar{\mathrm{e}}_{\bz_h^{\bPi}}$} & 
                {$r(\bar{\mathrm{e}}_{\bz_h^{\bPi}})$} & 
                {$\bar{\mathrm{e}}_{p_h}$} & 
                {$r(\bar{\mathrm{e}}_{p_h})$} & 
                {$\bar{\mathrm{e}}_{\bzeta_h^{\bPi}}$} & 
                {$r(\bar{\mathrm{e}}_{\bzeta_h^{\bPi}})$} & 
                {$\bar{\mathrm{e}}_{\varphi_h}$} & 
                {$r(\bar{\mathrm{e}}_{\varphi_h})$} & 
                {it} \\ [3pt]
                \hline 
                \hline
                \multirow{16}{*}{$1$}
                &\multirow{4}{*}{\rotatebox{90}{Quadrilateral \rule{0pt}{1pt}}} 
                
                & 1.00e-01 & 3.28e+01 & *    & 3.27e+01 & *    & 1.26e-01 & *    & 6.74e-01 & *    & 1.16e-01 & *    & 3.31e+00 & *    & 1.94e-02 & *    & 13\\
                & & 5.00e-02 & 8.59e+00 & 1.93 & 8.56e+00 & 1.93 & 2.89e-02 & 2.13 & 1.55e-01 & 2.12 & 1.62e-02 & 2.84 & 8.46e-01 & 1.97 & 4.86e-03 & 1.99 & 14\\
                & & 2.50e-02 & 2.17e+00 & 1.98 & 2.17e+00 & 1.98 & 6.97e-03 & 2.05 & 2.63e-02 & 2.56 & 1.77e-03 & 3.19 & 2.20e-01 & 1.94 & 1.22e-03 & 2.00 & 14\\
                & & 1.25e-02 & 5.45e-01 & 2.00 & 5.43e-01 & 2.00 & 1.73e-03 & 2.01 & 3.82e-03 & 2.78 & 3.16e-04 & 2.48 & 6.04e-02 & 1.86 & 3.04e-04 & 2.00 & 15\\ [8pt]
                \cline{2-18}
                & \multirow{4}{*}{\rotatebox{90}{Distorted \rule{0pt}{1pt}}}  
                & 1.03e-01 & 4.04e+01 & *    & 4.03e+01 & *    & 2.89e-01 & *    & 6.59e-01 & *    & 1.83e-01 & *    & 4.31e+00 & *    & 2.53e-02 & *    & 13\\
                & & 5.07e-02 & 1.03e+01 & 1.93 & 1.02e+01 & 1.93 & 4.24e-02 & 2.70 & 1.19e-01 & 2.41 & 1.81e-02 & 3.26 & 1.04e+00 & 2.01 & 5.98e-03 & 2.03 & 14\\
                & & 2.66e-02 & 2.78e+00 & 2.02 & 2.77e+00 & 2.02 & 9.60e-03 & 2.30 & 1.86e-02 & 2.87 & 2.27e-03 & 3.22 & 2.76e-01 & 2.05 & 1.58e-03 & 2.06 & 15\\
                & & 1.32e-02 & 6.87e-01 & 2.00 & 6.85e-01 & 2.00 & 2.22e-03 & 2.10 & 2.59e-03 & 2.83 & 4.26e-04 & 2.40 & 7.04e-02 & 1.96 & 3.89e-04 & 2.01 & 16\\ [2pt]
                \cline{2-18}
                & \multirow{4}{*}{\rule{0pt}{10ex}\rotatebox{90}{Hexahedral \rule{0pt}{1pt}}} 
                & 1.03e-01 & 3.16e+01 & *    & 3.15e+01 & *    & 2.03e-01 & *    & 4.66e-01 & *    & 1.48e-01 & *    & 3.21e+00 & *    & 1.90e-02 & *    & 13\\
                & & 5.07e-02 & 7.75e+00 & 1.98 & 7.72e+00 & 1.98 & 3.10e-02 & 2.64 & 6.09e-02 & 2.86 & 9.58e-03 & 3.86 & 7.51e-01 & 2.04 & 4.44e-03 & 2.05 & 15\\
                & & 2.66e-02 & 2.06e+00 & 2.05 & 2.05e+00 & 2.05 & 7.05e-03 & 2.30 & 8.77e-03 & 3.00 & 1.33e-03 & 3.06 & 1.97e-01 & 2.08 & 1.16e-03 & 2.07 & 15\\
                & & 1.32e-02 & 5.09e-01 & 2.01 & 5.07e-01 & 2.01 & 1.64e-03 & 2.09 & 1.25e-03 & 2.79 & 3.00e-04 & 2.13 & 4.87e-02 & 2.00 & 2.86e-04 & 2.01 & 16\\
                \cline{2-18}
                & \multirow{4}{*}{\rotatebox{90}{Triangular \rule{0pt}{1pt}}} 
                & 4.36e-02 & 8.45e+00 & *    & 8.42e+00 & *    & 2.86e-02 & *    & 1.86e-01 & *    & 2.90e-02 & *    & 8.33e-01 & *    & 4.71e-03 & *    & 14\\
                & & 2.55e-02 & 2.90e+00 & 1.99 & 2.89e+00 & 1.99 & 9.41e-03 & 2.07 & 6.74e-02 & 1.89 & 2.51e-03 & 4.55 & 2.85e-01 & 2.00 & 1.61e-03 & 2.00 & 14\\
                & & 1.79e-02 & 1.45e+00 & 1.96 & 1.44e+00 & 1.96 & 4.63e-03 & 2.00 & 3.38e-02 & 1.94 & 9.14e-04 & 2.84 & 1.43e-01 & 1.93 & 8.07e-04 & 1.95 & 15\\
                & & 1.39e-02 & 8.76e-01 & 1.98 & 8.73e-01 & 1.98 & 2.79e-03 & 2.00 & 2.09e-02 & 1.91 & 5.11e-04 & 2.30 & 8.59e-02 & 2.03 & 4.84e-04 & 2.02 & 15\\
                \hline
                \hline
                \multirow{16}{*}{$2$}
                & \multirow{4}{*}{\rotatebox{90}{Quadrilateral \rule{0pt}{1pt}}} 
                & 1.00e-01 & 6.63e+00 & *    & 6.62e+00 & *    & 3.03e-02 & *    & 1.36e-01 & *    & 9.62e-03 & *    & 4.03e-01 & *    & 2.21e-03 & *    & 13\\
                & & 5.00e-02 & 8.67e-01 & 2.94 & 8.65e-01 & 2.94 & 2.95e-03 & 3.36 & 1.47e-02 & 3.21 & 4.62e-04 & 4.38 & 5.08e-02 & 2.99 & 2.73e-04 & 3.02 & 14\\
                & & 2.50e-02 & 1.10e-01 & 2.98 & 1.09e-01 & 2.98 & 3.52e-04 & 3.07 & 1.18e-03 & 3.64 & 3.36e-05 & 3.78 & 6.29e-03 & 3.02 & 3.41e-05 & 3.00 & 15\\
                & & 1.25e-02 & 1.37e-02 & 3.00 & 1.37e-02 & 3.00 & 4.36e-05 & 3.01 & 9.01e-05 & 3.71 & 3.87e-06 & 3.12 & 7.83e-04 & 3.01 & 4.27e-06 & 3.00 & 15\\ [8pt]
                \cline{2-18}
                & \multirow{4}{*}{\rotatebox{90}{Distorted \rule{0pt}{1pt}}} 
                & 1.03e-01 & 9.60e+00 & *    & 9.59e+00 & *    & 6.67e-02 & *    & 1.67e-01 & *    & 5.16e-02 & *    & 5.69e-01 & *    & 3.80e-03 & *    & 13\\
                & & 5.07e-02 & 1.19e+00 & 2.94 & 1.19e+00 & 2.94 & 5.15e-03 & 3.61 & 2.00e-02 & 2.99 & 6.50e-04 & 6.16 & 6.56e-02 & 3.04 & 3.72e-04 & 3.28 & 14\\
                & & 2.66e-02 & 1.66e-01 & 3.05 & 1.66e-01 & 3.05 & 5.84e-04 & 3.37 & 2.93e-03 & 2.97 & 4.87e-05 & 4.01 & 9.07e-03 & 3.07 & 5.11e-05 & 3.07 & 15\\
                & & 1.32e-02 & 2.04e-02 & 3.01 & 2.03e-02 & 3.01 & 6.61e-05 & 3.12 & 3.78e-04 & 2.93 & 5.68e-06 & 3.08 & 1.12e-03 & 3.00 & 6.25e-06 & 3.01 & 16\\ [2pt]
                \cline{2-18}
                & \multirow{4}{*}{\rule{0pt}{10ex}\rotatebox{90}{Hexahedral \rule{0pt}{1pt}}} 
                & 1.03e-01 & 5.87e+00 & *    & 5.87e+00 & *    & 3.74e-02 & *    & 8.12e-02 & *    & 8.70e-03 & *    & 3.38e-01 & *    & 2.40e-03 & *    & 13\\
                & & 5.07e-02 & 6.99e-01 & 3.00 & 6.98e-01 & 3.00 & 2.87e-03 & 3.62 & 1.02e-02 & 2.92 & 2.40e-04 & 5.05 & 3.73e-02 & 3.10 & 2.11e-04 & 3.42 & 14\\
                & & 2.66e-02 & 9.49e-02 & 3.09 & 9.48e-02 & 3.09 & 3.27e-04 & 3.37 & 1.41e-03 & 3.07 & 2.84e-05 & 3.30 & 5.09e-03 & 3.09 & 2.84e-05 & 3.10 & 15\\
                & & 1.32e-02 & 1.16e-02 & 3.01 & 1.16e-02 & 3.01 & 3.76e-05 & 3.10 & 1.79e-04 & 2.96 & 3.30e-06 & 3.09 & 6.19e-04 & 3.02 & 3.45e-06 & 3.02 & 16\\
                \cline{2-18}
                & \multirow{4}{*}{\rotatebox{90}{Triangular \rule{0pt}{1pt}}} 
                & 4.36e-02 & 8.49e-01 & *    & 8.48e-01 & *    & 2.85e-03 & *    & 2.00e-02 & *    & 4.56e-04 & *    & 6.51e-02 & *    & 3.42e-04 & *    & 14\\
                & & 2.55e-02 & 1.70e-01 & 3.00 & 1.69e-01 & 3.00 & 5.48e-04 & 3.06 & 3.76e-03 & 3.11 & 6.67e-05 & 3.58 & 1.27e-02 & 3.05 & 6.50e-05 & 3.09 & 14\\
                & & 1.79e-02 & 6.33e-02 & 2.77 & 6.32e-02 & 2.77 & 2.02e-04 & 2.81 & 1.37e-03 & 2.83 & 2.01e-05 & 3.38 & 4.68e-03 & 2.80 & 2.39e-05 & 2.82 & 15\\
                & & 1.39e-02 & 2.91e-02 & 3.08 & 2.90e-02 & 3.08 & 9.25e-05 & 3.10 & 6.39e-04 & 3.03 & 8.44e-06 & 3.43 & 2.15e-03 & 3.08 & 1.08e-05 & 3.13 & 15\\
                \hline
            \end{tabular}
        }
    \end{center}
    \vspace{-0.5cm}
    \caption{Example 1. Convergence history and fixed-point iteration count for a variety of 2D meshes with polynomial degrees $k=1,2$. The parameters are set to unity, except for $\eta_1 = 10^{-3}$, $\eta=5\times10^{-4}$, and $\bkappa=10^{-2}\bbI$.}
    \label{tab:convergence2D}
\end{table}

\begin{table}[!t]
    \setlength{\tabcolsep}{2pt}
    \begin{center}
        \resizebox{\textwidth}{!}{ 
            \begin{tabular}{| c | c | c | c | c | c | c | c | c | c | c | c | c | c | c | c | c |}
                \hline
                {} &
                {$h$} & 
                {$\bar{\mathrm{e}}_h$} & 
                {$r(\bar{\mathrm{e}}_h)$} & 
                {$\bar{\text{e}}_{\bsigma_h^{\bbPi}}$} & 
                {$r(\bar{\text{e}}_{\bsigma_h^{\bbPi}})$} & 
                {$\bar{\mathrm{e}}_{\bu_h}$} & 
                {$r(\bar{\mathrm{e}}_{\bu_h})$} & 
                {$\bar{\mathrm{e}}_{\bz_h^{\bPi}}$} & 
                {$r(\bar{\mathrm{e}}_{\bz_h^{\bPi}})$} & 
                {$\bar{\mathrm{e}}_{p_h}$} & 
                {$r(\bar{\mathrm{e}}_{p_h})$} & 
                {$\bar{\mathrm{e}}_{\bzeta_h^{\bPi}}$} & 
                {$r(\bar{\mathrm{e}}_{\bzeta_h^{\bPi}})$} & 
                {$\bar{\mathrm{e}}_{\varphi_h}$} & 
                {$r(\bar{\mathrm{e}}_{\varphi_h})$} & 
                {it} \\ [2pt]
                \hline 
                \hline
                \multirow{4}{*}{\rotatebox{90}{$\lambda = 10^6$}} 
                & 1.03e-01 & 4.60e+02 & *    & 4.58e+02 & *    & 4.10e+01 & *    & 1.55e-02 & *    & 1.90e-02 & *    & 1.70e+00 & *    & 1.90e-02 & *    & 5\\
                & 5.07e-02 & 1.01e+02 & 2.13 & 1.01e+02 & 2.13 & 4.45e+00 & 3.13 & 3.53e-03 & 2.09 & 4.39e-03 & 2.06 & 4.00e-01 & 2.04 & 4.44e-03 & 2.05 & 5\\
                & 2.66e-02 & 2.75e+01 & 2.02 & 2.75e+01 & 2.02 & 6.15e-01 & 3.06 & 9.27e-04 & 2.07 & 1.16e-03 & 2.06 & 1.05e-01 & 2.07 & 1.16e-03 & 2.07 & 5\\
                & 1.32e-02 & 6.67e+00 & 2.03 & 6.67e+00 & 2.03 & 7.44e-02 & 3.03 & 2.30e-04 & 2.00 & 2.86e-04 & 2.00 & 2.60e-02 & 2.00 & 2.86e-04 & 2.01 & 6\\
                \cline{1-17}
                \multirow{4}{*}{\rotatebox{90}{$s_0=10^{-8}$}} 
                & 1.03e-01 & 3.16e+01 & *    & 3.15e+01 & *    & 2.03e-01 & *    & 4.13e-01 & *    & 6.13e-02 & *    & 3.21e+00 & *    & 1.90e-02 & *    & 13\\
                & 5.07e-02 & 7.75e+00 & 1.98 & 7.72e+00 & 1.98 & 3.10e-02 & 2.64 & 5.97e-02 & 2.72 & 5.52e-03 & 3.39 & 7.51e-01 & 2.04 & 4.44e-03 & 2.05 & 15\\
                & 2.66e-02 & 2.06e+00 & 2.05 & 2.05e+00 & 2.05 & 7.05e-03 & 2.30 & 8.71e-03 & 2.98 & 1.18e-03 & 2.39 & 1.97e-01 & 2.08 & 1.16e-03 & 2.07 & 15\\
                & 1.32e-02 & 5.09e-01 & 2.01 & 5.07e-01 & 2.01 & 1.64e-03 & 2.09 & 1.24e-03 & 2.79 & 2.85e-04 & 2.03 & 4.87e-02 & 2.00 & 2.86e-04 & 2.01 & 16\\
                \cline{1-17}
                \multirow{4}{*}{\rotatebox{90}{$\alpha=10^{-6}$}} 
                & 1.03e-01 & 3.16e+01 & *    & 3.15e+01 & *    & 2.03e-01 & *    & 1.55e-02 & *    & 1.90e-02 & *    & 3.30e+00 & *    & 1.90e-02 & *    & 14\\
                & 5.07e-02 & 7.75e+00 & 1.98 & 7.72e+00 & 1.98 & 3.10e-02 & 2.64 & 3.53e-03 & 2.09 & 4.39e-03 & 2.06 & 7.74e-01 & 2.04 & 4.44e-03 & 2.05 & 15\\
                & 2.66e-02 & 2.06e+00 & 2.05 & 2.05e+00 & 2.05 & 7.05e-03 & 2.30 & 9.27e-04 & 2.07 & 1.16e-03 & 2.06 & 2.03e-01 & 2.07 & 1.16e-03 & 2.07 & 16\\
                & 1.32e-02 & 5.09e-01 & 2.01 & 5.07e-01 & 2.01 & 1.64e-03 & 2.09 & 2.30e-04 & 2.00 & 2.86e-04 & 2.00 & 5.02e-02 & 2.00 & 2.86e-04 & 2.01 & 17\\
                \hline
            \end{tabular}
        }
    \end{center}
    \vspace{-0.5cm}
    \caption{Example 1. Convergence history and fixed-point iteration counts are shown for the Hexahedral mesh with polynomial degree $k=1$ and extreme values for the parameters $\lambda$, $s_0$, and $\alpha$. In each test, the remaining parameters are set to unity, except $\eta_1=10^{-3}$, $\eta=5\times10^{-4}$, and $\bkappa=10^{-2}\bbI$.}
    \label{tab:convergence2DRobust}
\end{table}

\begin{figure}[!h]
    \centering
    \subfigure[Cubical \label{fig:cube}]{\includegraphics[width=0.24\textwidth,trim={7.cm 1.85cm 8.05cm 2.15cm},clip]{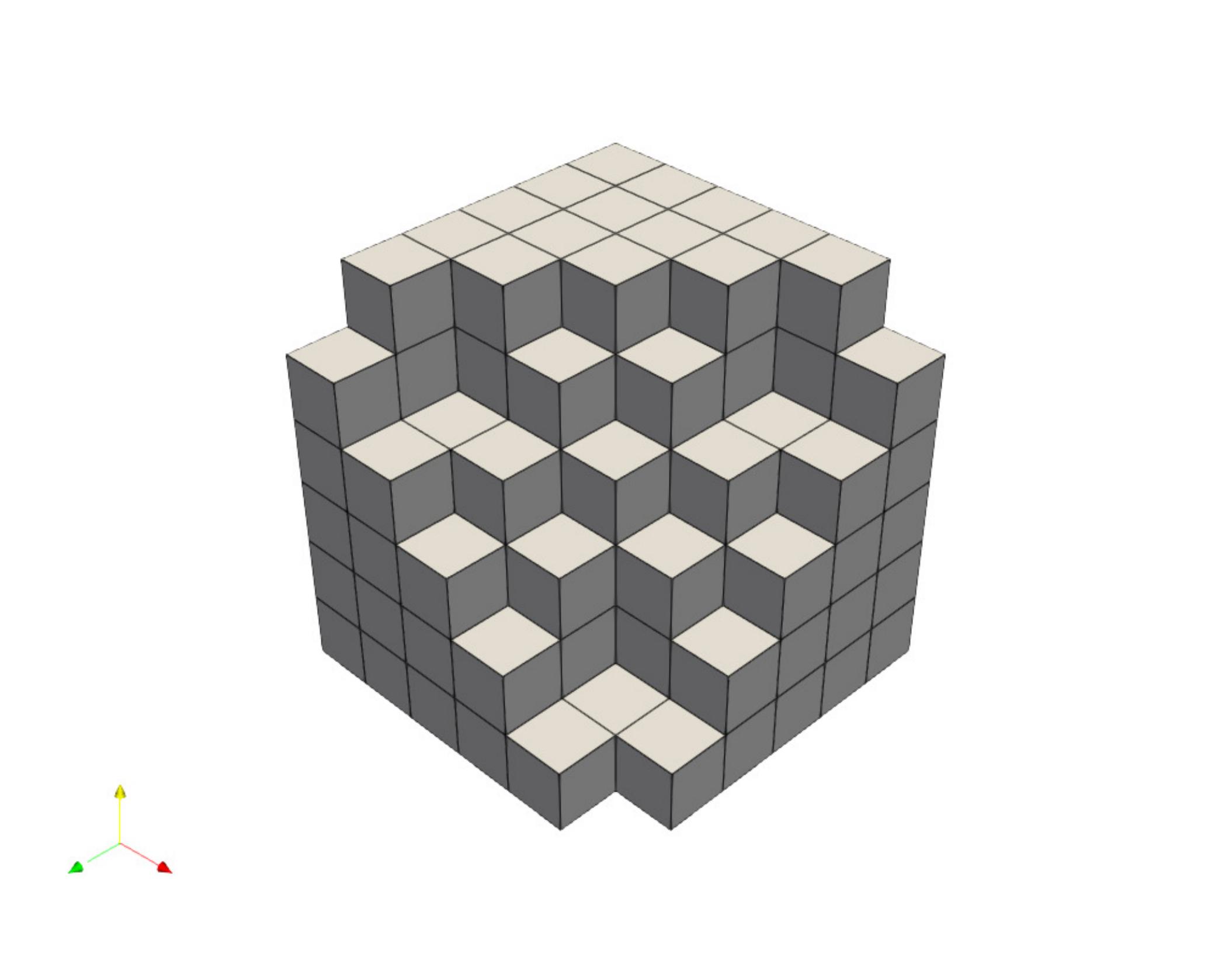}}  
    \subfigure[Octahedral \label{fig:octa}]{\includegraphics[width=0.24\textwidth,trim={7.cm 1.85cm 8.05cm 2.15cm},clip]{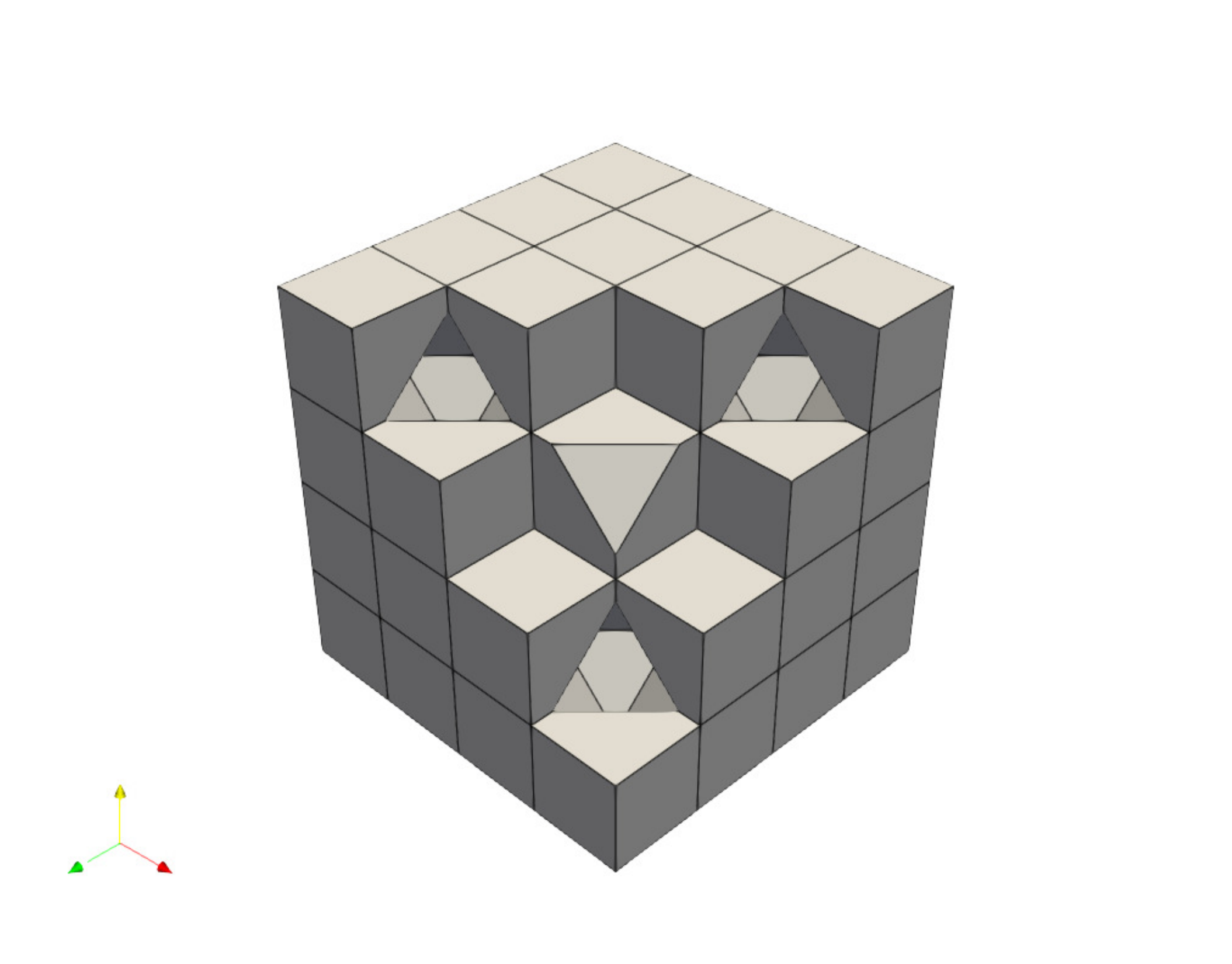}} 
    \subfigure[Voronoi \label{fig:voro}]{\includegraphics[width=0.24\textwidth,trim={7.cm 1.85cm 8.05cm 2.15cm},clip]{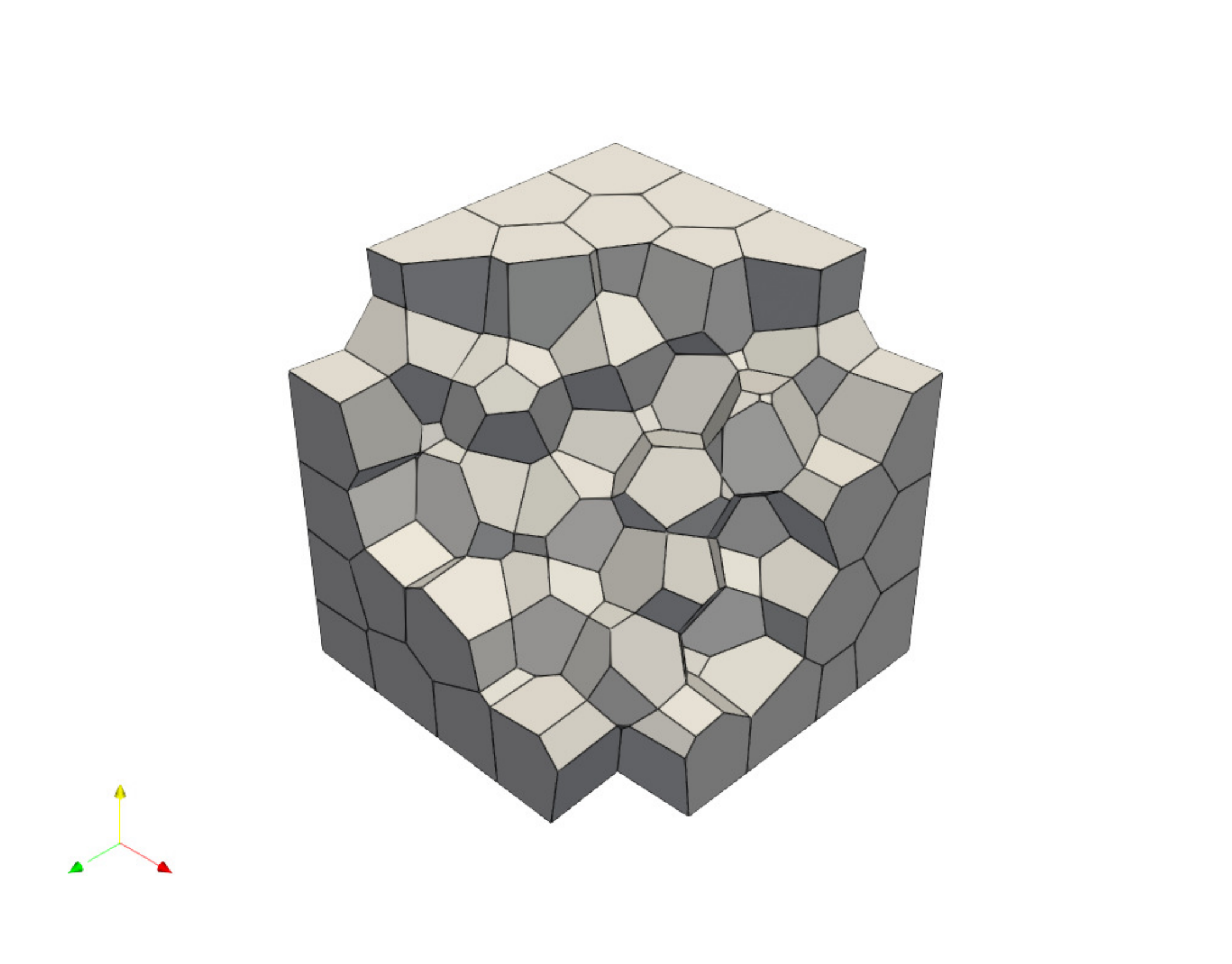}}
    \subfigure[Nonahedral \label{fig:nove}]{\includegraphics[width=0.24\textwidth,trim={7.cm 1.85cm 8.05cm 2.15cm},clip]{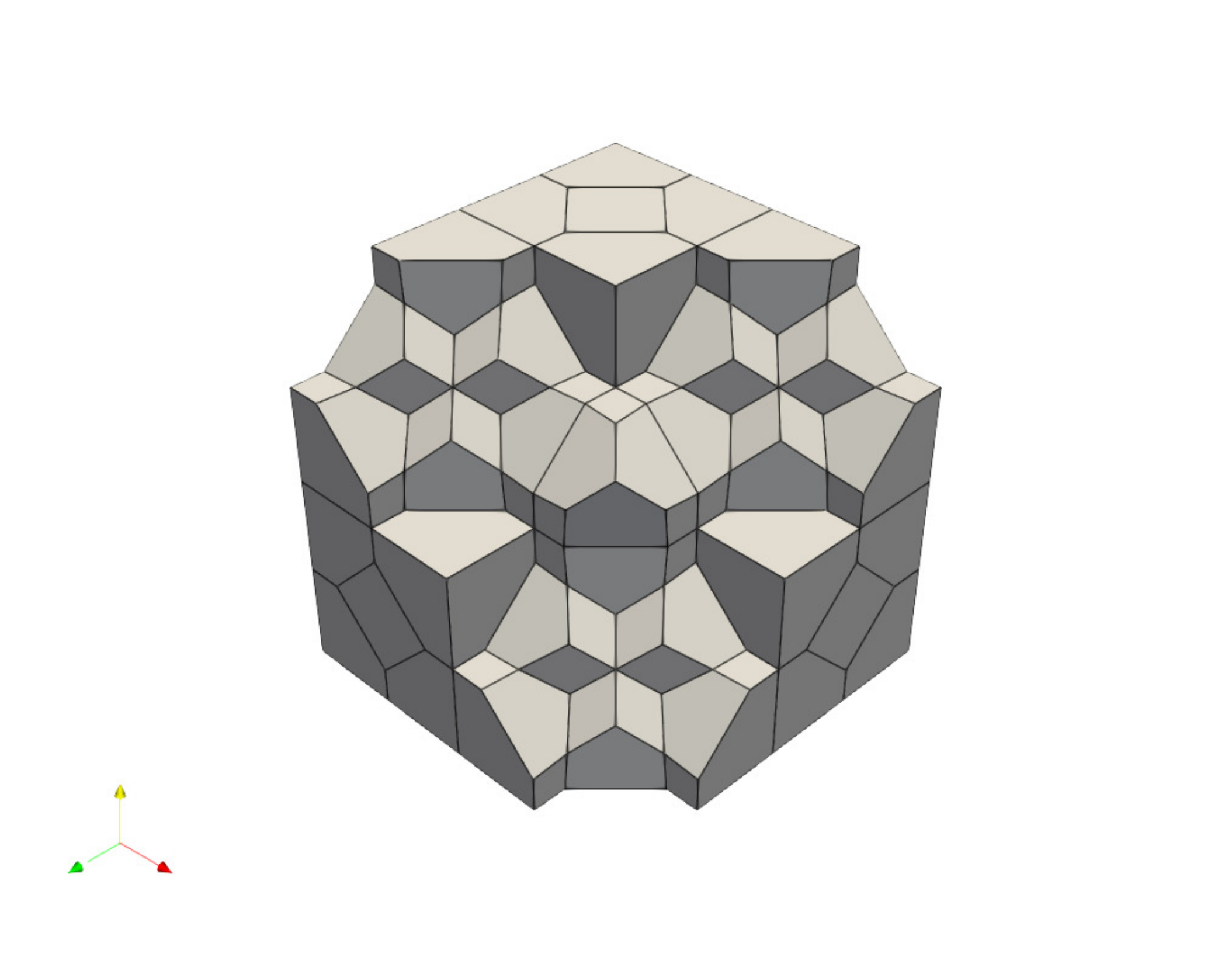}}
    \caption{Example 2. Cross-section of a variety of 3D meshes used in the uniform refinement convergence test.}\label{fig:meshes3D}
\end{figure}

\subsection{Example 2. Convergence rates under uniform refinement: 3D case} \label{sec:ex2}
We extend Example 1 by {considering} the unit cube domain $\Omega = (0,1)^3$ discretised using the polyhedral meshes illustrated in Figure~\ref{fig:meshes3D}, the sub-boundaries are defined by the sets $\Gamma_{\mathrm{N}} = \left\{ (x_1,x_2,x_3)\in \partial \Omega : x_1=0 \text{ or } x_2=0 \text{ or } x_3=0\right\}$ and $\Gamma_{\mathrm{D}} = \partial \Omega \setminus \Gamma_{\mathrm{N}}$. We set unity model parameters and define the manufactured solutions by
\begin{gather*}
  \bu(x_1,x_2,x_3) = \left( \cos(4\pi x_2)\cos(4\pi x_3) + e^{x_1}, \sin(4\pi x_1)\sin(4\pi x_3) + e^{x_2}, cos(4 \pi x_3)sin(4\pi x_1) + e^{x_3}\right)^{\tt t}, \\
  p(x_1,x_2,x_3) = \sin(2\pi x_2)\sin(2\pi x_3) + e^{x_1},\quad
  \varphi(x_1,x_2,x_3) = \cos(2\pi x_1)\cos(2\pi x_2) + e^{x_3}, \\  \bzeta(x_1,x_2,x_3) = -\rho(\bsigma(x_1,x_2,x_3))\nabla \varphi(x_1,x_2,x_3),
\end{gather*}
where the polynomial order in this case is given by $k$ for the \gls{hr} and $k+1$ for the mixed \gls{vem} spaces. One more time, all the model parameters are fixed to unity except for the modulation and Forchheimer parameters which are given by $\eta_1=5\times10^{-6}$ and $\eta=10^{-5}$.

In three dimensions, the computational cost increases substantially, even in the lowest-case order $k=1$; for example, the Hellinger–Reissner subsystem alone yields a linear system of dimension $542{,}925 \times 542{,}925$, with $100{,}405{,}246$ nonzero entries in the last refinement step of the Voronoi mesh (cf. Figure~\ref{fig:meshes3D}\subref{fig:voro}). Such system sizes would normally pose a considerable challenge, both in terms of memory requirements and solution time. However, \texttt{VEM++} exploits parallelisation through \texttt{MPI} and its interface with \texttt{PETSc-MUMPS} (see \cite{petsc2025}), which allows distributed assembly and the efficient solution of large-scale sparse systems. For this study, computations were performed on the NCI Gadi HPC cluster using the hugemem queue ($1{,}5$ TB of RAM per node with 48 CPUs), with 16 CPUs for the first two refinements, 32 CPUs for the third, and 64 CPUs for the final refinement. {Furthermore, algebraic efficiency within the fixed-point solver is enhanced by decoupling static and dynamic operators: bilinear forms are assembled only once at initialisation, leaving only coupling terms ($D_h$ and $A_{h,\boldsymbol{\sigma}_h^\Pi}$) to be updated in subsequent iterations. While beneficial in 2D, this optimisation becomes vital in 3D, where repeated matrix assembly would otherwise impose prohibitive computational overhead.}

\begin{figure}[!t]
    \centering
    \subfigure[$|\bsigma_h^{\bbPi}|$.] {\includegraphics[width=0.325\textwidth,trim={2.cm 0cm 0cm .8cm},clip]{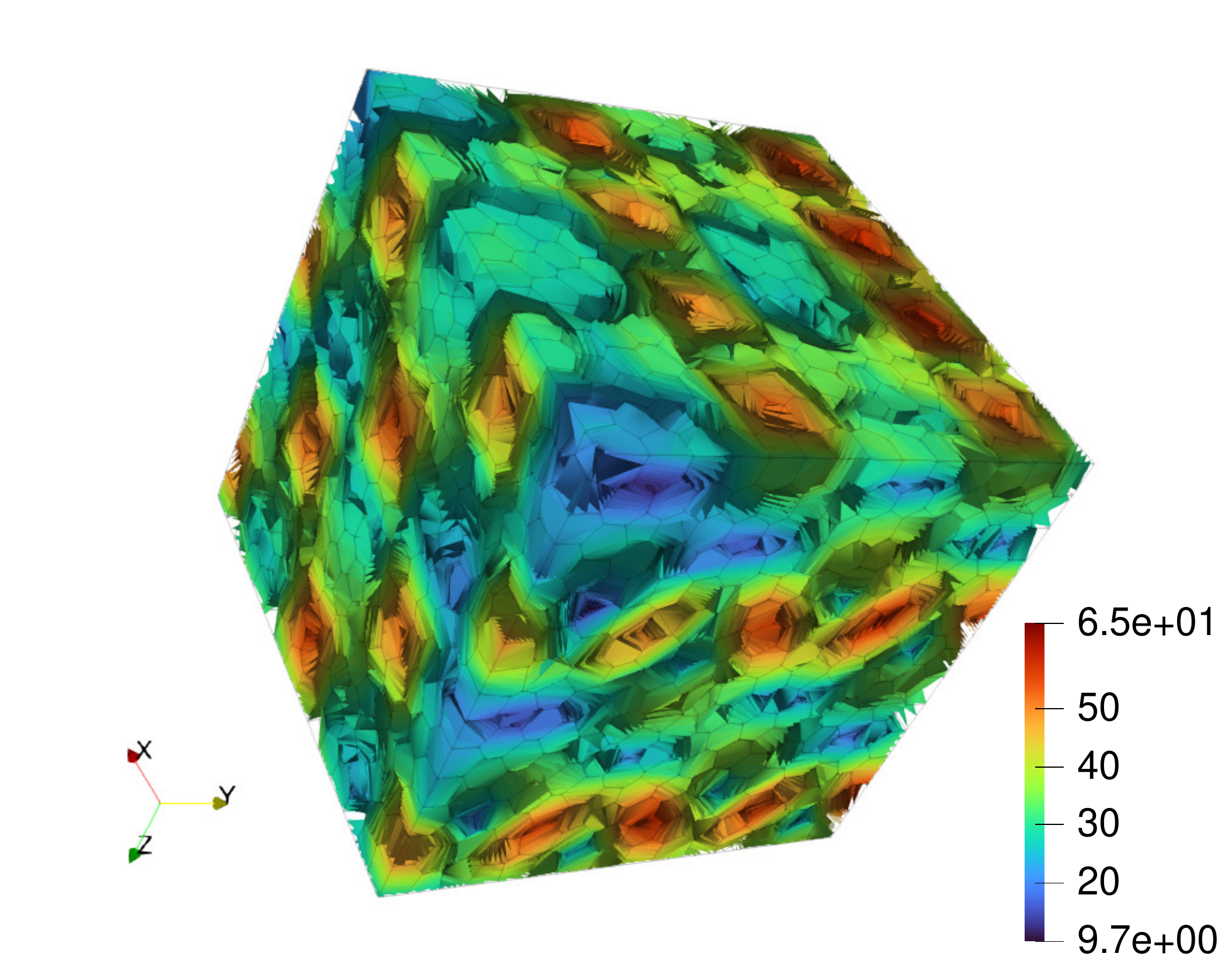}}  
    \subfigure[$|\bz_h^{\bPi}|$.]     {\includegraphics[width=0.325\textwidth,trim={2.cm 0cm 0cm .8cm},clip]{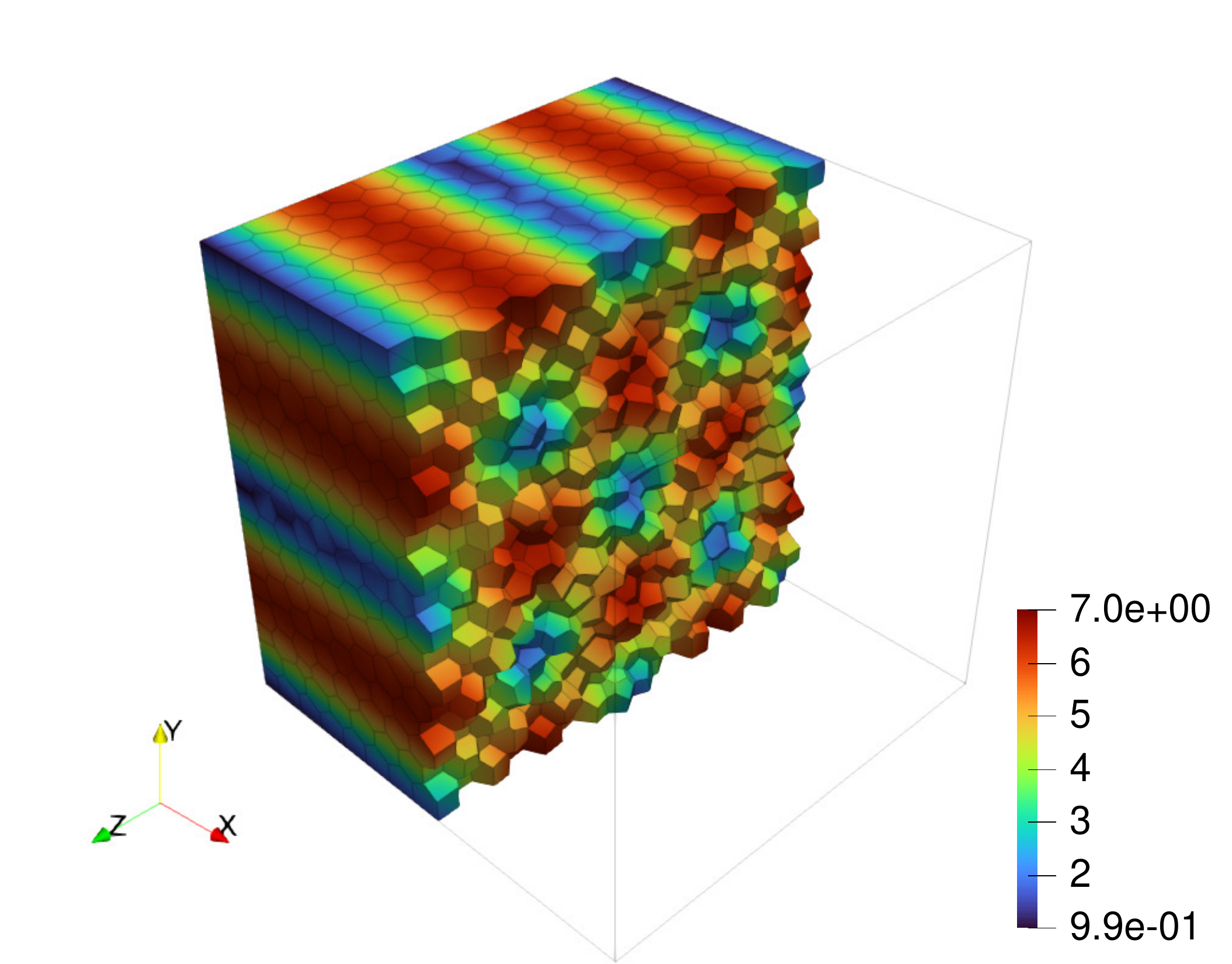}} 
    \subfigure[$|\bzeta_h^{\bPi}|$.]  {\includegraphics[width=0.325\textwidth,trim={2.cm 0cm 0cm .8cm},clip]{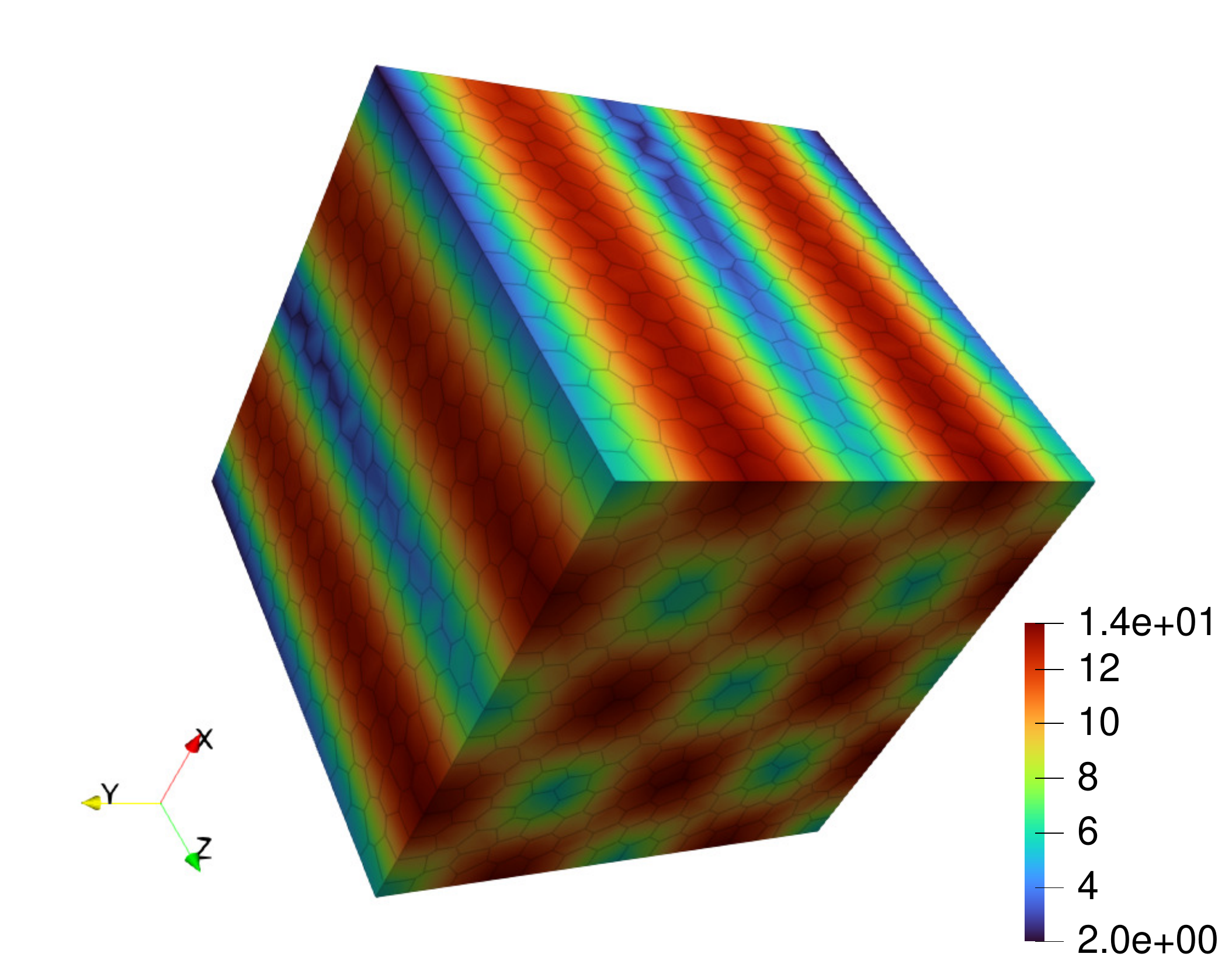}}
    \subfigure[$|\bu_h|$.]         {\includegraphics[width=0.325\textwidth,trim={2.cm 0cm 0cm .8cm},clip]{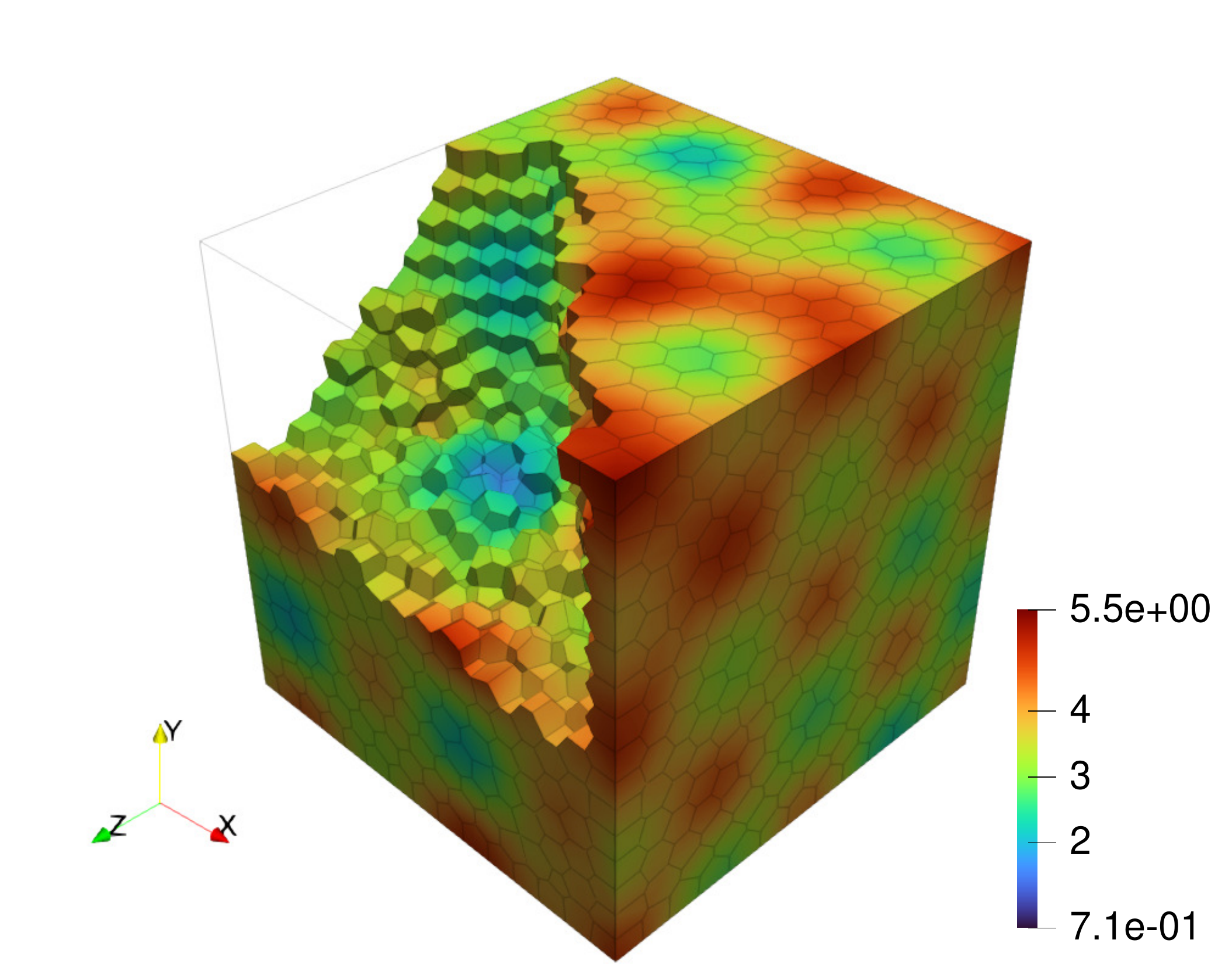}}
    \subfigure[$p_h$.]             {\includegraphics[width=0.325\textwidth,trim={2.cm 0cm 0cm .8cm},clip]{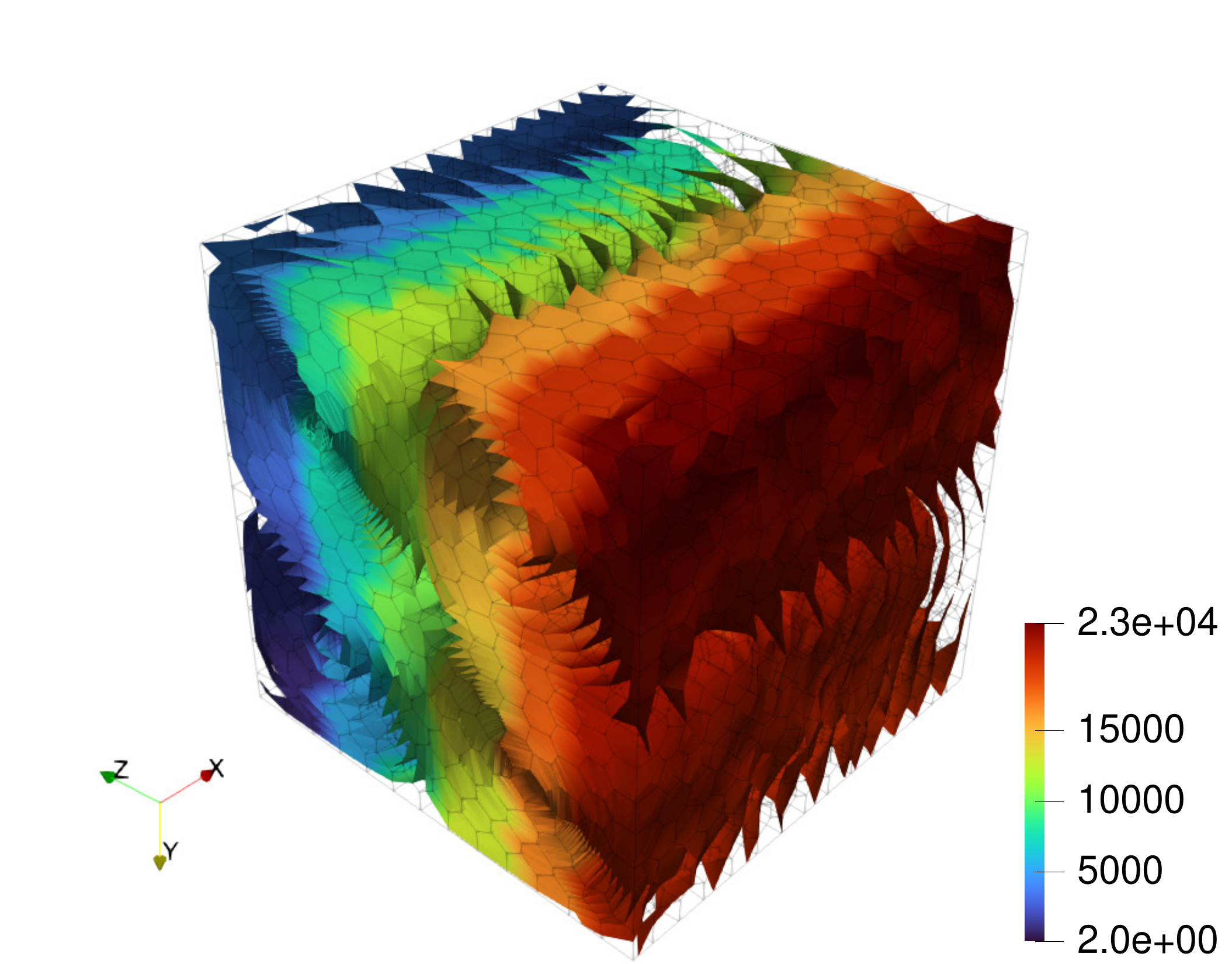}}
    \subfigure[$\varphi_h$.]       {\includegraphics[width=0.325\textwidth,trim={2.cm 0cm 0cm .8cm},clip]{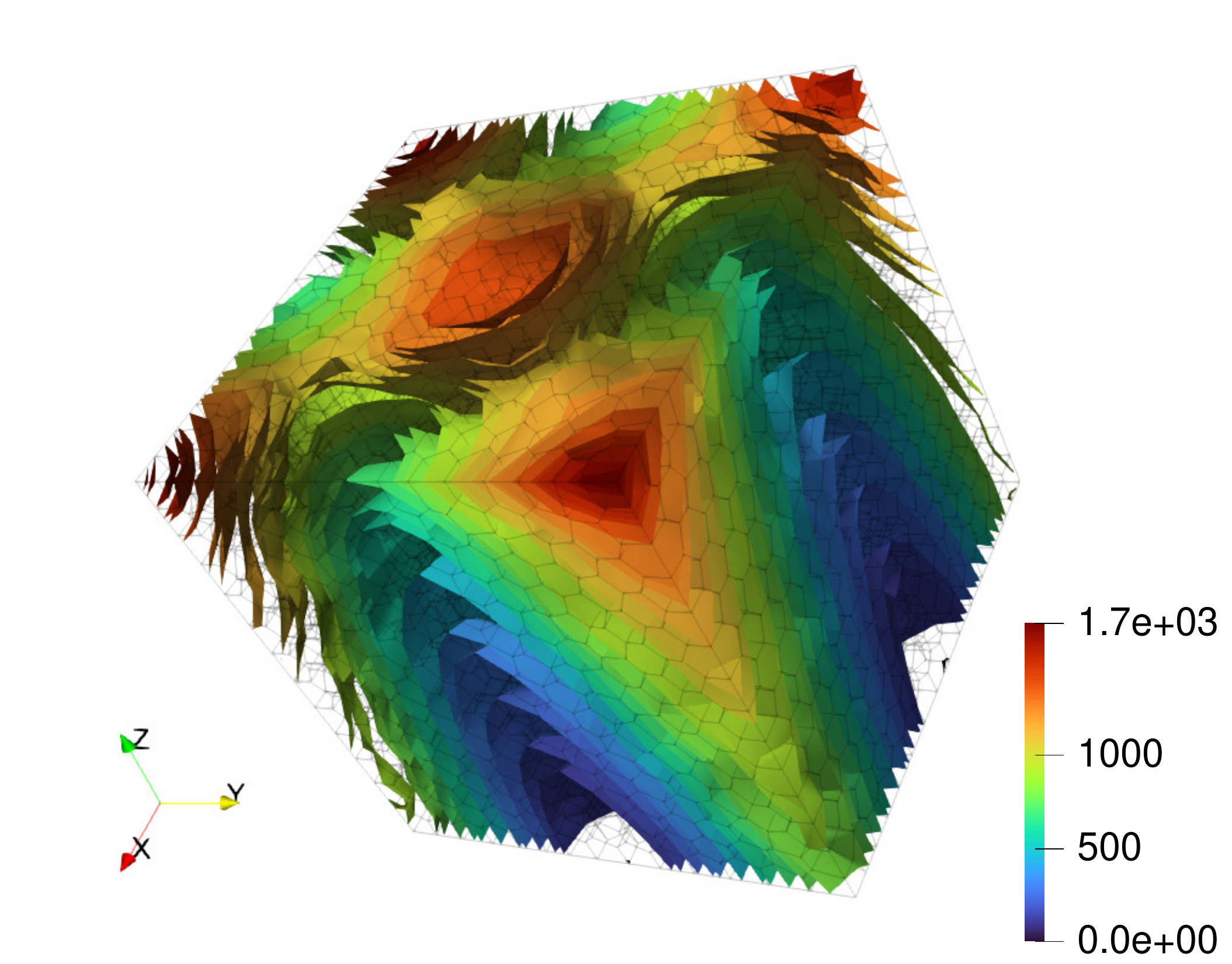}}
    \caption{Example 2. Snapshots of the variables of interest for the Voronoi mesh in the last refinement step with $k=1$. The modulation parameter is set to $\eta_1 = 10^{-5}$, while the remaining parameters are set to unity.}\label{fig:manufacturedSols3D}
\end{figure}

\begin{table}[!h]
    \setlength{\tabcolsep}{2pt}
    \begin{center}
        \resizebox{\textwidth}{!}{ 
            \begin{tabular}{| c | c | c | c | c | c | c | c | c | c | c | c | c | c | c | c | c |}
                \hline
                {$\cT_h$} & 
                {$h$} & 
                {$\bar{\mathrm{e}}_h$} & 
                {$r(\bar{\mathrm{e}}_h)$} & 
                {$\bar{\text{e}}_{\bsigma_h^{\bbPi}}$} & 
                {$r(\bar{\text{e}}_{\bsigma_h^{\bbPi}})$} & 
                {$\bar{\mathrm{e}}_{\bu_h}$} & 
                {$r(\bar{\mathrm{e}}_{\bu_h})$} & 
                {$\bar{\mathrm{e}}_{\bz_h^{\bPi}}$} & 
                {$r(\bar{\mathrm{e}}_{\bz_h^{\bPi}})$} & 
                {$\bar{\mathrm{e}}_{p_h}$} & 
                {$r(\bar{\mathrm{e}}_{p_h})$} & 
                {$\bar{\mathrm{e}}_{\bzeta_h^{\bPi}}$} & 
                {$r(\bar{\mathrm{e}}_{\bzeta_h^{\bPi}})$} & 
                {$\bar{\mathrm{e}}_{\varphi_h}$} & 
                {$r(\bar{\mathrm{e}}_{\varphi_h})$} & 
                {it} \\ [2pt]
                \hline 
                \hline
                \multirow{4}{*}{\rotatebox{90}{Cubical}} 
                & 2.17e-01 & 6.08e+01 & *    & 6.06e+01 & *    & 2.05e-01 & *    & 2.36e+00 & *    & 2.97e-02 & *  & 4.87e+00 & *    & 2.98e-02 & *    & 3\\
                & 1.44e-01 & 2.85e+01 & 1.87 & 2.84e+01 & 1.87 & 9.14e-02 & 1.99 & 1.06e+00 & 1.97 & 1.34e-02 & 1.97 & 2.23e+00 & 1.93 & 1.34e-02 & 1.97 & 3\\
                & 1.08e-01 & 1.63e+01 & 1.94 & 1.63e+01 & 1.94 & 5.18e-02 & 1.97 & 6.01e-01 & 1.98 & 7.56e-03 & 1.98 & 1.27e+00 & 1.94 & 7.57e-03 & 1.98 & 3\\
                & 8.66e-02 & 1.05e+01 & 1.96 & 1.05e+01 & 1.96 & 3.33e-02 & 1.98 & 3.85e-01 & 1.99 & 4.85e-03 & 1.99 & 8.24e-01 & 1.95 & 4.86e-03 & 1.99 & 3\\
                \hline
                \hline
                \multirow{4}{*}{\rotatebox{90}{Octahedral \rule{0pt}{1ex}}} 
                & 2.08e-01 & 5.82e+01 & *    & 5.80e+01 & *    & 1.95e-01 & *    & 2.22e+00 & *    & 2.78e-02 & *  & 4.60e+00 & *    & 2.79e-02 & *    & 3\\
                & 1.04e-01 & 1.53e+01 & 1.93 & 1.52e+01 & 1.93 & 4.86e-02 & 2.01 & 5.69e-01 & 1.96 & 7.08e-03 & 1.98 & 1.18e+00 & 1.96 & 7.08e-03 & 1.98 & 3\\
                & 8.33e-02 & 9.87e+00 & 1.96 & 9.84e+00 & 1.96 & 3.12e-02 & 1.98 & 3.68e-01 & 1.96 & 4.54e-03 & 1.99 & 7.60e-01 & 1.97 & 4.54e-03 & 1.99 & 3\\
                & 6.94e-02 & 6.89e+00 & 1.97 & 6.86e+00 & 1.97 & 2.18e-02 & 1.98 & 2.57e-01 & 1.96 & 3.16e-03 & 1.99 & 5.32e-01 & 1.95 & 3.16e-03 & 1.99 & 3\\
                \hline
                \hline
                \multirow{4}{*}{\rotatebox{90}{Voronoi}} 
                & 3.46e-01 & 1.28e+02 & *    & 1.27e+02 & *    & 4.68e-01 & *    & 5.55e+00 & *    & 7.01e-02 & *    & 1.29e+00 & *    & 7.00e-02 & *    & 2\\
                & 1.73e-01 & 3.65e+01 & 1.81 & 3.64e+01 & 1.80 & 1.26e-01 & 1.89 & 1.42e+00 & 1.96 & 1.77e-02 & 1.99 & 3.24e-01 & 2.00 & 1.70e-02 & 2.05 & 3\\
                & 1.37e-01 & 2.33e+01 & 1.95 & 2.32e+01 & 1.95 & 8.00e-02 & 1.97 & 8.98e-01 & 1.99 & 1.10e-02 & 2.04 & 1.98e-01 & 2.14 & 1.07e-02 & 1.99 & 3\\
                & 1.09e-01 & 1.48e+01 & 1.97 & 1.47e+01 & 1.97 & 5.06e-02 & 1.99 & 5.77e-01 & 1.91 & 6.90e-03 & 2.03 & 1.29e-01 & 1.84 & 6.79e-03 & 1.97 & 3\\
                \hline
                \hline
                \multirow{4}{*}{\rule{0pt}{3ex} \rotatebox{90}{Nonahedral \rule{0pt}{1ex}}} 
                & 3.59e-01 & 1.43e+02 & *    & 1.42e+02 & *    & 5.10e-01 & *    & 6.94e+00 & *    & 8.81e-02 & *  & 1.39e+01 & *    & 8.47e-02 & *    & 3\\
                & 1.80e-01 & 4.40e+01 & 1.70 & 4.39e+01 & 1.70 & 1.48e-01 & 1.78 & 1.64e+00 & 2.08 & 2.04e-02 & 2.11 & 3.31e+00 & 2.07 & 2.04e-02 & 2.06 & 3\\
                & 8.98e-02 & 1.12e+01 & 1.97 & 1.12e+01 & 1.97 & 3.70e-02 & 2.01 & 4.42e-01 & 1.89 & 5.16e-03 & 1.98 & 8.37e-01 & 1.98 & 5.16e-03 & 1.98 & 3\\
                & 7.19e-02 & 7.23e+00 & 1.96 & 7.21e+00 & 1.96 & 2.36e-02 & 2.00 & 2.95e-01 & 1.81 & 3.31e-03 & 1.99 & 5.43e-01 & 1.94 & 3.31e-03 & 1.99 & 3\\
                \hline
            \end{tabular}
        }
    \end{center}
    \vspace{-0.5cm}
    \caption{Example 2. Convergence history and fixed-point iteration count for a variety of 3D meshes for the lowest-case order $k=1$. We considered unit parameters except for $\eta_1=5\times10^{-6}$ and $\eta=10^{-5}$.}
    \label{tab:convergence3D}
\end{table}

We summarise the error history in Table~\ref{tab:convergence3D}. One more time, the prediction provided by Corollary~\ref{th:convergence} holds for all the meshes listed in Figure~\ref{fig:meshes3D}, we observe optimal rate of convergence $O(h^2)$. Snapshots of the variables of interest are shown in Figure~\ref{fig:manufacturedSols3D} for the Voronoi mesh with 4,000 elements (last refinement step).

\subsection{Example 3. Convergence rates for the pure \gls{sad} case} \label{ex3}
Although the pure \gls{sad} model is not covered by the present theoretical analysis, we investigate here the approximation properties of the scheme in the particular case where the Forchheimer parameter is set to zero. {In this setting, fluid flow follows classical linear Darcy poroelasticity, while solute flux is directly modulated by mechanical stress gradients via $\mathbf{J} = -\mathbf{D}(\boldsymbol{\sigma})\nabla c$. Here, local volumetric stress $\text{tr}(\boldsymbol{\sigma})$ acts as a driving potential that modulates spatial diffusion and induces directional drift. This regime is relevant across several multiphysics applications including 
(i) brain mechanics and glymphatics: Mechanosensitive clearance of metabolic waste (e.g., Amyloid-$\beta$ and Tau proteins) in deformable parenchymal tissue subject to local edema or intracranial pressure gradients; 
(ii) Lithium-ion battery electrodes: Intercalation-induced stress diffusion, where volumetric expansion and contraction of active electrode particles creates localised stress fields that alter chemical potential gradients and lithium-ion diffusivity; 
(iii) deformable cardiac electromechanics: Mechanosensitive transport of metabolic substrates and signalling molecules through deforming myocardium during contractile cycles; and 
(iv) cartilage and musculoskeletal biomechanics: Nutrient and ionic transport within cartilage and intervertebral discs driven by mechanical compression and interstitial fluid pressure.}
 For the numerical experiment, all physical parameters are taken equal to one, except for $\eta_1=10^{-3}$ and $\eta=0$. The manufactured solutions employed in both the 2D and 3D tests are the same as those introduced in Sections~\ref{sec:ex1} and \ref{sec:ex2}.

The corresponding error histories are reported in Tables~\ref{tab:convergence2DOnlyStress} and \ref{tab:convergence3DOnlyStress}. Even in the case $\eta=0$, the method recovers the optimal convergence rate $O(h^{k+1})$ for all mesh families considered. {Notably, the fixed-point algorithm requires fewer iterations to converge in this setting than in the altered diffusion model analysed in Sections~\ref{sec:ex1} and \ref{sec:ex2}. This efficiency is anticipated because the latter model involves two nonlinear contributions.}

Finally, the robustness of the method with respect to large variations in some of the physical parameters is also preserved in the pure \gls{sad} regime. This behaviour is illustrated in Table~\ref{tab:convergence2DRobustOnlyStress}.

\begin{table}[!t]
    \setlength{\tabcolsep}{2pt}
    \begin{center}
        \resizebox{\textwidth}{!}{ 
            \begin{tabular}{| c | c | c | c | c | c | c | c | c | c | c | c | c | c | c | c | c | c |}
                \hline
                {$k$} & 
                {$\cT_h$} & 
                {$h$} & 
                {$\bar{\mathrm{e}}_h$} & 
                {$r(\bar{\mathrm{e}}_h)$} & 
                {$\bar{\text{e}}_{\bsigma_h^{\bbPi}}$} & 
                {$r(\bar{\text{e}}_{\bsigma_h^{\bbPi}})$} & 
                {$\bar{\mathrm{e}}_{\bu_h}$} & 
                ${r(\bar{\mathrm{e}}_{\bu_h})}$ & 
                {$\bar{\mathrm{e}}_{\bz_h^{\bPi}}$} & 
                {$r(\bar{\mathrm{e}}_{\bz_h^{\bPi}})$} & 
                {$\bar{\mathrm{e}}_{p_h}$} & 
                {$r(\bar{\mathrm{e}}_{p_h})$} & 
                {$\bar{\mathrm{e}}_{\bzeta_h^{\bPi}}$} & 
                {$r(\bar{\mathrm{e}}_{\bzeta_h^{\bPi}})$} & 
                {$\bar{\mathrm{e}}_{\varphi_h}$} & 
                {$r(\bar{\mathrm{e}}_{\varphi_h})$} & 
                {it} \\ [3pt]
                \hline 
                \hline
                \multirow{16}{*}{$1$}
                &\multirow{4}{*}{\rotatebox{90}{Quadrilateral \rule{0pt}{1pt}}} 
                
                & 1.00e-01 & 3.31e+01 & *    & 3.29e+01 & *    & 5.28e-01 & *    & 1.66e+00 & *    & 1.92e-02 & *    & 2.90e+00 & *    & 1.95e-02 & *    & 3\\
                & & 5.00e-02 & 8.65e+00 & 1.93 & 8.61e+00 & 1.93 & 7.61e-02 & 2.79 & 4.15e-01 & 2.00 & 4.85e-03 & 1.99 & 7.46e-01 & 1.96 & 4.88e-03 & 2.00 & 3\\
                & & 2.50e-02 & 2.19e+00 & 1.99 & 2.17e+00 & 1.99 & 1.14e-02 & 2.74 & 1.00e-01 & 2.05 & 1.22e-03 & 2.00 & 1.97e-01 & 1.92 & 1.22e-03 & 2.00 & 3\\
                & & 1.25e-02 & 5.48e-01 & 2.00 & 5.45e-01 & 2.00 & 2.06e-03 & 2.46 & 2.45e-02 & 2.03 & 3.04e-04 & 2.00 & 5.59e-02 & 1.82 & 3.05e-04 & 2.00 & 4\\ [8pt]
                \cline{2-18}
                & \multirow{4}{*}{\rotatebox{90}{Distorted \rule{0pt}{1pt}}}  
                & 1.03e-01 & 4.07e+01 & *    & 4.04e+01 & *    & 9.22e-01 & *    & 1.97e+00 & *    & 2.42e-02 & *    & 3.74e+00 & *    & 2.54e-02 & *    & 3\\
                & & 5.07e-02 & 1.03e+01 & 1.93 & 1.03e+01 & 1.93 & 1.15e-01 & 2.93 & 4.64e-01 & 2.04 & 5.70e-03 & 2.04 & 8.95e-01 & 2.01 & 5.99e-03 & 2.04 & 3\\
                & & 2.66e-02 & 2.79e+00 & 2.02 & 2.78e+00 & 2.02 & 1.81e-02 & 2.87 & 1.23e-01 & 2.06 & 1.53e-03 & 2.03 & 2.38e-01 & 2.05 & 1.59e-03 & 2.06 & 3\\
                & & 1.32e-02 & 6.90e-01 & 2.00 & 6.87e-01 & 2.00 & 2.90e-03 & 2.63 & 3.00e-02 & 2.02 & 3.77e-04 & 2.01 & 5.88e-02 & 2.00 & 3.89e-04 & 2.01 & 3\\ [2pt]
                \cline{2-18}
                & \multirow{4}{*}{\rule{0pt}{10ex}\rotatebox{90}{Hexahedral \rule{0pt}{1pt}}} 
                & 1.03e-01 & 3.18e+01 & *    & 3.16e+01 & *    & 7.09e-01 & *    & 1.49e+00 & *    & 1.81e-02 & *    & 2.83e+00 & *    & 1.90e-02 & *    & 3\\
                & & 5.07e-02 & 7.79e+00 & 1.98 & 7.75e+00 & 1.98 & 8.71e-02 & 2.95 & 3.48e-01 & 2.04 & 4.32e-03 & 2.02 & 6.63e-01 & 2.04 & 4.44e-03 & 2.05 & 3\\
                & & 2.66e-02 & 2.07e+00 & 2.05 & 2.06e+00 & 2.05 & 1.35e-02 & 2.88 & 9.16e-02 & 2.07 & 1.15e-03 & 2.05 & 1.74e-01 & 2.07 & 1.16e-03 & 2.07 & 3\\
                & & 1.32e-02 & 5.11e-01 & 2.01 & 5.09e-01 & 2.01 & 2.16e-03 & 2.63 & 2.27e-02 & 2.00 & 2.84e-04 & 2.00 & 4.28e-02 & 2.01 & 2.86e-04 & 2.01 & 3\\
                \cline{2-18}
                & \multirow{4}{*}{\rotatebox{90}{Triangular \rule{0pt}{1pt}}} 
                & 4.36e-02 & 8.51e+00 & *    & 8.47e+00 & *    & 6.32e-02 & *    & 3.96e-01 & *    & 4.44e-03 & *    & 7.12e-01 & *    & 4.74e-03 & *    & 3\\
                & & 2.55e-02 & 2.92e+00 & 1.99 & 2.91e+00 & 1.99 & 1.48e-02 & 2.70 & 1.42e-01 & 1.90 & 1.58e-03 & 1.92 & 2.44e-01 & 1.99 & 1.61e-03 & 2.00 & 3\\
                & & 1.79e-02 & 1.46e+00 & 1.96 & 1.45e+00 & 1.96 & 6.15e-03 & 2.47 & 7.18e-02 & 1.93 & 7.97e-04 & 1.93 & 1.22e-01 & 1.94 & 8.08e-04 & 1.95 & 3\\
                & & 1.39e-02 & 8.82e-01 & 1.98 & 8.78e-01 & 1.98 & 3.38e-03 & 2.37 & 4.36e-02 & 1.97 & 4.82e-04 & 1.99 & 7.33e-02 & 2.02 & 4.85e-04 & 2.02 & 3\\
                \hline
                \hline
                \multirow{16}{*}{$2$}
                & \multirow{4}{*}{\rotatebox{90}{Quadrilateral \rule{0pt}{1pt}}} 
                & 1.00e-01 & 6.68e+00 & *    & 6.66e+00 & *    & 2.23e-01 & *    & 1.98e-01 & *    & 1.94e-03 & *    & 3.32e-01 & *    & 2.16e-03 & *    & 3\\
                & & 5.00e-02 & 8.71e-01 & 2.94 & 8.70e-01 & 2.94 & 1.47e-02 & 3.92 & 2.40e-02 & 3.05 & 2.45e-04 & 2.99 & 4.20e-02 & 2.98 & 2.72e-04 & 2.99 & 3\\
                & & 2.50e-02 & 1.10e-01 & 2.99 & 1.10e-01 & 2.99 & 9.68e-04 & 3.92 & 2.70e-03 & 3.15 & 3.07e-05 & 3.00 & 5.25e-03 & 3.00 & 3.41e-05 & 3.00 & 3\\
                & & 1.25e-02 & 1.38e-02 & 3.00 & 1.38e-02 & 3.00 & 7.12e-05 & 3.77 & 3.18e-04 & 3.09 & 3.84e-06 & 3.00 & 6.56e-04 & 3.00 & 4.27e-06 & 3.00 & 4\\ [8pt]
                \cline{2-18}
                & \multirow{4}{*}{\rotatebox{90}{Distorted \rule{0pt}{1pt}}} 
                & 1.03e-01 & 9.66e+00 & *    & 9.63e+00 & *    & 4.39e-01 & *    & 2.70e-01 & *    & 2.87e-03 & *    & 4.73e-01 & *    & 3.16e-03 & *    & 3\\
                & & 5.07e-02 & 1.20e+00 & 2.94 & 1.20e+00 & 2.94 & 2.64e-02 & 3.96 & 3.34e-02 & 2.94 & 3.35e-04 & 3.03 & 5.62e-02 & 3.00 & 3.70e-04 & 3.02 & 3\\
                & & 2.66e-02 & 1.67e-01 & 3.05 & 1.67e-01 & 3.05 & 1.98e-03 & 4.01 & 4.71e-03 & 3.03 & 4.62e-05 & 3.07 & 7.78e-03 & 3.06 & 5.11e-05 & 3.07 & 3\\
                & & 1.32e-02 & 2.05e-02 & 3.01 & 2.04e-02 & 3.01 & 1.31e-04 & 3.89 & 5.88e-04 & 2.98 & 5.64e-06 & 3.01 & 9.53e-04 & 3.01 & 6.25e-06 & 3.01 & 3\\ [2pt]
                \cline{2-18}
                & \multirow{4}{*}{\rule{0pt}{10ex}\rotatebox{90}{Hexahedral \rule{0pt}{1pt}}} 
                & 1.03e-01 & 5.91e+00 & *    & 5.89e+00 & *    & 2.68e-01 & *    & 1.58e-01 & *    & 1.71e-03 & *    & 2.80e-01 & *    & 1.84e-03 & *    & 3\\
                & & 5.07e-02 & 7.02e-01 & 3.00 & 7.01e-01 & 3.00 & 1.52e-02 & 4.04 & 1.88e-02 & 3.00 & 1.98e-04 & 3.04 & 3.20e-02 & 3.05 & 2.09e-04 & 3.06 & 3\\
                & & 2.66e-02 & 9.53e-02 & 3.09 & 9.52e-02 & 3.09 & 1.11e-03 & 4.06 & 2.54e-03 & 3.10 & 2.66e-05 & 3.11 & 4.36e-03 & 3.09 & 2.84e-05 & 3.09 & 3\\
                & & 1.32e-02 & 1.17e-02 & 3.01 & 1.17e-02 & 3.01 & 7.39e-05 & 3.88 & 3.15e-04 & 3.00 & 3.25e-06 & 3.01 & 5.28e-04 & 3.02 & 3.45e-06 & 3.02 & 4\\
                \cline{2-18}
                & \multirow{4}{*}{\rotatebox{90}{Triangular \rule{0pt}{1pt}}} 
                & 4.36e-02 & 8.58e-01 & *    & 8.56e-01 & *    & 1.29e-02 & *    & 3.27e-02 & *    & 2.54e-04 & *    & 5.41e-02 & *    & 3.42e-04 & *    & 3\\
                & & 2.55e-02 & 1.71e-01 & 3.00 & 1.71e-01 & 3.00 & 1.57e-03 & 3.91 & 6.21e-03 & 3.09 & 4.86e-05 & 3.07 & 1.03e-02 & 3.08 & 6.50e-05 & 3.09 & 3\\
                & & 1.79e-02 & 6.39e-02 & 2.78 & 6.37e-02 & 2.78 & 4.19e-04 & 3.71 & 2.27e-03 & 2.83 & 1.77e-05 & 2.85 & 3.79e-03 & 2.82 & 2.39e-05 & 2.82 & 3\\
                & & 1.39e-02 & 2.93e-02 & 3.08 & 2.93e-02 & 3.08 & 1.62e-04 & 3.75 & 1.05e-03 & 3.06 & 8.27e-06 & 3.01 & 1.72e-03 & 3.12 & 1.08e-05 & 3.13 & 4\\
                \hline
            \end{tabular}
        }
    \end{center}
    \vspace{-0.5cm}
    \caption{Example 3. Convergence history and fixed-point iteration count for a variety of 2D meshes with polynomial degrees $k=1,2$. The parameters are set to unity, except for $\eta_1 = 10^{-3}$ and $\eta=0$.}
    \label{tab:convergence2DOnlyStress}
\end{table}

\begin{table}[!h]
    \setlength{\tabcolsep}{2pt}
    \begin{center}
        \resizebox{\textwidth}{!}{ 
            \begin{tabular}{| c | c | c | c | c | c | c | c | c | c | c | c | c | c | c | c | c |}
                \hline
                {$\cT_h$} & 
                {$h$} & 
                {$\bar{\mathrm{e}}_h$} & 
                {$r(\bar{\mathrm{e}}_h)$} & 
                {$\bar{\text{e}}_{\bsigma_h^{\bbPi}}$} & 
                {$r(\bar{\text{e}}_{\bsigma_h^{\bbPi}})$} & 
                {$\bar{\mathrm{e}}_{\bu_h}$} & 
                {$r(\bar{\mathrm{e}}_{\bu_h})$} & 
                {$\bar{\mathrm{e}}_{\bz_h^{\bPi}}$} & 
                {$r(\bar{\mathrm{e}}_{\bz_h^{\bPi}})$} & 
                {$\bar{\mathrm{e}}_{p_h}$} & 
                {$r(\bar{\mathrm{e}}_{p_h})$} & 
                {$\bar{\mathrm{e}}_{\bzeta_h^{\bPi}}$} & 
                {$r(\bar{\mathrm{e}}_{\bzeta_h^{\bPi}})$} & 
                {$\bar{\mathrm{e}}_{\varphi_h}$} & 
                {$r(\bar{\mathrm{e}}_{\varphi_h})$} & 
                {it} \\ [2pt]
                \hline 
                \hline
                \multirow{4}{*}{\rotatebox{90}{Cubical}} 
                & 2.17e-01 & 6.07e+01 & *    & 6.04e+01 & *    & 3.14e-01 & *    & 2.36e+00 & *    & 2.97e-02 & *    & 4.69e+00 & *    & 2.98e-02 & *    & 2\\
                & 1.44e-01 & 2.84e+01 & 1.87 & 2.83e+01 & 1.87 & 1.05e-01 & 2.71 & 1.06e+00 & 1.97 & 1.34e-02 & 1.97 & 2.11e+00 & 1.97 & 1.34e-02 & 1.97 & 2\\
                & 1.08e-01 & 1.63e+01 & 1.94 & 1.62e+01 & 1.93 & 5.44e-02 & 2.29 & 6.01e-01 & 1.98 & 7.56e-03 & 1.98 & 1.19e+00 & 1.99 & 7.57e-03 & 1.99 & 2\\
                & 8.66e-02 & 1.05e+01 & 1.96 & 1.05e+01 & 1.96 & 3.40e-02 & 2.11 & 3.85e-01 & 1.99 & 4.85e-03 & 1.99 & 7.65e-01 & 1.99 & 4.85e-03 & 1.99 & 2\\
                \hline
                \hline
                \multirow{4}{*}{\rotatebox{90}{Octahedral \rule{0pt}{1ex}}} 
                & 2.08e-01 & 5.81e+01 & *    & 5.79e+01 & *    & 2.88e-01 & *    & 2.22e+00 & *    & 2.78e-02 & *    & 4.39e+00 & *    & 2.79e-02 & *    & 2\\
                & 1.04e-01 & 1.53e+01 & 1.93 & 1.52e+01 & 1.93 & 5.07e-02 & 2.50 & 5.67e-01 & 1.97 & 7.08e-03 & 1.98 & 1.12e+00 & 1.98 & 7.08e-03 & 1.98 & 2\\
                & 8.33e-02 & 9.85e+00 & 1.96 & 9.82e+00 & 1.96 & 3.18e-02 & 2.09 & 3.66e-01 & 1.96 & 4.54e-03 & 1.99 & 7.16e-01 & 1.99 & 4.54e-03 & 1.99 & 3\\
                & 6.94e-02 & 6.87e+00 & 1.97 & 6.85e+00 & 1.97 & 2.19e-02 & 2.04 & 2.56e-01 & 1.96 & 3.16e-03 & 1.99 & 4.98e-01 & 1.99 & 3.16e-03 & 1.99 & 2\\
                \hline
                \hline
                \multirow{4}{*}{\rotatebox{90}{Voronoi}} 
                & 3.59e-01 & 1.43e+02 & *    & 1.42e+02 & *    & 1.50e+00 & *    & 6.94e+00 & *    & 8.81e-02 & *    & 1.33e+01 & *    & 8.53e-02 & *    & 2\\
                & 1.80e-01 & 4.39e+01 & 1.70 & 4.38e+01 & 1.70 & 1.91e-01 & 2.98 & 1.63e+00 & 2.09 & 2.04e-02 & 2.11 & 3.21e+00 & 2.05 & 2.04e-02 & 2.07 & 2\\
                & 8.98e-02 & 1.12e+01 & 1.97 & 1.11e+01 & 1.97 & 3.71e-02 & 2.36 & 4.33e-01 & 1.91 & 5.16e-03 & 1.98 & 8.13e-01 & 1.98 & 5.16e-03 & 1.98 & 2\\
                & 7.19e-02 & 7.21e+00 & 1.97 & 7.18e+00 & 1.97 & 2.35e-02 & 2.05 & 2.85e-01 & 1.87 & 3.31e-03 & 1.99 & 5.21e-01 & 2.00 & 3.31e-03 & 1.99 & 4\\
                \hline
                \hline
                \multirow{4}{*}{\rule{0pt}{3ex} \rotatebox{90}{Nonahedral \rule{0pt}{1ex}}} 
                & 3.46e-01 & 1.27e+02 & *    & 1.27e+02 & *    & 1.41e+00 & *    & 5.54e+00 & *    & 7.01e-02 & *    & 1.08e+01 & *    & 6.99e-02 & *    & 2\\
                & 1.73e-01 & 3.64e+01 & 1.81 & 3.63e+01 & 1.80 & 1.59e-01 & 3.15 & 1.42e+00 & 1.97 & 1.77e-02 & 1.99 & 2.66e+00 & 2.03 & 1.69e-02 & 2.05 & 2\\
                & 1.37e-01 & 2.32e+01 & 1.95 & 2.32e+01 & 1.95 & 8.70e-02 & 2.60 & 8.90e-01 & 2.01 & 1.10e-02 & 2.04 & 1.68e+00 & 1.98 & 1.07e-02 & 1.98 & 2\\
                & 1.09e-01 & 1.48e+01 & 1.97 & 1.47e+01 & 1.97 & 5.12e-02 & 2.29 & 5.67e-01 & 1.95 & 6.90e-03 & 2.03 & 1.07e+00 & 1.97 & 6.78e-03 & 1.97 & 2\\
                \hline
            \end{tabular}
        }
    \end{center}
    \vspace{-0.5cm}
    \caption{Example 3. Convergence history and fixed-point iteration count for a variety of 3D meshes for the lowest-case order $k=1$. We considered unit parameters except for $\eta_1=10^{-5}$ and $\eta=0$.}
    \label{tab:convergence3DOnlyStress}
\end{table}

\begin{table}[!t]
    \setlength{\tabcolsep}{2pt}
    \begin{center}
        \resizebox{\textwidth}{!}{ 
            \begin{tabular}{| c | c | c | c | c | c | c | c | c | c | c | c | c | c | c | c | c |}
                \hline
                {} &
                {$h$} & 
                {$\bar{\mathrm{e}}_h$} & 
                {$r(\bar{\mathrm{e}}_h)$} & 
                {$\bar{\text{e}}_{\bsigma_h^{\bbPi}}$} & 
                {$r(\bar{\text{e}}_{\bsigma_h^{\bbPi}})$} & 
                {$\bar{\mathrm{e}}_{\bu_h}$} & 
                {$r(\bar{\mathrm{e}}_{\bu_h})$} & 
                {$\bar{\mathrm{e}}_{\bz_h^{\bPi}}$} & 
                {$r(\bar{\mathrm{e}}_{\bz_h^{\bPi}})$} & 
                {$\bar{\mathrm{e}}_{p_h}$} & 
                {$r(\bar{\mathrm{e}}_{p_h})$} & 
                {$\bar{\mathrm{e}}_{\bzeta_h^{\bPi}}$} & 
                {$r(\bar{\mathrm{e}}_{\bzeta_h^{\bPi}})$} & 
                {$\bar{\mathrm{e}}_{\varphi_h}$} & 
                {$r(\bar{\mathrm{e}}_{\varphi_h})$} & 
                {it} \\ [2pt]
                \hline 
                \hline
                \multirow{4}{*}{\rotatebox{90}{$\lambda = 10^6$}} 
                & 1.03e-01 & 7.52e+02 & *    & 7.22e+02 & *    & 2.11e+02 & *    & 1.43e+00 & *    & 1.81e-02 & *    & 1.50e+00 & *    & 1.90e-02 & *    & 1\\
                & 5.07e-02 & 1.61e+02 & 2.17 & 1.59e+02 & 2.13 & 2.30e+01 & 3.12 & 3.43e-01 & 2.01 & 4.32e-03 & 2.02 & 3.53e-01 & 2.04 & 4.44e-03 & 2.05 & 1\\
                & 2.66e-02 & 4.29e+01 & 2.05 & 4.27e+01 & 2.04 & 3.19e+00 & 3.06 & 9.12e-02 & 2.05 & 1.15e-03 & 2.05 & 9.26e-02 & 2.07 & 1.16e-03 & 2.07 & 1\\
                & 1.32e-02 & 1.04e+01 & 2.03 & 1.04e+01 & 2.03 & 3.85e-01 & 3.03 & 2.26e-02 & 2.00 & 2.84e-04 & 2.00 & 2.28e-02 & 2.01 & 2.86e-04 & 2.01 & 1\\
                \cline{1-17}
                \multirow{4}{*}{\rotatebox{90}{$s_0=10^{-8}$}} 
                & 1.03e-01 & 3.18e+01 & *    & 3.16e+01 & *    & 7.09e-01 & *    & 1.49e+00 & *    & 1.81e-02 & *    & 2.83e+00 & *    & 1.90e-02 & *    & 3\\
                & 5.07e-02 & 7.79e+00 & 1.98 & 7.75e+00 & 1.98 & 8.71e-02 & 2.95 & 3.48e-01 & 2.04 & 4.32e-03 & 2.02 & 6.63e-01 & 2.04 & 4.44e-03 & 2.05 & 3\\
                & 2.66e-02 & 2.07e+00 & 2.05 & 2.06e+00 & 2.05 & 1.35e-02 & 2.88 & 9.16e-02 & 2.07 & 1.15e-03 & 2.05 & 1.74e-01 & 2.07 & 1.16e-03 & 2.07 & 3\\
                & 1.32e-02 & 5.11e-01 & 2.01 & 5.09e-01 & 2.01 & 2.16e-03 & 2.63 & 2.27e-02 & 2.00 & 2.84e-04 & 2.00 & 4.28e-02 & 2.01 & 2.86e-04 & 2.01 & 3\\
                \cline{1-17}
                \multirow{4}{*}{\rotatebox{90}{$\alpha=10^{-6}$}} 
                & 1.03e-01 & 3.18e+01 & *    & 3.16e+01 & *    & 7.08e-01 & *    & 1.43e+00 & *    & 1.81e-02 & *    & 2.91e+00 & *    & 1.90e-02 & *    & 3\\
                & 5.07e-02 & 7.79e+00 & 1.98 & 7.75e+00 & 1.98 & 8.70e-02 & 2.95 & 3.43e-01 & 2.01 & 4.32e-03 & 2.02 & 6.83e-01 & 2.04 & 4.44e-03 & 2.05 & 3\\
                & 2.66e-02 & 2.07e+00 & 2.05 & 2.06e+00 & 2.05 & 1.35e-02 & 2.88 & 9.12e-02 & 2.05 & 1.15e-03 & 2.05 & 1.79e-01 & 2.07 & 1.16e-03 & 2.07 & 3\\
                & 1.32e-02 & 5.11e-01 & 2.01 & 5.09e-01 & 2.01 & 2.16e-03 & 2.63 & 2.26e-02 & 2.00 & 2.84e-04 & 2.00 & 4.41e-02 & 2.01 & 2.86e-04 & 2.01 & 3\\
                \hline
            \end{tabular}
        }
    \end{center}
    \vspace{-0.5cm}
    \caption{Example 3. Convergence history and fixed-point iteration counts are shown for the Hexahedral mesh with polynomial degree $k=1$ and extreme values for the parameters $\lambda$, $s_0$, and $\alpha$. In each test, the remaining parameters are set to unity, except $\eta_1=10^{-3}$ and $\eta=0$.}
    \label{tab:convergence2DRobustOnlyStress}
\end{table}

\subsection{Example 4. Sleep-driven molecular clearance within brain tissue}\label{sec:brainExample1}
Neurodegenerative diseases such as Alzheimer’s and dementia are linked to the accumulation of proteins (functional molecules) and metabolites (intermediate or residual products of metabolism) within brain tissue. To mitigate this, the brain enhances its clearance mechanisms during sleep. Studies indicate that sleep deprivation impairs molecular clearance, and this effect cannot be compensated for by an extra night’s sleep \cite{evpmr2021}. Moreover, it has been shown that the cortical interstitial space in mice (the narrow, irregularly shaped region between neurons and blood vessels in the cerebral cortex) increases by more than $60 \%$ during sleep, resulting in more efficient clearance \cite{Xie2013-me}.

\begin{figure}[!t]
    \centering
    \includegraphics[width=.9\textwidth,trim={0.2cm 0.9cm 0.2cm 0.2cm},clip]{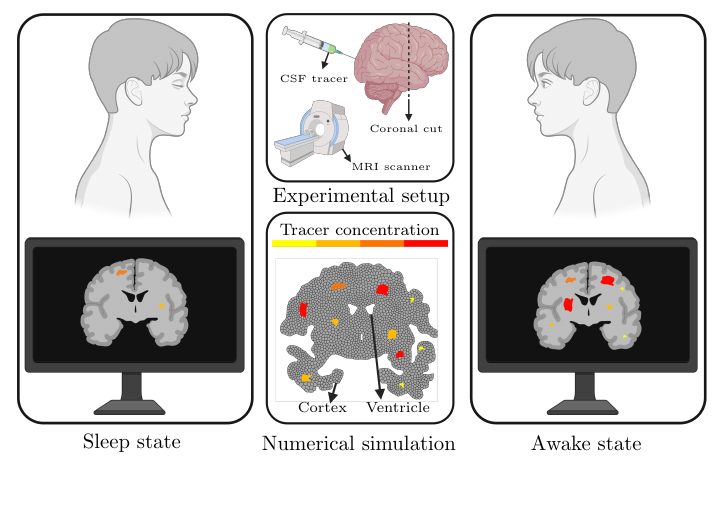}
    \caption{Example 4. Two-dimensional schematic illustration of molecular clearance in brain tissue of a fluorescent CSF tracer. The experimental setup is shown at the top middle panel. The MRI scans for the sleep and awake states are shown on the left and right panels. The bottom middle panel show the expected CSF tracer concentration computational simulations in a polytopal mesh of a coronal slice of the brain with 1,999 Voronoi cells.}\label{fig:brainExampleExperiment}
\end{figure}

Experimentally, MRI scans can visualise the distribution of CSF tracers such as Gadobutrol as transported through both cerebrospinal fluid (CSF) outside the brain and in interstial fluid (ISF) within brain tissue under various conditions, including sleep and awake states \cite{evpmr2021} (see Figure~\ref{fig:brainExampleExperiment}). In this example, we focus on the mathematical modelling of this process by tracking the concentration of CSF tracer under sleep and awake states in coronal slices of the brain. The mesh originally introduced in \cite{bhkmr_jcp22} provides the geometry of the coronal slice boundary. We extend this geometry by including the left, right, and bottom ventricles, and employ the capabilities of \texttt{PolyMesher} \cite{Talischi2012} to discretise the domain with 19,999 Voronoi cells.

Following \cite{gomez23}, we neglect convection and assume that stress-dependent diffusion is the dominant transport mechanism. This assumption is supported by experimental data indicating that transport within brain tissue occurs $3.5 \pm 1.5$ faster than predicted by Fickian diffusion within the extra-cellular matrix\cite{vinje2023human}. The expansion of the cortical interstitial space during sleep leads to an increase in the volume fraction of brain tissue \cite{Xie2013-me}. This can be measured as a porosity of $\phi = 0.14$ in the awake state and $\phi = 0.23$ in the sleep state.

\begin{figure}[!t]
    \centering
    \subfigure[$\varphi_{h,\text{sleep}}$.]       {\includegraphics[width=0.245\textwidth,trim={0.5cm 2.25cm 4.75cm 1.25cm},clip]{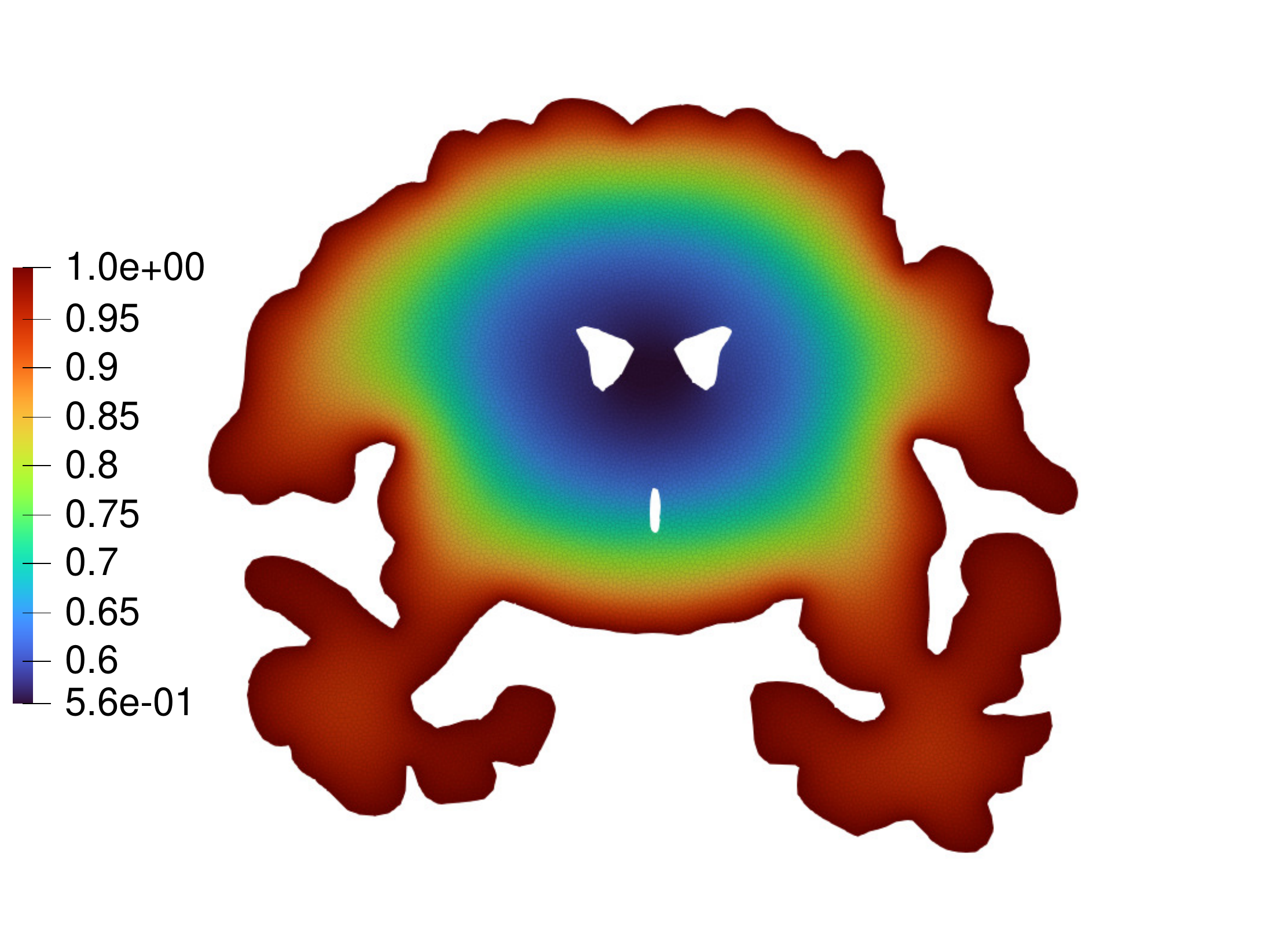}}
    \subfigure[$\varphi_{h,\text{awake}}$.]       {\includegraphics[width=0.245\textwidth,trim={0.5cm 2.25cm 4.75cm 1.25cm},clip]{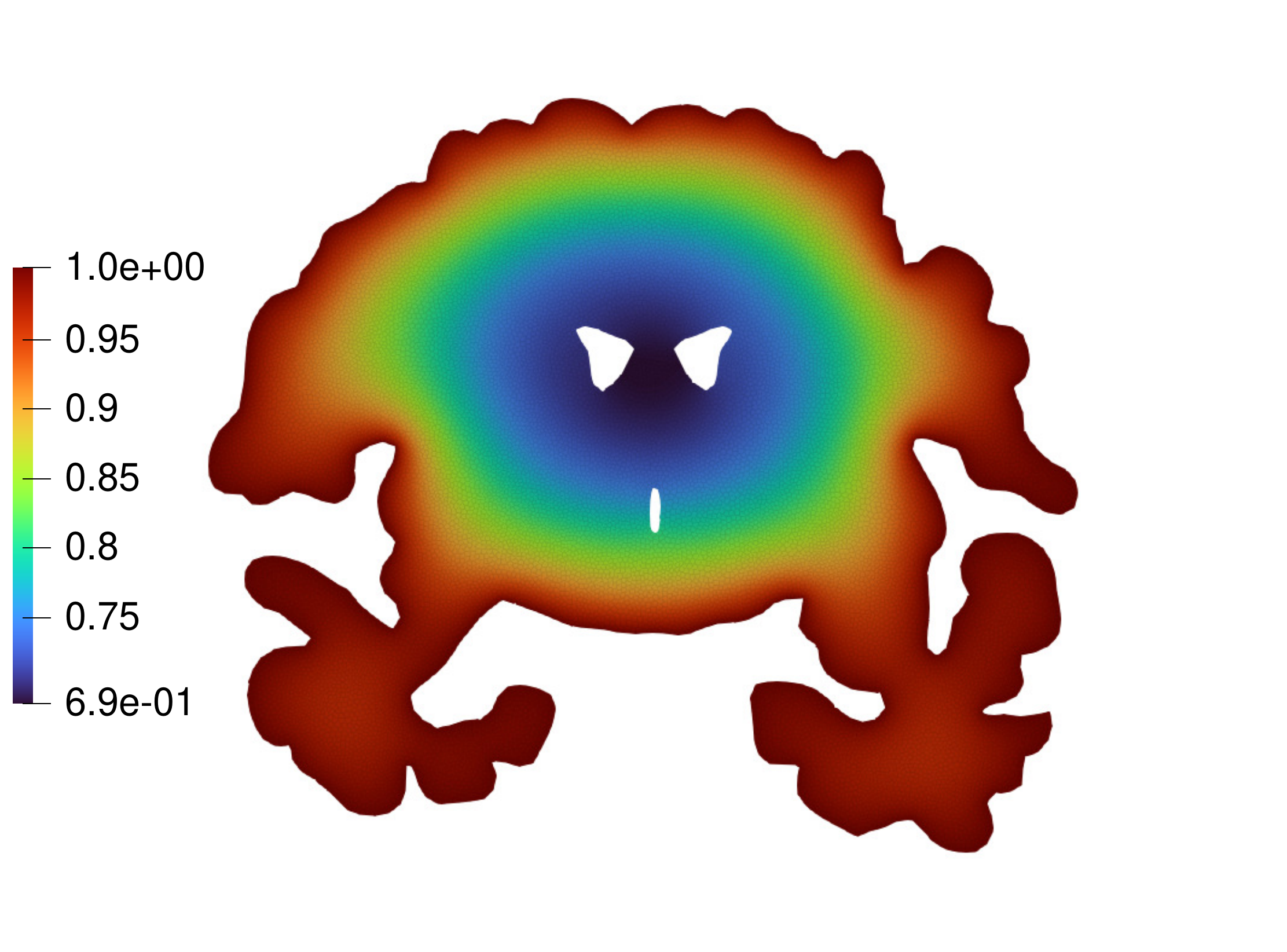}}
    \subfigure[$|\varphi_{h,\text{sleep}}-\varphi_{h,\text{awake}}|$.]       {\includegraphics[width=0.245\textwidth,trim={0.5cm 2.25cm 4.75cm 1.25cm},clip]{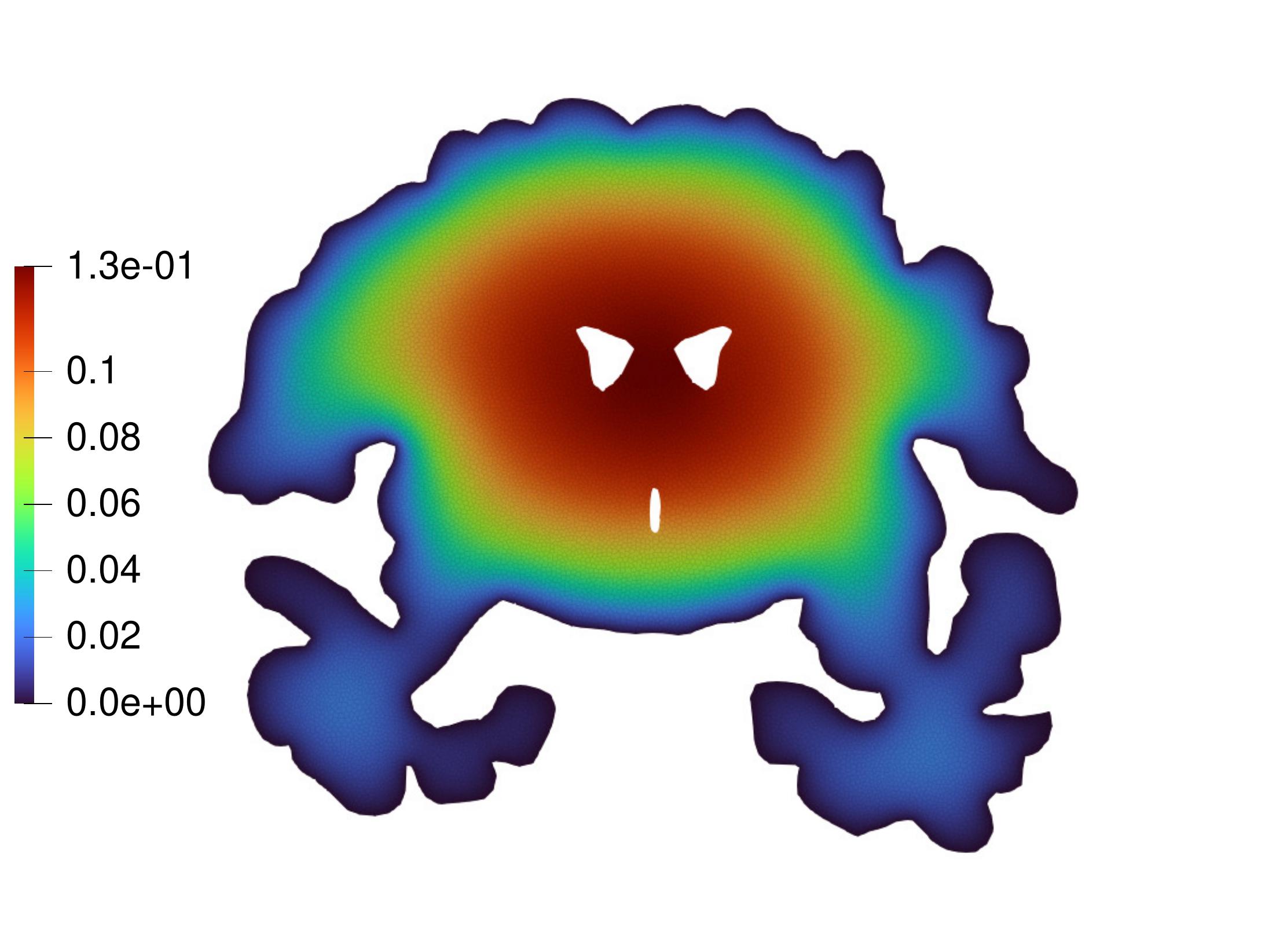}}
    \subfigure[$|\bu_{h,\text{sleep}}|$.]         {\includegraphics[width=0.245\textwidth,trim={0.5cm 2.25cm 4.75cm 1.25cm},clip]{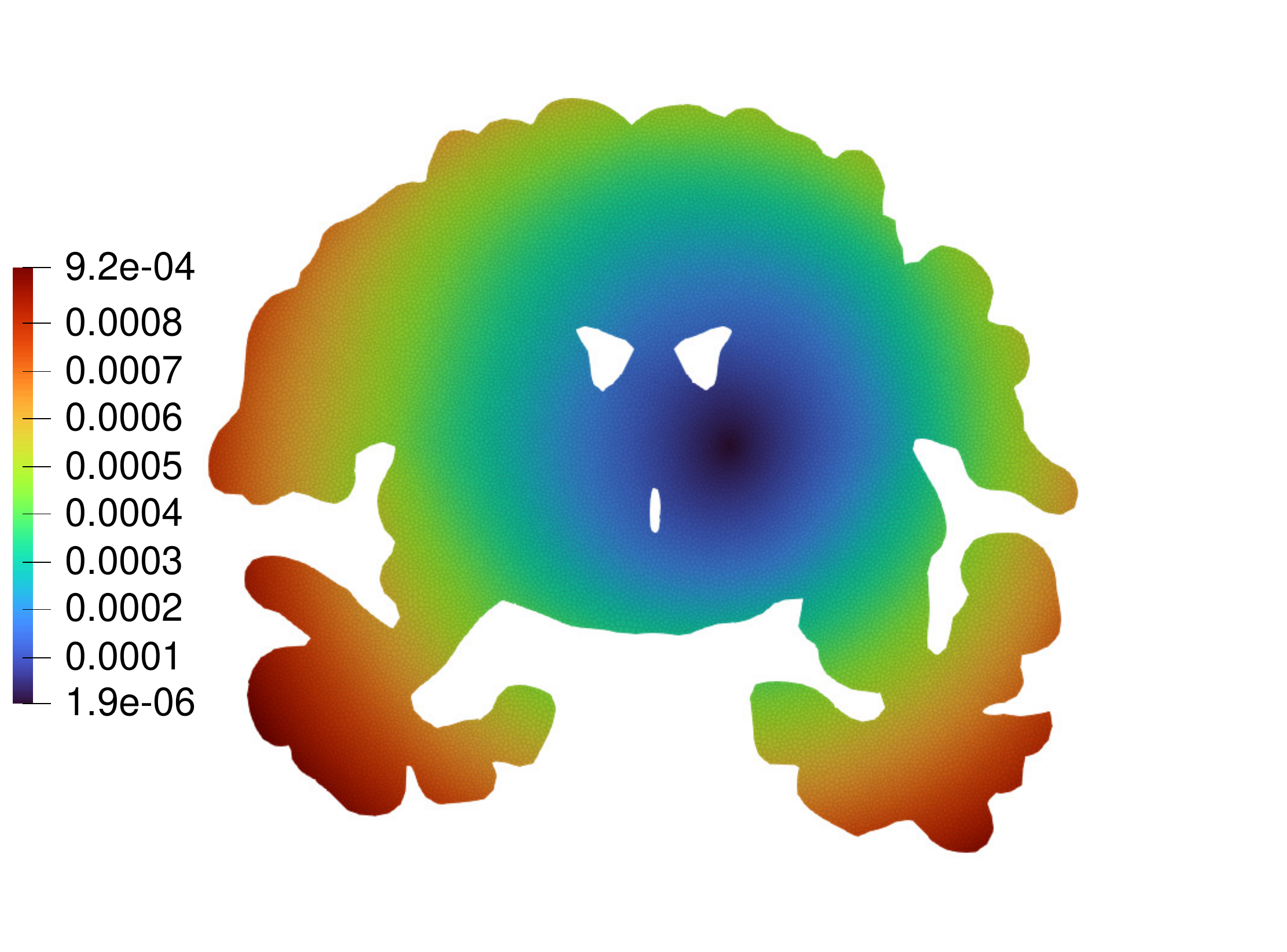}}
    \subfigure[$|\bu_{h,\text{awake}}|$.]         {\includegraphics[width=0.245\textwidth,trim={0.5cm 2.25cm 4.75cm 1.25cm},clip]{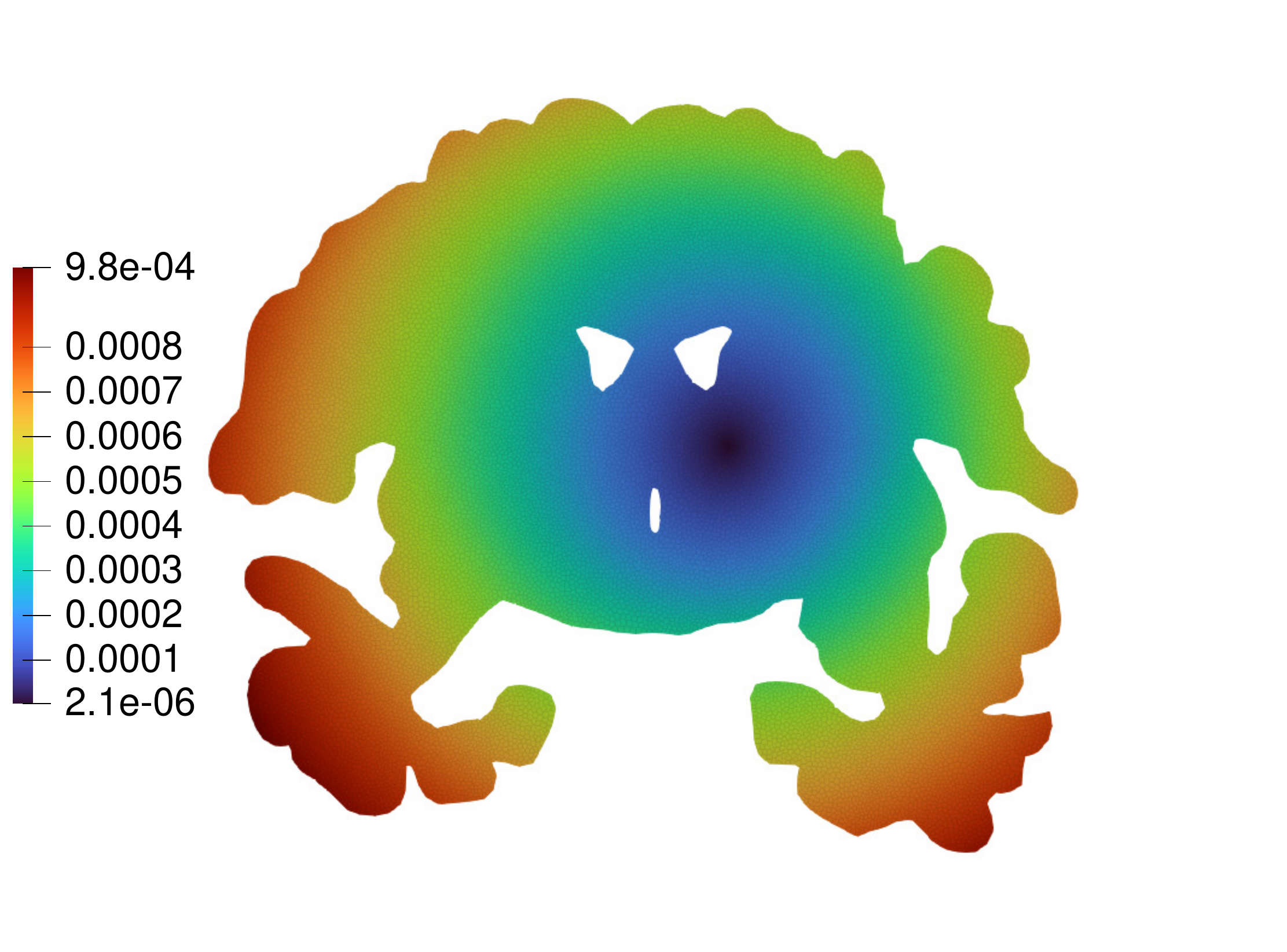}}
    \subfigure[$|\bu_{h,\text{sleep}}-\bu_{h,\text{awake}}|$.]         {\includegraphics[width=0.245\textwidth,trim={0.5cm 2.25cm 4.75cm 1.25cm},clip]{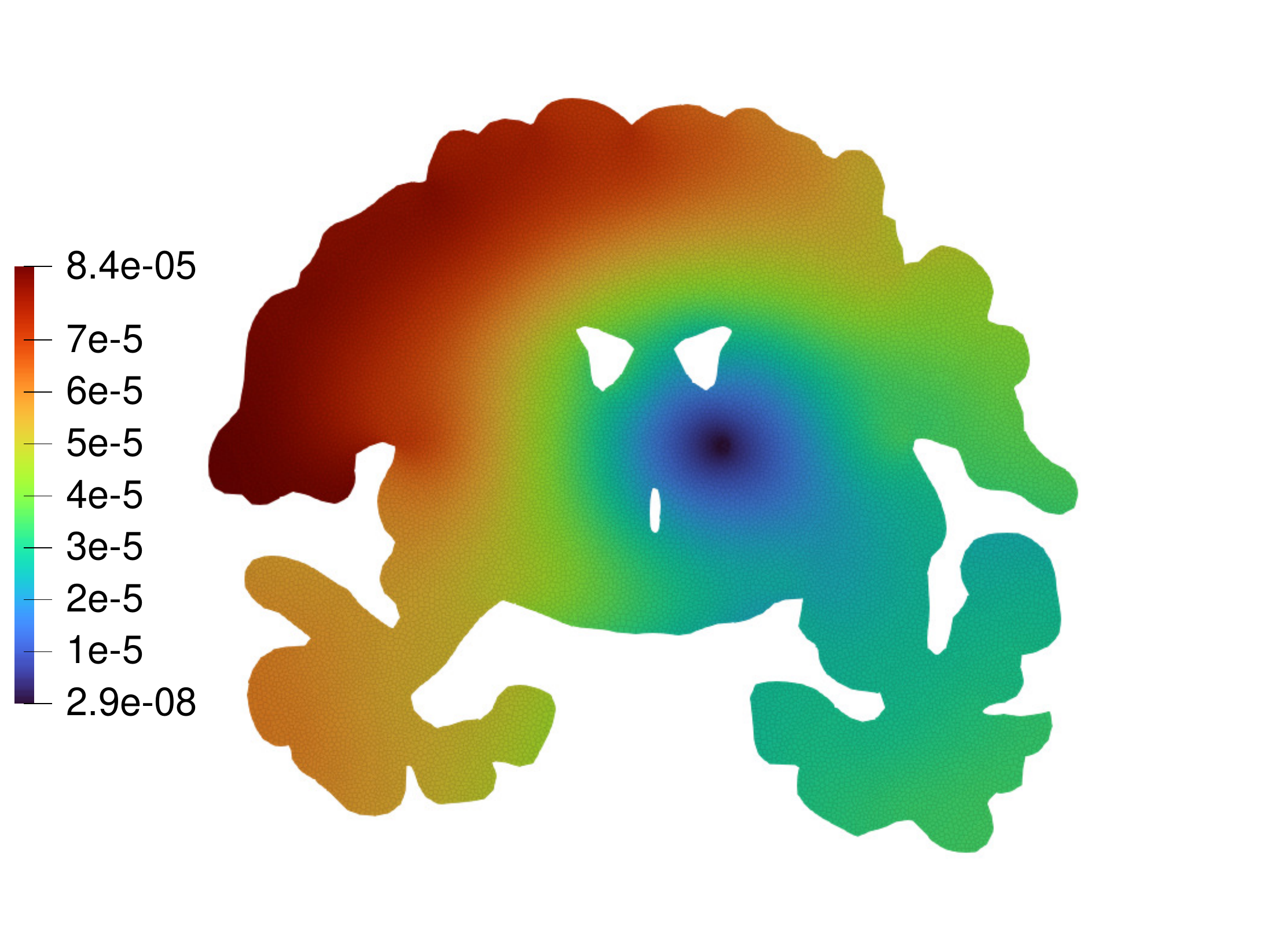}}
    \subfigure[$|\bsigma_{h,\text{sleep}}^{\bbPi}|$.] {\includegraphics[width=0.245\textwidth,trim={0.5cm 2.25cm 4.75cm 1.25cm},clip]{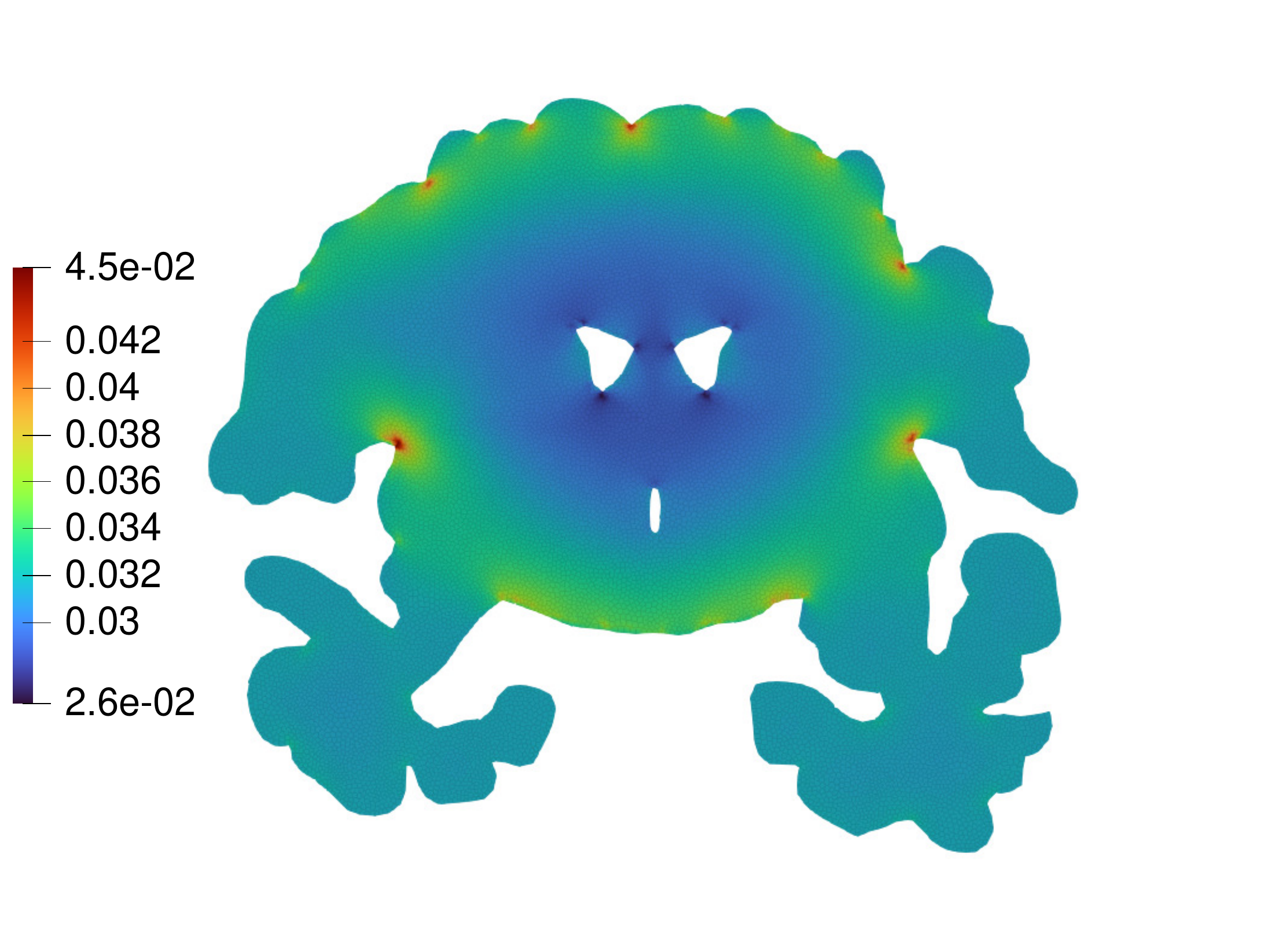}}
    \subfigure[$|\bsigma_{h,\text{awake}}^{\bbPi}|$.] {\includegraphics[width=0.245\textwidth,trim={0.5cm 2.25cm 4.75cm 1.25cm},clip]{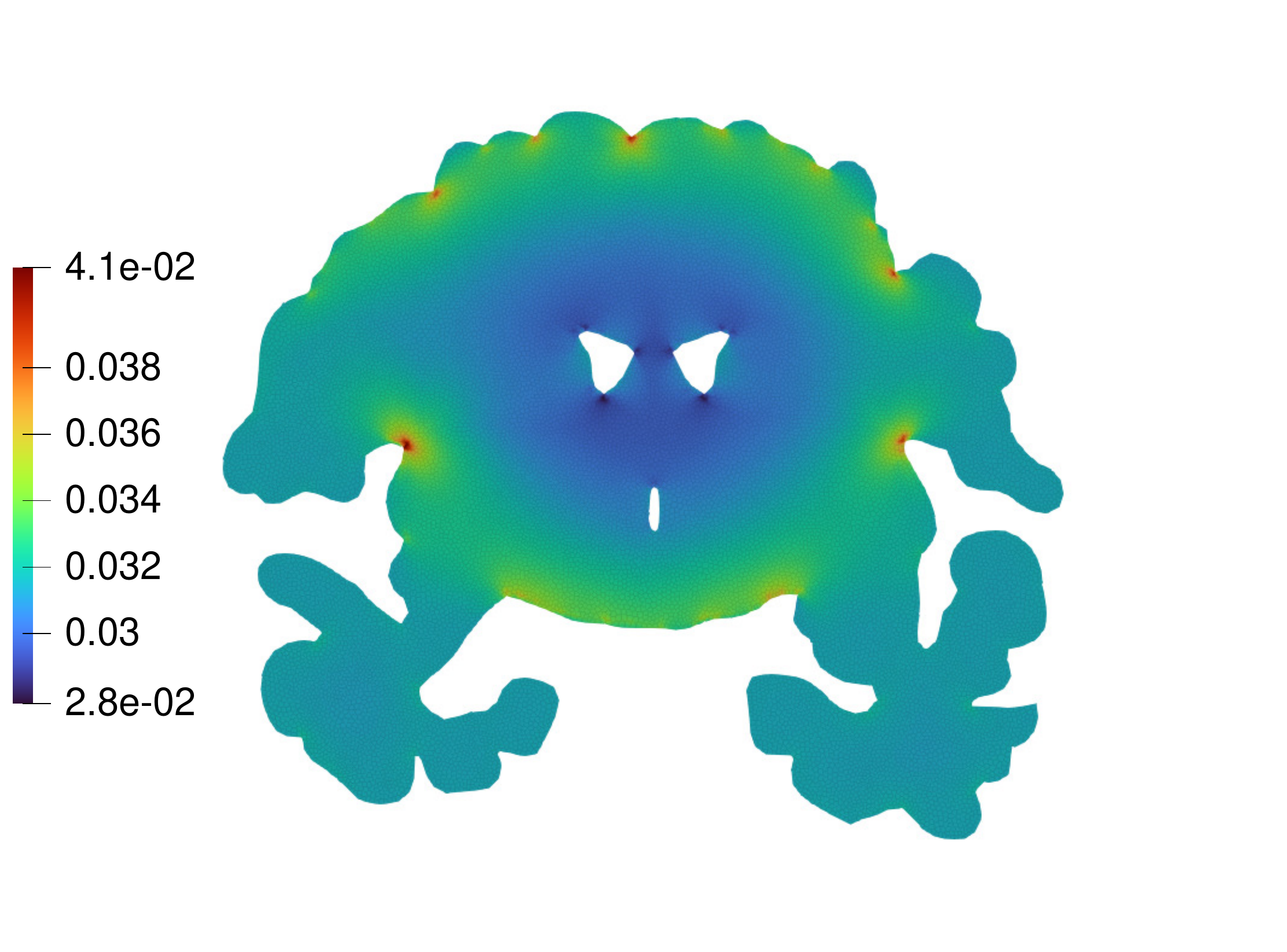}}
    \subfigure[$|\bsigma_{h,\text{sleep}}^{\bbPi}-\bsigma_{h,\text{awake}}^{\bbPi}|$.] {\includegraphics[width=0.245\textwidth,trim={0.5cm 2.25cm 4.75cm 1.25cm},clip]{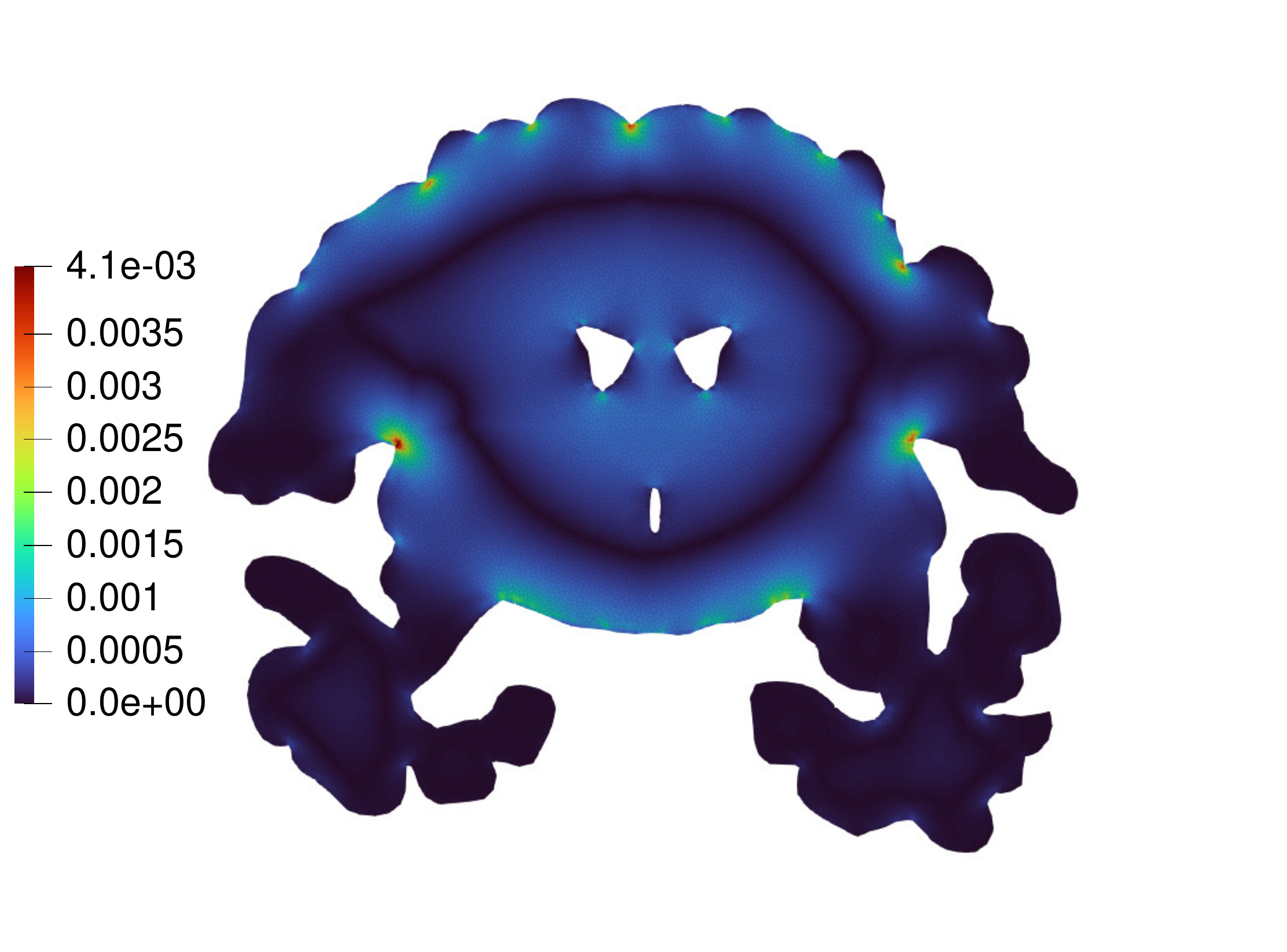}}
    \subfigure[$|\bz_{h,\text{sleep}}^{\bPi}|$.] {\includegraphics[width=0.245\textwidth,trim={0.5cm 2.25cm 4.75cm 1.25cm},clip]{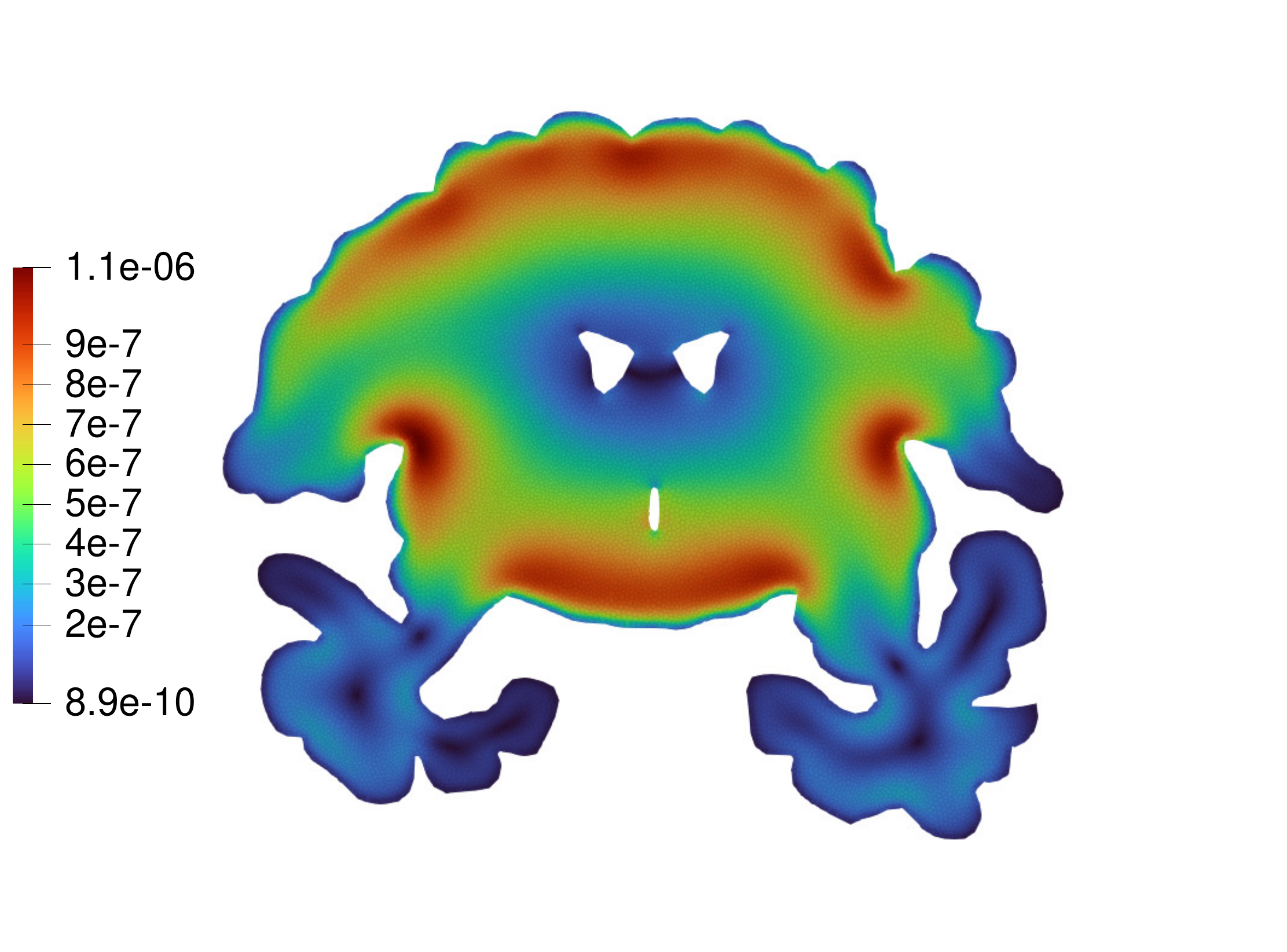}}
    \subfigure[$|\bz_{h,\text{awake}}^{\bPi}|$.] {\includegraphics[width=0.245\textwidth,trim={0.5cm 2.25cm 4.75cm 1.25cm},clip]{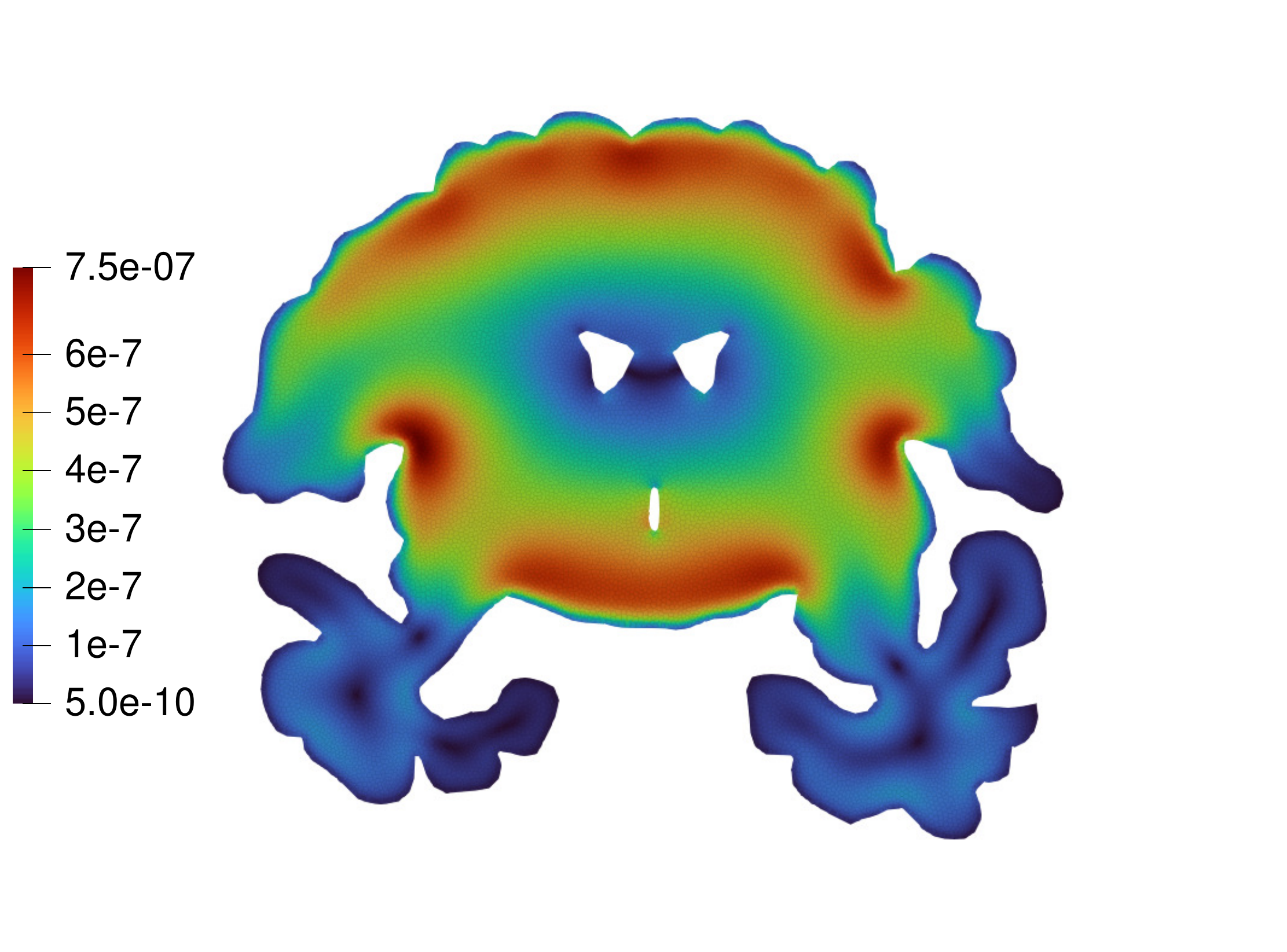}}
    \subfigure[$|\bz_{h,\text{sleep}}^{\bPi}-\bz_{h,\text{awake}}^{\bPi}|$.] {\includegraphics[width=0.245\textwidth,trim={0.5cm 2.25cm 4.75cm 1.25cm},clip]{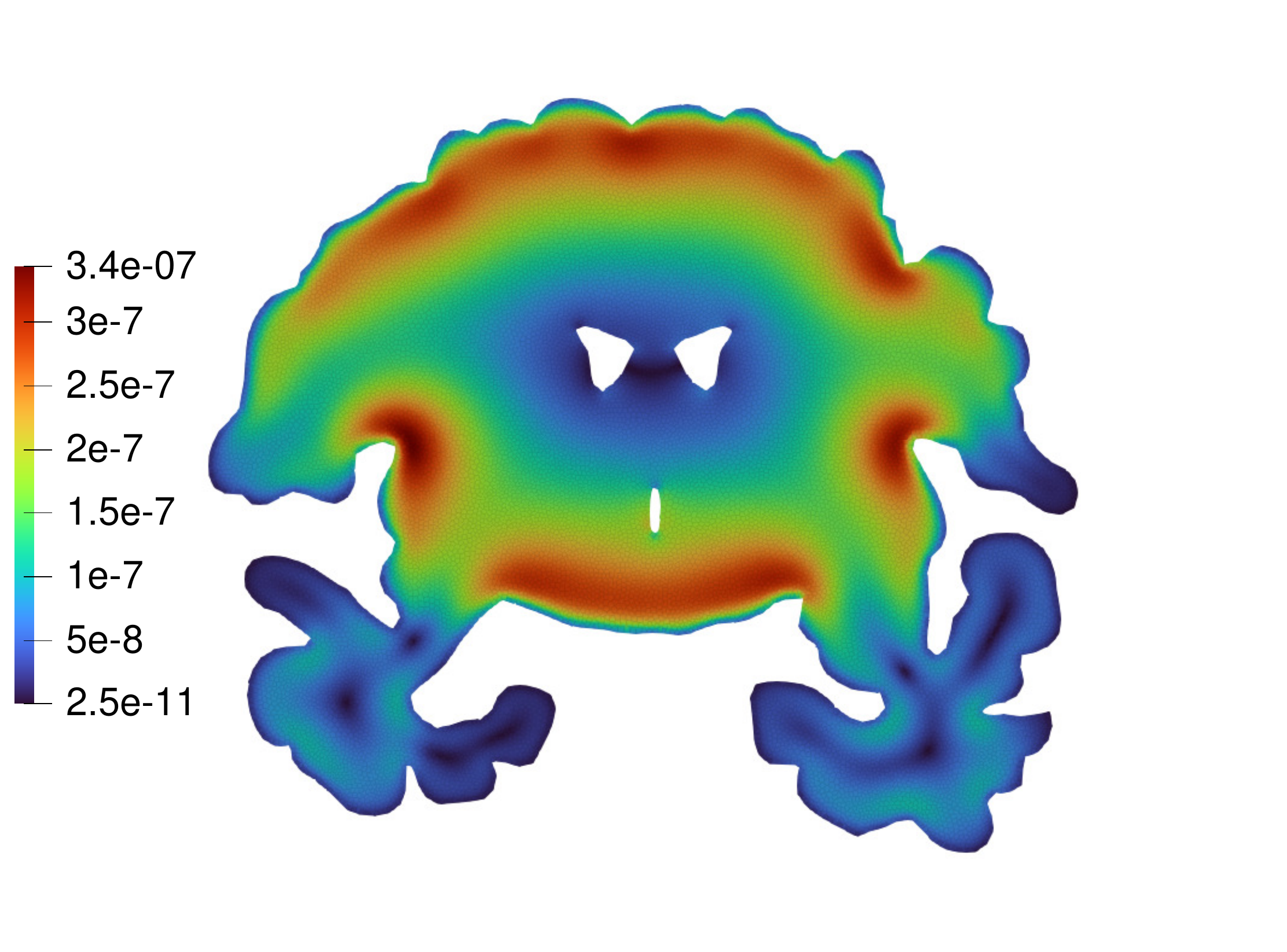}}
    \caption{Example 4. Snapshots of the variables of interest for the sleep-driven metabolite clearance within brain tissue simulation. The geometry of the human brain is discretised with 19,999 Voronoi cells and the polynomial degree for the \gls{vem} is set to $k=1$.}\label{fig:brainSols2D}
\end{figure}

We imposed a compression condition $\bsigma \bn = -(p_0 \pi/2)\bn$ on the brain cortex and clamped the brain tissue along the three ventricles. In addition, a ventricular pressure of $p_D = p_0$ is prescribed, while no filtration flux is imposed in the cortex. The concentration of the CSF tracer Gadobutrol is assumed to be uniformly distributed within the brain cortex, with a value of $\varphi_0 = \num[round-mode=figures, round-precision=3]{6.0471e-4} \unit{\g}/\unit{\mm^3}$. We adimensionalise the concentration using this quantity and set the boundary condition $\varphi_D = 1$; no CSF tracer flux is allowed through the ventricles. Finally, we assume that no external forces act in the simulation; that is, the right-hand sides $\ff$, $g$, and $\ell$ are set to zero.

The \gls{sad} coefficient is modified from \eqref{eq:D} and now takes the form
\begin{equation}\label{eq:DExample5}
    \varrho(\bsigma) =  \frac{\varrho_0}{\varphi_0\phi} \left(1+ \exp(-\eta_1[ \tr\bsigma]^2)\right),
\end{equation}
 where $\varrho_0 = \num[round-mode=figures, round-precision=2]{5.3e-2} \unit{\mm^3}/\unit{\s}$ and $\eta_1 = \num[round-mode=figures, round-precision=1]{2.e-1}$. The coupling parameters are set to $\alpha = 1$, $\beta = 0.35$ and Forchheimer effects are neglected for this test. Relevant parameters associated with the brain transport problem including mechanical properties of brain tissue are given next (see \cite{Budday2015-jx,vemrs2020,cadq2023}):
 \begin{gather*}
     E = \num[round-mode=figures, round-precision=3]{8.e2}  \unit{\g}/\unit{\mm \, \s^2}, \quad \nu = 0.495, \quad \lambda = \frac{E\nu}{(1+\nu)(1-2\nu)}, \quad \mu = \frac{E}{2(1+\nu)},\\
     \kappa_0 =  \num[round-mode=figures, round-precision=3]{1.4285e-5}  \unit{\g \, \mm}/\unit{\s^3}, \quad \bkappa = \kappa_0\bbI, \quad s_0 = \num[round-mode=figures, round-precision=3]{2.e-8}  \unit{\mm\, \s^2}/\unit{\g}.
 \end{gather*}

 We report snapshots of the approximated variables of interest $\varphi_h$, $\bu_h$, $\bsigma_h^{\bbPi}$, and $\bz_h^{\bPi}$ for the lowest order scheme (cf. Section~\ref{sec:discrete}) in Figure~\ref{fig:brainSols2D}. The sleep state is displayed in the left column, and the awake state in the middle column. The right column illustrates the difference between the two states. The simulations indicate that the CSF tracer concentration is approximately $13\%$ higher in the awake state, particularly in the region adjacent to the ventricles. The computational results reproduce the experimentally observed differences in CSF tracer concentration clearance between the awake and sleep states, indicating that stress-dependent diffusion can play a key role in this process. 

 {We remark that, although a uniform tracer concentration is prescribed on the inflow boundary (reflecting standard intrathecal infusion into the subarachnoid space), the resulting internal concentration distribution displayed in Figure \ref{fig:brainSols2D} exhibits  spatial heterogeneity. This pattern is physically consistent with \gls{sad}: spatially non-uniform stress fields   alter the local diffusion tensor   and displacement, directing transport along preferred paths of mechanical deformation. This could reproduce clinical observations from dynamic contrast-enhanced MRI (DCE-MRI), where uniform subarachnoid tracer supply leads to non-uniform parenchymal uptake dictated by local neuro-mechanical tissue properties.}

{The agreement reported above between the computed sleep/awake contrast in CSF tracer concentration and the experimental trends of \cite{evpmr2021,Xie2013-me} should be read as \emph{qualitative} validation of the proposed \gls{sad} mechanism, and not as a quantitative fit to subject-specific data. The mechanical parameters ($E$, $\nu$, $\bkappa$, $s_0$) are taken from independent ex vivo/in vivo characterisations of brain tissue \cite{Budday2015-jx,vemrs2020,cadq2023}, while the coefficients $\varrho_0$ and $\eta_1$ in \eqref{eq:DExample5} are chosen to reproduce the correct sign and order of magnitude of the sleep/awake contrast rather than calibrated against  specific tracer dynamics. None of these parameters were fitted to the tracer-clearance experiment itself, and the porosity values $\phi=0.14,0.23$ are the only inputs taken directly from the experimental reference \cite{Xie2013-me}. A more quantitative validation needs subject-specific geometry and poromechanical parameters (e.g., from MR elastography) together with time-resolved DCE-MRI concentration data, used to calibrate $\varrho_0$, $\eta_1$ and $\bkappa$ through an inverse problem. Our purpose is only to show that the new \gls{sad} model and its \gls{vem} discretisation are capable of reproducing, at least qualitatively, the physiologically observed differences in clearance behaviour between sleep and awake states.}

\subsection{Example 5. Nonlinear transport and accumulation of Amyloid-$\beta$ in brain tissue}\label{sec:brainExample2}

The brain consists of about $\num[round-mode=figures, round-precision=1]{8}$-$\num[round-mode=figures, round-precision=1]{10}$\% protein, i.e. around $\num[round-mode=figures, round-precision=1]{100} \unit{\g}$, of which the daily turnover is around $\num[round-mode=figures, round-precision=1]{3}$-$\num[round-mode=figures, round-precision=1]{4} \unit{\g}$~\cite{rasmussen2022fluid}. The resulting waste-protein fragments must then be re-cycled or removed away from the brain. In neurodegenerative diseases such as Alzheimer's (AD) and Parkinson's diseases the hallmark feature of the disease and progression is accumulation of waste proteins. For instance, while the normal amount of the waste-protein Amyloid-$\beta$ (A$\beta$) is a few {$\unit{\mg}$}, AD leads to two orders of magnitude higher amounts~\cite{bateman2006human}. A$\beta$ is known to accumulate in A$\beta$ plaque or fibrils~\cite{schmit2011drives} (see also Figure~\ref{fig:brainExampleExperiment2}) and will as such alter the mechanical properties of in particular the extra-cellular matrix. 
The exact mechanisms of these processes are not currently known. 
One theory {proposes that prion-like proteins} infect its neighbouring proteins in a similar manner as an infectious disease~\cite{weickenmeier2018multiphysics}. However, it is also well-known 
that  environmental factors such as pH and temperature
affect amyloid formation. More recently, focus is
put on fluid velocity, shear stress and membrane dynamics~\cite{iorio2025fluid, krausser2020physical, willis2025flow}. 
Specifically here we propose a change of permeability as A$\beta$ is accumulating, in response to stress. 

For this simulation, we recall the brain tissue properties and boundary conditions for the Biot block introduced in Example 4 (cf. Section~\ref{sec:brainExample1}). Furthermore, we explore two cases for the $A\beta$ concentration, which is baseline-measured at $\varphi_0 = \num[round-mode=figures, round-precision=3]{4.5e-8} \unit{\g}/\unit{\mm^3}$ (corresponding to approximately $\num[round-mode=figures, round-precision=1]{10} \unit{\micro\mol}$ of {pathogenic $A\beta$, see, e.g., \cite{Corti26}}). The first case assumes the concentration is uniformly distributed throughout the brain neocortex, the second case introduces a random distribution simulating that only 15\% of the region experiences accumulation. The altered diffusion equation (cf. \eqref{eq:diffWeak}) is then nondimensionalised using $\varphi_0$ as the reference concentration. No-flux boundary conditions are imposed in the ventricles and the remaining section of the cortex. Regarding the nonlinear terms, the \gls{sad} takes the form provided in \eqref{eq:DExample5} with $\rho_0=\num[round-mode=figures, round-precision=5]{3.170979e-7} \unit{\mm^3}/\unit{\s}$, $\eta_1=\num[round-mode=figures, round-precision=1]{1.e-2}$, and $\phi=0.2$; while the Forchheimer coefficient is set to $\eta = \num[round-mode=figures, round-precision=1]{1.e-13}/(\varphi_0 \phi)$.

 \begin{figure}[!t]
    \centering
    \includegraphics[width=.9\textwidth,trim={0.2cm 0.9cm 0.2cm 0.2cm},clip]{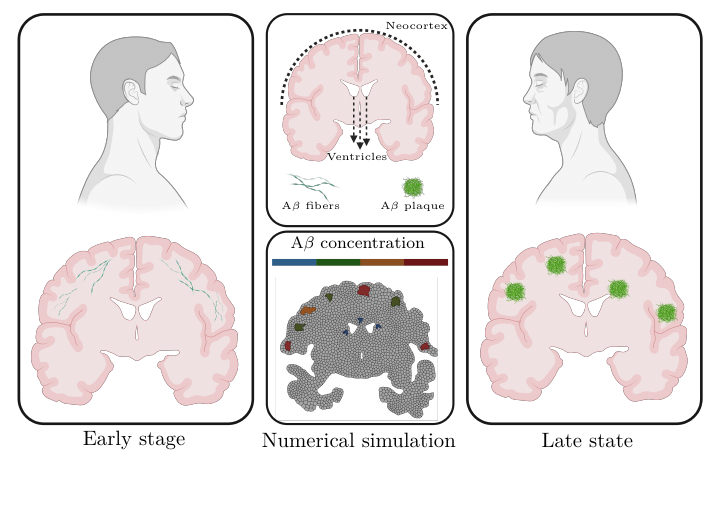}
    \caption{Example 5. Two-dimensional schematic illustration of nonlinear diffusion and accumulation of Amyloid-$\beta$ (A$\beta$) within brain tissue. The geometrical setup is shown in the top-middle panel. The left panel illustrates an early stage of Alzheimer's disease (AD), characterised by a small and localised accumulation of A$\beta$ fibres near the neocortex, whereas the right panel depicts the formation of accumulated A$\beta$ plaques. The bottom-middle panel illustrates the expected outcome of the computational simulation on a polytopal mesh of a coronal brain slice composed of 1,999 Voronoi cells.}\label{fig:brainExampleExperiment2}
\end{figure}

{Snapshots of the primary variables of interest ($\varphi_h$, $\bu_h$, $\bsigma_h^{\bbPi}$, and $\bz_h^{\bPi}$), obtained via the lowest-order \gls{vem} (cf. Section~\ref{sec:discrete}), are presented in Figures~\ref{fig:brainSols22D}--\ref{fig:brainSols22DRand} for both uniform and non-uniform seeding. Because Amyloid-$\beta$ ($A\beta$) is predominantly produced in metabolically active grey matter, its spatial deposition is   non-uniform, and initial plaque seeding in grey matter hotspots triggers a self-reinforcing cascade of impaired local clearance and accelerated accumulation. This feedback is captured by the \gls{sad} model through the reduction of diffusivity with accumulating stress in \eqref{eq:DExample5} (this is  motivated by histological evidence that amyloid fibrils and plaques stiffen the extracellular matrix and obstruct interstitial and perivascular pathways \cite{schmit2011drives,iorio2025fluid}). This corresponds to a tissue-mechanics alternative to macro-scale prion-like propagation models \cite{weickenmeier2018multiphysics}, and is qualitatively consistent with the  observed co-localisation of plaque burden and reduced interstitial/perivascular flow. The higher displacement magnitudes around these metabolically active regions further point to a direct  coupling between local plaque seeding, clearance degradation, and mechanical tissue deterioration. However we stress that \eqref{eq:DExample5} remains a phenomenological model: the precise dependence of $\varrho$ on stress and concentration is not yet constrained by direct measurement, and any alternative framework preserving the boundedness and Lipschitz properties imposed on $\varrho^{-1}$ (cf. \eqref{eq:D}) could be used instead, without changing the analysis of Sections~\ref{sec:continuous_wp}--\ref{sec:error}.

We have also examined the sensitivity of the results to  $\varrho_0$, $\eta_1$, and the Forchheimer coefficient $\eta$. Varying $\eta_1$ by an order of magnitude around the reported value preserves the qualitative feedback mechanism described above, but modulates the sharpness and spatial pattern of the resulting accumulation, with larger $\eta_1$ producing a more localised and higher-amplitude response; decreasing $\eta$ strengthens the stress-flow coupling and augments the local saturation of $\varphi_h$. In every case explored, the qualitative property of a self-reinforcing cascade of local clearance impairment and accumulation  with elevated tissue displacement, is preserved. The peak concentration and the rate  are indeed sensitive to these parameters. As in Example 4, these results are evidence that the model can reproduce the correct qualitative, feedback-driven phenomenology, but are not considered yet  a  validated prediction of $A\beta$ accumulation.}

\begin{figure}[!h]
    \centering
    \subfigure[$|\bsigma_h^{\bbPi}|$.] {\includegraphics[width=0.32\textwidth,trim={10.75cm 6.25cm 3.15cm 6.25cm},clip]{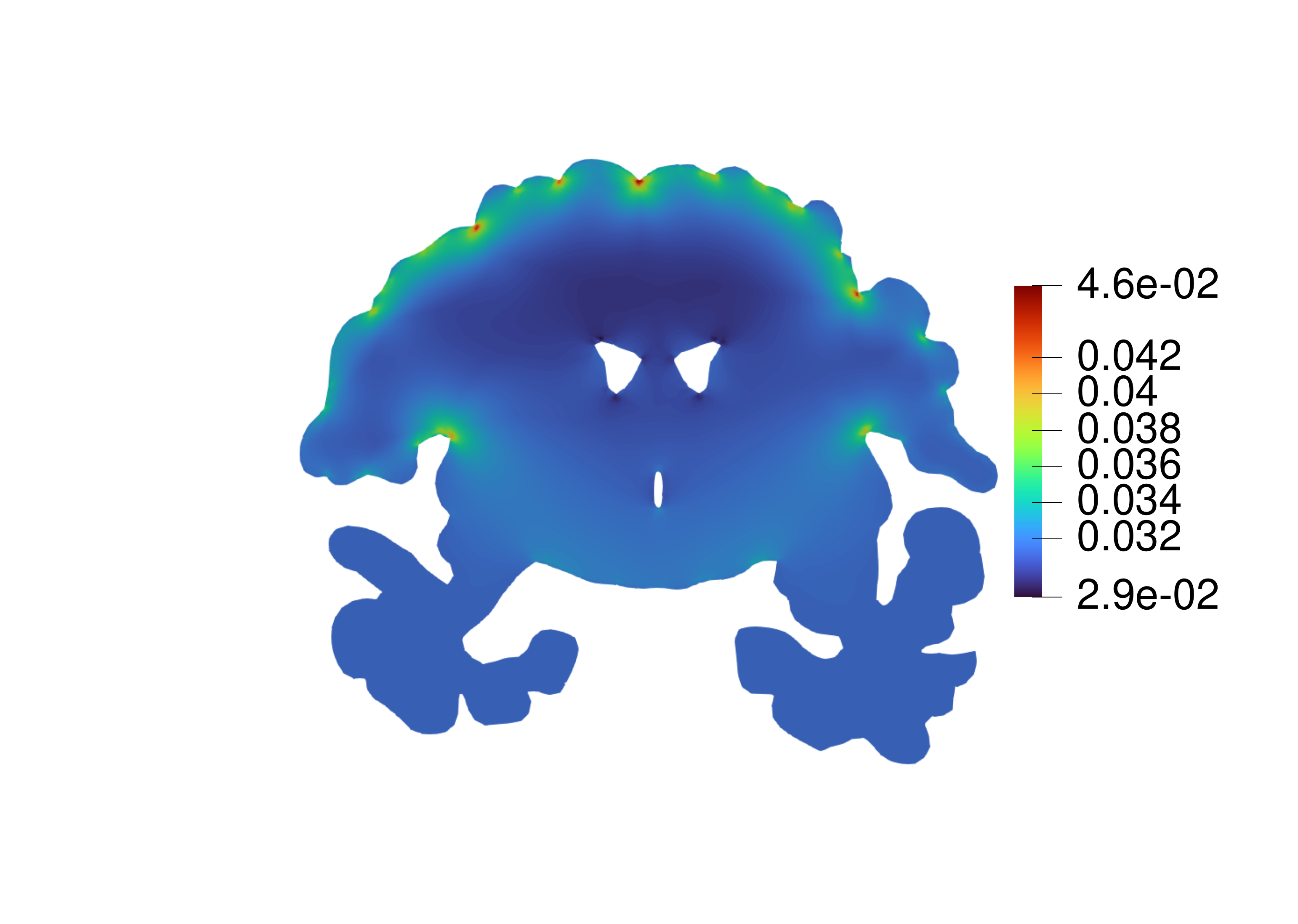}}  
    \subfigure[$|\bz_h^{\bPi}|$.]     {\includegraphics[width=0.32\textwidth,trim={10.75cm 6.25cm 3.15cm 6.25cm},clip]{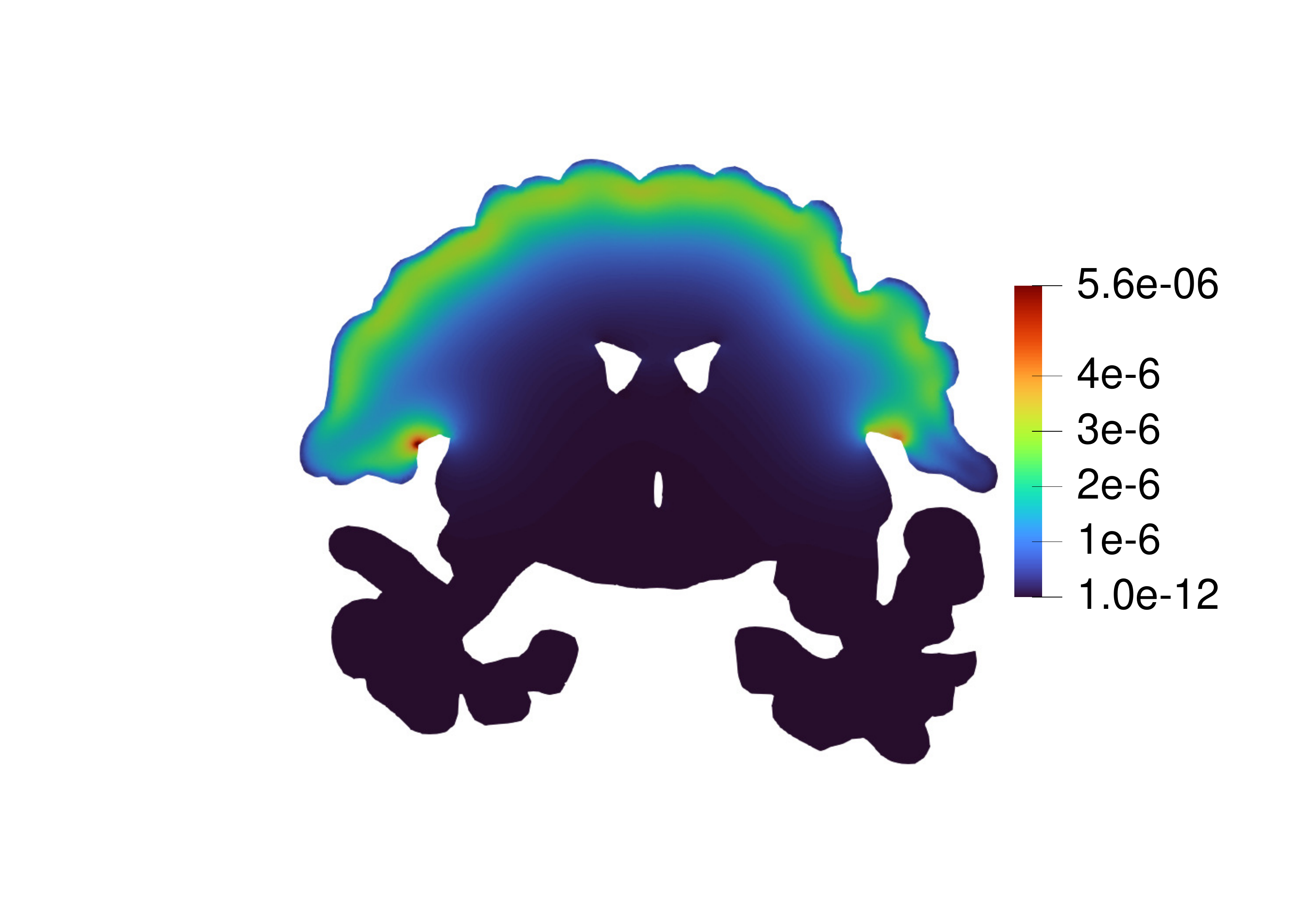}} 
    \subfigure[$|\bzeta_h^{\bPi}|$.]  {\includegraphics[width=0.32\textwidth,trim={10.75cm 6.25cm 2.75cm 6.25cm},clip]{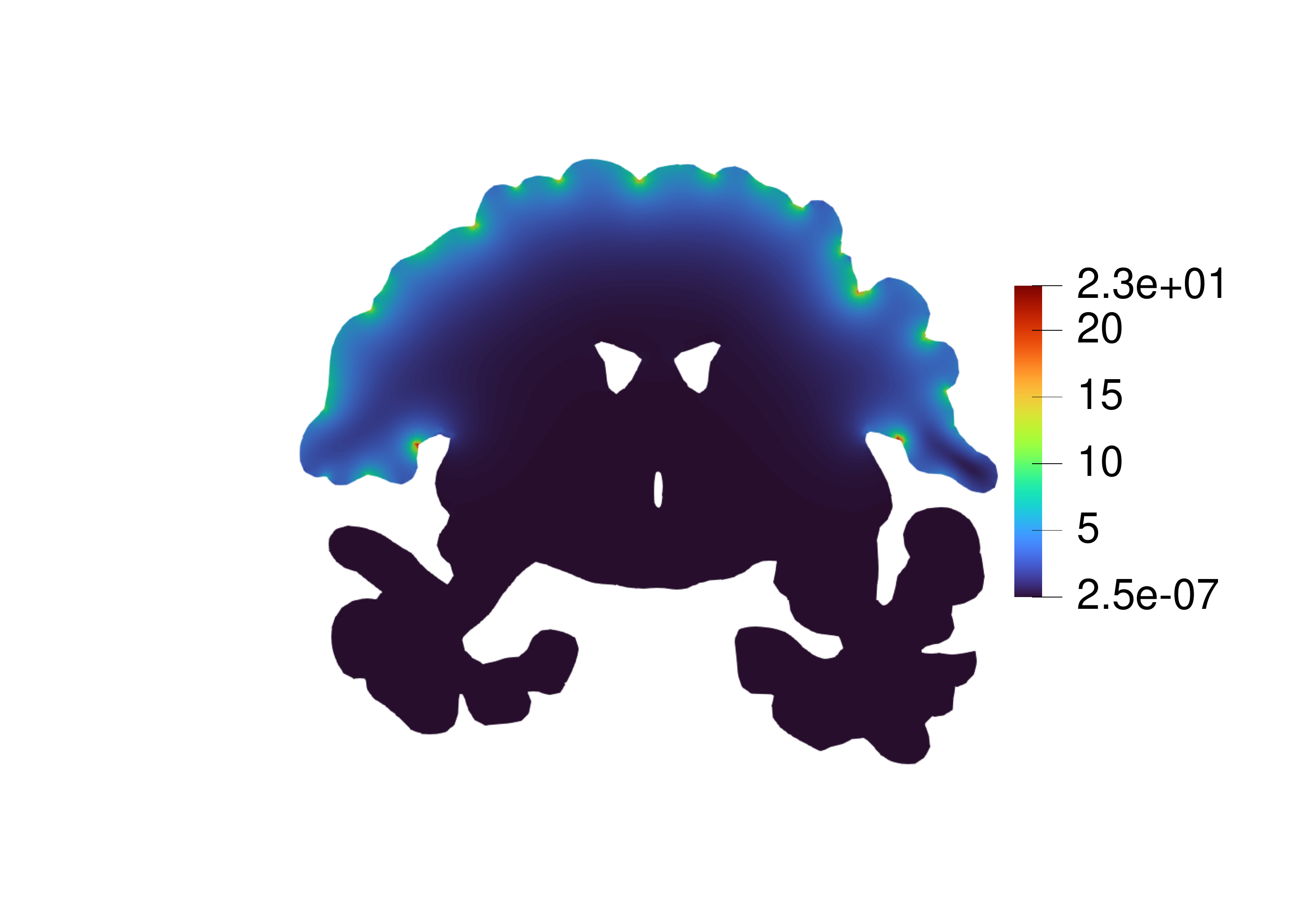}}
    \subfigure[$|\bu_h|$.]         {\includegraphics[width=0.32\textwidth,trim={10.75cm 6.25cm 2.75cm 6.25cm},clip]{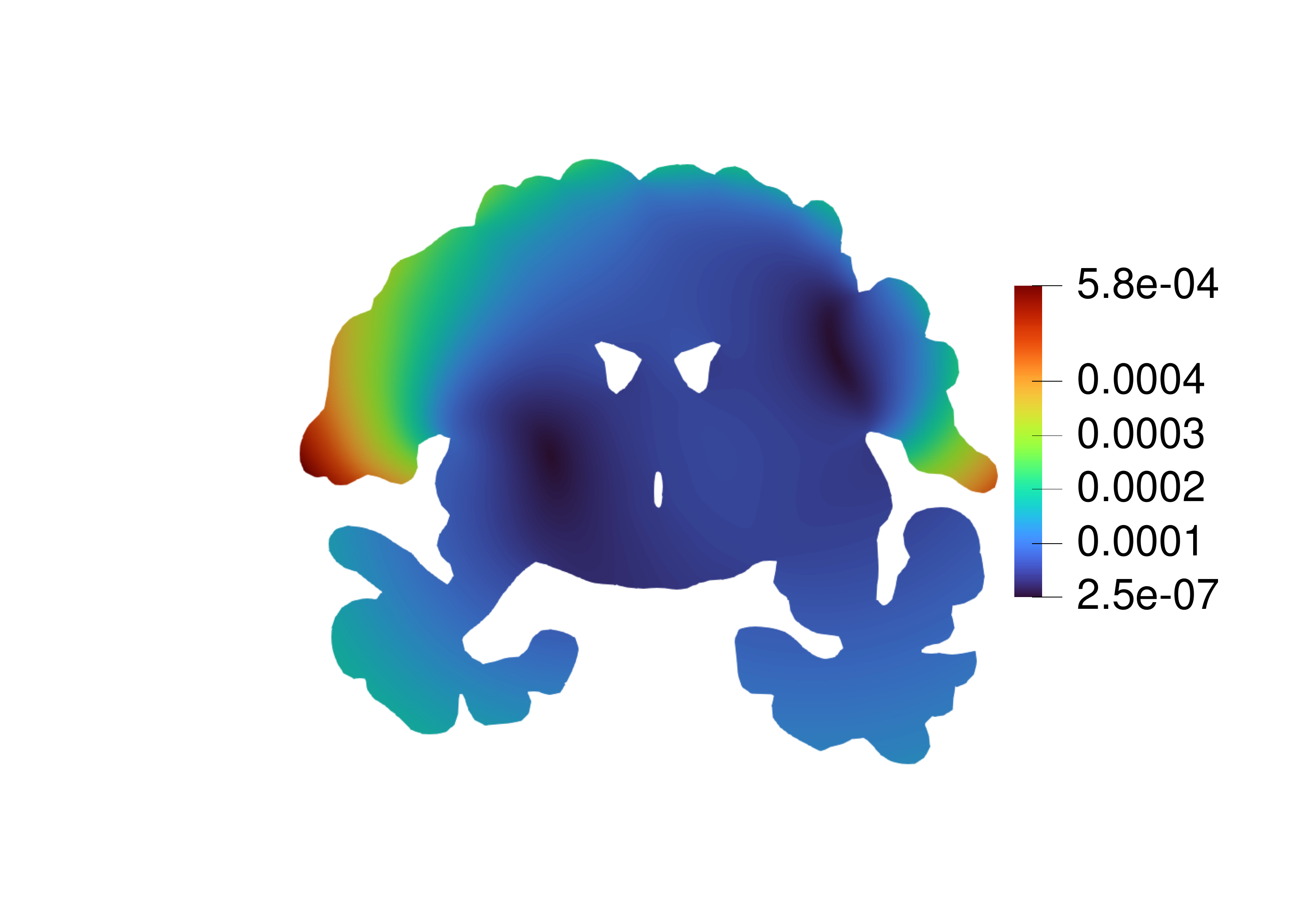}}
    \subfigure[$p_h$.]             {\includegraphics[width=0.32\textwidth,trim={10.75cm 6.25cm 2.75cm 6.25cm},clip]{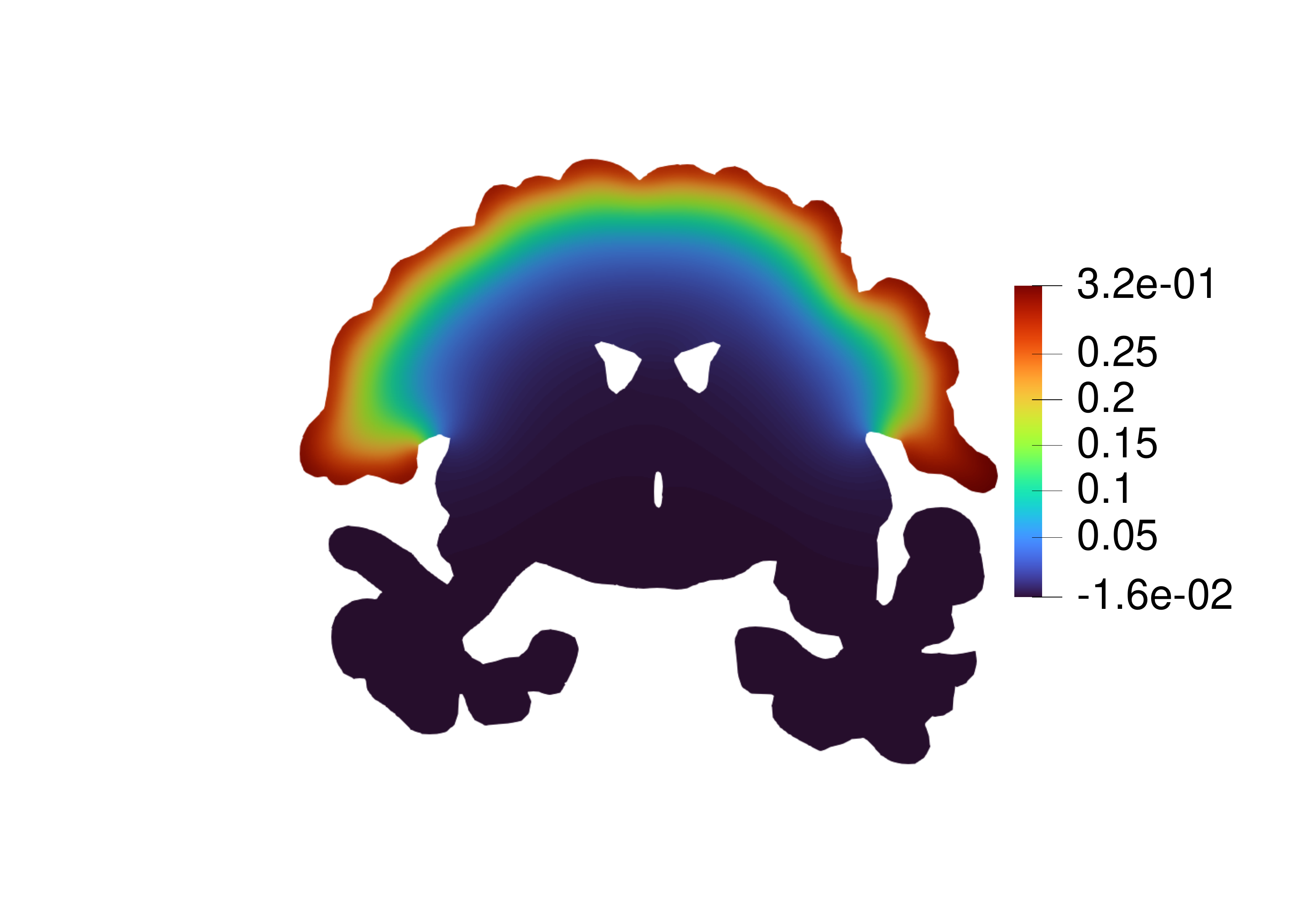}}
    \subfigure[$\varphi_h$.]       {\includegraphics[width=0.32\textwidth,trim={10.75cm 6.25cm 2.75cm 6.25cm},clip]{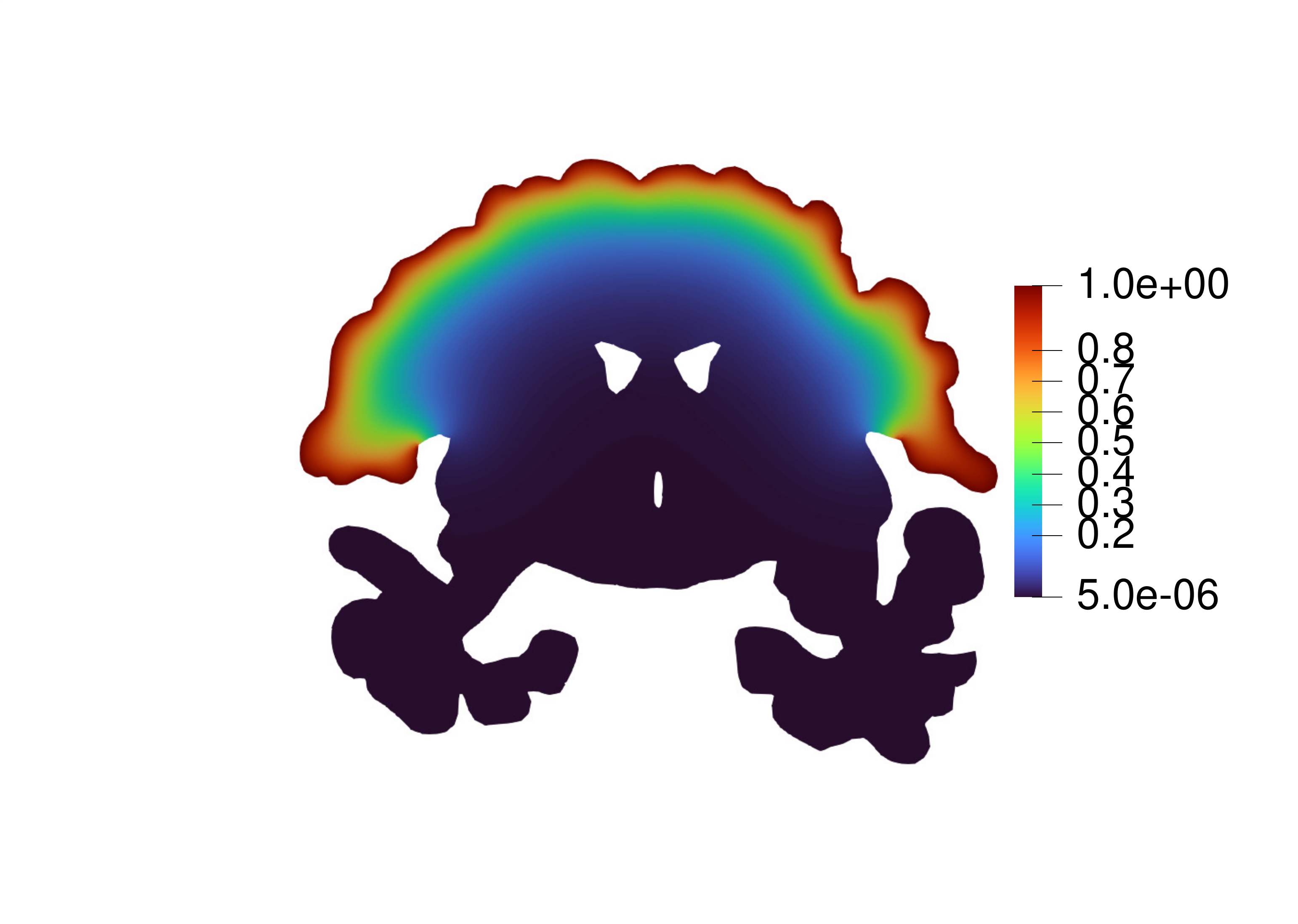}}
    \caption{Example 5. Snapshots of the variables of interest for the transport and accumulation of Amyloid-$\beta$ in brain tissue with uniformly distributed concentration in the neocortex. The geometry of the human brain is discretised with 19,999 Voronoi cells and the polynomial degree for \gls{vem} is set to $k=1$.}\label{fig:brainSols22D}
\end{figure}

\begin{figure}[!h]
    \centering
    \subfigure[$|\bsigma_h^{\bbPi}|$.] {\includegraphics[width=0.32\textwidth,trim={10.75cm 6.25cm 2.75cm 6.25cm},clip]{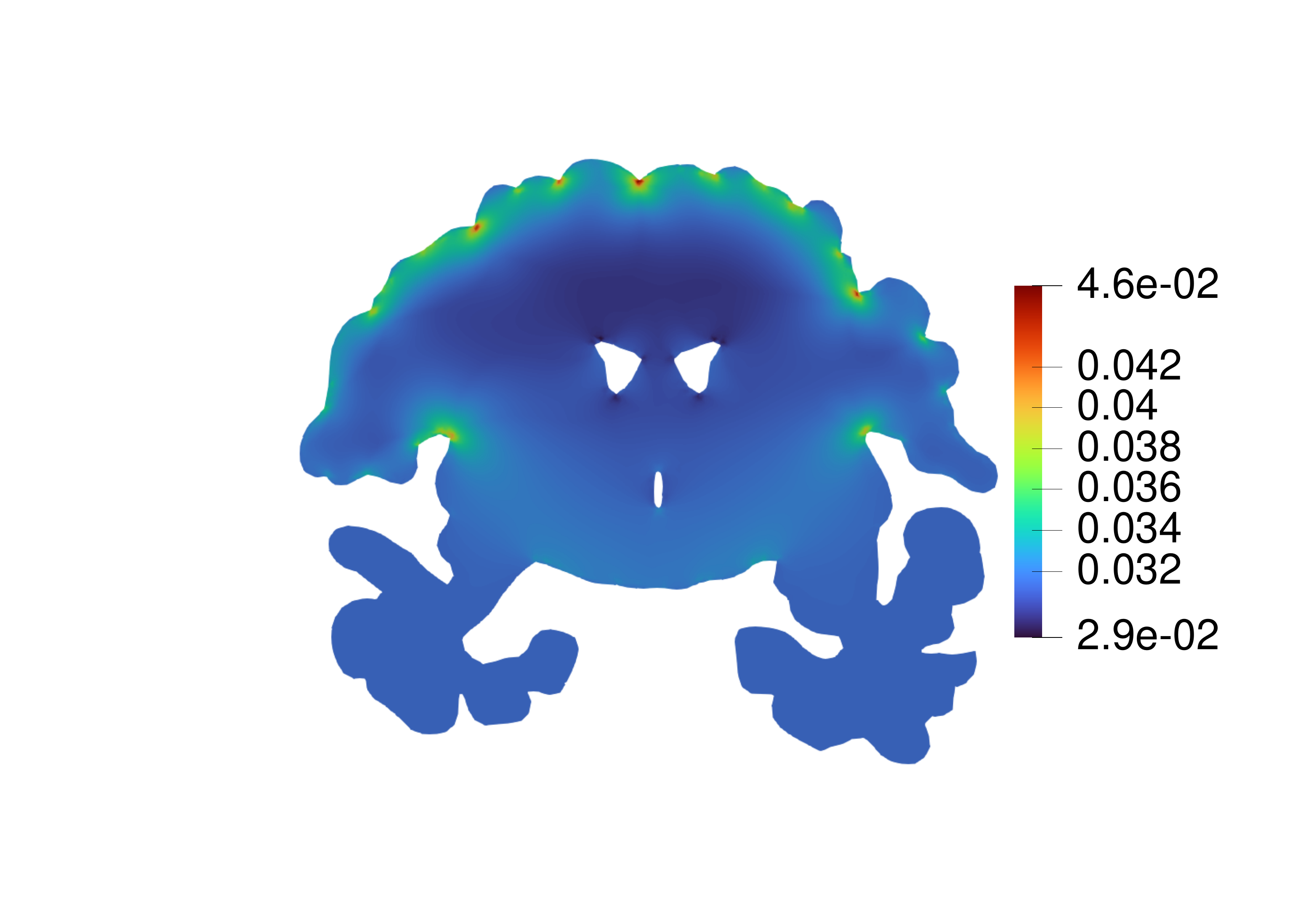}}  
    \subfigure[$|\bz_h^{\bPi}|$.]     {\includegraphics[width=0.32\textwidth,trim={10.75cm 6.25cm 2.75cm 6.25cm},clip]{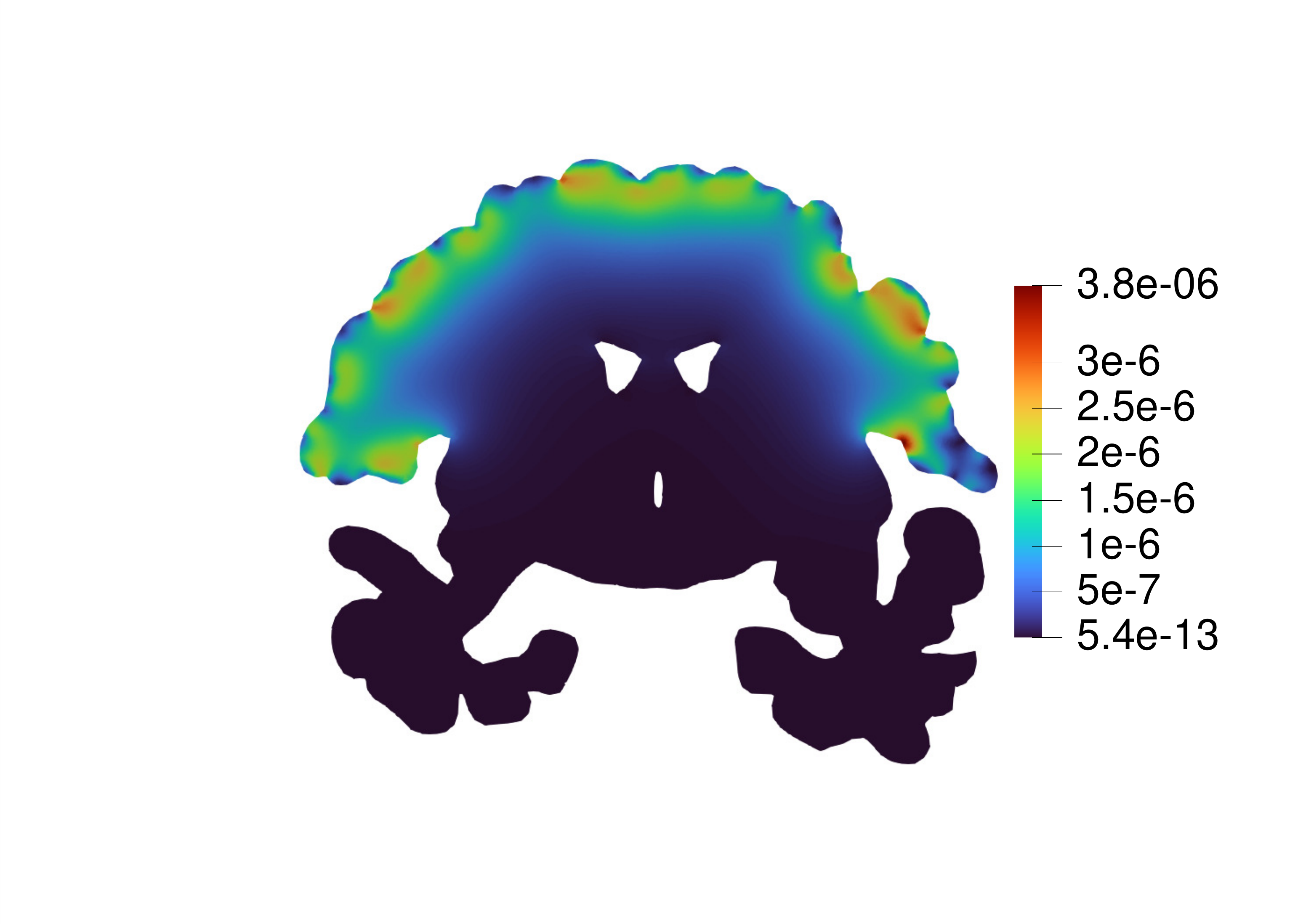}} 
    \subfigure[$|\bzeta_h^{\bPi}|$.]  {\includegraphics[width=0.32\textwidth,trim={10.75cm 6.25cm 2.75cm 6.25cm},clip]{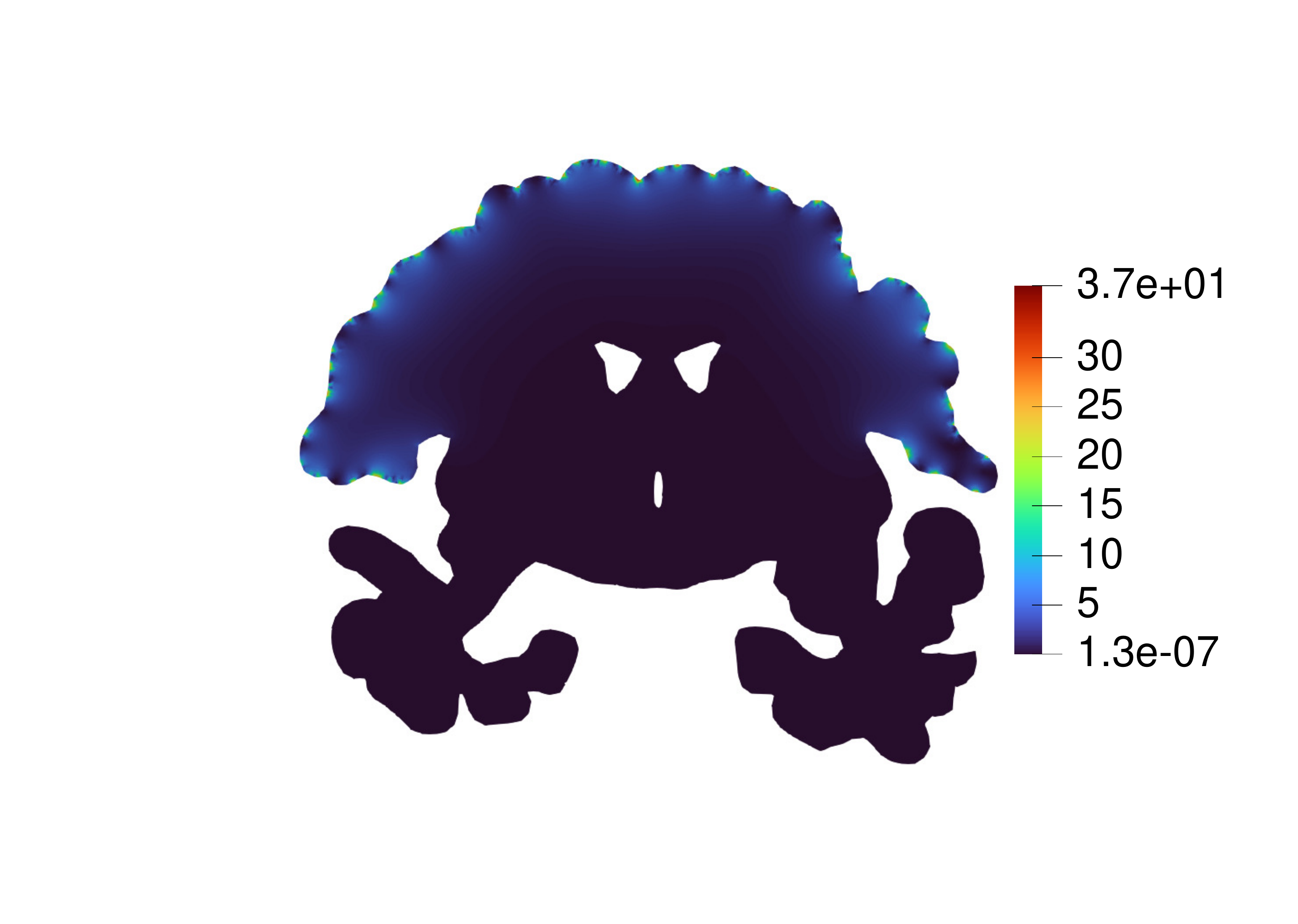}}
    \subfigure[$|\bu_h|$.]         {\includegraphics[width=0.32\textwidth,trim={10.75cm 6.25cm 2.75cm 6.25cm},clip]{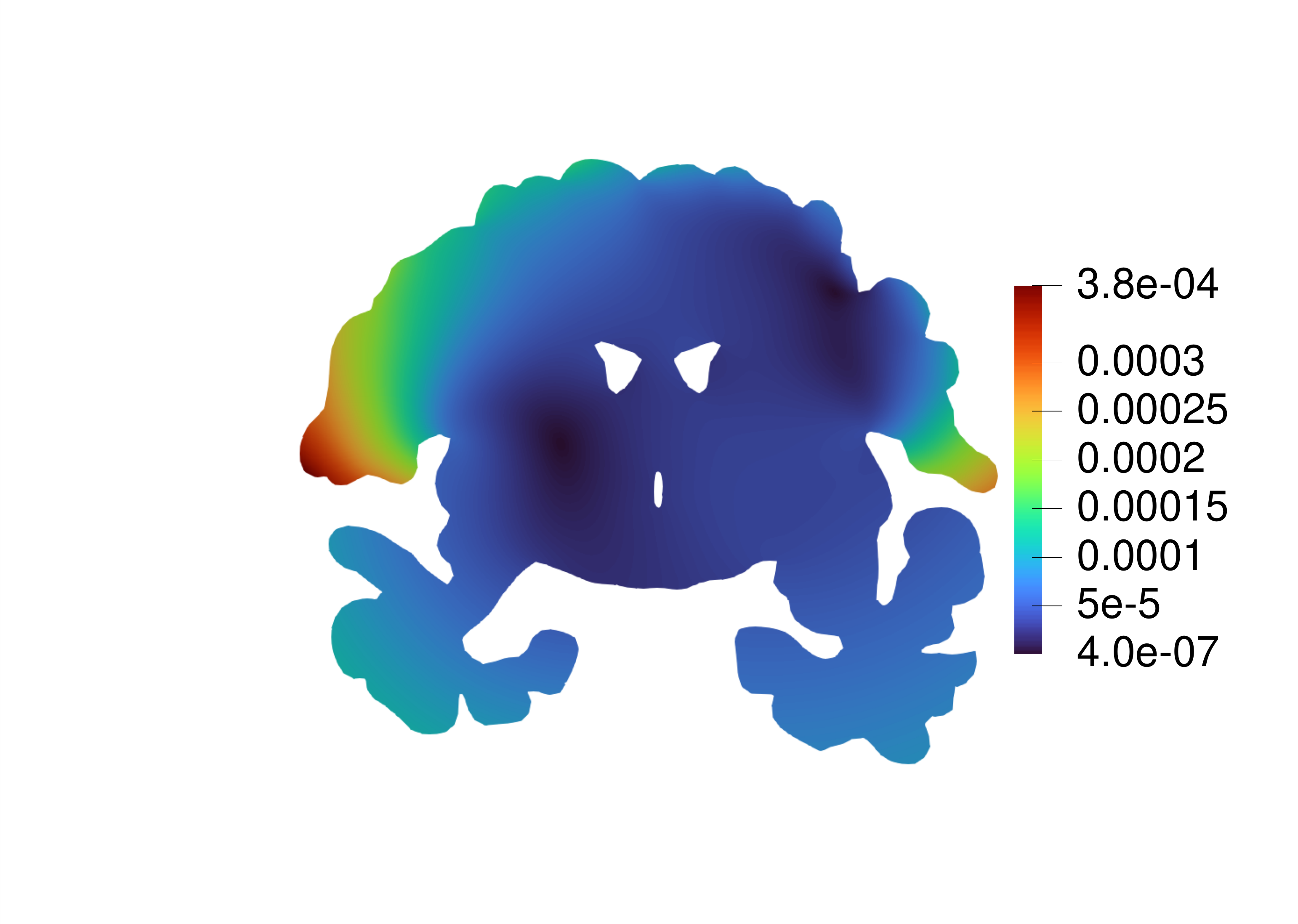}}
    \subfigure[$p_h$.]             {\includegraphics[width=0.32\textwidth,trim={10.75cm 6.25cm 2.75cm 6.25cm},clip]{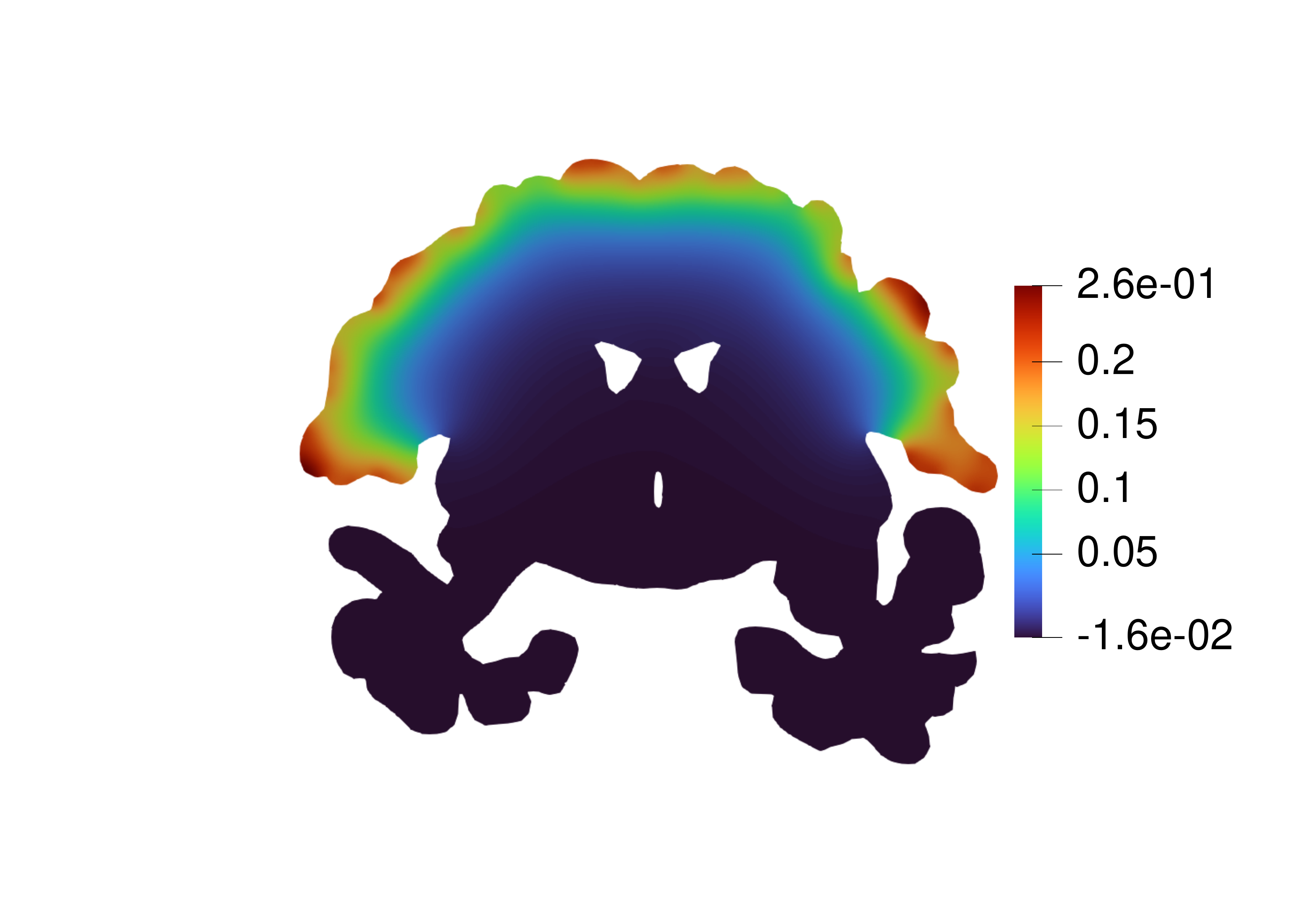}}
    \subfigure[$\varphi_h$.]       {\includegraphics[width=0.32\textwidth,trim={10.75cm 6.25cm 2.75cm 6.25cm},clip]{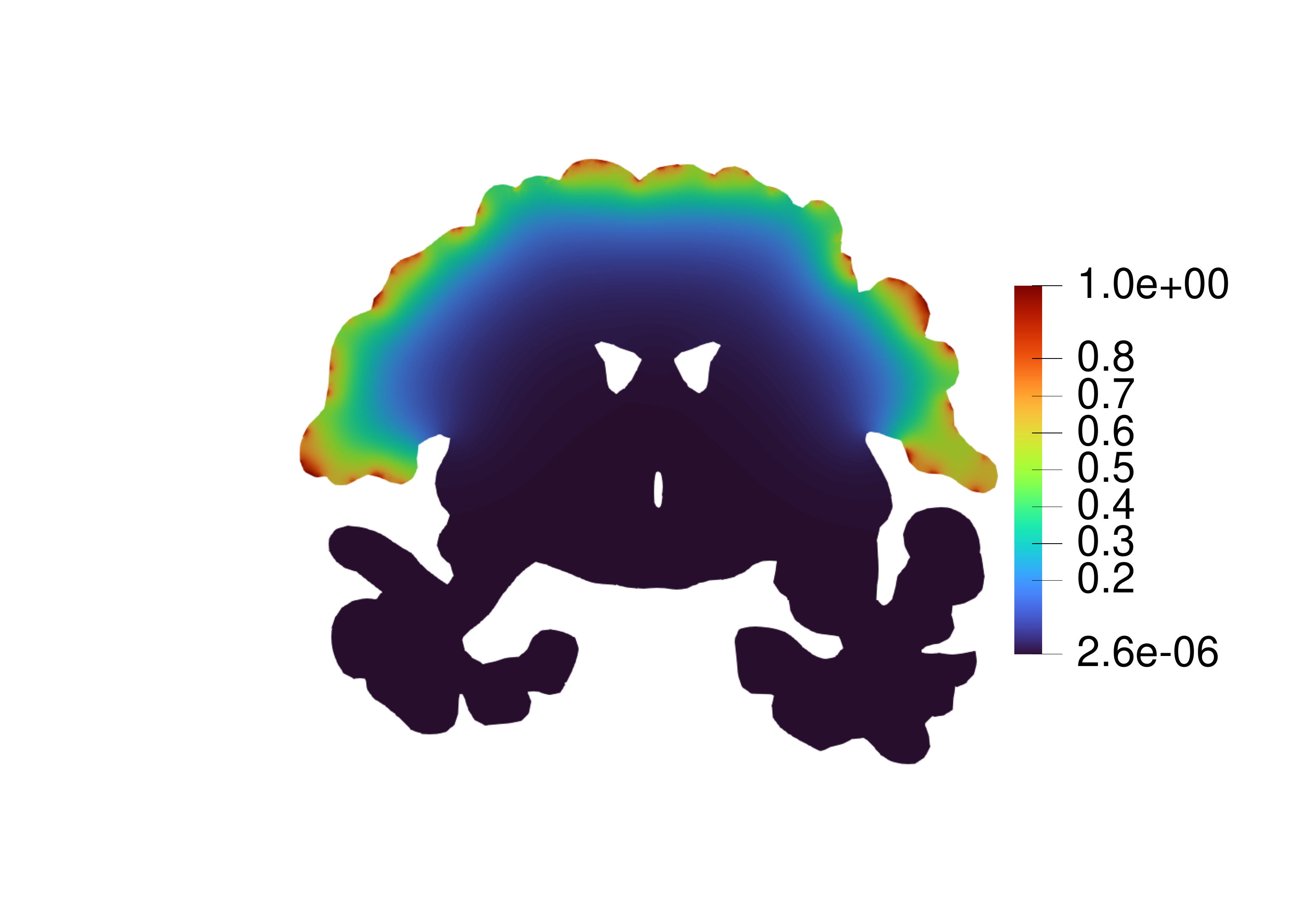}}
    \caption{Example 5. Snapshots of the variables of interest for the transport and accumulation of Amyloid-$\beta$ in brain tissue with randomly distributed concentration confined to 15\% of the neocortex. The geometry of the human brain is discretised with 19,999 Voronoi cells and the polynomial degree for \gls{vem} is set to $k=1$.}\label{fig:brainSols22DRand}
\end{figure}

\section{{Concluding remarks}}\label{sec:concl}
{In this paper we have advanced a model for fluid transport in the brain using an extended Forchheimer-type nonlinear \gls{sad}. We have constructed and analysed a \gls{vem} that is suitable for the coupled problem under consideration. 

A first limitation   concerns the time-dependence, which could be addressed using the general theory from \cite{showalter1997}, but the analysis requires a different framework (for instance, following the  Biot coupling arguments in \cite{caucao2024banach}). Another issue  is that the structure of the fixed-point argument requires the Forchheimer parameter $\eta$ to remain bounded away from zero. 
The case  $\eta=0$ corresponds to the pure \gls{sad} case, and then we could apply the approach from \cite{gomez23}. 
While we were not able to circumvent this strong assumption theoretically and develop a continuous limit of this analysis covering also the regime with $\eta=0$, we have confirmed numerically that this case is still stable.} 

{Several extensions of the present framework are of interest within the same mixed \gls{vem} setting, and we outline some of them.

First,   well-posedness of the coupled problem (Section~\ref{sec:wp}) relies on a smallness assumption on the data as a byproduct of the fixed-point approach. This restricts the theory to a regime that may not cover the most severe physiological regimes, even though the numerical scheme itself remains stable beyond it (cf. Section~\ref{sec:results}); relaxing this assumption, e.g., through a continuation or Leray--Schauder-type argument, is left for future work.
Second,  the present Biot--\gls{sad} system represents the fluid phase through a single effective compartment. Embedding it within a multiple-network poroelasticity  framework, distinguishing, e.g., vascular, CSF, and ISF compartments with their own pressures and exchange terms, would improve the model and would also allow a natural coupling with network-scale descriptions of cerebral hypoperfusion, such as the one recently proposed  in \cite{Corti26}. While  \cite{Corti26} focuses primarily on $A\beta$-driven alterations in vascular blood flow, our model centres on changes within the extracellular matrix. Because the precise biological pathways of $A\beta$ are still not fully understood, both modelling approaches necessarily rely on phenomenological representations of its mechanical impact; qualitatively, $A\beta$ accumulation is known to affect both vascular stiffness and extracellular space properties, including permeability and matrix compliance. A hypoperfusion-type source (e.g., entering through the right-hand sides $g,\ell$, or through a perfusion-dependent permeability $\bkappa$) could then act alongside the tissue-scale, stress-driven clearance-accumulation feedback studied in Example 5, 
bringing the two scales and their complementary mechanisms together within a single multiphysics model.
Third, the current boundary conditions on the cortex and ventricular walls are prescribed. Coupling \gls{sad} with an interfacial (Stokes--Biot or Stokes--Darcy) model of the surrounding CSF would let the exterior flow be resolved rather than imposed. A more involved analysis and discretisation are then required.
Finally, the 3D computations and brain examples  indicate that realistic, finer patient-specific geometries   require dedicated preconditioners exploiting the double saddle-point structure of the discrete system, rather than the direct solvers used here.}


\bigskip
\noindent\textbf{Acknowledgements.} 
We are thankful to Abner J. Salgado for pointing out an issue in the regularity of an auxiliary problem needed in a previous version of this manuscript. We also thank Michele Visinoni for his support in several aspects regarding the \gls{vem} implementation for the \gls{hr} subproblem, and  Miroslav Kuchta for providing the brain geometries used in the application examples. 

\noindent\textbf{Funding.} This work has been supported by the Australian Research Council through the \textit{Future Fellowship}  FT220100496; by Vicerrector\'{\i}a de Investigaci\'on project C3088, Sede de Occidente, Universidad de Costa Rica; 
by the European Research Council under grant 101141807 (aCleanBrain);  by the foundation Stiftelsen Kristian Gerhard Jebsen through its center for Brain Fluids and
and by  the Center of Advanced Study (CAS) at the Norwegian Academy of Science and Letters under the program \textit{Mathematical Challenges in Brain Mechanics}.  Computational resources were provided by Monash eResearch, in partnership with the Faculties of Science, Engineering and IT, through the Monash allocation scheme on the National Computational Infrastructure (NCI), Australia (2025).

\bibliographystyle{siam}
\bibliography{references}

\end{document}